\documentclass{amsbook}
\usepackage[all]{xy}
\usepackage{tikz,tikz-cd,amsbsy,mathrsfs,enumerate,mathtools,stmaryrd,imakeidx,lgreek,hyperref}

\newtheorem{lemma}{Lemma}[section]
\newtheorem{theorem}[lemma]{Theorem}
\newtheorem{conjecture}[lemma]{Conjecture}
\newtheorem{proposition}[lemma]{Proposition}
\newtheorem{corollary}[lemma]{Corollary}
\newtheorem{openproblem}[lemma]{Open problem}
\newcommand{\thistheoremname}{}
\newtheorem*{genericthm}{\thistheoremname}
\newenvironment{namedthm}[1]
  {\renewcommand{\thistheoremname}{#1}%
   \begin{genericthm}}
  {\end{genericthm}}
\theoremstyle{definition}
\newtheorem{definition}[lemma]{Definition}
\theoremstyle{remark}
\newtheorem{remark}[lemma]{Remark}
\newtheorem{example}[lemma]{Example}

\newcommand{\tr}[1]{\vphantom{#1}^t #1}
\newcommand{\CC}{\mathbb{C}}
\newcommand{\HH}{\mathbb{H}}
\newcommand{\NN}{\mathbb{N}}
\newcommand{\PP}{\mathbb{P}}
\newcommand{\QQ}{\mathbb{Q}}
\newcommand{\RR}{\mathbb{R}}
\newcommand{\ZZ}{\mathbb{Z}}
\newcommand{\TT}{\mathbb{T}}

\newcommand{\cA}{\mathcal{A}} 
\newcommand{\cC}{\mathcal{C}} 
\newcommand{\cD}{\mathcal{D}} 
\newcommand{\cF}{\mathcal{F}} 
\newcommand{\cL}{\mathcal{L}} 
\newcommand{\cM}{\mathcal{M}} 
\newcommand{\cP}{\mathcal{P}} 
\newcommand{\cQ}{\mathcal{Q}} 
\newcommand{\cR}{\mathcal{R}} 
\newcommand{\cS}{\mathcal{S}} 
\newcommand{\hA}{\widehat{A}} 
\newcommand{\hB}{\widehat{B}} 
\newcommand{\hF}{\widehat{F}} 
\newcommand{\hG}{\widehat{G}} 
\newcommand{\hX}{\widehat{X}} 
\newcommand{\hY}{\widehat{Y}} 
\newcommand{\hT}{\widehat{T}} 
\newcommand{\hC}{\widehat{\mathcal{C}}} 
\newcommand{\hR}{\widehat{\mathcal{R}}} 
\newcommand{\hchi}{\widehat{\chi}}
\newcommand{\hmu}{\widehat{\mu}}
\newcommand{\hnu}{\widehat{\nu}}
\newcommand{\hDelta}{\widehat{\Delta}}
\newcommand{\rAR}{\mathrm{AR}} 
\newcommand{\rB}{\mathrm{B}} 
\newcommand{\rF}{\mathrm{F}} 
\newcommand{\rG}{\mathrm{G}} 
\newcommand{\rM}{\mathrm{M}} 
\newcommand{\rU}{\mathrm{U}} 

\newcommand{\ba}{\mathbf{a}}
\newcommand{\bb}{\mathbf{b}}
\newcommand{\be}{\mathbf{e}}
\newcommand{\bp}{\mathbf{p}}
\newcommand{\bt}{\mathbf{t}}
\newcommand{\bu}{\mathbf{u}}
\newcommand{\bv}{\mathbf{v}}
\newcommand{\bw}{\mathbf{w}}
\newcommand{\bx}{\mathbf{x}}
\newcommand{\by}{\mathbf{y}}
\newcommand{\bz}{\mathbf{z}}
\newcommand{\bM}{\mathbf{M}}
\newcommand{\bN}{\mathbf{N}}
\newcommand{\bl}{\boldsymbol{\ell}}
\newcommand{\bone}{\mathbf{1}}
\newcommand{\bsigma}{\boldsymbol{\sigma}}
\newcommand{\bphi}{\boldsymbol{\varphi}}
\newcommand{\bpsi}{\boldsymbol{\psi}}
\newcommand{\btau}{\boldsymbol{\tau}}
\newcommand{\fr}{\mathfrak{r}}
\newcommand{\GL}{\operatorname{GL}}
\newcommand{\SL}{\operatorname{SL}}
\newcommand{\rint}{\operatorname{int}}
\newcommand{\balpha}{\boldsymbol{\alpha}}
\newcommand{\gen}{{\mathrm{gen}}}
\newcommand{\newlambda}{\mbox{\begin{greek}l\end{greek}}{}}
\makeatletter
\newcommand{\labitem}[2]{\def\@itemlabel{#1}\item\def\@currentlabel{#1}\label{#2}}
\makeatother

\numberwithin{equation}{section}

\counterwithout{section}{chapter}
\counterwithout{figure}{chapter}
\counterwithout{footnote}{chapter}

\newcommand{\indx}[1]{\index[index]{#1}} 
\newcommand{\notx}[3]{\index[notation]{#1@\makebox[5em][l]{#2}#3}} 
\makeindex[name=index, title=Index, columns=3]
\makeindex[name=notation, title=Notation, columns=2]

\title[Nonstationary Markov Partitions]{Nonstationary Markov Partitions and  Multidimensional Continued Fraction Algorithms}

\author[P.~Arnoux]{Pierre~Arnoux}
\address{(P.A.) Institut de Math\'ematiques de Marseille, Aix-Marseille Universit\'e, Campus de Luminy,  F-13288 Marseille, France}
\email{pierre@pierrearnoux.fr}
\author[V.~Berth\'e]{Val\'erie Berth\'e}
\address{(V.B. \& W.S.) Université Paris Cit\'e, CNRS, IRIF, F-75013 Paris, France}
\email{berthe@irif.fr, steiner@irif.fr}
\author[M.~Minervino]{Milton Minervino}
\address{(M.M.) 12 Rue de L\"ubeck, F-75116 Paris}
\email{milminervino@gmail.com}
\author[W.~Steiner]{Wolfgang Steiner}
\author[J. M. Thuswaldner]{J\"org M. Thuswaldner}
\address{(J.M.T.) Chair of Mathematics, Statistics and Geometry, Montan\-universit\"at Leoben, A-8700 Leoben, AUSTRIA}
\email{joerg.thuswaldner@unileoben.ac.at}

\date{\today}

\begin{document}

\begin{abstract}
It is well known from results of Sina\u{\i} and Bowen that a hyperbolic toral automorphism admits a Markov partition. Our aim is to generalize this concept to the nonstationary case, i.e., we associate Markov partitions to nonstationary sequences of toral automorphisms. Special emphasis is placed on sequences of toral automorphisms produced by strongly convergent multidimensional continued fraction algorithms. The convergence of the algorithms is expressed in terms of a Pisot type condition which yields hyperbolicity for the nonstationary dynamics with a splitting into two subspaces of dimension~1 and codimension~1, respectively. 

For a multidimensional  continued fraction map, we first consider its natural extension, whose orbits are given by bi-infinite sequences of matrices with determinant~$\pm 1$. The Pisot type condition allows us to interpret an orbit of this natural extension as an Anosov mapping family, i.e., as a bi-infinite sequence of toral automorphisms with well-defined stable and unstable manifolds. We prove that this Anosov mapping family admits a bi-infinite sequence of explicit nonstationary Markov partitions. To obtain the atoms of the Markov partitions, a combinatorial structure, expressed in terms of symbolic dynamical systems, namely substitutive and $\mathcal{S}$-adic shifts, has to be superimposed on the Anosov mapping family. In particular, the atoms of the Markov partitions are geometric realizations of $\mathcal{S}$-adic shifts, defined by suspensions of so-called $\mathcal{S}$-adic Rauzy fractals. 

These Markov partitions then provide a symbolic model as a nonstationary edge shift for the Anosov mapping family. Restacking of the Markov partition yields a renormalization process that allows us to interpret a multidimensional continued fraction algorithm as a sequence of iteratively induced toral rotations. 

We develop our theory locally in order to make it applicable to single orbits with a local Pisot property as well as generically in order to apply it to each generic orbit of a multidimensional continued fraction algorithm with a global Pisot property.

As a guiding example for our theory we study Anosov mapping families on 2- and 3-dimensional tori associated to various versions of the Brun continued fraction algorithm. As a realization result, these mapping families allow us to attach an explicit nonstationary Markov partition to almost all codimension~1 splittings of $\RR^3$ and $\RR^4$. 
\end{abstract}

\maketitle

\thanks{\vspace{20ex}
\noindent
This work was supported by the ERC grant DynAMiCs (101167561) of the European Research Council, by the bilateral grant SYMDYNAR (ANR-23-CE40-0024 and FWF~I~6750) of the Agence Nationale de la Recherche and the Austrian Science Fund, and by the ANR project IZES (ANR-22-CE40-0011).

\bigskip
\noindent
The authors are indebted to Niels Langeveld, Maria Clara Werneck, and Reem Yassawi for their careful reading of the manuscript.}

\setcounter{tocdepth}{3}
\tableofcontents

\chapter{Introduction}\label{sec:intro}
The classical theory of \emph{dynamical systems}\indx{dynamical system} deals with a space~$X$ together with a single transformation $f: X\to X$ acting on that space. The iterates of this transformation~$f$ on a given point give rise to \emph{orbits}, whose dynamical properties can be studied. It is also meaningful to study a ``nonstationary'' analog of this situation. Namely, in this monograph we consider a sequence of spaces $(X_n)_{n\in\ZZ}$ equipped with a sequence of transformations $(f_{n})_{n\in\ZZ}$ that map along this sequence of spaces according to the diagram  
\begin{equation}\tag{$*$}\label{eq:star}
\cdots\xrightarrow{f_{-2}} X_{-1}\xrightarrow{f_{-1}} X_{0}\xrightarrow{f_{0}} X_{1}\xrightarrow{f_{1}} \cdots.
\end{equation}
At first glance, because of the lack of recurrence, the dynamics of such a \emph{mapping family} is trivial. Nevertheless, for the case of so-called \emph{Anosov mapping families} one can find the whole richness of a classical recurrent dynamical system defined by a single map $f\colon X\to X$. An Anosov mapping family is given by a diagram of the form~\eqref{eq:star}, for which the transformations~$f_n$, $n\in \ZZ$, are diffeomorphisms along a sequence $(X_n)_{n\in \ZZ}$ of compact Riemannian manifolds equipped with an invariant sequence of splittings of the tangent bundle into expanding and contracting subspaces. Anosov mapping families have been introduced in \cite{AF:05} as a nonstationary generalization of Anosov diffeomorphisms and have been further studied for instance in \cite{Fisher:09,Stenlund:11,KL:16,MJ:19,CRV:19}. 
We focus here on the case of \emph{linear Anosov mapping families}, i.e., on Anosov mapping families whose transformations $f_n$, $n\in \ZZ$, are linear mappings. For a given linear Anosov mapping family we will construct a nonstationary version of a Markov partition, which provides a symbolic model of this mapping family in the form of a nonstationary symbolic shift.  

In nonstationary dynamics the different maps~$f_n$, $n\in\ZZ$, are often chosen according to some probability process, or by a dynamical system equipped with a certain skew product construction, such as developed for instance in the setting of random matrices or of differentiable dynamical systems; see e.g.\ \cite{Kifer,Arnold98,KiferLiu,Viana:book}.

Our motivation as well as our guiding examples come from the theory of \emph{unimodular multidimensional continued fraction algorithms}. For a given vector of real numbers, a multidimensional continued fraction algorithm provides increasingly good rational Diophantine approximations in form of vectors whose rational entries all have the same denominator. Classical continued fraction algorithms provide extremely good (and even the best) rational Diophantine approximations; see \cite{Cassels,Lagarias:93,Schweiger:00,WINE} and Chapter~\ref{sec:cf} for details. 

To be more precise, let $d\ge 2$ be fixed. A multidimensional continued fraction algorithm dynamically generates an infinite sequence $(M_n)_{n\in\ZZ}$ of $d {\times} d$ matrices with integer entries and determinant~$\pm 1$ in a way that each column of the product $M_0\cdots M_{n-1}$ contains the numerators and the denominator of an $n$-th rational Diophantine approximation ($n\in \NN$).\footnote{As we will see later, the  negative part of the sequence $(M_n)_{n\in\ZZ}$ is related to the natural extension of the algorithm.} The matrices $M_n$, $n\in\ZZ$, can  be regarded as linear automorphisms that map between elements of a sequence $(\TT_n)_{n\in\ZZ}$\notx{Td}{$\TT^d,\TT_n$}{$d$-dimensional torus} of \mbox{$d$-dimensional} tori, according to the diagram\footnote{It will become apparent later why we take the inverse of the matrices here.} 
\begin{equation}
\tag{$**$}
\label{eq:star2}
\cdots\xrightarrow{M_{-2}^{-1}} \TT_{-1}\xrightarrow{M_{-1}^{-1}} \TT_{0}\xrightarrow{M_{0}^{-1}} \TT_{1}\xrightarrow{M_{1}^{-1}} \cdots.
\end{equation}
Therefore, a multidimensional continued fraction algorithm defines mapping families. The Anosov property of these mapping families is related to convergence properties of the algorithm, which also 
play an important role in the theory of Diophantine approximation. Well known notions of convergence depend on whether the angles between the (column) vectors of the matrices $M_0\cdots M_{n-1}$ tend to~$0$ (\emph{weak convergence}), or whether the distances between the (column) vectors tend to~$0$ (\emph{strong convergence}); see Definition~\ref{def:wsc}.  A~gauge for the ``strongness'' of convergence of a multidimensional continued fraction algorithm is provided by its (first and second) Lyapunov exponents; see \cite{Lagarias:93}. In fact, Lyapunov exponents describe the asymptotic behavior of the singular values of large products of matrices. If the second Lyapunov exponent of an algorithm is (generically) negative, we say that it satisfies the \emph{Pisot condition}.
This condition, which plays a crucial role in our investigations, yields hyperbolicity for the nonstationary dynamics of the algorithm, and, more precisely, guarantees that the mapping families in \eqref{eq:star2} are (generically) ``codimension one'' linear Anosov mapping families in the sense that in the associated splittings of tangent spaces, the contracting spaces have dimension~$1$. 

As above, let  $(M_n)_{n\in\ZZ}$ be a sequence of $d{\times}d$-matrices of determinant~$\pm 1$ generated by the natural extension of a multidimensional continued fraction algorithm with a strong convergence property.  We call such a sequence an \emph{orbit} of the algorithm.
If we assume that $(M_n)_{n\in\ZZ}$ is periodic with period~$p$, we can block this sequence by forming products of blocks of consecutive matrices of length $p$. By this blocking, $(M_n)_{n\in\ZZ}$ becomes the constant sequence $(M)_{n\in \ZZ}$, where $M=M_0\cdots M_{p-1}$ is a hyperbolic integer matrix of determinant~$\pm 1$. In this case, the nonstationary diagram \eqref{eq:star2} defined by $(M_n)_{n\in\ZZ}$ becomes stationary, i.e., it degenerates to the hyperbolic toral automorphism $M^{-1}: \TT^d \to \TT^d$ on the $d$-dimensional torus~$\TT^d$. Since it is well known that each hyperbolic toral automorphism admits a \emph{Markov partition} \cite{Sinai:68,Sinai:68bis,Bow:70,Bedford:86}, one can  associate a Markov partition to a periodic orbit of a multidimensional continued fraction algorithm. As an illustration, we refer to \cite{IO:93}, where this was done for periodic orbits of the Brun continued fraction algorithm. 

In the present work, we generalize this and associate nonstationary Markov partitions with arbitrary orbits of multidimensional continued fraction algorithms with a strong convergence property (namely, the Pisot condition). As mentioned before, by strong convergence, a (generic) orbit of the algorithm $(M_n)_{n\in \ZZ}$ gives rise to an Anosov mapping family \eqref{eq:star2}. In particular, it implies the existence of a well-defined invariant sequence of splittings of the tangent bundles of the tori in stable and unstable manifolds with a uniform upper bound for the contraction and a uniform lower bound for the expansion. As in the stationary case, nonstationary Markov partitions then allow us to build symbolic models for linear Anosov mapping families and, hence, for orbits of multidimensional continued fraction algorithms, as nonstationary edge shifts. Nonstationary edge shifts can be considered as nonstationary subshifts of finite type; see e.g.\ \cite{AF:05,Fisher:09}. They are also related to the (two-sided)  constructions of Markov compacta and Bratteli diagrams introduced in \cite{Bufetov13,Bufetov:14}. 

For the case $d=2$, this has been done already in \cite{AF:05}; see also \cite{AF:01}. In these papers a linear Anosov mapping family on $2$-dimensional tori is associated to each orbit of the classical continued fraction algorithm, and nonstationary Markov partitions with polygonal atoms are constructed for the associated sequences of $2{\times}2$ matrices. The situation becomes much more delicate for $d\ge 3$. Indeed, as in the case of a single toral hyperbolic automorphism, the atoms of nonstationary Markov partitions are no longer polygonal in dimension $d\ge 3$. Our aim is to prove that these atoms are provided by so-called \emph{$\cS$-adic Rauzy fractals} that were first defined in \cite{BST:19}; see also \cite{BST:23}. 

The sections of this introduction are organized as follows. Section~\ref{sec:defMarkov}  briefly introduces Markov partitions for hyperbolic toral automorphisms. In Section~\ref{sec:introsturm}, we review the rich dynamical interpretations of the classical continued fraction algorithm, Section~\ref{sec:introcf}  then discusses how to generalize this setting to  higher dimensions.
Section~\ref{sec:introresults} introduces the main results of this monograph, and Section~\ref{sec:introoutline} lays out its overall organization.

\section{Markov partitions for hyperbolic toral automorphisms}\label{sec:defMarkov}
One of our main aims is to associate a nonstationary Markov partition to a sequence of toral automorphisms. This generalizes the concept of a (stationary) Markov partition for a single toral automorphism, which we discuss briefly in the present section.  A~\emph{toral automorphism}\indx{toral automorphism}~$f$ is  an invertible linear map on the  torus $\TT^d = \RR^d/\ZZ^d$, i.e., $f: \TT^d\to \TT^d$, $\bx \mapsto M \bx \pmod{\ZZ^d}$, for some integer matrix~$M$ with determinant~$\pm1$. It is called \emph{hyperbolic}\index{toral automorphism!hyperbolic}  if none of the eigenvalues of the matrix~$M$ is on the unit circle.
As already mentioned before, it is well known since the work of Sina\u{\i}  and Bowen \cite{Sinai:68,Sinai:68bis,Bow:70} that each hyperbolic toral automorphism~$f$ admits a \emph{Markov partition}. If the Markov partition is fine enough, then it provides a symbolic representation of~$f$ in the sense that it allows to conjugate the dynamical system $(\TT^d,f)$ to a symbolic dynamical system; cf.\ Section~\ref{subsec:nsft}. To illustrate this, we discuss Markov partitions of (the inverse of)\footnote{We consider $f(\bx) = M^{-1}\bx$ instead of $f(\bx) = M\bx$, with $M = \big(\begin{smallmatrix}2&1\\1&1\end{smallmatrix}\big)$, for consistency with the rest of the paper.} Arnold's famous \emph{cat map}\indx{cat map}
\begin{equation}\label{eq:catmap}
f:\, \TT^2 \to \TT^2, \quad \begin{pmatrix}x_1\\x_2\end{pmatrix} \mapsto \begin{pmatrix}2&1\\1&1\end{pmatrix}^{-1} \begin{pmatrix}x_1\\x_2\end{pmatrix} \pmod{\ZZ^2},
\end{equation}
a hyperbolic toral automorphism of~$\TT^2$. For an extensive survey of the theory of Markov partitions for toral automorphisms, we refer to \cite{A98}. 

\begin{figure}[ht] 
\begin{tikzpicture}[scale=2.25]
\filldraw[fill=red](-.171,.276)--(.276,-.447)--(1,0)--(.553,.724)--cycle;
\node at (.41,.14){$R_1$};
\filldraw[fill=blue](-.171,.276)--(-.447, .724)--(0,1)--(.276,.553)--cycle;
\node at (-.09, .64){$R_2$};
\draw[dotted](0,0)--(0,1)--(1,1)--(1,0)--cycle;
\node at (1.25,.33){$\xrightarrow[\displaystyle f^{-1}]{\big(\begin{smallmatrix}2&1\\1&1\end{smallmatrix}\big)}$};
\begin{scope}[shift={(1.75,0)}]
\begin{scope}[cm={2,1,1,1,(0,0)}]
\filldraw[fill=red](-.171,.276)--(.276,-.447)--(1,0)--(.553,.724)--cycle;
\node at (.41,.14){$M R_1$};
\filldraw[fill=blue](-.171,.276)--(-.447, .724)--(0,1)--(.276,.553)--cycle;
\node at (-.09, .64){$M R_2$};
\end{scope}
\draw(.276,-.447)--(1,0)--(.553,.724)--(.276,.553)--(0,1)--(-.447, .724)--cycle;
\draw[dotted](0,0)--(0,1)--(1,1)--(1,0)--cycle;
\node at (1.5,.33){$=$};
\end{scope}
\begin{scope}[shift={(4,0)}]
\filldraw[fill=blue](-.447,.724)--(-.342,.553)--(.106,.829)--(0,1)--cycle;
\filldraw[fill=red](-.342,.553)--(-.171,.276)--(.276,.553)--(.106,.829)--cycle;
\filldraw[fill=blue](-.171,.276)--(-.065,.106)--(.659,.553)--(.553,.724)--cycle;
\filldraw[fill=red](-.065,.106)--(.106,-.171)--(.829,.276)--(.659,.553)--cycle;
\filldraw[fill=red](.106,-.171)--(.276,-.447)--(1,0)--(.829,.276)--cycle;
\draw[dotted](0,0)--(0,1)--(1,1)--(1,0)--cycle;
\end{scope}
\end{tikzpicture}
\caption{The Markov partition $\{R_1,R_2\}$ for the cat map and its image.} \label{fig:cat1}
\end{figure}
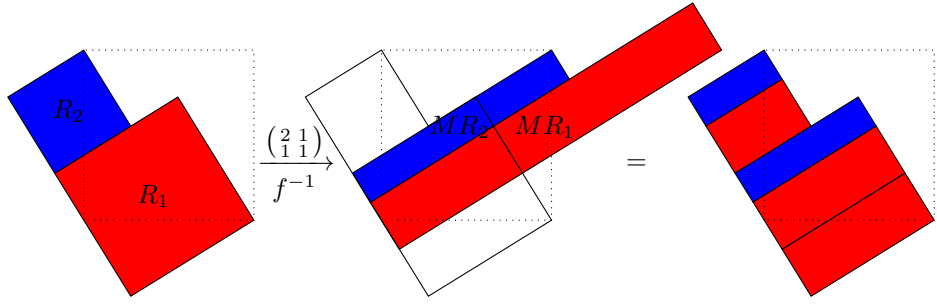

Consider the collection $\{R_1,R_2\}$ of the two disjoint open squares depicted in the left side of Figure~\ref{fig:cat1}. The construction of the squares~$R_1$ and~$R_2$ is surveyed for instance in \cite[Chapter~2, Section~5d]{KaHa:95}. The collection $\{R_1,R_2\}$ forms a \emph{topological partition}\indx{partition!topological} of~$\TT^2$ in the sense that $L = \overline{R_1} \cup \overline{R_2}$ forms a fundamental domain of $\RR^2/\ZZ^2$ and $R_1 \cap R_2 = \emptyset$; see Figure~\ref{fig:cat1}. 
To illustrate what happens when we apply the cat map~$f^{-1}$ to~$R_i$, $i \in \{1,2\}$, we do this in two steps.
We first consider~$R_i$ as a subset of~$\RR^2$ and apply the matrix~$M$ to this set. Because $M$ is hyperbolic and $R_i$ is chosen in a way that its edges are parallel to eigenvectors of~$M$, the image~$M R_i$ is a thinner and longer version of~$R_i$ whose edges are parallel to the edges of~$R_i$; see the middle of Figure~\ref{fig:cat1}. To get the representative of the image $f^{-1}(R_i)$ of~$R_i$ in the fundamental region~$L$, we have to ``restack'' pieces of $MR_i$ back to~$L$. This leads to the right side of Figure~\ref{fig:cat1}, which shows an important feature of the collection $\{R_1,R_2\}$: The collection $\{f^{-1}(R_1),f^{-1}(R_2)\}$ of its images tessellates~$L$ in a way that $f^{-1}(R_i)$ consists of finitely many rectangles each of which is a subset of some set~$R_j$ and has ``full height''. A~similar ``full width'' property holds if we apply~$f$ to the partition $\{R_1,R_2\}$, see Figure~\ref{fig:cat2}. 
\begin{figure}[ht] 
\begin{tikzpicture}[scale=2.25]
\filldraw[fill=red](-.171,.276)--(.276,-.447)--(1,0)--(.553,.724)--cycle;
\node at (.41,.14){$R_1$};
\filldraw[fill=blue](-.171,.276)--(-.447, .724)--(0,1)--(.276,.553)--cycle;
\node at (-.09, .64){$R_2$};
\draw[dotted](0,0)--(0,1)--(1,1)--(1,0)--cycle;
\node at (1.35,.33){$\xrightarrow[\displaystyle f]{\big(\begin{smallmatrix}2&1\\1&1\end{smallmatrix}\big)^{-1}}$};
\begin{scope}[shift={(2,0)}]
\begin{scope}[cm={1,-1,-1,2,(0,0)}]
\filldraw[fill=red](-.171,.276)--(.276,-.447)--(1,0)--(.553,.724)--cycle;
\filldraw[fill=blue](-.171,.276)--(-.447, .724)--(0,1)--(.276,.553)--cycle;
\end{scope}
\node at (-.15,-.25){$M^{-1} R_1$};
\node at (-.35,1.45){$M^{-1} R_2$};
\draw(.276,-.447)--(1,0)--(.553,.724)--(.276,.553)--(0,1)--(-.447, .724)--cycle;
\draw[dotted](0,0)--(0,1)--(1,1)--(1,0)--cycle;
\node at (1.25,.33){$=$};
\end{scope}
\begin{scope}[shift={(3.75,0)}]
\filldraw[fill=red](-.447,.724)--(.276,-.447)--(.553,-.276)--(-.171,.894)--cycle;
\filldraw[fill=blue](-.171,.894)--(.553,-.276)--(.724,-.171)--(0,1)--cycle;
\filldraw[fill=red](.276,.553)--(.724,-.171)--(1,0)--(.553,.724)--cycle;
\draw[dotted](0,0)--(0,1)--(1,1)--(1,0)--cycle;
\end{scope}
\end{tikzpicture}
\caption{The Markov partition $\{R_1,R_2\}$ and its image under~$f$.} \label{fig:cat2}
\end{figure}
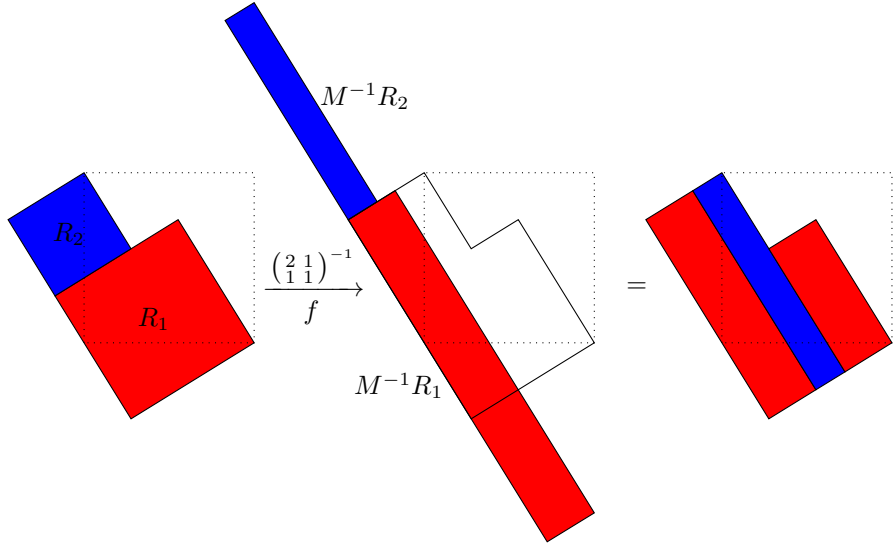
The ``full height'' and ``full width'' property have an important implication: If we choose $i_{-1},i_0,i_1\in\{1,2\}$, then 
\[
f(R_{i_{-1}}) \cap R_{i_{0}} \neq \emptyset, \,
R_{i_0} \cap f^{-1} (R_{i_{1}}) \neq \emptyset
\ \mbox{implies} \
f(R_{i_{-1}}) \cap  (R_{i_{0}})  \cap f^{-1} (R_{i_{1}}) \neq \emptyset.
\]
This implication is illustrated in Figure~\ref{fig:cat3}. 
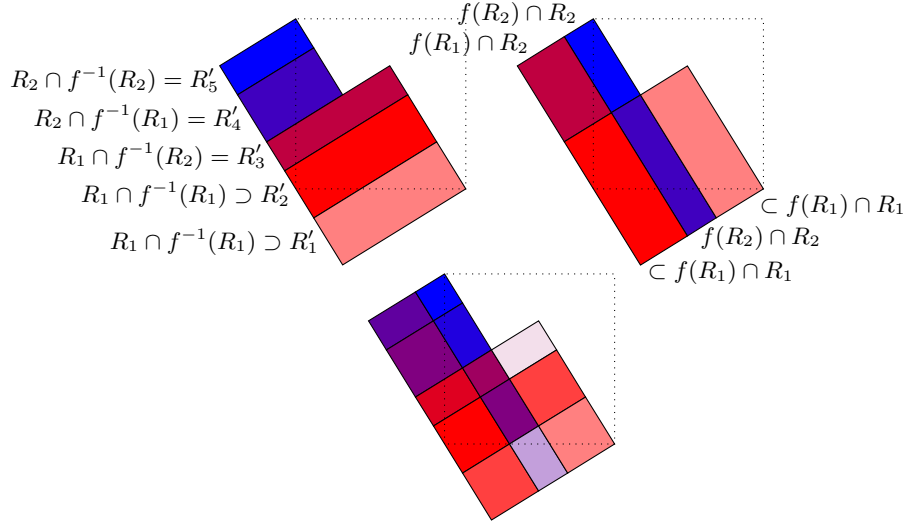
\begin{figure}[ht] 
\begin{tikzpicture}[scale=2.25]
\filldraw[fill=blue](-.447,.724)--(-.342,.553)--(.106,.829)--(0,1)--cycle;
\node[left] at (-.395,.639){\small$R_2 \cap f^{-1}(R_2) = R'_5$};
\filldraw[fill=blue!75!red](-.342,.553)--(-.171,.276)--(.276,.553)--(.106,.829)--cycle;
\node[left] at (-.257,.415){\small$R_2 \cap f^{-1}(R_1) = R'_4$};
\filldraw[fill=red!75!blue](-.171,.276)--(-.065,.106)--(.659,.553)--(.553,.724)--cycle;
\node[left] at (-.118,.191){\small$R_1 \cap f^{-1}(R_2) = R'_3$};
\filldraw[fill=red](-.065,.106)--(.106,-.171)--(.829,.276)--(.659,.553)--cycle;
\node[left] at (.021,-.033){\small$R_1 \cap f^{-1}(R_1) \supset R'_2$};
\filldraw[fill=red!50](.106,-.171)--(.276,-.447)--(1,0)--(.829,.276)--cycle;
\node[left] at (.191,-.309){\small$R_1 \cap f^{-1}(R_1) \supset R'_1$};
\draw[dotted](0,0)--(0,1)--(1,1)--(1,0)--cycle;
\begin{scope}[shift={(1.75,0)}]
\begin{scope}[cm={1,-1,-1,2,(1,-1)}]
\filldraw[fill=blue](-.447,.724)--(-.342,.553)--(.106,.829)--(0,1)--cycle;
\node[above left=-3pt] at (-.224,.862){\small$f(R_2) \cap R_2$};
\filldraw[fill=blue!75!red](-.342,.553)--(-.171,.276)--(.276,.553)--(.106,.829)--cycle;
\node[below right=-4pt] at (.053,.415){\small$f(R_2) \cap R_2$};
\end{scope}
\begin{scope}[cm={1,-1,-1,2,(0,0)}]
\filldraw[fill=red!75!blue](-.171,.276)--(-.065,.106)--(.659,.553)--(.553,.724)--cycle;
\node[above left=-3pt] at (.029,.4){\small$f(R_1) \cap R_2$};
\filldraw[fill=red](-.065,.106)--(.106,-.171)--(.829,.276)--(.659,.553)--cycle;
\node[below right=-2pt] at (.196,-.1){\small$\subset f(R_1) \cap R_1$};
\end{scope}
\begin{scope}[cm={1,-1,-1,2,(0,1)}]
\filldraw[fill=red!50](.106,-.171)--(.276,-.447)--(1,0)--(.829,.276)--cycle;
\node[below right=-3pt] at (.92,-.05){\small$\subset f(R_1) \cap R_1$};
\end{scope}
\draw[dotted](0,0)--(0,1)--(1,1)--(1,0)--cycle;
\end{scope}
\begin{scope}[shift={(.875,-1.5)}] \filldraw[fill=red!37.5!blue](-.447,.724)--(-.342,.553)--(-.065,.724)--(-.171,.894)--cycle;
\filldraw[fill=red!50!blue](-.342,.553)--(-.171,.276)--(.106,.447)--(-.065,.724)--cycle;
\filldraw[fill=red!87.5!blue](-.171,.276)--(-.065,.106)--(.211,.276)--(.106,.447)--cycle;
\filldraw[fill=red](-.065,.106)--(.106,-.171)--(.382,0)--(.211,.276)--cycle;
\filldraw[fill=red!75](.106,-.171)--(.276,-.447)--(.553,-.276)--(.382,0)--cycle;
\filldraw[fill=blue](-.171,.894)--(-.065,.724)--(.106,.829)--(0,1)--cycle;
\filldraw[fill=red!12.5!blue](-.065,.724)--(.106,.447)--(.276,.553)--(.106,.829)--cycle;
\filldraw[fill=red!62.5!blue](.106,.447)--(.211,.276)--(.382,.382)--(.276,.553)--cycle;
\filldraw[fill=red!50!blue](.211,.276)--(.382,0)--(.553,.106)--(.382,.382)--cycle;
\filldraw[fill=red!37.5!blue!37.5](.382,0)--(.553,-.276)--(.724,-.171)--(.553,.106)--cycle;
\filldraw[fill=red!62.5!blue!12.5](.276,.553)--(.382,.382)--(.659,.553)--(.553,.724)--cycle;
\filldraw[fill=red!75](.382,.382)--(.553,.106)--(.829,.276)--(.659,.553)--cycle;
\filldraw[fill=red!50](.553,.106)--(.724,-.171)--(1,0)--(.829,.276)--cycle;
\draw[dotted](0,0)--(0,1)--(1,1)--(1,0)--cycle;
\end{scope}
\end{tikzpicture}
\caption{Illustration of the intersection properties of the Markov partition $\{R_1,R_2\}$.}\label{fig:cat3}
\end{figure}
This tessellation property is even maintained if we consider iterates~$f^n$, $n\in\ZZ$: The collection $\{f^n(R_1),f^n(R_2)\}$ tessellates~$L$ in a way that the images $f^n(R_i)$ consist of finitely many ``full height'' (if $n< 0$) or ``full width'' (if $n> 0$) rectangles, each of which is a subset of some~$R_j$.  
This fact can be made exact in the following way. For each $n \ge 1$ and each sequence of elements $i_0,\ldots,i_{n}\in \{1,2\}$, the property
\begin{equation}\label{eq:classicalmarkovdef}
R_{i_k} \cap f^{-1} (R_{i_{k+1}}) \neq \emptyset \quad \mbox{for all}\ k\in\{0,\ldots, n{-}1\}
\end{equation}
implies that
\begin{equation}\label{eq:classicalmarkovdef2}
R_{i_0} \cap f^{-1} (R_{i_{1}}) \cap \cdots \cap f^{-n} (R_{i_{n}}) \neq \emptyset.
\end{equation}
This implication makes, by definition, $\{R_1,R_2\}$ a \emph{Markov partition}\indx{Markov partition!stationary} for the hyperbolic toral automorphism~$f$.

The partition $\{R_1,R_2\}$ has a drawback. As can be seen in the first panel of Figure~\ref{fig:cat3}, the image $f^{-1}(R_1)$ intersects~$R_1$ in the two open rectangles $R'_1$ and~$R'_2$. (This is due to the $(1,1)$ entry of~$M$ being~$2$.)
This leads to ambiguities when one tries to use this partition to construct a symbolic representation of~$f$.
More precisely, if we consider a sequence $(i_n)_{n\in\ZZ} \in \{1,2\}^\ZZ$ such that $f^n(x) \in R_{i_n}$ for some $x \in R'_1$, then we also find a point $x \in R'_2$ with the same property.
To avoid such ambiguities, we need to refine the partition $\{R_1,R_2\}$. 
It turns out that the partition $\{R'_1,\dots,R'_5\}$ consisting of the five rectangles in the first panel of Figure~\ref{fig:cat3} is an appropriate refinement. Indeed, $\{R'_1,\dots,R'_5\}$ forms a Markov partition again, and this time $f(R'_j)$ intersects~$R'_i$ in at most one rectangle. This entails that the Markov partition $\{R'_1,\dots,R'_5\}$ is \emph{generating}\indx{Markov partition!stationary!generating} in the sense that 
\[
\bigcap_{n\in\NN}\overline{\bigcap_{|k|\le n} f^{-k}(S_{i_k})}
\]
contains at most one point for each sequence $(i_n)_{n\in\ZZ} \in \{1,\dots,5\}^\ZZ$. Thus almost every element of $\bx\in \TT^2$ can be ``coded'' by a sequence $(i_n)_{n\in\ZZ} \in \{1,\dots,5\}^\ZZ$. It turns out that the closure of the set  of sequences in $\{1,\dots,5\}^\ZZ$ that occur as codings of orbits is a \emph{shift of finite type}\indx{shift!of finite type}, i.e., it is a subset $X_M\subset \{1,\dots,5\}^\ZZ$ whose elements are characterized by the fact that they do not contain certain substrings (of length~$2$). In other words, the shift space $(X_M,\Sigma)$ with the shift map $\Sigma: X_M\to X_M$, $(i_n)_{n\in\ZZ}\mapsto (i_{n+1})_{n\in\ZZ}$, is a symbolic representation of $(\TT^2,f)$ in the sense that it allows to conjugate~$f$ to a shift of finite type. The set~$X_M$ is explicitly described in \cite[Chapter~2, Section~5d]{KaHa:95}.

We develop the theory of nonstationary Markov partitions in Section~\ref{sec:MarkovTheory}, consider  symbolic  representations in Section~\ref{subsec:nsft} and apply it to sequences of toral automorphisms in Sections~\ref{sec:markov} and \ref{sec:metricMP}. Its relation with multidimensional continued fraction algorithms will be described in the subsequent section of this introduction.
 
\section[Classical continued fraction algorithm]{Dynamical interpretations of the classical continued fraction algorithm} \label{sec:introsturm}
One motivation for the present work is to give conditions on multidimensional continued fraction algorithms that allow an extension of the dynamical interpretations of the classical continued fraction algorithm, such as described for instance in \cite{Arnoux-Nogueira,Arnoux:94,AF:01}, to these algorithms. In the classical case, these dynamical interpretations rely on two ingredients. Firstly, the classical continued fraction algorithm can be interpreted in terms of a cascade of \emph{inductions} (i.e., \emph{first return maps}) of circle rotations (cf.\ Definition~\ref{def:induction}), an insight that goes back to Rauzy, see e.g.~\cite{Arnoux-Rauzy:91}. Secondly, it can be viewed as a first return map to a section of the \emph{geodesic flow} on the unit tangent bundle of the modular surface. Geometrically, this return map can be viewed as a restacking process that is interlinked with a nonstationary Markov partition. This idea dates back to \cite{Artin:24} and has been pursued e.g.\ in \cite{AF:82,AF:84,Series:85,BKS,AF:01,AF:05}. We will now give a short account of these interpretations.

The classical continued fraction algorithm can be viewed as a map acting on the parameter space $[0,1)$ of a family of circle rotations, in particular, of the rotations $\fr_\alpha:\, x \mapsto x {+} \alpha \mod 1$\notx{rot}{$\fr_{\balpha}, \fr_{\bx}$}{toral rotation} on the one-dimensional torus $\TT = \RR/\ZZ$; here, $\alpha\in[0,1)$ is the parameter. The continued fraction expansion of a real number~$\alpha$ is linked with the recurrence properties of the rotation~$\fr_\alpha$. In particular, the iterates $n$ for which $\fr_\alpha^n(0)$ is closer to~$0$ than each of the previous iterates~$\fr_\alpha^m(0)$, $m\in\{1,\dots,n{-}1\}$, are just the denominators of the convergents in the continued fraction expansion of~$\alpha$.

When viewing the rotation~$\fr_\alpha$ as an exchange of two intervals on $[0,1)$, the \emph{additive} or \emph{Farey}\indx{continued fraction algorithm!Farey}\indx{Farey!algorithm} continued fraction algorithm (see e.g.\ \cite{Schweiger:00} and Section~\ref{subsec:cfdef}) can be recovered by considering the first return map  of~$\fr_\alpha$ on the larger one of the intervals $[0,1{-}\alpha)$ and $[1{-}\alpha,1)$ (cf.\ Definition~\ref{def:induction}), renormalized in a way that the length of this interval becomes~$1$.  Since this return map is again an exchange of two intervals, it can again be viewed as a rotation $\fr_{\alpha'}$ for some other parameter $\alpha'\in[0,1)$. The ensuing map $\alpha \mapsto \alpha'$ is the \emph{Farey map}\indx{Farey!map}
\begin{equation}\label{eq:FareyIntro}
T_{\rF}:\, [0,1) \to [0,1], \quad x \mapsto 
\begin{cases}
\frac{x}{1-x}, & x < \frac12, \\
\frac{1-x}{x}, & x \ge \frac12 \\
\end{cases}
\notx{F}{$\rF$}{object related to the Farey algorithm}
\end{equation}
and, iterating this induction process yields a cascade of rotations that models the iterates of the Farey map; see e.g.\ \cite{AFH:99,Arnoux:01}.

One can give a very explicit symbolic model of this process: The rotation~$\fr_\alpha$, viewed as an exchange of two intervals, can be coded by these two intervals on a two-letter alphabet. This coding gives the classical \emph{Sturmian sequences}\indx{Sturmian!sequence}; see Definition~\ref{def:sturm}.
Sturmian sequences can also be defined by attaching substitutions to the $2{\times}2$ integer matrices generated by the Farey continued fraction algorithm.  This is done by using the two \emph{Sturmian substitutions}\indx{Sturmian!substitution}\indx{substitution!Sturmian}
\begin{equation}\label{eq:sturmsubs}
\sigma_{\rF,1}:  \begin{cases}1 \mapsto 1,\\ 2 \mapsto 21,\end{cases} \qquad
\sigma_{\rF,2}: \begin{cases}1 \mapsto 12, \\ 2 \mapsto 2,\end{cases}
\notx{F}{$\rF$}{object related to the Farey algorithm}
\end{equation}
that are related to the two alternatives of the Farey map, that is, to the two matrices \begin{equation}\label{eq:sturmmat}
M_{\rF,1} = \begin{pmatrix}1&1\\0&1\end{pmatrix} \quad \mbox{and} \quad M_{\rF,2} = \begin{pmatrix}1&0\\1&1\end{pmatrix}
\end{equation}
corresponding to the homographies in \eqref{eq:FareyIntro}.
With help of the substitutions in~\eqref{eq:sturmsubs}, we can define an induction process that forms a symbolic version of the Farey map; see \cite{AF:01,Fog02} and Example~\ref{ex:sturm}.
This is a special instance of Rauzy induction, as performed more generally in the setting of interval exchange transformations \cite{Yoccoz:2006}. Roughly speaking, a sequence which is obtained as a coding of a map can be recovered by applying a substitution to the sequence coding the induced map; each step in the induction process corresponds to a ``desubstitution'' of the original sequence.  
 
By accelerating the Farey map, we obtain the \emph{Gauss map}\indx{Gauss!map}
\begin{equation}\label{eq:GaussMapIntro}
T_{\rG}:\, (0,1) \to [0,1), \quad \alpha \mapsto \frac1\alpha - \Big\lfloor \frac1\alpha\Big\rfloor
\notx{G}{$\rG$}{object related to the classical continued fraction algorithm}
\end{equation}
of the classical continued fraction algorithm  (cf.\ Example~\ref{ex:classicalCF}) that is related to an analogous induction process leading again to a symbolic interpretation in terms of Sturmian sequences. In this case, one ``desubstitutes'' by certain compositions of the substitutions in \eqref{eq:sturmsubs} that correspond to the partial quotients of the continued fraction expansion; see Example \ref{ex:sturm} for more details on Sturmian sequences. Summing up, by ``superimposing'' the combinatorics of Sturmian sequences, we can interpret the classical continued fraction algorithm as acting on a sequence of induced rotations. 

To illustrate the link between the classical continued fraction algorithm and nonstationary Markov partitions via first return maps of the geodesic flow on the unit tangent bundle of the modular surface, we start with \cite{AF:05}. In this paper, another type of symbolic coding associated to continued fraction algorithms is considered. The authors start from a $2$-dimensional hyperbolic toral automorphism. Clearly, the slopes of the stable and unstable foliations of such an automorphism are always quadratic irrationals. As we saw in Section~\ref{sec:defMarkov}, the associated  Markov partition is made of  two boxes, with axes parallel to these slopes, and whose union forms an $L$-shaped region. In \cite{AF:05}, the authors ask if this picture can be generalized to arbitrary slopes. This leads them to replace a single toral automorphism by a sequence of toral automorphisms of the form \eqref{eq:star2}. Starting from two given arbitrary slopes $G^s$ and~$G^u$ for the stable and unstable foliations, respectively, they use the classical continued fraction algorithm in order to set up a sequence of  $2{\times}2$-matrices $(M_n)_{n\in\ZZ}$ having determinant $\pm 1$ whose associated linear Anosov mapping family \eqref{eq:star2} has $G^s$ as stable and $G^u$ as unstable foliation. Besides that, they are able to extend classical Markov partitions and associate a sequence of (nonstationary) Markov partitions to these Anosov mapping families. As in the stationary case illustrated in Section~\ref{sec:defMarkov}, these nonstationary Markov partitions are $L$-shaped regions whose boundary pieces are parallel to $G^s$ and~$G^u$. 
\begin{figure}[ht] 
\begin{tikzpicture}[scale=1.75]
\filldraw[fill=red](.515,-.319)--(-.3,.185)--(0,1)--(.815,.496)--cycle;
\draw[dashed](.215,-.134)--(.515,.681) (-.085,.051)--(.215,.866);
\filldraw[fill=blue](.515,-.319)--(.815,-.504)--(1,0)--(.7,.185)--cycle; 
\draw[dotted](0,0)--(0,1)--(1,1)--(1,0)--cycle;
\node at (1.5,.33){$\xrightarrow{\displaystyle f_0}$};
\begin{scope}[shift={(3.2,0)}]
\begin{scope}[cm={-2,1,1,0,(0,0)}]
\filldraw[fill=red](.515,-.319)--(-.3,.185)--(0,1)--(.815,.496)--cycle;
\draw[dashed](.215,-.134)--(.515,.681) (-.085,.051)--(.215,.866);
\filldraw[fill=blue](.515,-.319)--(.815,-.504)--(1,0)--(.7,.185)--cycle; 
\end{scope}
\draw[dotted](0,0)--(0,1)--(1,1)--(1,0)--cycle;
\node at (1.25,.5){$\xrightarrow{\mbox{restacking}}$};
\end{scope}
\begin{scope}[shift={(5.5,0)}]
\filldraw[fill=red](-.564,.215)--(-.128,.815)--(.657,.515)--(.436,.215)--(1,0)--(.785,-.3)--cycle;
\filldraw[fill=blue](.657,.515)--(-.128,.815)--(0,1)--(.785,.7)--cycle;
\draw[dashed](.221,-.085)--(.436,.215)--(-.349,.515);
\draw[dotted](0,0)--(1,0)--(1,1)--(0,1)--cycle;
\end{scope}
\begin{scope}[shift={(1.5,-1.75)}]
\node at (-1.25,.5){$\xrightarrow{\mbox{recoloring}}$};
\filldraw[fill=red](.221,-.085)--(-.564,.215)--(0,1)--(.785,.7)--cycle;
\draw[dashed](-.343,.13)--(.221,.915);
\filldraw[fill=blue](.221,-.085)--(.785,-.3)--(1,0)--(.436,.215)--cycle;
\draw[dotted](0,0)--(1,0)--(1,1)--(0,1)--cycle;
\node at (1.5,.5){$\xrightarrow{\displaystyle f_1}$};
\end{scope}
\begin{scope}[shift={(4.25,-1.75)}]
\begin{scope}[cm={-1,1,1,0,(0,0)}]
\filldraw[fill=red](.221,-.085)--(-.564,.215)--(0,1)--(.785,.7)--cycle;
\draw[dashed](-.343,.13)--(.221,.915);
\filldraw[fill=blue](.221,-.085)--(.785,-.3)--(1,0)--(.436,.215)--cycle;
\end{scope}
\draw[dotted](0,0)--(1,0)--(1,1)--(0,1)--cycle;
\end{scope}
\end{tikzpicture}
\caption{Illustration of the restacking and renormalization process on the atoms of the nonstationary Markov partition furnished by the classical continued fraction algorithm. The ``recoloring'' is just done in order to visualize the Markov partition of the restacked $L$-shaped region.} \label{fig:Lshaped}
\end{figure}
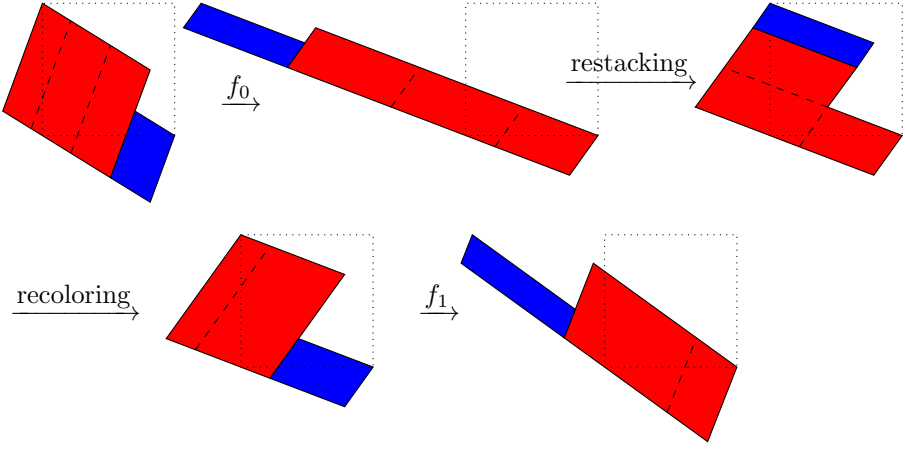
The $(n{+}1)$-st Markov partition of the sequence emerges from the $n$-th one by a \emph{restacking and renormalization process}\indx{restacking}\indx{renormalization} (see e.g.~\cite{AF:05,thuswaldner2019boldsymbolsadic}), which is illustrated in Figure~\ref{fig:Lshaped}. However, contrary to the stationary Markov partition (like the one in Figure~\ref{fig:cat1}), the first $L$-shaped region in Figure~\ref{fig:Lshaped} is not the same as the fourth $L$-shaped region. This reflects the ``nonstationarity'' of this sequence of Markov partitions. 

On this restacking and renormalization process visualized in Figure~\ref{fig:Lshaped} we can see the Gauss map on the axis parallel to~$G_s$. Moreover, it can be reinterpreted as a section of a geodesic flow on the unit tangent bundle of the modular surface. To make this precise, it is convenient to transform the $L$-shaped regions affinely (and volume preserving) in a way that the boundary pieces become axis parallel; see Figure~\ref{fig:Lshaped2}.
\begin{figure}[ht]
\begin{tikzpicture}[scale=2]
\filldraw[fill=red](0,0)--node[below]{$1$}(1,0)--(1,.815)--(0,.815)--cycle;
\draw[dashed](.368,0)--(.368,.815) (.736,0)--(.736,.815);
\filldraw[fill=blue](-.368,0)--node[below]{$\vphantom{1}\alpha$}(0,0)--(0,.504)--(-.368,.504)--cycle; 
\node at (1.5,.5){$\xrightarrow{\mbox{restacking}}$};
\begin{scope}[shift={(2,0)}]
\filldraw[fill=red](0,0)--(.632,0)--(.632,.815)--(.368,.815)--(.368,1.63)--(0,1.63)--cycle;
\node[below] at (.184,0){$\vphantom{\frac{1}{\alpha}}\alpha$};
\node[below] at (.65,0){$1{-}\lfloor\frac{1}{\alpha}\rfloor\alpha$};
\draw[dashed](.368,0)--(.368,.815)--(0,.815);
\filldraw[fill=blue](0,1.63)--(.368,1.63)--(.368,2.133)--(0,2.133)--cycle;
\node at (1.3,.5){$\xrightarrow[\mbox{and recoloring}]{\mbox{renormalizing}}$};
\end{scope}
\begin{scope}[shift={(4.7,0)}]
\filldraw[fill=red](0,0)--node[below]{$\vphantom{\frac{1}{\alpha}}1$}(1,0)--(1,.785)--(0,.785)--cycle;
\draw[dashed](.718,0)--(.718,.785);
\filldraw[fill=blue](-.718,0)--node[below]{$\frac{1}{\alpha}{-}\lfloor\frac{1}{\alpha}\rfloor$}(0,0)--(0,.3)--(-.718,.3)--cycle; 
\end{scope}
\end{tikzpicture}
\caption{The axis-parallel version of the restacking process of Figure~\ref{fig:Lshaped}.}
\label{fig:Lshaped2}
\end{figure}
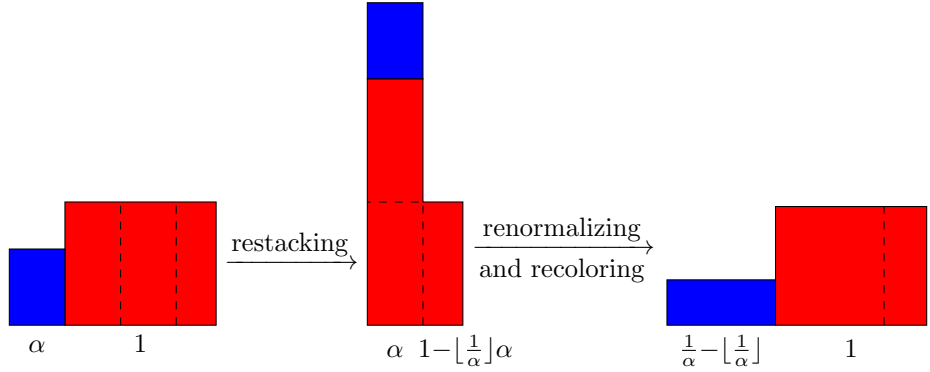
On the $x$-axis, this transformed restacking process performs one step of the underlying classical continued fraction algorithm, i.e., the Gauss map~$g$. By restacking, the interval on  $x$-axis becomes smaller. It is this smaller interval to which we induced the rotation earlier, so this restacking process is intimately related to the induction process described above. The whole restacking process can be seen as a \emph{natural extension of the Gauss map}, that is, an invertible dynamical system that admits the original noninvertible (Gauss) dynamics as a factor; see Section~\ref{sec:natex} for precise definitions. 

Because each of these $L$-shaped regions are fundamental domains of a lattice, it induces a tiling of~$\RR^2$ of covolume~$1$; see Figure~\ref{fig:Ltiling}. Using this lattice viewpoint, the above restacking and renormalization process can now be reinterpreted as a section of a geodesic flow. Indeed, it is well known that the unit tangent bundle of the modular surface can be modeled by the space $\SL(2,\ZZ) \backslash \SL(2,\RR)$ that we just encountered. 
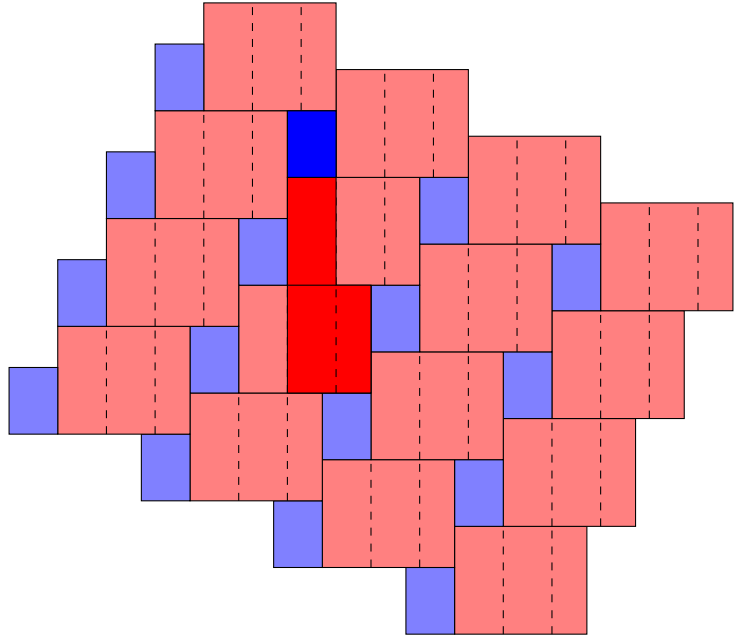
\begin{figure}[ht]
\begin{tikzpicture}[scale=1.75]
\newcommand{\cftile}{\filldraw[fill=red!50](0,0)--(1,0)--(1,.815)--(0,.815)--cycle;
\draw[dashed](.368,0)--(.368,.815) (.736,0)--(.736,.815);
\filldraw[fill=blue!50](-.368,0)--(0,0)--(0,.504)--(-.368,.504)--cycle;}
 \foreach \i in {0,1,2,3} \foreach \j in {0,1,2,3}
{\begin{scope}[shift={(\j*.368-\i,\i*.504+\j*.815)}]
\cftile
\end{scope}}
\begin{scope}[shift={(-1.264,1.823)}]
\filldraw[fill=red](0,0)--(.632,0)--(.632,.815)--(.368,.815)--(.368,1.63)--(0,1.63)--cycle;
\draw[dashed](.368,0)--(.368,.815);
\draw(.368,.815)--(0,.815);
\filldraw[fill=blue](0,1.63)--(.368,1.63)--(.368,2.133)--(0,2.133)--cycle;
\end{scope}
\end{tikzpicture}
\caption{A tiling with covolume~$1$ induced by an $L$-shaped Markov partition. The restacked partition, which is illustrated in bright colors, represents the same lattice.}
\label{fig:Ltiling}
\end{figure}
In this model, the geodesic flow is given by the right action of the diagonal group $g_t = \mathrm{diag}(e^{t/2},e^{-t/2})$, $t\in\RR$. The action of~$g_t$ consists in expanding the horizontal coordinates and contracting the vertical ones. Each point in $\SL(2,\ZZ) \backslash \SL(2,\RR)$ corresponds to a flat torus represented by an appropriate $L$-shaped region. One can now find a transverse section to the geodesic flow on the modular surface such that the first return map of the flow can be viewed geometrically by restacking and renormalizing the $L$-shaped nonstationary Markov partitions as illustrated in Figure~\ref{fig:Lshaped2}. This return map provides a geometric model of the natural extension of the Gauss map by coding geodesics as bi-infinite sequences of symbols in terms of the continued fraction expansions.

On top of this, in each $L$-shaped region, i.e., in each point in $\SL(2,\ZZ) \backslash \SL(2,\RR)$, there is a vertical flow, whose first return map to a well chosen transverse section (which is given by the bottom line of the $L$-shaped region) is the rotation~$\fr_\alpha$ we considered in the first place. The particular case of periodic orbits has been studied in detail, and according to the classical theorem of Galois, they correspond to special quadratic units. For periodic orbits, the corresponding rotations are self-induced, i.e., they are conjugate to their first return map on some subspace. They also correspond to periodic geodesics on the modular surface. 

Note also that the modular surface can be interpreted as the moduli space of the torus (as a quotient of the Teichm\"uller space by the modular group), and an element of its tangent space corresponds to a torus equipped with two transverse measured foliations, a \emph{vertical} and a \emph{horizontal} foliation; the first return map of the flow along the vertical foliation on a suitable horizontal interval is again the rotation~$\fr_\alpha$ considered above. In this moduli space of the torus, a starting point on a closed geodesic corresponds to a pair of transverse measured foliations on the torus, and the action of the geodesic flow along the geodesic corresponds to a hyperbolic automorphism of the torus, which admits these two foliations as stable and unstable foliation. 
One can even consider a fiber bundle, with torus fiber, over the unit tangent bundle of the modular surface. This fiber bundle also has an algebraic model, which is given by the affine quotient $\mathrm{SA}(2,\ZZ) \backslash \mathrm{SA}(2,\RR)$. The right action of $\mathrm{SA}(2,\RR)$ on this quotient contains all the dynamics we have seen above. This action is the \emph{scenery flow} that has been studied in detail in \cite{AF:01}. 

\section[Multidimensional continued fraction algorithms]{Generalization to multidimensional continued fraction algorithms}  \label{sec:introcf}
In the present work, we provide an extension of the theory described in Section~\ref{sec:introsturm} to multidimensional continued fraction algorithms. The most classical examples of such algorithms are the Jacobi--Perron \cite{Bernstein:71,Heine1868,Perron:07,Schweiger:73}, the Brun \cite{Brun19,Brun20,BRUN}, and the Selmer algorithms \cite{Selmer:61, Schweiger:00}.  
Our goal is to view a multidimensional continued fraction algorithm in terms of inductions of toral rotations and associate nonstationary Markov partitions to it. Again, it turns out that this induction process and the Markov partitions are intimately related.  

More precisely, let $F$ be a multidimensional continued fraction map acting on a subset~$X$ of the projective space~$\PP^{d-1}$. (Since the matrices aim at providing rational approximations, it is natural to work in the projective space.) 
To a point $\bx \in X$, we associate a rotation~$\fr_\bx$ acting on the torus~$\TT^{d-1}$.
Using the sequence of matrices $(M_n)_n$ corresponding to the $F$-orbit of~$\bx$, we will find an explicit subset~$\cR_\bx$ of the torus such that, up to some  renormalization process, the induced map of~$\fr_\bx$ on~$\cR_\bx$ is~$\fr_{F(\bx)}$. It turns out that the set can also be used to define suitable atoms for nonstationary Markov partitions for the underlying linear Anosov mapping families~\eqref{eq:star2}; see Section~\ref{sec:induced}. Indeed, it occurs in a restacking process of the atoms of these Markov partitions as it did in the classical case.

Besides the classical case treated in the previous section, also the case of periodic orbits is a strong source of inspiration. It is indeed  known how to set up the induction process mentioned above for periodic orbits of the continued fraction map~$F$. In fact, periodic orbits correspond to self-induced dynamical systems, i.e., systems that are topologically conjugate to their first return map on some subspace. For self-induced systems, the induction process is studied for instance in \cite{KV:98,Arnoux-Ito:01,BS:05,Ito-Rao:06,ST:09}, and it gives rise to  symbolic codings of a particular class of rotations on~$\TT^{d-1}$.  Moreover, in the periodic case, the linear Anosov mapping family \eqref{eq:star2} degenerates to a hyperbolic toral automorphism $M^{-1}\colon \TT^{d-1} \to \TT^{d-1}$, where $M$ is a nonnegative integer matrix with determinant~$\pm 1$. It is known that, from the subset~$\cR_\bx$ of~$ \TT^{d-1}$  on which one induces, one can build a stationary Markov partition for this hyperbolic toral automorphism  by a suspension. Important geometric objects that play a decisive role in these results are the \emph{Rauzy fractals} together with their suspensions, the \emph{Rauzy boxes}. Indeed, Rauzy fractals can be used to define appropriate sets for the induction, while atoms for the Markov partitions can be defined using Rauzy boxes. They play the role of the intervals and of the $L$-shaped region in the classical case, respectively. Rauzy fractals were first introduced in \cite{Rauzy:82} for the so-called Tribonacci substitution; see also \cite{Thurston:1989}. See Figure~\ref{fig:Tribo} for an illustration and Section~\ref{sec:rauzy} for precise definitions.
\begin{figure}[ht]
\includegraphics{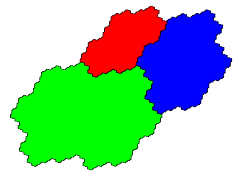}   \quad
\includegraphics{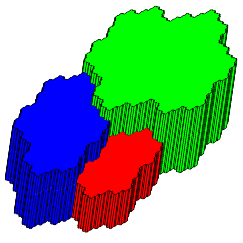}
\caption{The  Rauzy  fractal and the Rauzy box  for the Tribonacci substitution $\sigma: 1 \mapsto 12, 2 \mapsto 13, 3 \mapsto 1$.}
\label{fig:Tribo}
\end{figure}
They can more generally be associated to Pisot substitutions (see \cite{Arnoux-Ito:01} and the surveys \cite{Fog02,BS:05,ST:09,CANTBST}), as well as to Pisot beta-transformations and beta-shifts under the name of central tiles \cite{Aki02,BSSST2011}. The idea is as follows. To a matrix~$M$ defining a hyperbolic \emph{Pisot} toral automorphism  (i.e., the characteristic polynomial of~$M$ is the minimal polynomial of a Pisot number), one attaches a \emph{substitution}~$\sigma$, i.e., a combinatorial rule that replaces letters by words, in a way that the matrix~$M$ counts the number of occurrences of letters in the images of each letter under the substitution~$\sigma$; see Section~\ref{sec:matrices} and the survey \cite{AkiBBLS} for more on  the Pisot assumption. We can view $\sigma$ as a combinatorial version of $M$ and we can even attach a \emph{broken line} to~$\sigma$. This broken line connects two successive points in~$\ZZ^d$ and stays within bounded distance of the line in~$\RR^d$ defined by the Perron--Frobenius eigenvector of~$M$ under the Pisot property of the matrix~$M$; see Figure~\ref{fig:brokenline}. 
\begin{figure}[ht]
\includegraphics[trim=0 60 0 60,clip,width = 0.6\textwidth]{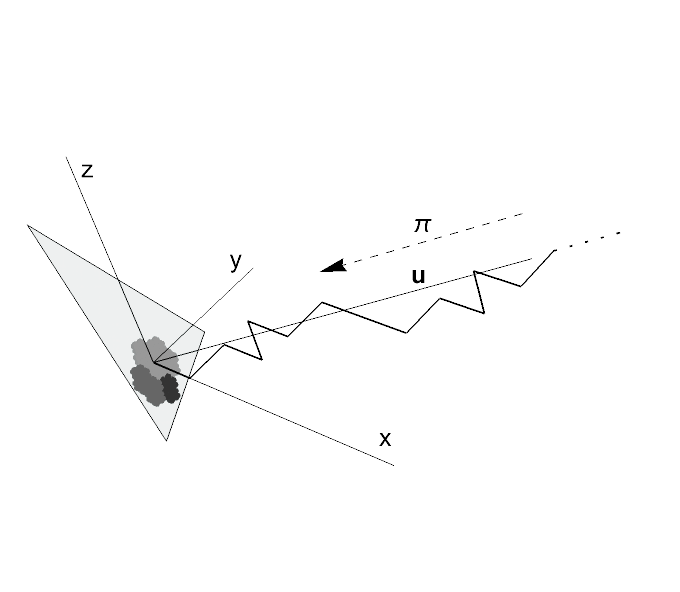} 
\caption{An illustration of a broken line defining a Rauzy fractal.}
\label{fig:brokenline}
\end{figure}
By taking the closure of the projection of the vertices of this broken line on a transverse direction, one obtains the Rauzy fractal of~$\sigma$, which codes the desired rotation on~$\TT^{d-1}$; see again \cite{Fog02,BS:05,ST:09,CANTBST} and also Section~\ref{sec:rauzy}.  Note  that from the periodic substitutive case, one only obtains a very particular set of algebraic rotations.  Note also  that substitutive dynamical systems defined in terms of Pisot substitutions are conjectured to have pure discrete spectrum, that is, to be  measurably conjugate to a rotation on a compact abelian group \cite{AkiBBLS}. 
So, as in the classical case, substitutions are used and by superimposing a combinatorial structure on our periodic orbits again leads to the desired induction process as well as to atoms for the Markov process.

Two strongly interlinked problems arise in the attempt to generalize this to the nonperiodic case. 

\begin{namedthm}{Problem I} 
\makeatletter\def\@currentlabelname{I}\makeatother
\label{p:I}
For a rotation on a torus~$\TT^{d-1}$ of dimension~$d\ge 3$, find a subset of~$\TT^{d-1}$ on which the induced map is measurably conjugate to a rotation.
\end{namedthm}
 
\begin{namedthm}{Problem II}
\makeatletter\def\@currentlabelname{II}\makeatother
\label{p:II}
Find atoms of a nonstationary Markov partition associated to a nonperiodic orbit of a multidimensional continued fraction algorithm.
\end{namedthm}

Problem~\nameref{p:I} is difficult because one usually gets an exchange of pieces with some discontinuities when inducing on a ``simple'' subset of~$\TT^{d-1}$. Moreover, we know from the periodic case that the construction of the subsets on which the induced rotation lives is intimately related to the atoms of the Markov partition for a hyperbolic toral automorphism.
Hence, we expect that Problem~\nameref{p:I} is strongly related to Problem~\nameref{p:II}. Thus we are left with finding atoms of a nonstationary Markov partition associated to a nonperiodic orbit, which is our main goal. However, a well-known result of \cite{Bow78} states that, already in the periodic case (except in trivial cases), such a partition cannot be smooth and has ``fractal'' boundary.
In particular, the boundaries of the atoms of a Markov partition for an algebraically irreducible linear Anosov diffeomorphisms of the $3$-torus cannot be smooth. This indicates that, in the nonperiodic higher-dimensional case, it is hopeless to try simple sets, as for instance polytopes, in order to tackle the above problems. We are thus lead to look more closely at the Rauzy fractals discussed before. 

Indeed, the following strategy, which was developed in \cite{BST:19,BST:23,Fogg:24}, combines features from the classical Sturmian and the higher dimensional periodic situation and will thereby lead to a solution to Problems~\nameref{p:I} and~\nameref{p:II}. Inspired by the known situations described above, one introduces symbolic dynamical systems defined in terms of substitutions, by superimposing a combinatorial structure upon the multidimensional continued fraction map~$F$. This is done by a so-called \emph{substitutive realization} of the algorithm~$F$: One assigns to each square matrix~$M$ that can be produced by the algorithm $F$ a  substitution~$\sigma$. Using this assignment, to each sequence $\bM = (M_n)_n$ of matrices produced by the map~$F$, we can associate a sequence of substitutions $\bsigma = (\sigma_n)_n$; see Definition~\ref{d:realization} below. This \emph{$\cS$-adic} framework (see also \cite{Berthe-Delecroix}) then gives rise to a family of symbolic dynamical systems $(X_{\bsigma},\Sigma)$ (where $\Sigma$ stands for the shift map) consisting of infinite sequences of substitutions generated by the multidimensional continued fraction algorithm~$F$. 

Let us emphasize how this symbolic $\cS$-adic superstructure can be used to solve our problems.  Continued fraction algorithms provide better and better rational approximation vectors of a given vector $(\alpha_1,\dots,\alpha_d) \in \RR^d$. Geometrically, this amounts to exhibiting integer points in~$\ZZ^d$ that lie close to the line~$L$ defined by $(\alpha_1,\dots,\alpha_d)$. These points are column vectors of products $M_1 \cdots M_n$ of matrices $M_1, \ldots, M_n$ produced by the continued fraction algorithm~$F$, with increasing length~$n$. We recall that the algorithm is \emph{strongly convergent} if these column vectors converge in distance (and not only in angle) to the line~$L$; see Definition~\ref{def:wsc}.  As in the periodic case, the substitutions of the substitutive realization provide a broken line connecting points in~$\ZZ^d$. By strong convergence, this broken line stays (usually) within bounded distance of~$L$; see again Figure~\ref{fig:brokenline}. In particular, in this general case, the infinite broken line is given by an element of the symbolic dynamical system $(X_{\bsigma}, \Sigma)$. By taking the closure of the projection of the vertices of this broken line on a transverse direction, one obtains a bounded  $\cS$-adic Rauzy fractal. These fractals were first defined in \cite{BST:19,thuswaldner2019boldsymbolsadic,BST:23}, where they were used to relate orbits of multidimensional continued fraction algorithms to rotations on a torus~$\TT^{d-1}$. Contrary to the periodic case discussed above, $\cS$-adic Rauzy fractals do not admit a self-affine structure, which makes them harder to deal with. Nevertheless, they can be used to solve Problems~\nameref{p:I} and~\nameref{p:II} stated above. 

As a first application, $\cS$-adic Rauzy fractals lead to a solution to a problem that has been open for a long time (see \cite{Arnoux-Rauzy:91}): One wants to find good symbolic codings for rotations~$\fr_\bx$ on the torus that enjoy the beautiful properties of Sturmian sequences like low factor complexity and good symbolic discrepancy properties; see Example~\ref{ex:sturmbalance}. It was shown in  \cite{BST:19,BST:23,Fogg:24} that $\cS$-adic Rauzy fractals are the right objects to obtain such codings. This allows us to associate a rotation on~$\TT^{d-1}$ to an orbit of multidimensional continued fraction algorithms, and it is this rotation on which we start our induction process. In particular, we revisit \cite{BST:19,BST:23} to get the induced rotation~$\fr_{F(\bx)}$ from~$\fr_\bx$ by  taking first return maps on natural subtiles of $\cS$-adic Rauzy fractals for strongly convergent multidimensional continued fraction maps~$F$; see Section~\ref{sec:induced}. Thus, the continued fraction map $F$ renormalizes rotations of the torus by a hyperbolic system. This situation generalizes the way the Gauss map renormalizes rotations of the circle and solves Problem~\nameref{p:I}. As explained in \cite[Introduction]{Forni}, this renormalization allows one to understand the dynamics of a generic system in a given family. 

We want to relate sequences of matrices produced by multidimensional continued fraction algorithms to mapping families of the form~\eqref{eq:star2}. Since mapping families are defined in terms of bi-infinite sequences of matrices, we need to equip the orbit $(M_n)_{n\in\NN}$ with a ``past''. This can be done by considering its \emph{natural extension}, which is discussed in Section~\ref{sec:natex}. Indeed, the natural extension of a multidimensional continued fraction algorithm generates orbits of the form $\bM = (M_n)_{n\in\ZZ}$. We can now use these orbits to define linear Anosov mapping families of the form~\eqref{eq:star2}.  We then use $\cS$-adic Rauzy fractals and their suspensions to construct the desired nonstationary Markov partitions and thereby solve Problem~\nameref{p:II}; see Chapter~\ref{chapter:markov}. 
Using these nonstationary Markov partitions, we are able to associate symbolic models as nonstationary edge and vertex shifts (see Definition~\ref{def:nfst}) to such sequences, extending the well-known construction for a single toral automorphism; see e.g.\ \cite{A98}.  

More generally, our approach via symbolic dynamics is motivated by the search for symbolic models for automorphisms of  compact abelian groups, which elicited a wide literature. Algebraic constructions were initiated by Vershik for hyperbolic toral automorphisms with the use of homoclinic points; see \cite{Vershik:92,KV:98,SV:98}.  Among these symbolic models, the one associated  with Markov partitions are particularly interesting. Note also that the present approach is reminiscent of the generating method of fractal curves due to Dekking \cite{Dekking:82}, used in \cite{Bedford:86,IO:93} to produce Markov partitions with fractal boundaries for a class of hyperbolic toral automorphism on~$\TT^3$.

Besides that, with our theory we want to set the stage for viewing multidimensional continued fraction algorithms as Poincar\'e sections of flows in the same spirit as this was done in \cite{AF:01} for the classical continued fraction algorithm. This is evoked in Chapter~\ref{sec:conclude}.

\section{Discussion of the main results}\label{sec:introresults}
Most of our main results need preparations and definitions. For this reason, their exact statements will be given later. However, for the convenience of the reader, we provide informal versions of these results already here.

We use the  notation
\[
\cM_d = \{M \in \NN^{d\times d} \,:\, |\!\det M| = 1\} \notx{Ma}{$M,M_n$}{matrix}
\]
for the set of square matrices of determinant $\pm 1$ with nonnegative integer entries.
We consider a multidimensional continued fraction map~$F$ acting on a subset of the projective space~$\PP^{d-1}$ that produces sequences $(M_n)_{n\in\NN}\in\cM_d^\NN$\notx{Mb}{$\bM$}{bi-infinite sequence of matrices} of matrices. 
When considering (a geometric realization of) the natural extension~$\hF$ of the continued fraction map~$F$, even \emph{bi-infinite} sequences $\bM = (M_n)_{n\in\ZZ}\in\cM_d^\ZZ$ of matrices are produced. These sequences are often \emph{primitive} in the sense that products $M_k\cdots M_{\ell-1}$ of large blocks of matrices are positive; see Definition~\ref{def:primivite} for details.
In the theory of multidimensional continued fraction algorithms, it is desirable that the produced sequences of matrices~$\bM$ have good convergence properties. These convergence properties are efficiently described in terms of negativity of the second Lyapunov exponent; see e.g.\ \cite{Lagarias:93,Schweiger:00}. We thoroughly study such convergence properties  even if the Lyapunov exponents of a single sequence~$\bM$ do not exist by providing a weaker condition, stating that the second largest singular value $\delta_2(M_{0}\cdots M_{n-1})$ of $M_{0}\cdots M_{n-1}$  tends to zero exponentially fast; see Definition~\ref{def:genPisot}. This condition turns out to be much easier to verify.  We then  superimpose a  substitutive realization of a sequence $\bM = (M_n)_{n\in\ZZ}\in\cM_d^\ZZ$  which gives a sequence of substitutions $\bsigma = (\sigma_n)_{n\in\ZZ}$; see Section~\ref{subsec:realization} for details. Sequences of substitutions are of interest for us as they are needed to define $\cS$-adic Rauzy fractals, the central objects required to obtain induced rotations as well as nonstationary Markov partitions. Good convergence properties entail that these $\cS$-adic Rauzy fractals are bounded, which will be of great importance in our theory. 

\subsection{Convergence results}
We provide results  for single sequences of matrices and single sequences of substitutions, as well as metric results that hold  in the setting of   measure-theoretic dynamical systems of the form $(D,\Sigma,\nu)$, where $\Sigma$ stands for the shift and $D$ is a shift invariant subset of the set  of sequences of unimodular $d{\times}d$ matrices~$\cM_d^{\ZZ}$. The following theorem concerning single sequences of matrices is stated as part of Theorem~\ref{th:matrixPisot}; note that in this theorem only the future of the bi-binfinite sequence $\bM = (M_n)_{n\in\ZZ}$ is relevant.  This theorem is a first step toward the proof of the existence of a hyperbolic splitting; see Section~\ref{subsec:resultsMarkov}.

\begin{namedthm}{Theorem A} 
\makeatletter\def\@currentlabelname{A}\makeatother
\label{t:A}
Let $\bM = (M_n)_{n\in\ZZ}$ be a primitive sequence of matrices in~$\cM_d$ for which the second largest singular value $\delta_2(M_{0}\cdots M_{n-1})$ of $M_{0}\cdots M_{n-1}$ tends to zero exponentially fast in the sense that $\limsup_{n\to\infty} \frac{1}{n} \log \delta_2(M_{0}\cdots M_{n-1}) < 0$. 
Then, under the mild growth condition $\lim_{n\to\infty} \frac{1}{n} \log \lVert M_n\rVert = 0$,  the sequence~$\bM$ converges exponentially fast to a vector with rationally independent entries.
\end{namedthm}

The fact that only the $\limsup$ is required to be less than zero (i.e., the limit and, hence, the second Lyapunov exponent, need not exist) makes this condition checkable and we are able to use this theorem to construct concrete sequences of matrices that admit exponentially fast convergence; see e.g.\ Example~\ref{ex:AR1} or Propositions~\ref{prop:brunpisot:2} and~\ref{prop:brunpisot:1}. 

Because the $\limsup$ condition ensures that the sequence of matrices~$\bM$ grows only in one direction, it is reminiscent of the growth behavior of powers of a matrix having a characteristic polynomial that is the minimal polynomial of a Pisot number. For this reason, the $\limsup$ condition will be called  \emph{local Pisot condition} (for a single sequence of matrices); see Definition~\ref{def:genPisot}. In terms of symbolic dynamics, when imposing a symbolic structure upon sequences of matrices with substitutions, the analog of Theorem~\nameref{t:A} is stated as Theorem~\ref{theo:sufcondPisotS}.  In this case, the local Pisot condition implies on top of the exponentially fast convergence also the combinatorial property of \emph{balance} (see Definition~\ref{def:balance}), which can also be phrased as \emph{bounded symbolic discrepancy} (see Section~\ref{sec:combprop}).

Besides that, Lyapunov exponents are studied from a metric and ergodic viewpoint for measure-theoretic dynamical systems of the form $(D,\Sigma,\nu)$, where $\Sigma$ stands for the shift and $D$ is a shift invariant subset of the set  of sequences of unimodular $d{\times}d$ matrices~$\cM_d^{\ZZ}$. We say that such a dynamical system satisfies the \emph{metric version of the Pisot condition} (see Definition~\ref{def:gengenPisot2}) if its second Lyapunov exponent is less than zero. Similarly as Theorem~\nameref{t:A} does for  single sequences of matrices, this condition implies exponentially fast convergence for a generic sequence $\bM = (M_n)_{n\in\ZZ} \in D$, including strong convergence and a certain irreducibility property; see Theorem~\ref{thm:oseledetsmat}. The metric theory for sequences of substitutions is covered in Section~\ref{subsec:balanceSub}, see in particular, Theorem~\ref{thm:oseledetssub}.

In Theorem~\nameref{t:B} below, we use the local Pisot condition to get a linear Anosov mapping family that can be regarded as a nonstationary version of a hyperbolic toral automorphism.
In particular, the underlying hyperbolic splitting is ``Pisot'' in the sense that the stable space has dimension~$1$, and the unstable space co-dimension~$1$.
Theorem~\nameref{t:B} will be made precise in Theorem~\ref{cor:anosov}. 

\begin{namedthm}{Theorem B} 
\makeatletter\def\@currentlabelname{B}\makeatother
\label{t:B}
Let $\bM = (M_n)_{n\in\ZZ} \in \cM_d^{\ZZ}$ be a primitive sequence of matrices satisfying the local Pisot condition. Then the sequence of linear toral automorphisms
\[  
\cdots \xrightarrow{M_{-2}^{-1}} \TT_{-1}\xrightarrow{M_{-1}^{-1}}  \TT_{0}\xrightarrow{M_{0}^{-1}} \TT_1\xrightarrow{M_{1}^{-1}}  \cdots 
 \]
associated to $\bM$ is a linear Anosov mapping family. In particular, for each $n \in \ZZ$, in the tangent bundle of each $d$-dimensional torus~$\TT_n$, there are directions $\bu_n$ and~$\bv_n$ satisfying
\[
\RR_{\ge0} \bu_n = \bigcap_{k\ge n} M_{n}\cdots M_{k-1}\, \RR_{\ge0}^d \quad \text{ and } \quad
\RR_{\ge0} \bv_n = \bigcap_{k< n} \tr{(M_k\cdots M_{n-1})}\, \RR_{\ge0}^d,
\]
such that the mappings $M_n^{-1} \cdots M_m^{-1}$ \emph{contract}~$\bu_n$ and \emph{expand}~$\bv_n^\perp$ for large $m{-}n$.
 \end{namedthm}

Also this theorem has a metric version, see again Theorem~\ref{thm:oseledetsmat}, in which the generic hyperbolic splitting is related to the Oseledets splitting induced by a certain cocycle which is provided by Oseledets' Multiplicatice Ergodic Theorem. Both versions of Theorem~\nameref{t:B}, the one for single sequences of matrices and the metric one, can also be rephrased for sequences of substitutions, see Theorem~\ref{th:anosovS} and, once again, Theorem~\ref{thm:oseledetssub}.

\subsection{Pisot conjecture, pure discrete spectrum, and induced rotations}
In Section~\ref{sec:rauzy}, we review the theory of $\cS$-adic Rauzy fractals and Rauzy boxes attached to a sequence $\bsigma$ of substitutions over $d$ letters. Under mild assumptions, $\cS$-adic Rauzy fractals are fundamental domains for the action of rotations on~$\TT^{d-1}$, which can be also stated as a tiling property: $\cS$-adic Rauzy fractals periodically tile the $(d{-}1)$-dimensional plane; see Definition~\ref{def:tilingcond} for a precise statement in the present context. This \emph{tiling condition} has been studied in \cite{BST:19,BST:23} in the $\cS$-adic setting,  and its validity is conjectured in an even wider scope. The according conjecture is strongly related to an $\cS$-adic generalization of the well-known \emph{Pisot conjecture}; see e.g.~\cite{AkiBBLS}. It says that each $\cS$-adic dynamical system satisfying a Pisot property has pure discrete spectrum, in other words, it is measurably conjugate to a rotation on a compact abelian group.
This important conjecture is formulated below as Conjecture~\ref{c:SadicPisotconj}. Many partial results exist for the substitutive case (see e.g.\ \cite{HolSol:03,AkiBBLS,Barge:16,Barge15}). Results on the $\cS$-adic case are contained in \cite{BST:19,BST:23}, and we provide slight variants of these results in Theorems~\ref{t:tilingpds},~\ref{prop:combcond2}, and~\ref{theo:metrictilingaccel}.

If we assume that a given sequence $\bsigma$ of substitutions gives rise to a dynamical system $(X_{\bsigma},\Sigma)$ that has pure discrete spectrum, the associated rotation can be ``seen'' on the $\cS$-adic Rauzy fractal $\cR_{\bsigma}$ of $\bsigma$. It turns out that if we induce this rotation on a shrunk version of this $\cS$-adic Rauzy fractal, we obtain again a rotation. Iterating this induction process leads to a cascade of rotations that allows the interpretation of a multidimensional continued fraction algorithm as an action on a family of torus rotations; see Section~\ref{sec:induced}. This is analogous to the situation of the classical continued fraction algorithm discussed in Section~\ref{sec:introsturm}. The following is an informal statement of Theorem~\ref{t:Frot}. The tiling condition below refers to Definition~\ref{def:tilingcond}.

\begin{namedthm}{Theorem C}
\makeatletter\def\@currentlabelname{C}\makeatother
\label{t:C}
Let $\bsigma=(\sigma_n)_{n\in\ZZ}$ be a primitive sequence of substitutions. If the $\cS$-adic Rauzy fractal of $\bsigma$ satisfies a certain tiling condition, then we can associate a sequence of rotations $(\fr_{\Sigma^n\bsigma})_{n\in\ZZ}$ to $\bsigma$ such that each $\fr_{\Sigma^n\bsigma}$, $n\in \NN$, is the rotation obtained by inducing the rotation $\fr_{\bsigma}$ on the subset $M_{\sigma_0}\cdots M_{\sigma_{n-1}}\cR_{\Sigma^n\bsigma}$ of~$\cR_{\bsigma}$.
\end{namedthm}

In Section~\ref{sec:induced} we also state a metric version of this result, as well as versions for multidimensional continued fraction algorithms.

\subsection{Results on nonstationary Markov partitions}\label{subsec:resultsMarkov}
In Sections~\ref{sec:markov} and~\ref{sec:metricMP}, we provide the core results motivating  the present work by exhibiting ``nonstationary Markov partitions'' and symbolic models as ``nonstationary edge shifts'' for the hyperbolic linear Anosov mapping families \eqref{eq:star2} of a sequence $\bM\in\cM^\ZZ$. The pieces of the corresponding generating nonstationary Markov partitions are obtained by associating a sequence~$\bsigma$ of substitutions to~$\bM$. More precisely, these pieces are the Rauzy boxes, defined as suspensions of the $\cS$-adic Rauzy fractals of the sequence $\bsigma = (\sigma_n)_{n\in\ZZ}$. 
In particular, the Rauzy fractals are given by the ``future'' $(\sigma_0,\sigma_1,\dots)$ and the heights of the suspensions are given by the ``past'' $(\sigma_{-1},\sigma_{-2},\dots)$.
For an exact statement of the next theorem, we refer to Theorem~\ref{thm:MarkovFine}, which is a refinement in terms of generating partitions of Theorem~\ref{thm:MarkovCoarse}.  
Again, the tiling condition below refers to Definition~\ref{def:tilingcond}.

\begin{namedthm}{Theorem D}
\makeatletter\def\@currentlabelname{D}\makeatother
\label{t:D}
Let $\bM = (M_n)_{n\in\ZZ}  \in \cM_d^{\ZZ}$ be a primitive sequence of matrices with a  superimposed substitutive structure~$\bsigma$. If the $\cS$-adic Rauzy fractals of $\bsigma$ satisfy a certain tiling condition, then the associated linear Anosov mapping family 
\[  
\cdots \xrightarrow{M_{-2}^{-1}} \TT_{-1}\xrightarrow{M_{-1}^{-1}}  \TT_{0}\xrightarrow{M_{0}^{-1}} \TT_1\xrightarrow{M_{1}^{-1}}  \cdots 
 \]
admits a generating nonstationary Markov partition, whose atoms are explicitly given by Rauzy boxes associated to the substitutive structure~$\bsigma$. This Markov partition also provides a symbolic model for this Anosov mapping family as a nonstationary edge shift. 
\end{namedthm}
 
Theorem~\nameref{t:E}  below provides a metric version of Theorem~\nameref{t:D} in which the crucial condition is again the generic Pisot condition. Indeed, we will show that the generic Pisot condition implies the tiling condition almost everywhere under mild additional requirements. The exact statement of Theorem~\nameref{t:E} is  contained in Theorem~\ref{theo:metricmarkovM}.

\begin{namedthm}{Theorem E}
\makeatletter\def\@currentlabelname{E}\makeatother
\label{t:E}
Let $(D,\Sigma,\nu)$, with $D \subset \cM_d^{\ZZ}$, be an ergodic dynamical system.
Suppose that $(D,\Sigma,\nu)$ satisfies the Pisot condition, and that there exists a single substitutive dynamical system that has pure discrete spectrum that corresponds to a periodic sequence in~$D$. Then, for $\nu$-almost every $\bM = (M_n)_{n\in\ZZ}  \in D$, the associated sequence
\[  
\cdots \xrightarrow{M_{-2}^{-1}} \TT_{-1}\xrightarrow{M_{-1}^{-1}}  \TT_{0}\xrightarrow{M_{0}^{-1}} \TT_1\xrightarrow{M_{1}^{-1}}  \cdots 
 \]
is a linear Anosov mapping family that admits a generating nonstationary Markov partition, whose atoms are explicitly given by Rauzy boxes. This Markov partition also provides a symbolic model for this Anosov mapping family as a nonstationary edge shift. 
\end{namedthm}

Theorems~\nameref{t:D} and~\nameref{t:E} can be formulated also for sequences of substitutions, see for instance Theorem~\ref{theo:metricmarkov}. A version of Theorem~\nameref{t:E} for multidimensional continued fraction algorithms is stated as Theorem~\ref{theo:FC}. 

The following Theorem~\nameref{t:F} provides an ``accelerated'' version of Theorem~\nameref{t:E}.  We state this theorem in the framework of multidimensional continued fraction algorithms. In particular, if $(X,F,\nu)$ is a positive $(d{-}1)$-dimensional continued fraction algorithm satisfying the Pisot condition, then it can be ``accelerated'' in such a way that the generated Markov partition exists almost everywhere without further conditions. By an acceleration we mean that, instead of applying $F$ to an element of $X$, we apply $F^k$ for some $k \in \NN$ that has to be chosen appropriately. This is a classical approach; see for instance \cite{Yoccoz:2005} for accelerations of Rauzy induction in the setting of interval exchange transformations. The reason is that such an acceleration gives us matrices with larger entries and, hence, more freedom in choosing substitutions for the construction of the required Markov partition. Summing up, Theorem~\nameref{t:F} provides Markov partitions under the Pisot condition without further combinatorial conditions. The exact statement for Theorem~\nameref{t:F} is given in Theorem~\ref{theo:FCAcc}, a~version for sequences of matrices is stated as Theorem~\ref{theo:metricmarkovAcc}. 

\begin{namedthm}{Theorem F}
\makeatletter\def\@currentlabelname{F}\makeatother
\label{t:F}
Let $(X,F,\nu)$ be a positive $(d{-}1)$-dimensional continued fraction algorithm satisfying the Pisot condition, and let $(\hX,\hF,\hnu)$ be a natural extension of $(X,F,\nu)$. 
Then there exists an acceleration $F^k$ of $F$ such that the cocycle of the acceleration $(\hX,\hF^k,\hnu)$ generates sequences of matrices $\bM = (M_n)_{n\in\ZZ}  \in \cM_d^{\ZZ}$ that $\hnu$-almost always give rise to an  Anosov mapping family
\[  
\cdots \xrightarrow{M_{-2}^{-1}} \TT_{-1}\xrightarrow{M_{-1}^{-1}}  \TT_{0}\xrightarrow{M_{0}^{-1}} \TT_1\xrightarrow{M_{1}^{-1}}  \cdots 
\]
admitting a generating nonstationary Markov partition, with explicitly given Rauzy boxes as atoms. This Markov partition also provides a symbolic model for this Anosov mapping family as a nonstationary edge shift. 
\end{namedthm}

In the previous theorems, we needed two conditions.  Firstly, the generic Pisot condition, which is formulated in terms of negativity of the second Lyapunov exponent. Secondly, only in Theorem~\nameref{t:E}, we require that our set of $\cS$-adic sequences contains a substitutive dynamical system that has pure discrete spectrum. These conditions suffice to show that the tiling condition holds true almost everywhere and, hence, enable us to derive Theorems~\nameref{t:E} and~\nameref{t:F} from Theorem~\nameref{t:D}. Since there exist many algorithms for checking pure discrete spectrum for a single unimodular Pisot substitution (see e.g.~\cite{Livshits:87,Livshits:92,SS:02,CANTBST,AL11}), the only crucial condition is the Pisot condition. 

\subsection{The Brun algorithm, our running example}
The Brun algorithm, which will be discussed in detail in Section~\ref{sec:Brun}, serves as our running example.
It is a generalization of the Farey continued fraction algorithm to higher dimensions.
For a given vector in~$\RR_{\ge0}^d$, it produces a new vector by subtracting the second largest coordinate from the largest one.
We will consider versions of the Brun algorithm where the coordinates of the vector are ordered or unordered.
We also study the multiplicative acceleration known as modified Jacobi--Perron algorithm, which generalizes the classical continued fraction algorithm.
For all these versions of the Brun algorithm, the Pisot condition holds in dimensions $d{-}1 = 2$ and $d{-}1 = 3$; see Proposition~\ref{prop:Brun23Pisot}.

We mainly focus on the Anosov mapping family on the $3$-dimensional tori~$\TT^3$ associated to the $2$-dimensional Brun continued fraction algorithm, but also discuss the $4$-dimensional case (related to the $3$-dimensional Brun algorithm). With help of the Brun algorithm we are able to construct Anosov mapping families with generating Markov partitions for hyperbolic foliations in dimension~$3$ and~$4$ as detailed in the following realization result; see Corollary~\ref{cor:foliMarkov} for the exact statement.
The restriction on  the dimension comes from the  fact that  these   results  require,   for the Pisot condition to hold,   the  negativity of the second Lyapunov exponent, which is only established for small dimensions.
Observe that numerical results from \cite{BST21} indicate that classical multidimensional continued fraction algorithms cease to be strongly convergent for high dimensions; cf.\ Conjecture~\ref{conj:convfc}.

\begin{namedthm}{Corollary G}
\makeatletter\def\@currentlabelname{G}\makeatother
\label{c:G}
For almost every pair of subspaces $(G^s,G^u)$ with $G^s \oplus G^u = \RR^3$, there is a linear Anosov mapping family \eqref{eq:star2} with stable foliation~$G^s$ and unstable foliation~$G^u$ admitting a generating Markov partition and, hence, a symbolic model as a nonstationary edge shift.

For almost every pair of subspaces $(G^s,G^u)$ with $G^s \oplus G^u = \RR^4$ and $\dim(G^s) \,{\neq}\, 2$, there is a linear Anosov mapping family \eqref{eq:star2} with stable foliation~$G^s$ and unstable foliation~$G^u$ admitting a generating Markov partition and, hence, a symbolic model as a nonstationary edge shift.
\end{namedthm}

The sequences $\bM = (M_n)_{n\in\ZZ}\in\cM^\ZZ$ of matrices defining the claimed Anosov mapping families in this corollary correspond to expansions of pairs of vectors in the natural extension of the Brun continued fraction algorithm.
Therefore, Corollary~\nameref{c:G} is formulated only for splittings where one of the foliations has dimension~$1$.

In the cases covered by Corollary~\nameref{c:G}, it makes sense to speak of a Markov partition of a hyperbolic foliation $(G^s,G^u)$ because such a foliation gives rise to a sequence of matrices to which we can associate a Markov partition. This is in the spirit of \cite[Introduction]{AF:05}. If one looks at hyperbolic toral automorphisms in $\RR^3$ and~$\RR^4$, the slopes of the stable and unstable manifolds can only be cubic and quartic algebraic integers, respectively. What we do here completes the picture in the sense that (almost all) arbitrary slopes are covered when we turn to Anosov mapping families. 
Moreover, in  this context, for our realization result from Corollary~\nameref{c:G}, it is important that there exists a set of positive Lebesgue measure on which the natural extension of the underlying continued fraction algorithm is defined; see Remark~\ref{rem:nu>0}. This is crucial for  being able to cover almost all slopes. 
For the Brun continued fraction algorithm, positivity of the measure is proved in \cite{Arnoux-Nogueira}, and this motivates choosing this algorithm as the main example for our realization results.

\section{Outline} \label{sec:introoutline}
Let us briefly sketch the contents of this monograph. Chapter~\ref{chapter:matrix} is organized as follows. Section~\ref{sec:mapfam} recalls general results on Anosov mapping families and on nonstationary Markov partitions. In Section~\ref{sec:matrices}, convergence properties for bi-infinite sequences of matrices under the Pisot condition are investigated. These results are revisited in Section~\ref{sec:metricmat} from a metric viewpoint by considering shifts of sequences of matrices endowed with an invariant measure.

Chapter~\ref{sec:cf} provides definitions and results concerning  (unimodular) multidimensional continued fraction algorithms.
General results concerning their natural extension, their dual versions, their invariant measures, and their associated linear Anosov mapping families are given in Section~\ref{sec:cfgentheory}. 
Section~\ref{sec:Brun} aims at applying this formalism and the previous results to the Brun algorithm, which will be used as a running example in the monograph, in its additive (ordered and unordered) and multiplicative form.  

Chapter~\ref{chapter:substitution} deals with sequences of substitutions. Section~\ref{sec:subs}  develops symbolic dynamics for a single bi-infinite sequence of substitutions. Here, properties of the incidence matrices of the substitutions play a key role. Rauzy fractals and Rauzy boxes are defined in Section~\ref{sec:rauzy}, where also their tiling properties and related pure discrete spectrum results are thoroughly discussed. Section~\ref{sec:metricS} covers the metric theory of bi-infinite sequences of substitutions and gives metric conditions on tiling properties of Rauzy fractals and Rauzy boxes as well as metric results on pure discrete spectrum.

Chapter~\ref{chapter:markov} is devoted to the construction of induced rotations an nonstationary Markov partitions for sequences of matrices, sequences of substitutions, and multidimensional continued fraction algorithms. The theory of induced rotations is covered in Section~\ref{sec:induced}.
Section~\ref{sec:markov} shows how to use Rauzy boxes in order to define nonstationary Markov partitions for $\cS$-adic linear Anosov mapping families. In Section~\ref{sec:metricMP}, a metric theory of nonstationary Markov partitions for Anosov mapping families and multidimensional continued fraction algorithms is laid out. This section ends with a thorough discussion of our standard example, the Brun algorithm, in Section~\ref{subsec:BrunS}. We apply our theory to this algorithm and give the exact statement and proof of Corollary~\nameref{c:G}.

We conclude this monograph with Chapter~\ref{sec:conclude}, which contains a discussion of perspectives and directions on further research. This includes a list of open problems and research questions, viewing the natural extension of a multidimensional continued fraction algorithm as a Poincar\'e section of a flow, and the discussion of the interplay between different dynamical systems that played a major role in the present work.

\chapter{Sequences of integer matrices and mapping families}
\label{chapter:matrix}

In this chapter, we deal with properties of bi-infinite sequences of (nonnegative) matrices that are relevant for our theory of nonstationary Markov partitions. We organize these sequences into linear mapping families. If certain convergence properties hold, these linear mapping families turn out to be Anosov. We start in Section~\ref{sec:mapfam} with a discussion of mapping families with special emphasis on Anosov mapping families, and introduce the concept of nonstationary Markov partition for these objects. Section~\ref{sec:matrices} is devoted to bi-infinite sequences of matrices. Our main aim is to introduce notions of convergence of such sequences and to give criteria to check convergence. In this context, the so-called Pisot condition (see Definition~\ref{def:genPisot}) plays a major role. This Pisot condition can be viewed as a ``codimension one'' Anosov property of a mapping family associated to a bi-infinite sequence of matrices in the sense that the splitting of the tangent space gives   contracting spaces with dimension~$1$, and thus expanding space with codimension~$1$. Section~\ref{sec:metricmat} finally gives metric results for bi-infinite sequences of matrices and the associated mapping families. Our main tool here will be Oseledets' Multiplicative Ergodic Theorem.

\section{Mapping families} \label{sec:mapfam}
In this section, we define basic objects that will play an important role throughout the manuscript. Section~\ref{subsec:mapf} is devoted to the concept of a \emph{mapping family} which forms the framework for our nonstationary dynamics. It first appears in \cite{AF:05}. In Section~\ref{sec:MarkovTheory} we introduce the notion of a nonstationary Markov partition of a mapping family and define the so-called Property~M, a checkable criterion for a sequence of partitions to form a nonstationary Markov partition. In Section~\ref{subsec:nsft}, using Bratteli diagrams, we define nonstationary edge and vertex shifts. By using nonstationary Markov partitions, these shifts will be shown later to serve as a symbolic representations of a mapping family. 

\subsection{Eventually Anosov mapping families} \label{subsec:mapf}
Consider a bi-infinite sequence of $d$-dimensional Riemannian manifolds $(X_n)_{n\in\ZZ}$ with uniformly bounded diameters.
Let $X = \coprod_{n\in\ZZ} X_n$\notx{0union}{$\amalg$}{disjoint union} be the disjoint union of these spaces equipped with a metric whose restriction to~$X_n$ equals the metric on~$X_n$. 
For each $n \in \ZZ$, let $f_n: X_n \to X_{n+1}$ be a $C^1$-diffeomorphism and define $f: X \to X$ by $f(x) = f_n(x)$ for $x \in X_n$. 
We call the pair $(X,f)$ a \emph{mapping family}\indx{mapping family}. 
We will sometimes write out $(X,f)$ in the form
\begin{equation}\label{eq:mappingfamilydefinition}
\cdots \xrightarrow{f_{-2}} X_{-1} \xrightarrow{f_{-1}}  X_0\xrightarrow{\;f_0\;} X_1\xrightarrow{\;f_1\;}  \cdots . 
\end{equation}

Let $TX$ be the tangent space of~$X$, $T_pX \cong \RR^d$ its fiber over $p \in X$.
The derivative of~$f$ (defined on~$TX$) will be denoted by~$D(f)$\notx{D}{$D(\cdot)$}{derivative}. 

In the sequel, for an unspecified norm, we will use the notation $\lVert\cdot\rVert$\notx{0norm}{$\lVert\cdot\rVert$}{arbitrary norm}. The following definition of an \emph{eventually Anosov mapping family}\indx{mapping family!Anosov} is inspired by \cite[Definition~2.10]{AF:05}. 

Let $(X,f)$ be a mapping family.  Assume that there exists an $f$-invariant splitting\indx{splitting!invariant} $G^s \oplus G^u$ of the tangent space $TX$ of~$X$, and set $G^s_p = G^s \cap T_pX$, $G^u_p = G^u \cap T_pX$ for $p\in X$. \notx{Gs}{$G^s,G_p^s,G_n^s$}{stable subspace of a hyperbolic splitting} \notx{Gu}{$G^u,G_p^u,G_n^u$}{unstable subspace of a hyperbolic splitting}
If, for all $p\in X$,
\begin{enumerate}[\upshape (i)]
\itemsep.5ex
\item \label{i:anosov2}
$\lim\limits_{n\to+\infty}  \sup\{ \lVert D(f^n)\,\bx\rVert / \lVert \bx\rVert \,:\, \bx \in G^s_p \setminus \{\mathbf{0}\} \} = 0$,
\item \label{i:anosov1}
$\lim\limits_{n\to+\infty}  \inf \{ \lVert D(f^n)\,\bx\rVert / \lVert \bx\rVert \,:\, \bx \in G^u_p \setminus \{\mathbf{0}\} \} = +\infty$,
\item \label{i:anosov4} 
$\lim\limits_{n\to-\infty}  \inf\{ \lVert D(f^n)\,\bx\rVert / \lVert \bx\rVert \,:\, \bx \in G^s_p \setminus \{\mathbf{0}\} \} = +\infty$,
\item \label{i:anosov3}
$\lim\limits_{n\to-\infty}  \sup\{ \lVert D(f^n)\,\bx\rVert / \lVert \bx\rVert \,:\, \bx \in G^u_p \setminus \{\mathbf{0}\} \} = 0$,
\end{enumerate}
we say that $(X,f)$ is \emph{(two-sided) eventually Anosov}. 

\begin{remark}
For an eventually Anosov mapping family, none of the conditions (\ref{i:anosov2})--(\ref{i:anosov3}) follows from the others; see for instance \cite[Remark~2.11 and Example~9]{AF:05}. If just (\ref{i:anosov2}) and~(\ref{i:anosov1}) hold, then $(X,f)$ is ``eventually Anosov in the future'', and if only (\ref{i:anosov4}) and~(\ref{i:anosov3}) are in force, then it is ``eventually Anosov in the past''. 
However, we do not need these one-sided versions in the present monograph. 

For the more specialized concept of \emph{strictly Anosov mapping family} defined in \cite[Definition~2.7]{AF:05}, where the contraction and expansion ratios are uniformly bounded away from~$1$,  the $f$-invariance implies that contraction for $n \to {+}\infty$ is equivalent to expansion for $n \to {-}\infty$ and vice versa. 
Here, we use eventually Anosov rather than strictly Anosov because the latter property is not maintained by most sequences generated by multidimensional continued fraction algorithms.
\end{remark}

\begin{remark}
In an eventually Anosov mapping familiy $(X,f)$, a single function $f_n$, $n \in \ZZ$, may well be nonhyperbolic. This is illustrated for instance by the generators of $2{\times}2$ integer matrix with determinant $1$ related to the additive continued fraction algorithm, see~\cite[Example~6]{AF:05} and Example~\ref{ex:classicalMappingFamily} below. 
Also, even if $f_n$ is hyperbolic, the associated stationary splitting of the tangent space~$TX_n$ may be different from the (local) splitting of~$TX_n$ induced by the (global) splitting $G^s \oplus G^u$ of~$TX$.
\end{remark}

\begin{example}\label{ex:classicalMappingFamily}
Let $(a_n)_{n\in\ZZ}$ be a bi-infinite sequence of positive integers. To this sequence, we can attach the sequence of matrices $(M_n)_{n\in\ZZ}$ with
\[
M_n=\begin{pmatrix} 
0&1\\1&a_n
\end{pmatrix} \qquad(n\in\ZZ).
\]
Since the matrices $M_n$ have determinant $-1$, they can be regarded as automorphisms on the torus $\TT^2$. Therefore they give rise to the mapping family
\begin{equation}\label{eq:classicalMF}
\cdots\xrightarrow{M_{-2}^{-1}} \TT^2\xrightarrow{M_{-1}^{-1}} \TT^2\xrightarrow{M_{0}^{-1}} \TT^2\xrightarrow{M_{1}^{-1}} \cdots,
\end{equation}
where $X_n=\TT^2$ and $f_n=M_n^{-1}$ for each $n\in\ZZ$. This mapping family is related to the classical continued fraction algorithm. Indeed, write
\[
[a_0;a_1,a_2,a_3\ldots] = a_0+\cfrac1{a_1+\cfrac1{a_2 +\cfrac1{a_3 + \ddots}}}.
\medskip
\] 
Then, because the classical continued fraction algorithm has good convergence properties, the mapping family \eqref{eq:classicalMF} is Anosov with the hyperbolic splitting
\[
G^s \oplus G^u = \bigg(\RR \binom{1}{[a_n;a_{n+1},a_{n+2},\ldots]} \oplus \RR\binom{[a_{n-1};a_{n-2},a_{n-3},\ldots]}{-1}
\bigg)_{n\in\ZZ}.
\]
In view of generalizations to higher dimensions, note that $\RR([a_{n-1};a_{n-2},\ldots],-1)$ is the subspace of~$\RR^2$ orthogonal to the vector $(1,[a_{n-1};a_{n-2},\ldots])$.
\end{example}

\subsection{Nonstationary Markov partitions}\label{sec:MarkovTheory}
We shall now define Markov partitions for mapping families. To this matter, we will generalize the classical stationary Markov partitions\indx{Markov partition!stationary} discussed in Section~\ref{sec:defMarkov}; see also \cite[Definition~6.1]{A98}. We first recall that a \emph{topological partition}\indx{partition!topological} of a compact topological space~$R$ is a finite collection $\{R_1,\dots,R_q\}$ of pairwise disjoint open sets satisfying $\overline{R_1} \cup \cdots \cup \overline{R_q} = X$. Using this terminology, we are able to state the following definition (compare the analogy to the stationary case in Section~\ref{sec:defMarkov}), which  roughly speaking states that  for any finite pseudo-orbit, there exists a finite orbit with the same symbolic dynamics.
All definitions of this section will be illustrated in Example~\ref{ex:NMexpl2}.

\begin{definition}[Nonstationary Markov partition]\label{nsMarkov}\indx{Markov partition!nonstationary}\indx{Markov partition!nonstationary!generating}
Let $(X,f)$ be the mapping family 
\[
\cdots \xrightarrow{f_{-2}} X_{-1} \xrightarrow{f_{-1}}  X_0\xrightarrow{\;f_0\;} X_1\xrightarrow{\;f_1\;}  \cdots . 
\]
For each $n \in \ZZ$, let $\cP_n = \{R_{n,1},\dots,R_{n,q_n}\}$ be a topological partition of~$X_n$. The sequence $(\cP_n)_{n\in\ZZ}$\notx{Pa}{$\cP_n$}{nonstationary Markov partition} is called a \emph{(nonstationary) Markov partition} for $(X,f)$ if, for any sequence $(i_n)_{n\in\ZZ} \in \prod_{n\in\ZZ} \{1,\dots,q_n\}$,  for each $n \in \ZZ$ and each $m \ge 2$, we have that
\[
R_{k,i_k} \cap f^{-1} (R_{k+1,i_{k+1}}) \neq \emptyset \quad \mbox{for all}\ k\in\{n,\dots,n{+}m{-}1\}
\]
implies that
\[
R_{n,i_n} \cap f^{-1} (R_{n+1,i_{n+1}}) \cap \cdots \cap f^{-m}(R_{n+m,i_{n+m}}) \neq \emptyset.
\]
The nonstationary Markov partition $(\cP_n)_{n\in\ZZ}$ is called \emph{generating} if, for any sequence $(i_n)_{n\in\ZZ} \in \prod_{n\in\ZZ} \{1,\dots,q_n\}$,
\[
\bigcap_{n\in\NN}\overline{\bigcap_{|k|<n}  f^{-k}(R_{k,i_{k}})}
\]
contains at most one point.
\end{definition}

\begin{remark}\label{rem:nsMarkovExplain}
Note that, by the definition of the mapping family $(X,f)$, the mapping $f:X\to X$ equals $f_n$ on $X_n$. Thus Definition~\ref{nsMarkov} can be written out in the following way: A~sequence $(\cP_n)_{n\in\ZZ}$, with $\cP_n = \{R_{n,1},\dots,R_{n,q_n}\}$, is called a \emph{(nonstationary) Markov partition} for $(X,f)$ if, for any sequence $(i_n)_{n\in\ZZ} \in \prod_{n\in\ZZ} \{1,\dots,q_n\}$,  for each $n \in \ZZ$ and each $m \ge 2$, we have that
\[
R_{k,i_k} \cap f_k^{-1} (R_{k+1,i_{k+1}}) \neq \emptyset \quad \mbox{for all}\ k\in\{n,\dots,n{+}m{-}1\}
\]
implies that
\[
R_{n,i_n} \cap f_n^{-1} (R_{n+1,i_{n+1}}) \cap \cdots \cap (f_n^{-1} \circ \cdots \circ f_{n+m-1}^{-1}) (R_{n+m,i_{n+m}}) \neq \emptyset.
\]
The nonstationary Markov partition $(\cP_n)_{n\in\ZZ}$ is called \emph{generating} if, for any sequence $(i_n)_{n\in\ZZ} \in \prod_{n\in\ZZ} \{1,\dots,q_n\}$,
\[
\bigcap_{n\in\NN}\overline{\bigcap_{0\le k<n}  (f_{-1} \circ \cdots \circ f_{-k})(R_{-k,i_{-k}}) \cap \cdots \cap R_{0,i_0} \cap \cdots \cap (f_0^{-1} \circ \cdots \circ f_{k-1}^{-1})(R_{k,i_k})}
\]
contains at most one point.
\end{remark}

Nonstationary Markov partitions can also be interpreted in the following way:
For any sequence $(i_n)_{n\in\ZZ}$, if we have $(x_k)_{n\le k<n+m}$ such that $x_k \in  R_{k,i_k}$ and $f_k(x_k)\in R_{k+1,i_{k+1}}$ for all $k \in \{n,\dots,n{+}m{-}1\}$, then there exists a point $y \in R_{n,i_n}$ such that $f_k \circ \cdots \circ f_n(y) \in R_{k,i_k}$ for all $k \in \{n,\dots,n{+}m\}$. In other words, as mentioned above, for any finite pseudo-orbit, there exists a finite orbit with the same symbolic dynamics. If the partition is generating, then the symbolic dynamics uniquely defines the point.

We now introduce a nonstationary version of \emph{Property~M} from Adler \cite[Section~7]{A98}. This property is easy to check, and we will see below that  it implies the nonstationary Markov property of Definition~\ref{nsMarkov}. We need some terminology first.

\begin{definition}[Transverse partitions; cf.\ {\cite[Definitions~7.1 and~7.2]{A98}}]\label{def:transv}\indx{partition!transverse}\indx{partition!horizontal}\indx{partition!vertical}
To each point~$p$ in a set~$R$, we associate two subsets $h(p), v(p) \subset R$ that both contain~$p$ in a way that each of the collections 
\[
H = \{h(p) : p \in R\}\quad\text{and}\quad V = \{v(p) : p \in R\}
\]
forms a partition of~$R$. We call $H$ and~$V$ \emph{transverse partitions} of~$R$ if $h(p) \cap v(q) \neq \emptyset$ for all $p,q \in R$.
\end{definition}

Using transverse partitions, we can now define a nonstationary version of Property~M given in \cite[Definition~7.4]{A98}. 

\begin{definition}[Nonstationary Property~M]\label{def:M}\indx{Property M}
Let $(X,f)$ be a mapping family. 
For each $n \in \ZZ$, let $\cP_n = \{R_{n,1},\dots,R_{n,q_n}\}$ be a topological partition of~$X_n$. 
The sequence $(\cP_n)_{n\in\ZZ}$ is said to have \emph{(nonstationary) Property~M} for $(X,f)$ if, for each $n \in \ZZ$ and $i \in \{1,\dots,q_n\}$, the atom~$R_{n,i}$ has a pair of transverse partitions 
\[
H_{n,i} = \{h_n(p) \,:\, p \in R_{n,i}\}, \quad V_{n,i} = \{v_n(p) \,:\, p \in R_{n,i}\}, \notx{hnatom}{$h_n(\cdot), h_n^\gen(\cdot)$}{\hspace{.75em}atom of a horizontal partition} \notx{Hn}{$H_{n,\cdot}, H_{n,\cdot}^\gen$}{horizontal partition}
\notx{vnatom}{$v_n(\cdot), v_n^\gen(\cdot)$}{\hspace{.75em}atom of a vertical partition} \notx{Vn}{$V_{n,\cdot}, V_{n,\cdot}^\gen$}{vertical partition}
\]
such that
\begin{equation}\label{eq:propMone}
h_{n+1}(p) \subset f_n\big(h_n(f_n^{-1}p)\big) \quad \mbox{and} \quad f_n\big(v_n(f_n^{-1}p)\big) \subset v_{n+1}(p)
\end{equation}
for all $p \in f_n(R_{n,i} ) \cap  R_{n+1,j}$, $j \in \{1,\dots,q_{n+1}\}$. 
\end{definition}

Property~M implies the nonstationary Markov property of Definition~\ref{nsMarkov}. Indeed, the following proposition can be derived in exactly the same way as its stationary analog in \cite[Theorem~7.9]{A98}.\footnote{Contrarily to \cite{A98}, we do not put the index~$i$ in $h_n(p)$ and $v_n(p)$ because $i$ is given by the location of $p$ within the topological partition $\cP_n = \{R_{n,1},\dots,R_{n,q_n}\}$.}

\begin{proposition}\label{prop:MMarkov}
Let $(X,f)$ be a mapping family and assume that the sequence of partitions $(\cP_n)_{n\in\ZZ}$ has the nonstationary Property~M for $(X,f)$. 
Then $(\cP_n)_{n\in\ZZ}$ is a nonstationary Markov partition for $(X,f)$.
\end{proposition}

We illustrate the concepts of this section for a special instance of Example~\ref{ex:classicalMappingFamily}. 

\begin{example}\label{ex:NMexpl2}
Two well-known continued fraction expansions are given by
\[
\begin{aligned}
\phi&=[1;1,1,1,1,\ldots] \quad\text{and} \\
e&=[2;1,2,1,1,4,1,1,6,1,1,8,\ldots],
\end{aligned}
\]
where $\phi=\frac{1+\sqrt{5}}{2}$ is the golden ratio.\notx{phi}{$\phi$}{golden ratio}
We combine these expansions to a bi-infinite sequence 
\[
(a_n)_{n\in\ZZ}=(\ldots,1,1,1,1;2,1,2,1,1,4,1,1,6,1,1,8,\ldots),
\]
where the semicolon is between $a_{-1}$ and~$a_0$. 
As in Example~\ref{ex:classicalMappingFamily}, this sequence gives rise to a sequence $(M_n)_{n\in\ZZ}$ of matrices that defines the mapping family
\begin{equation*}
\cdots\xrightarrow{M_{-2}^{-1}} \TT^2\xrightarrow{M_{-1}^{-1}} \TT^2\xrightarrow{M_{0}^{-1}} \TT^2\xrightarrow{M_{1}^{-1}} \cdots,
\end{equation*}
where $X_n=\TT^2$ and $f_n=M_n^{-1}$ for each $n\in\ZZ$.
The nonstationary Markov partition $(\cP_n)_{n\in\ZZ}$ with
$\cP_n = \{R_{n,1},R_{n,2}\}$ for this mapping family (at levels $-1,0,1$) is shown in Figure~\ref{fig:Markovpartition}. We will see later that these parallelogram shaped atoms are special cases of $\cS$-adic Rauzy fractals; see also \cite{AF:05}, where these atoms are constructed for mapping families related to the classical continued fraction algorithm.

\begin{figure}[ht] 
\begin{tikzpicture}[scale=2.25]
\begin{scope}[shift={(-2,0)}]
\filldraw[fill=red](.185,-.115)--(-.504,.311)--(0,1)--(.689,.574)--cycle;
\node at (.09,.44){$R_{-1,2}$};
\filldraw[fill=blue](.185,-.115)--(.689,-.426)--(1,0)--(.496,.311)--cycle;
\node at (.59,-.06){$R_{-1,1}$};
\draw[dotted](0,0)--(1,0)--(1,1)--(0,1)--cycle;
\draw(0,0)--node[pos=.9,right=2pt]{$\RR\binom{1}{1+1/e}$}(1,1.368);
\draw(0,0)--(1,-.618)node[above right=-5pt]{$\RR\binom{\phi}{-1}$};
\end{scope}
\node at (-.6,.5){$\xrightarrow[\displaystyle f_{-1}]{\left(\begin{smallmatrix}-1&1\\1&0\end{smallmatrix}\right)}$};
\filldraw[fill=red](.515,-.319)--(-.3,.185)--(0,1)--(.815,.496)--cycle;
\node at (.28,.34){$R_{0,2}$};
\filldraw[fill=blue](.515,-.319)--(.815,-.504)--(1,0)--(.7,.185)--cycle; 
\node at (.75,-.16){$R_{0,1}$};
\draw[dotted](0,0)--(0,1)--(1,1)--(1,0)--cycle;
\draw(0,0)--node[pos=.9,right]{$\RR\binom{1}{e}$}(.5,1.36);
\draw(0,0)--(1,-.618)node[above right=-5pt]{$\RR\binom{\phi}{-1}$};
\node at (1.3,.5){$\xrightarrow[\displaystyle f_0]{\left(\begin{smallmatrix}-2&1\\1&0\end{smallmatrix}\right)}$};
\begin{scope}[shift={(2,0)}]
\filldraw[fill=red](.221,-.085)--(-.564,.215)--(0,1)--(.785,.7)--cycle;
\node at (.11,.46){$R_{1,2}$};
\filldraw[fill=blue](.221,-.085)--(.785,-.3)--(1,0)--(.436,.215)--cycle;
\node at (.61,-.04){$R_{1,1}$};
\draw[dotted](0,0)--(1,0)--(1,1)--(0,1)--cycle;
\draw(0,0)--node[pos=.9,left]{$\RR\binom{1}{1/(e{-}2)}$}(1,1.392);
\draw(0,0)--node[pos=.75,below=3pt]{$\RR\binom{\phi^2}{-1}$}(1,-.382);
\end{scope}
\end{tikzpicture}
\caption{Levels $-1$, $0$, and~$1$ of the nonstationary Markov partition of Example~\ref{ex:NMexpl2}. 
The $L$-shaped regions are fundamental domains of~$\TT^2$ in~$\RR^2$ corresponding to the respective levels of the mapping family.}
\label{fig:Markovpartition} 
\end{figure}
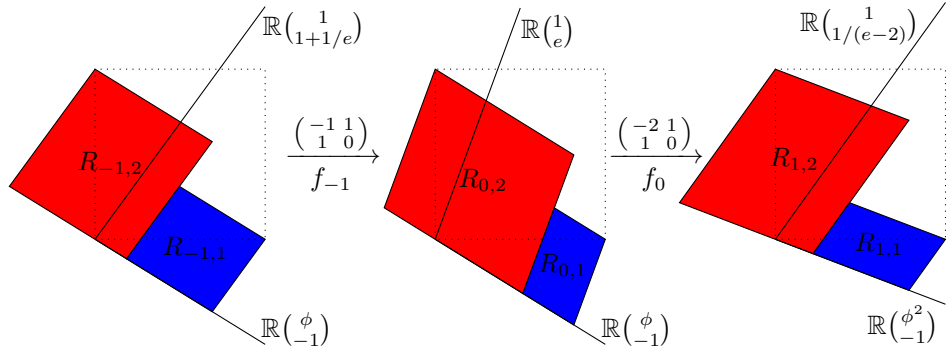

Figure~\ref{fig:transverse} shows how the atoms of the horizontal and vertical partitions are mapped by the transformations of the mapping family. Indeed, since our mapping family is Anosov by Example~\ref{ex:classicalMappingFamily}, the horizontal atoms get longer (on the long run), while the vertical atoms get shorter (on the long run).

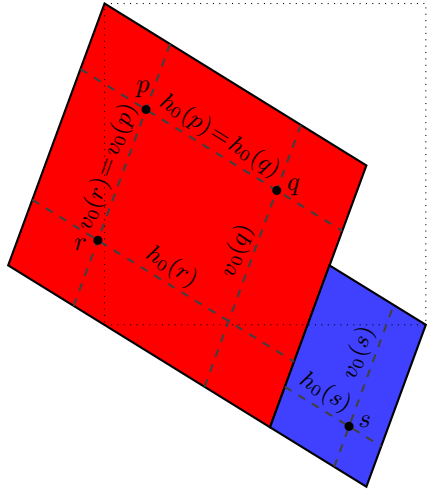
\begin{figure}[ht] 
\begin{tikzpicture}[scale=4.25]
\filldraw[fill=red,thick](.515,-.319)--(-.3,.185)--(0,1)--(.815,.496)--cycle;
\filldraw[fill=blue!75,thick](.515,-.319)--(.815,-.504)--(1,0)--(.7,.185)--cycle; 
\draw[dotted](0,0)--(1,0)--(1,1)--(0,1)--cycle;
\draw[thick,dashed,black!75](-.096,.059)--node[pos=.51,rotate=70,above=-2pt,black]{\small$v_0(r)\!=\!v_0(p)$}(.204,.874)
(.311,-.193)--node[rotate=70,above=-2pt,black]{\small$v_0(q)$}(.611,.622) (.715,-.442)--node[pos=.67,rotate=70,above=-2pt,black]{\small$v_0(s)$}(.9,.062)
(-.075,.796)--node[rotate=-32,above=-2pt,black]{\small$h_0(p)\!=\!h_0(q)$}(.74,.292)
(-.225,.389)--node[rotate=-32,above=-2pt,black]{\small$h_0(r)$}(.59,-.115)
(.561,-.193)--node[pos=.33,rotate=-32,above=-2pt,black]{\small$h_0(s)$}(.861,-.378);
\fill(.129,.67) circle (.4pt);
\node[above=1pt] at (.12,.67){$p$};
\fill(.536,.418) circle (.4pt);
\node[right] at (.54,.43){$q$};
\fill(-.021,.263) circle (.4pt);
\node[left] at (-.025,.25){$r$};
\fill(.761,-.316) circle (.4pt);
\node[right] at (.761,-.3){$s$};
\end{tikzpicture}
\caption{Transverse partitions of the nonstationary Markov Partition $\{R_{0,1},R_{0,2}\}$ of Example~\ref{ex:NMexpl2}.}
\label{fig:transverse}
\end{figure}

Figure~\ref{fig:propertyM} illustrates the inclusions \eqref{eq:propMone} for the atoms of the transverse partitions in the definition of Property~M. We start with the second inclusion. Consider the $L$-shaped region at level 0. The dashed line passing through~$p$ ``vertically'' is the atom~$v_0(p) \subset R_{0,2}$. Its preimage $f_{-1}^{-1}(v_0(p))$ is the grey line passing through~$f_{-1}^{-1}(p)$. The solid part of this grey line is $v_{-1}(f_{-1}^{-1}(p)) \subset R_{-1,2}$. This yields $v_{-1}(f_{-1}^{-1}(p)) \subset f_{-1}^{-1}(v_0(p))$ and, hence, $f_{-1}(v_{-1}(f_{-1}^{-1}(p))) \subset v_0(p)$. This illustrates the second inclusion in \eqref{eq:propMone} on level $n=-1$. To get the first inclusion, we start with the dashed line passing through $p$ ``horizontally'', which is the atom $h_0(p) \subset R_{0,2}$. Its image $f_0(h_0(p))$ is the grey line passing through~$f_0(p)$. The solid part of this line is $h_1(f(p)) \subset R_{1,2}$, which shows that $h_1(f_0(p)) \subset f_0(h_0(p))$. Setting $q = f_0(p)$, this implies that $q \in f_0(R_{0,2})\cap R_{1,2}$ and $h_1(q) \subset f_0(h_0(f_0^{-1}(q)))$. This illustrates the first inclusion in \eqref{eq:propMone}, with $p$ replaced by $q = f_0(p)$, on level $n=0$.

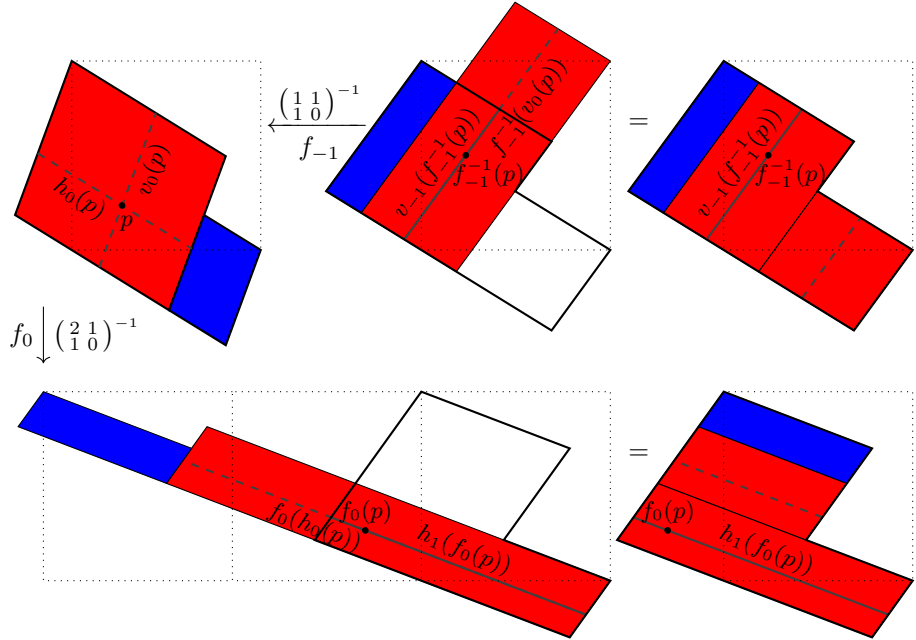
\begin{figure}[ht] 
\begin{tikzpicture}[scale=2.5]
\filldraw[fill=red,thick](.515,-.319)--(-.3,.185)--(0,1)--(.815,.496)--cycle;
\filldraw[fill=blue,thick](.515,-.319)--(.815,-.504)--(1,0)--(.7,.185)--cycle; 
\draw[thick,dashed,black!75](-.18,.511)--node[pos=.33,rotate=-32,below=-2pt,black]{\small$h_0(p)$}(.635,.007)
(.146,-.091)--node[pos=.67,rotate=70,below=-2pt,black]{\small$v_0(p)$}(.446,.724);
\fill(.266,.2354) circle (.5pt);
\node[below] at (.29,.23){\small$p$};
\draw[dotted](0,0)--(1,0)--(1,1)--(0,1)--cycle;
\node at (1.3,.67){$\xleftarrow[\displaystyle f_{-1}]{\big(\begin{smallmatrix}1&1\\1&0\end{smallmatrix}\big)^{-1}}$};
\draw[->](-.15,-.3)--node[left]{$f_0$}node[right]{$\big(\begin{smallmatrix}2&1\\1&0\end{smallmatrix}\big)^{\scriptscriptstyle-1}$}(-.15,-.6);
\begin{scope}[shift={(1.85,0)}]
\begin{scope}[cm={0,1,1,1,(0,0)}]
\filldraw[fill=red](.515,-.319)--(-.3,.185)--(0,1)--(.815,.496)--cycle;
\filldraw[fill=blue](.515,-.319)--(.815,-.504)--(1,0)--(.7,.185)--cycle; 
\draw[thick,dashed,black!75](.146,-.091)--node[pos=.69,rotate=54,below=-1pt,black]{\small$f_{-1}^{-1}(v_0(p))$}(.446,.724);
\draw[thick,black!75](.146,-.091)--node[rotate=54,above=-2pt,black]{\small$v_{-1}(f_{-1}^{-1}(p))$}(.331,.413);
\end{scope}
\fill(.2354,.5014) circle (.5pt);
\node[below] at (.355,.525){\small$f_{-1}^{-1}(p)$};
\draw[thick](.689,-.426)--(-.504,.311)--(0,1)--(.689,.574)--(.496,.311)--(1,0)--cycle;
\draw[dotted](0,0)--(1,0)--(1,1)--(0,1)--cycle;
\node at (1.15,.67){$=$};
\end{scope}
\begin{scope}[shift={(3.45,0)}]
\begin{scope}[cm={0,1,1,1,(0,0)}]
\filldraw[fill=red](.515,-.319)--(-.3,.185)--(-.115,.689)--(.7,.185)--cycle;
\filldraw[fill=blue](.515,-.319)--(.815,-.504)--(1,0)--(.7,.185)--cycle; 
\draw[thick,black!75](.146,-.091)--node[rotate=54,above=-2pt,black]{\small$v_{-1}(f_{-1}^{-1}(p))$}(.331,.413);
\end{scope}
\begin{scope}[cm={0,1,1,1,(0,-1)}]
\filldraw[fill=red](.7,.185)--(-.115,.689)--(0,1)--(.815,.496)--cycle;
\draw[thick,dashed,black!75](.331,.413)--(.446,.724);
\end{scope}
\fill(.2354,.5014) circle (.5pt);
\node[below] at (.355,.525){\small$f_{-1}^{-1}(p)$};
\draw[thick](.689,-.426)--(-.504,.311)--(0,1)--(.689,.574)--(.496,.311)--(1,0)--cycle;
\draw[dotted](0,0)--(1,0)--(1,1)--(0,1)--cycle;
\end{scope}
\begin{scope}[shift={(1.85,-1.75)}]
\begin{scope}[cm={-2,1,1,0,(0,0)}]
\filldraw[fill=red](.515,-.319)--(-.3,.185)--(0,1)--(.815,.496)--cycle;
\draw[thick,dashed,black!75](-.18,.511)--node[pos=.66,rotate=-21,below=-2.5pt,black]{\small$f_0(h_0(p))$}(.635,.007);
\draw[thick,black!75](-.18,.511)--node[rotate=-21,above=-2pt,black]{\small$h_1(f_0(p))$}(.335,.192);
\filldraw[fill=blue](.515,-.319)--(.815,-.504)--(1,0)--(.7,.185)--cycle; 
\end{scope}
\fill(-.2966,.266) circle (.5pt);
\node[above=-1pt] at (-.2966,.266){\small$f_0(p)$};
\draw[thick](.785,-.3)--(-.564,.215)--(0,1)--(.785,.7)--(.436,.215)--(1,0)--cycle;
\draw[dotted](0,0)--(1,0)--(1,1)--(0,1)--cycle;
\begin{scope}[shift={(-1,0)}]
\draw[dotted](1,1)--(0,1)--(0,0)--(1,0);
\end{scope}
\begin{scope}[shift={(-2,0)}]
\draw[dotted](1,1)--(0,1)--(0,0)--(1,0);
\end{scope}
\node at (1.15,.67){$=$};
\end{scope}
\begin{scope}[shift={(3.45,-1.75)}]
\begin{scope}[cm={-2,1,1,0,(0,0)}]
\filldraw[fill=red](.215,-.134)--(-.3,.185)--(0,1)--(.515,.681)--cycle;
\draw[thick,black!75](-.18,.511)--node[rotate=-21,above=-2pt,black]{\small$h_1(f_0(p))$}(.335,.192);
\end{scope}
\begin{scope}[cm={-2,1,1,0,(1,0)}]
\filldraw[fill=red](.515,-.319)--(.215,-.134)--(.515,.681)--(.815,.496)--cycle;
\draw[thick,dashed,black!75](.335,.193)--(.635,.007);
\end{scope}
\begin{scope}[cm={-2,1,1,0,(2,0)}]
\filldraw[fill=blue](.515,-.319)--(.815,-.504)--(1,0)--(.7,.185)--cycle; 
\end{scope}
\fill(-.2966,.266) circle (.5pt);
\node[above=-1pt] at (-.2966,.266){\small$f_0(p)$};
\draw[thick](.785,-.3)--(-.564,.215)--(0,1)--(.785,.7)--(.436,.215)--(1,0)--cycle;
\draw[dotted](0,0)--(1,0)--(1,1)--(0,1)--cycle;
\end{scope}
\end{tikzpicture}
\caption{Illustration of transverse partitions and Property~M for Example~\ref{ex:NMexpl2}.}
\label{fig:propertyM}
\end{figure}

In Figure~\ref{fig:propertyM}, the $L$-shaped region at level $0$ is mapped via linear maps in~$\RR^2$. If we ``restack'' these linear images (at levels $-1$ and~$1$) back to the according fundamental domain of the torus, we obtain Figure~\ref{fig:transverse}.
\end{example}

Suppose that we got a Markov partition $(\cP_n)_{n\in\ZZ}$ of a mapping family $(X,f)$. The following result, which will be needed in Section~\ref{sec:markov}, shows how to construct a refined Markov partition $(\cQ_n)_{n\in\ZZ}$ from $(\cP_n)_{n\in\ZZ}$ having the property that $(\cP_n)_{n\in\ZZ}$ is generating if and only if $(\cQ_n)_{n\in\ZZ}$ is.

\begin{proposition}\label{prop:refinePtoQ}
Let $(X,f)$ be a mapping family that is given as in \eqref{eq:mappingfamilydefinition} and let $(\cP_n)_{n\in\ZZ}$  with $\cP_n = \{R_{n,1},\dots,R_{n,q_n}\}$ be a nonstationary Markov partition for $(X,f)$. Define $(\cQ_n)_{n\in\ZZ}$ by
$\cQ_n=\{S_{n,i,j} = R_{n,i} \cap f^{-1}(R_{n+1,j}) \colon S_{n,i,j} \not=\emptyset\}$. Then the following assertions are true.
\begin{enumerate}[\upshape (i)]
\itemsep.5ex
\item \label{i:Qn}
$(\cQ_n)_{n\in\ZZ}$ is a nonstationary Markov partition for $(X,f)$.
\item \label{i:Pn}
$(\cP_n)_{n\in\ZZ}$ is generating if and only if $(\cQ_n)_{n\in\ZZ}$ is generating.
\end{enumerate}
\end{proposition}

\begin{proof}
To prove~(\ref{i:Qn}), fix a sequence sequence $(i_n,j_n)_{n\in\ZZ} \in \prod_{n\in\ZZ} \{1,\dots,q_n\}^2$ satisfying
\[
S_{k,i_k,j_k} \cap f^{-1} (S_{k+1,i_{k+1}j_{k+1}}) \neq \emptyset \quad \mbox{for all}\ k\in\{n,\ldots, n+m-1\}.
\]
By the definition of $S_{n,i,j}$ this can be written as
\[
R_{k,i_k} \cap f^{-1}(R_{k+1,j_k}) \cap  f^{-1}(R_{k+1,i_{k+1}}) \cap f^{-2}(R_{k+2,j_{k+1}}) \neq \emptyset
\]
for all $k\in\{n,\ldots, n+m-1\}$. Because $f^{-1}(R_{n,i})$, $i\in\{1,\ldots,q_n\}$, are pairwise disjoint this implies that $j_k=i_{k+1}$ for $k\in\{n,\ldots, n+m-1\}$. Hence, redefining $i_{n+m+1}=j_{n+m}$ for convenience, we gain
\[
R_{k,i_k} \cap  f^{-1}(R_{k+1,i_{k+1}}) \cap f^{-2}(R_{k+2,i_{k+2}}) \neq \emptyset \quad \mbox{for all}\ k\in\{n,\ldots, n+m-1\}.
\]
Because $(\cP_n)_{n\in\ZZ}$ is a nonstationary Markov partition for $(X,f)$, using the definition of $S_{i,j,n}$ we finally get
\[
\begin{aligned}
S_{n,i_n,i_{n+1}} \cap f^{-1}(S_{n+1,i_{n+1},i_{n+2}})\cap \cdots  \cap f^{-m}(S_{n+m,i_{n+m},i_{n+m+1}}) &= \\
R_{n,i_n} \cap f^{-1}(R_{n+1,i_{n+1}}) \cap  \cdots  \cap f^{-m-1}(R_{n+m,i_{n+m+1}}) & \neq \emptyset.
\end{aligned}
\]
This proves that  $(\cP_n)_{n\in\ZZ}$ is a nonstationary Markov partition for $(X,f)$.

To prove~(\ref{i:Pn}), assume first that $(\cP_n)_{n\in\ZZ}$ is generating. Because $\cQ_n$ is a refinement of $\cP_n$ this immediately implies that $(\cQ_n)_{n\in\ZZ}$ is generating as well. To see the other direction, assume that $(\cQ_n)_{n\in\ZZ}$ is generating. Then by the definition of $(\cQ_n)_{n\in\ZZ}$ in terms of $(\cP_n)_{n\in\ZZ}$ we have
\[
\bigcap_{n\in\NN}\overline{\bigcap_{|k|<n}  f^{-k}(R_{k,i_{k}})} = \bigcap_{n\in\NN}\overline{\bigcap_{|k|<n}  f^{-k}(S_{k,i_{k},i_{k+1}})},
\]
hence, also $(\cP_n)_{n\in\ZZ}$ is generating.
\end{proof}

In \cite{AF:05}, nonstationary Markov partitions are treated in a slightly different way building on Bowen's definition of a stationary Markov partition~\cite{Bow75}. Property~M is reminiscent of the definition of Bowen proposed in \cite[Section~3C]{Bow75}. Note that the partitions~$H$ (for ``horizontal'') and~$V$ (for ``vertical'') are abstract models for the unstable and stable foliations used by Bowen~\cite{Bow75}. 

\subsection{Symbolic representation of a mapping family} \label{subsec:nsft}
We now discuss how to associate appropriate symbolic representations to eventually Anosov mapping families; these symbolic representations are defined as nonstationary versions of (1-step) finite type shifts, and more precisely as  nonstationary versions of the well-known  vertex and edge  shifts; see e.g.\ \cite[Definitions~2.2.5 and~2.3.7]{Lind-Marcus}. 
 
We now define these objects.
Let $\cD$ be some set. We recall that the shift operator\indx{shift!operator}~$\Sigma$\notx{Sh}{$\Sigma$}{shift operator} maps a sequence $(\omega_n)_{n\in\ZZ} \in \cD^{\ZZ}$ to $(\omega_{n+1})_{n\in\ZZ}$.\footnote{The fact that shift operators on different spaces are always denoted by~$\Sigma$ should cause no confusion.} A~dynamical system $(X,\Sigma)$ with $X \subset \cD^{\ZZ}$ is a \emph{shift}\indx{shift} if $X$ is closed w.r.t.\ the product topology of the discrete topology on $X \subset \cD^{\ZZ}$. 
A~\emph{bi-infinite Bratteli diagram}\indx{Bratteli diagram} is an infinite directed graph $\mathscr{B} = (V,E)$\notx{B}{$\mathscr{B}$}{Bratteli diagram}, where $V = \coprod_{n\in\ZZ} V_n$, $E = \coprod_{n\in\ZZ} E_n$\notx{En}{$E_n$}{set of edges of a Bratteli diagram or nonstationary edge shift, prefixes of a substitution}, with source and range maps $s,r: E \to V$ satisfying $s(E_n) = V_n$ and $r(E_n) = V_{n+1}$ for all $n \in \ZZ$. We assume that $|V_n|$ and~$|E_n|$ are finite for all~$n$. 
With a Bratteli diagram~$\mathscr{B}$, we associate an infinite sequence of transition matrices $(A_n)_{n\in\ZZ}$. 
For each ``level'' $n\in\NN$, the matrix $A_n = (a_{n,w,v})$ has dimension $|V_{n+1}| {\times} |V_{n}|$ and entries $a_{n,w,v} = |\{e \in E_n : s(e) = v,\, r(e) = w\}|$, with $v\in V_n$, $w\in V_{n+1}$.

\begin{definition}[Nonstationary edge and vertex shift] \label{def:nfst}
\indx{shift!nonstationary!edge}\indx{shift!nonstationary!vertex}
Define \notx{X}{$X_{\mathscr{B}}, X_{\mathscr{B}}^{(n)}$}{nonstationary edge shift} \notx{X}{$\hat{X}_{\mathscr{B}}, \hat{X}_{\mathscr{B}}^{(n)}$}{nonstationary vertex shift}
\[
\begin{gathered}
\begin{aligned}
X_{\mathscr{B}}^{(0)} & = \bigg\{(e_n)_{n\in\ZZ} \in \prod_{n\in\ZZ} E_n \,:\, r(e_n) = s(e_{n+1})\ \text{for all}\ n \in \ZZ \bigg\}, & X_{\mathscr{B}}^{(n)} & = \Sigma^n\big(X_{\mathscr{B}}^{(0)}\big), \\
\hat{X}_{\mathscr{B}}^{(0)} & = \bigg\{ (v_n)_{n\in\ZZ} \in \prod_{n\in\ZZ} V_n \,:\, a_{n,v_{n+1},v_{n}} > 0\ \text{for all}\ n \in \ZZ \bigg\}, & \hat{X}_{\mathscr{B}}^{(n)} & = \Sigma^n\big(\hat{X}_{\mathscr{B}}^{(0)}\big),
\end{aligned} \\
X_{\mathscr{B}} = \coprod_{n\in\ZZ} X_{\mathscr{B}}^{(n)}  \qquad\text{and}\qquad
\hat{X}_{\mathscr{B}} = \coprod_{n\in\ZZ} \hat X_{\mathscr{B}}^{(n)}.
\end{gathered}
\]
We call $(X_\mathscr{B},\Sigma)$ and $(\hat{X}_\mathscr{B},\Sigma)$ the \emph{nonstationary edge shift} and the \emph{nonstationary vertex shift} associated to~$\mathscr{B}$, respectively.
\end{definition}

In \cite{AF:05,Fisher:09}, the nonstationary edge shifts we defined here are called \emph{nonstationary subshifts of finite type}.\footnote{We do not adopt this terminology because already in the stationary case edge shifts are only special cases of 1-step shifts of finite type.} Nonstationary edge shifts are also related to the two-sided  constructions of Markov compacta (and Bratteli diagrams) introduced in \cite{Bufetov13,Bufetov:14}. 

One can naturally associate a Bratelli diagram to a given sequence of nonnegative $d{\times}d$ integer matrices $(M_n)_{n\in\ZZ}$. Indeed, it follows from the definition that a Bratelli  diagram is completely defined by providing its sequence $(A_n)_{n\in\ZZ}$ of transition matrices. Thus, setting $A_n = \tr{\!M}_n$, $n\in \ZZ$, we get the \emph{Bratteli diagram $\mathscr{B}$ associated to the sequence of matrices $(M_n)_{n\in\ZZ}$}\indx{Bratteli diagram}.\footnote{It will become clear in Section~\ref{subsec:nsftSub} why we use the transposed matrices here.} The Bratteli diagram associated to the sequence of matrices $(M_n)_{n\in\ZZ}$ from Example~\ref{ex:NMexpl2} is depicted in Figure~\ref{fig:Bratt1}.

\begin{figure}[ht] 
\begin{tikzpicture}[scale=2,thick,->,minimum size=10pt,shorten >=.5pt]
\node (vm1) at (-2,1){1};
\node (vm2) at (-2,0){2};
\node (v01) at (0,1){1};
\node (v02) at (0,0){2};
\node (v11) at (2,1){1};
\node (v12) at (2,0){2};
\draw[->] 
(vm1) edge (v02)
(vm2) edge (v01)
(vm2) edge (v02)
(v01) edge (v12) 
(v02) edge (v11)
(v02) edge[bend left=10] (v12)
(v02) edge[bend right=10] (v12);
\node at (-2,-.25){$V_{-1}$};
\node at (0,-.25){$V_0$};
\node at (2,-.25){$V_1$};
\node at (-2.5,.5){$\cdots$};
\node at (2.5,.5){$\cdots$};
\end{tikzpicture}
\caption{The Bratteli diagram from Example~\ref{ex:NMexpl2} (levels $-1,0,1$).}
\label{fig:Bratt1}
\end{figure}
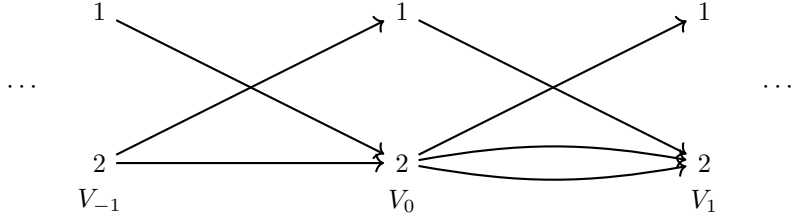

We now want to use nonstationary vertex shifts to provide a symbolic representation of a mapping family. Consider a mapping family $(X,f)$ defined as in~\eqref{eq:mappingfamilydefinition} and assume that $(\cP_n)_{n\in\ZZ}$  with $\cP_n = \{R_{n,1},\dots,R_{n,q_n}\}$ is a nonstationary Markov partition for $(X,f)$. With this Markov partition we associate a bi-infinite Bratteli diagram $\mathscr{B}=(V,E)$ in the following way. The sets of vertices are given by $V_n= \{1,\dots,q_n\}$ and the sets of edges are 
\[
E_n =\big\{(i,j) \;:\;  i\in V_n, j\in V_{n+1} \text{ with } R_{n,i} \cap f^{-1}(R_{n+1,j}) \not= \emptyset \big\}.
\]
Here $(i,j)$ denotes an edge with $s(i,j)=i$ and $r(i,j)=j$.

We provide the following result from \cite[Proposition~3.6]{AF:05}, which forms a nonstationary version of \cite[Theorem~6.5]{A98}.

\begin{proposition}[{\cite[Proposition~3.6]{AF:05}}]
\label{prop:genSymRep}
Let  $(X,f)$ be a mapping family with a generating nonstationary Markov partition $(\cP_n)_{n\in\ZZ}$. Let $(\hat X_{\mathscr{B}},\Sigma)$ be the vertex shift associated to this Markov partition. Then the maps
\[
\psi_n:\, X_{\mathscr{B}}^{(n)} \to X_n;\quad (i_k)_{k\in\ZZ}  \mapsto \bx, \;\text{ where}\;
\bigcap_{\ell\in\NN}\overline{\bigcap_{|k|<\ell}  f^{n-k}(R_{k,i_{k}})}=\{\bx\}.
\]
are continuous and one-to-one except on the set $\psi_n^{-1}\big(\bigcup_{k\in \ZZ}\bigcup_{i\in\{1,\ldots,q_k\}}f^{n-k}(\partial R_{k,i})\big)$ of pullbacks of the boundaries of the atoms $R_{k,i}$, and the diagram
\[  
\xymatrix{  
\cdots \ar[r]^{\Sigma}& X_{\mathscr{B}}^{(-1)} \ar[d]_{\psi_{-1}}\ar[r]^{\Sigma} &  X_{\mathscr{B}}^{(0)} \ar[d]_{\psi_0}\ar[r]^{\Sigma} & X_{\mathscr{B}}^{(1)} \ar[d]_{\psi_1}\ar[r]^{\Sigma} &  X_{\mathscr{B}}^{(2)}\ar[d]_{\psi_2}  \ar[r]^{\Sigma}&\cdots  \\
\cdots \ar[r]^{f}& X_{-1} \ar[r]^{f} &  X_0 \ar[r]^{f} & X_1 \ar[r]^{f} &  X_2 \ar[r]^{f}   &\cdots }  
\]
commutes. In this case $(\hat X_{\mathscr{B}},\Sigma)$ is called a \emph{topological symbolic representation}\indx{mapping family!topological symbolic representation} of $(X,f)$. 
\end{proposition}

Consider a mapping family $(X,f)$ defined as in~\eqref{eq:mappingfamilydefinition} and assume that $(\cP_n)_{n\in\ZZ}$  with $\cP_n = \{R_{n,1},\dots,R_{n,q_n}\}$ is a nonstationary Markov partition for $(X,f)$. Sometimes one wants to associate a nonstationary edge shift with a Markov partition of $(X,f)$. In this case, the edges of the Bratteli diagram are given by $E_n= \{i,\dots,q_n\}$. If the vertices are suitably defined, one gets a result that is analogous to Proposition~\ref{prop:genSymRep}; see Theorem~\ref{th:symbMod}.

\section{Bi-infinite sequences of matrices}\label{sec:matrices}
In this section, we  develop the basic theory of bi-infinite sequences of matrices with special emphasis on convergence properties of such sequences. In particular, our main aim is to give natural criteria for strong convergence of a bi-infinite sequence of matrices. These criteria are of independent interest, as strong convergence is of importance for instance in the context of Diophantine approximation properties of multidimensional continued fractions. The results of this section lay the foundation for our main results in Chapter~\ref{chapter:markov}.

Section~\ref{subsec:primrec} contains the definition of primitivity and recurrence, two properties of a sequence of matrices that will be basic for all that follows. In Section~\ref{subsec:eigen} we attach generalized eigenvectors to a sequence of matrices. Moreover, we show that these eigenvalues always exist for primitive and recurrent sequences of matrices. Using the generalized right eigenvector, we then define three notions of convergence: weak, strong, and exponential convergence. Section~\ref{sec:LyapunovPisot} defines Lyapunov exponents for a sequence of matrices and introduces the so-called (local) Pisot condition. This condition assures that a sequence of matrices has one expanding direction and contracts in all the other directions, a behavior that is reminiscent of a Pisot number. As we see in Section~\ref{sec:LyapunovPisot2}, the Pisot condition is crucial when deriving convergence properties of a sequence of matrices. The main part of this subsection is occupied by the proof of Proposition~\ref{p:localOseledets}, a local version of the famous Oseledets Multiplicative Ergodic Theorem. This lemma is the key ingredient of the proof of the fact that the Pisot condition is equivalent to exponential convergence under quite mild conditions. Another result shows the effect of the Pisot condition on the behaviour of the eigenvalues of products of consecutive matrices of the underlying sequence of matrices. Section~\ref{sec:LyapunovPisot3} defines the concept of algebraic irreducibility. Algebraic irreducibility entails that the generalized right eigenvector has rationally independent coordinated, a property that will be crucial later, when we need that certain Rauzy fractals have non-empty interior. This subsection also contains our main convergence result, formulated as Theorem~\ref{th:matrixPisot}, which was already stated as Theorem~\nameref{t:A} in the introduction.
We need strong convergence in Section~\ref{subsec:anosovM} to get an Anosov splitting associated to a sequence of matrices. This will be crucial later when we define a general class of Rauzy fractals that is used in Chapter~\ref{chapter:markov} to define induced rotations and atoms of nonstationary Markov partitions for multidimensional continued fraction algorithms.

We start with some basic terminology. We say that a real number~$x$ is \emph{nonnegative} if $x\ge 0$, and \emph{positive} if $x>0$; we say that a vector or a matrix is \emph{nonnegative}\indx{matrix!nonnegative} if all its coefficients are nonnegative, and \emph{positive}\indx{matrix!positive} if all its coefficients are positive. 
Let $d\ge 1$. For a semiring~$R$, we denote the set of $d{\times}d$ matrices with entries taken from~$R$ by $R^{d\times d}$ \notx{ra}{$R^{d\times d}$}{matrix over a ring}. The \emph{general linear group} of degree~$d$ over $\ZZ$ and~$\RR$ will be denoted by \notx{GL}{$\GL(d,\cdot)$}{general linear group}
\[
\GL(d,\ZZ) = \{M \in \ZZ^{d\times d} : \det M = \pm 1 \} \ \ \mbox{and}\ \ \GL(d,\RR) = \{M \in \RR^{d\times d} : \det M \neq 0\},
\]
respectively. A matrix $M \in \RR^{d\times d}$ with $\det M = \pm 1$ is called \emph{unimodular}\indx{matrix!unimodular}. 

To avoid trivialities, in the sequel we will always assume that $d\ge 2$.

\subsection{Primitivity and recurrence} \label{subsec:primrec}
For $d\ge 2$, let $\bM = (M_n)_{n\in\ZZ}$ be a bi-infinite sequence of nonnegative matrices in $\GL(d,\RR)$. Because we are interested in the product of consecutive matrices, it is convenient to use the notation\notx{Mmn}{$M_{[m,n)}$}{product of matrices} 
\[
M_{[m,n)} = M_m M_{m+1}\cdots M_{n-1} \quad \text{for}\ m,n \in \ZZ\ \text{with}\ m \le n,
\]
where $M_{[m,m)}$ denotes the $d{\times}d$ identity matrix.
The first important property of a bi-infinite sequence $\bM$ of nonnegative matrices in $\GL(d,\RR)$, $d \ge 2$, is the following notion of primitivity.

\begin{definition}[Primitivity]\label{def:primivite}\indx{primitive!sequence!of matrices}
We say that a sequence $\bM = (M_n)_{n\in\ZZ}$ of nonnegative matrices in $\GL(d,\RR)$ is \emph{primitive (in the future)} if, for each $m \in \ZZ$, there is $n > m$ such that the matrix $M_{[m,n)}$ is  positive. 
We say that it is \emph{primitive in the past} if, for each $n \in \ZZ$, there is $m < n$ such that the matrix $M_{[m,n)}$ is positive. 
If $\bM = (M_n)_{n\in\ZZ}$ is primitive in the future and in the past, we say that it is \emph{two-sided primitive}.
\end{definition} 

\begin{remark}\label{rem:future}
In the course of this monograph we will define several notions that have two directions, one \emph{in the future} and one \emph{in the past}. Besides that, two-sided versions will occur. Since the future will be of greater importance for us, the ``default version'' for each of these properties is the future. 
Thus the suffix ``in the future'' will not be written out. 
So if we say that a sequence $\bM$ is \emph{primitive}, we mean that it is \emph{primitive in the future}.
\end{remark}

Let $\bM = (M_n)_{n\in\ZZ}$ be a primitive sequence of nonnegative matrices  with values  in $\GL(d,\RR)$. Then there is $n_1 \in \NN$ such that, for each $n \ge n_1$, $M_{[0,n)}$ is positive and, hence, by the Perron--Frobenius theorem, $M_{[0,n)}$ has a dominant positive eigenvalue $\newlambda_1(n)$\notx{lambda}{$\newlambda_1(n)$}{dominant eigenvalue}, which is simple and associated to a positive eigenvector. If the matrices in $\bM$ are integer matrices, then, since $d\ge2$, we even have $\newlambda_1(n) > 1$ because each entry of each positive vector increases strictly when multiplied by the positive integer matrix~$M_{[0,n)}$. 
Furthermore, for each $k > 0$ there is $n_k > 0$ such that for each $n \ge n_k$ the matrix $M_{[0,n)}$ can be written as the product of $k$ positive integer matrices. Thus each coefficient of $M_{[0,n)}$ (for $n \geq n_k$) is larger than~$d^{k-1}$. This implies that $\newlambda_1(n)$ tends to infinity for $n\to\infty$.

The next property we will need is the following notion of recurrence. 

\begin{definition}[Eventual recurrence] \label{def:recurrent}\indx{recurrent}
We say that a sequence $\bM=(M_n)_{n\in\ZZ}$ of matrices in $\GL(d,\RR)$ is \emph{eventually recurrent (in the future)} if there is $m \in \ZZ$ such that, for each $\ell \ge 1$, we have  $(M_n,\dots,M_{n+\ell}) = (M_m,\dots,M_{m+\ell})$ for some $n > m$. 
Analogously, we say that $(M_n)_{n\in\ZZ}$ is \emph{eventually recurrent in the past} if there is $m \in \ZZ$ such that, for each $\ell \ge 1$, we have $(M_{n-\ell},\dots,M_n) = (M_{m-\ell},\dots,M_m)$ for some $n < m$.
\end{definition}

In this Definition~\ref{def:recurrent}, we do not use the fact that the sequence $\bM$ is   specifically a sequence of matrices. Indeed, eventual recurrence is meaningful for arbitrary bi-infinite sequences $(\omega_n)_{n\in \ZZ}$ with values in any alphabet. 

\subsection{Generalized eigenvectors  and convergence} \label{subsec:eigen}
Primitivity and recurrence entail the following convergence result, which follows from \cite[pp.~91--95]{Furstenberg:60}. 

\begin{proposition}\label{prop:fur}
If a sequence $\bM = (M_n)_{n\in\ZZ}$ of nonnegative matrices in $\GL(d,\RR)$ is primitive and eventually recurrent, then there exists a positive vector~$\bu$ such that  
\begin{equation}\label{eq:furst}
\bigcap_{n\in\NN} M_{[0,n)}\RR_{\ge0}^d = \RR_{\ge0} \bu .
\end{equation}
If $\bM = (M_n)_{n\in\ZZ}$ is primitive in the past and eventually recurrent in the past, 
then there exists a positive vector~$\bv$ such that  
\begin{equation}\label{eq:furstpast}
\bigcap_{n\in\NN}{} \tr{\!M}_{[-n,0)}\RR_{\ge0}^d = \RR_{\ge0} \bv .
\end{equation}
\end{proposition}

\begin{proof}
Because Assertion \eqref{eq:furst}  concerns only matrices with nonnegative index, its proof is the same as the proof of \cite[Proposition~3.5.5]{thuswaldner2019boldsymbolsadic}. Assertion \eqref{eq:furstpast} follows by applying \eqref{eq:furst} to the sequence $(\tr{\!M}_{-n-1})_{n\in\ZZ}$.
\end{proof}

In the important case where the primitive sequence $\bM = (M)_{n\in\ZZ}$ is constant, that is, if we consider the powers of a primitive matrix~$M$ in the intersections in \eqref{eq:furst} and \eqref{eq:furstpast}, the vectors~$\bu$ and $\bv$ are the right and left Perron--Frobenius eigenvector of~$M$, respectively. 
This motivates the following definition. 

\begin{definition}[Generalized right and left eigenvector]\label{def:gea}\indx{eigenvector!generalized}
We say that a sequence $\bM = (M_n)_{n\in\ZZ}$ of matrices in $\GL(d,\RR)$ \emph{admits a generalized eigenvector (in the future)} if there exists a vector $\bu \in \RR_{\ge0}^d\setminus\{\mathbf{0}\}$\notx{u}{$\bu,\bu_n$}{generalized right eigenvector} such that $\bigcap_{n\in\NN} M_{[0,n)} \RR_{\ge0}^d = \RR_{\ge0} \bu$. 
The vector~$\bu$ is called a \emph{generalized right eigenvector} of~$\bM$.

We say that $\bM$ \emph{admits a generalized eigenvector in the past} if there exists a vector $\bv \in \RR_{\ge0}^d \setminus\{\mathbf{0}\}$\notx{v}{$\bv,\bv_n$}{generalized left eigenvector} such that $\bigcap_{n\in\NN} \tr{\!M}_{[-n,0)} \RR_{\ge0}^d = \RR_{\ge0} \bv$, and $\bv$ is called a \emph{generalized left eigenvector} of~$\bM$.
\end{definition}

Transposed matrices are used for generalized eigenvectors in the past in order to get nested cones. The generalized eigenvectors can be seen as a nonstationary analog of the Perron--Frobenius eigenvectors of a primitive matrix.

If a given bi-infinite sequence $\bM$ of matrices in $\GL(d,\RR)$ admits a generalized right eigenvector~$\bu$ and a generalized left eigenvector $\bv$, respectively, then we easily see that, for each $m \in \ZZ$, the shifted sequence $\Sigma^m \bM=(M_{n+m})_{n\in\ZZ}$ admits the generalized right eigenvector $\bu_m$ and the generalized left eigenvector $\bv_m$, respectively, with
\begin{equation}\label{eq:unvn}
\bu_m = \begin{cases}M_{[0,m)}^{-1}\, \bu & \text{for } m \ge 0, \\[.5ex] 
M_{[m,0)}\, \bu & \text{for}\ m < 0,\end{cases} \qquad 
\bv_m = \begin{cases}\tr{\!M}_{[0,m)}\, \bv & \text{for } m \ge 0, \\[.5ex] \tr{\!M}_{[m,0)}^{-1}\, \bv & \text{for}\ m < 0;
\end{cases}
\end{equation} \notx{u}{$\bu,\bu_n$}{generalized right eigenvector} \notx{v}{$\bv,\bv_n$}{generalized left eigenvector}
note that $\bu_0=\bu$ and $\bv_0 = \bv$.

\begin{example}[Farey matrices]\label{ex:farey1} \indx{Farey!matrix}\indx{matrix!Farey}
Consider a two-sided eventually recurrent sequence $\bM$ that contains both of the \emph{Farey matrices}\indx{Farey!matrix} $M_{\rF,1}$ and $M_{\rF,2}$ from \eqref{eq:sturmmat}.
Since $M_{\rF,1}M_{\rF,2}$ as well as $M_{\rF,2}M_{\rF,1}$ is positive, $\bM$~is two-sided primitive and, hence, satisfies the conditions of Proposition~\ref{prop:fur} in the past and in the future. Thus $\bM$ admits a generalized left and a generalized right eigenvector. The same is true for all shifts $\Sigma^m\bM$, $m\in \ZZ$, of this sequence. The matrices $M_{\rF,1}$ and~$M_{\rF,2}$ are intimately related to the Farey algorithm \eqref{eq:FareyIntro}, see also Section~\ref{subsec:cfdef}.
\end{example}

We now want to define some notions of convergence for sequences of matrices that play a role in the context of multidimensional continued fraction algorithms; cf.\ e.g.\ \cite[Definitions~10 and~11]{Schweiger:00}. To state these definitions, we need some notations. The standard basis vectors of $\RR^d$ are denoted by $\be_1,\ldots,\be_d$ \notx{e}{$\be_1,\ldots,\be_d$}{standard basis vectors}. We will write $\mathrm{d}(\cdot,\cdot)$\notx{d}{$\mathrm{d}(\cdot,\cdot)$}{distance function} for the Euclidean distance on~$\RR^d$ and if $\bw \in \RR^d$ and $X \subset \RR^d$, then $\mathrm{d}(\bw,X) := \inf_{\bx\in X}\mathrm{d}(\bw,\bx)$.  Moreover, for $\bw \in \RR^d\setminus\{\mathbf{0}\}$ let \notx{0orthogonal}{$\bw^\perp$}{hyperplane orthogonal to~$\bw$}
\[
\bw^\perp =\{\bx \in \RR^d \,:\, \langle \bx,\bw\rangle = 0\}. 
\]
For any $\bu, \bv \in \RR^d \setminus \{\mathbf{0}\}$ with $\bu \notin \bv^\perp$, let\notx{pe}{$\pi_{\bu,\bv}$}{projection along $\bu$ onto~$\bv^\perp$}\notx{pe}{$\tilde{\pi}_{\bu,\bv}$}{projection along $\bv^\perp$ onto~$\RR\bu$}
\begin{equation}\label{eq:projections}
\begin{aligned}
\pi_{\bu,\bv}:\ & \RR^n \to \bv^\perp, & \bx & \mapsto  \bx - \frac{\langle \bx, \bv \rangle}{\langle\bu ,\bv\rangle}\bu, \quad \text{and} \\
\tilde{\pi}_{\bu,\bv}:\ &\RR^n \to \RR\bu, & \bx & \mapsto \frac{\langle \bx, \bv \rangle }{\langle\bu, \bv\rangle}\bu,
\end{aligned}
\end{equation}
denote the projection along~$\RR\bu$ to~$\bv^\perp$ and the projection along~$\bv^\perp$ to~$\RR\bu$, respectively.

\begin{definition}[Weak, strong, exponential convergence]\label{def:wsc}
\indx{convergence!weak}
\indx{convergence!strong}
\indx{convergence!exponential}
Let $\bM = (M_n)_{n\in\ZZ}$ be a sequence of matrices in $\GL(d,\RR)$ and let $\bu \in \RR^d \setminus \{\mathbf{0}\}$. 
We say that $\bM$ is \emph{weakly convergent (in the future)} to~$\bu$ if
\begin{equation}\label{eq:weakconv}
\lim_{n\to+\infty} \mathrm{d}\bigg(\frac{M_{[0,n)}\be_i}{\lVert M_{[0,n)}\be_i\rVert},\frac{\bu}{\lVert\bu\rVert}\bigg) = 0\quad \hbox{for all}\ i\in \{1,\dots,d\}.
\end{equation} 
We say that $\bM$ is \emph{strongly convergent (in the future)} to~$\bu$ if
\begin{equation}\label{eq:strongconv}
\lim_{n\to+\infty} \mathrm{d}(M_{[0,n)}\be_i,\RR\bu) = 0\quad \hbox{for all}\ i\in \{1,\dots,d\}.
\end{equation} 
We say that $\bM$ is \emph{exponentially convergent (in the future)} to~$\bu$ if there exist $C,\alpha>0$ such that 
\begin{equation}\label{eq:expconv}
\mathrm{d}(M_{[0,n)}\be_i,\RR\bu) \le C e^{-\alpha n} \quad \hbox{for all}\ n\in \NN,\, i\in \{1,\dots,d\}.
\end{equation} 
If $\bM$ is replaced by $(\tr{\!}M_{-n-1})_{n\in\ZZ}$ in these definitions of convergence, we say that $\bM$ is weakly, strongly, and exponentially convergent in the past, respectively.
\end{definition}

These notions of convergence are invariant under shifts in the sense that, for each $m \in \ZZ$, a sequence~$(M_n)_{n\in\ZZ}$ converges weakly [strongly, exponentially] if and only if the same is true for the shifted sequence $\Sigma^m \bM=(M_{n+m})_{n\in\ZZ}$.

Weak convergence is convergence in the projective space, while strong convergence is convergence in~$\RR^d$ of the sequence of vectors $M_{[0,n)}\be_i$ to the line~$\RR\bu$, which is the same as convergence of $\pi_{\bu,\bu}(M_{[0,n)}\be_i)$ to~$\mathbf{0}$, for all $i\in \{1,\ldots,d\}$. Exponential convergence says that the speed of strong convergence is geometric. Obviously, exponential convergence implies strong convergence, which in turn implies weak convergence. According to Proposition~\ref{prop:fur}, primitivity and eventual recurrence imply weak convergence. It is easy to give examples of nonprimitive not eventually recurrent sequences that are weakly convergent; for example, many sequences of  diagonal matrices converge weakly.

\begin{remark} \label{rem:expconv}
Exponential convergence also implies absolute convergence of any series of the form $\sum_{n\ge 0} \alpha_n \lVert\mathrm{d}(M_{[0,n)}\be_{i_n},\RR\bu)\rVert$ with bounded (or slowly growing) coefficients~$\alpha_n$; this fact will be of importance later when we give a criterion for balance; see e.g.\ Theorem~\ref{theo:sufcondPisotS} below.
\end{remark}

\subsection{Pisot condition and Lyapunov exponents} \label{sec:LyapunovPisot}
Given a sequence~$\bM$ of matrices in $\GL(d,\ZZ)$, $d \ge 2$, our aim is now to give a sufficient condition for strong and even exponential convergence.  This condition, which is formulated in Theorem~\ref{th:matrixPisot} (see also the informal version in Theorem~\nameref{t:A} in the introduction), is quite general. For that purpose, we first recall basic definitions  on convergence and  Lyapunov exponents  in Section~\ref{sec:LyapunovPisot},   in particular the local Pisot condition, before proving Theorem~\ref{th:matrixPisot} in Sections~\ref{sec:LyapunovPisot2} and~\ref{sec:LyapunovPisot3}.

Let us first recall some definitions. A~\emph{Pisot number}\indx{Pisot!number} is a real algebraic integer $\newlambda>1$ whose Galois conjugates are all strictly smaller than~$1$ in modulus. An integer matrix is a \emph{Pisot matrix}\indx{Pisot!matrix} if its characteristic polynomial is the minimal polynomial of a Pisot number. This motivates the following definition of a Pisot condition for a sequence of matrices (see also Remark~\ref{rem:PisotJustification}). Contrary to the eigenvalues used for a single matrix, in the setting of sequences of matrices, singular values turn out to be more convenient than eigenvalues. Recall that, for a matrix $M\in \GL(d,\ZZ)$, $d\ge 2$, the \emph{singular values}\indx{singular value} $\delta_1(M),\ldots,\delta_d(M)$\notx{delta}{$\delta_k(\cdot)$}{$k$-th singular value} are the eigenvalues of the matrix $(\tr{\!M}M)^{1/2}$; these values are positive and real, and we always assume that they are ordered in a way that 
\begin{equation}\label{eq:singularordered}
\delta_1(M) \ge \delta_2(M) \ge \cdots \ge \delta_d(M).
\end{equation}

\begin{definition}[Local Pisot condition] 
\label{def:genPisot}\indx{Pisot!condition!local}
We say that a sequence $\bM=(M_n)_{n\in\ZZ}$ of nonnegative matrices in $\GL(d,\ZZ)$ satisfies the \emph{local Pisot condition (in the future)}\footnote{The terminology ``local'' refers to the fact we consider here a single sequence of matrices.} if 
\[
\limsup_{n\to\infty} \frac{1}{n} \log \delta_2(M_{[0,n)}) < 0.
\]
We say that $\bM$ satisfies the \emph{local Pisot condition in the past} if $(\tr{\!}M_{-n-1})_{n\in\ZZ}$ satisfies the local Pisot condition in the future.
If $\bM$ satisfies the local Pisot condition in the future and in the past, we say that it satisfies the \emph{two-sided local Pisot condition}.
\end{definition}

\begin{remark}\label{rem:PisotJustification}
For a matrix $M\in \GL(d,\ZZ)$, the product of all singular values of $M$ equals $\lvert\det(M)\rvert\!=\!1$. Hence, $\delta_2(M)<1$ implies that $\delta_1(M)> 1$, because otherwise, in view of \eqref{eq:singularordered}, the product of all singular values of $M$ would be strictly less than $1$. Therefore, if a sequence $(M_n)_{n\in\ZZ} \in \GL(d,\ZZ)$ satisfies the local Pisot condition, we gain 
\[
\begin{aligned}
\limsup_{n\to\infty} \frac{1}{n} \log \delta_1(M_{[0,n)}) &> 0, \\
\limsup_{n\to\infty} \frac{1}{n} \log \delta_j(M_{[0,n)}) &< 0 \qquad (j\in \{2,\ldots, d\}).
\end{aligned}
\]
This illustrates the analogy to the definition of a Pisot number.
\end{remark}

The  (local)  Pisot condition is invariant under shifts of sequences. Moreover, in this condition we consider a $\limsup$  instead of a limit since the existence of the limit  is hard to prove in general, and negativity of the $\limsup$ is strong enough for our purposes.  However, if limits of this kind exist for a sequence $\bM = (M_n)_{n\in\ZZ}$ of matrices in $\GL(d,\ZZ)$, one  defines the \emph{(local) $i$-th Lyapunov exponent}\indx{Lyapunov!local exponent} $\vartheta_i$ as the limit 
\begin{equation}\label{def:Lyapu}
\vartheta_i := \lim_{n\to\infty}\frac{1}{n} \log(\delta_i(M_{[0,n)})). \notx{theta}{$\vartheta_i,(\Theta_i,d_i)$}{Lyapunov exponent}
\end{equation}
The \emph{(local) Lyapunov spectrum}\indx{Lyapunov!local spectrum} is the set of all Lyapunov exponents, each one counted with its multiplicity. The Lyapunov exponents can also be defined recursively using exterior powers~$\wedge^k$\notx{0exterior}{$\wedge$}{exterior product} (see for example \cite[Proposition~3.2.7]{Arnold98}) by
\begin{equation}\label{def:Lyapu2}
\vartheta_1 + \cdots +\vartheta_k = \lim_{n\to\infty} \frac{1}{n}\log \lVert\wedge^k M_{[0,n)}\rVert \qquad (1 \le k \le d), 
\end{equation}
provided that the limits exist. Each of the two equivalent definitions, \eqref{def:Lyapu} and \eqref{def:Lyapu2}, yield that $\vartheta_1 \ge \vartheta_2 \ge \cdots \ge \vartheta_d$. For the powers of a single invertible matrix, i.e., when $\bM=(M)_{n\in\ZZ}$ is a constant sequence, it is well known that the Lyapunov exponents are just the logarithms of the moduli of the eigenvalues of $M$; see e.g.\ \cite[3.2.3~Example~(ii)]{Arnold98}. If $M$ has complex eigenvalues, they come in pairs of conjugates. Thus the corresponding Lyapunov exponent of $\bM$ cannot be simple.

\begin{definition}[Strong Pisot condition]\label{def:strongPisot}\indx{Pisot!condition!strong}
A~sequence~$\bM=(M_n)_{n\in\ZZ}$ of nonnegative matrices in $\GL(d,\ZZ)$ satisfies the \emph{strong (local) Pisot condition} if it admits Lyapunov exponents $\vartheta_1 \ge \cdots \ge \vartheta_d$ (listed according to their multiplicity) such that $\vartheta_1>0>\vartheta_2$.
The sequence~$\bM$ satisfies the \emph{strong (local) Pisot condition in the past} if $(\tr{\!M}_{-n-1})_{n\in\ZZ}$ satisfies the strong Pisot condition. If $\bM$ satisfies the strong local Pisot condition in the future and in the past, we say that it satisfies the \emph{strong two-sided local Pisot condition}.
\end{definition}

The following lemma is immediate from the definitions.

\begin{lemma}\label{lem:strongweak}
Let $\bM$ be a sequence of nonnegative matrices in $\GL(d,\ZZ)$. If $\bM$ satisfies the strong local Pisot condition, then it satisfies the local Pisot condition; the analogous result is true for the~past.
\end{lemma}

\begin{remark}\label{rem:PisotGloLoc}
Here we use the terminlolgy {\em local} Lyapunov exponent and {\em local} Pisot condition, since these objects are defined for a single sequence of matrices. We will later define {\em global} versions of these notions that will concern generic elements of families of sequences of matrices that are defined by some condition (see for instance Definitions~\ref{def:genly} and~\ref{def:gengenPisot}). Nevertheless, if there is no risk of confusion, we omit the prefix ``local''.
\end{remark}

In the remaining part of this section, we only work with the local Pisot condition, not the strong one. As mentioned before, this is essential because it is fairly easy to come up with concrete sequences~$\bM$ satisfying the local Pisot condition; see for instance Example~\ref{ex:AR1}, or Propositions~\ref{prop:brunpisot:2} and~\ref{prop:brunpisot:1}. Exhibiting a nontrivial aperiodic sequence~$\bM$ satisfying the strong Pisot condition seems to be hard (of course, examples using matrices that are simultaneously diagonalizable can easily be constructed). 

In all that follows, $\lVert \cdot \rVert$ denotes an arbitrary norm, \notx{0norm}{$\lVert\cdot\rVert$}{arbitrary norm} $\lVert\cdot\rVert_p$ \notx{0norm}{$\lVert\cdot\rVert_p$}{$p$-norm} the $p$-norm for $p \in [1,\infty)$ (which is the Euclidean norm for $p=2$) and $\lVert\cdot\rVert_\infty$ \notx{0norm}{$\lVert\cdot\rVert_\infty$}{maximum norm} the maximum norm on~$\RR^d$. The matrix norms associated to these norms will also be denoted by $\lVert \cdot \rVert\|$, $\lVert\cdot\rVert_p$ and $\lVert\cdot\rVert_\infty$, respectively. The standard scalar product is denoted by~$\langle\cdot,\cdot\rangle$\notx{0scalar product}{$\langle\cdot,\cdot\rangle$}{scalar product}. We will confine ourselves to unimodular integer matrices with nonnegative entries, i.e., to matrices taken from the class 
\[
\cM_d = \NN^{d\times d} \cap \GL(d,\ZZ) \notx{Md}{$\cM_d$}{set of nonnegative unimodular integer matrices}.
\]

\subsection{Pisot condition and exponential convergence} \label{sec:LyapunovPisot2}
We will  now prove that, for a sequence $\bM = (M_n)_{n\in\ZZ} \in \cM_d^\ZZ$, the Pisot condition together with the (mild) \emph{growth condition}\indx{growth condition} $\lim_{n\to\infty} \frac{1}{n} \log \lVert M_n\rVert = 0$ implies exponential convergence. This result will be stated as Proposition~\ref{prop:strongcv2sided} and forms a key stem on our way to prove Theorem~\ref{th:matrixPisot}.  
Note that this growth condition is reminiscent of the log-integrability assumption \eqref{eq:logint} that holds in the generic setting; it holds trivially in the case of an additive continued fraction algorithm, where the set of matrices $\{M_n : n\in\ZZ\}$ is finite. The following proposition is a first step toward Proposition~\ref{prop:strongcv2sided}. 

\begin{proposition} \label{p:localOseledets}
Let  $\bM = (M_n)_{n\in\ZZ} \in \cM_d^{\ZZ}$ be a sequence of nonnegative unimodular integer matrices satisfying the local Pisot condition and the growth condition   $\lim_{n\to\infty} \frac{1}{n} \log \lVert M_n\rVert = 0$.
Then the sequence $\bM$ admits a generalized right eigenvector $\bu \in \RR^d \setminus \{\mathbf{0}\}$ such that $\limsup_{n\to\infty} \frac{1}{n} \log \lVert\tr{\!M}_{[0,n)} \vert_{\bu^\perp}\rVert < 0$. 
\end{proposition}

The proof, inspired by the proofs of \cite[Proposition~1.3]{Ruelle:79} and \cite[Lemma~3.4.6]{Arnold98}, uses two lemmas. Before stating them, we adopt some notation.

Let $(M_n)_{n\in\ZZ} \in \cM^\ZZ$ and $n\in\NN$. Write $\delta_i(n) = \delta_i(M_{[0,n)})$, $1\le i \le d$, for the $i$-th singular value of the matrix $M_{[0,n)}$ (ordered decreasingly, as in~\eqref{eq:singularordered}). Let $\tr{\!M}_{[0,n)} = B(n) D(n) A(n)$ be the singular value decomposition of $\tr{\!M}_{[0,n)}$, with diagonal matrix $D(n) = \mathrm{diag}(\delta_1(n),\dots,\delta_d(n))$ and orthogonal  matrices $A(n), B(n)$. Set $\ba_i(n) = \tr{\!A}(n) \be_i$, where $\{\be_1, \ldots,\be_d\}$ is the standard basis of~$\RR^d$.
By equivalence of norms, it suffices to prove Proposition~\ref{p:localOseledets} for the Euclidean norm.

\begin{lemma} 
Let $(M_n)_{n\in\ZZ} \in \cM_d^{\ZZ}$ and $n \in \NN$.
Then we have 
\begin{equation}\label{eq:Mdeltainorm}
\lVert\tr{\!M}_{[0,n)} \ba_i(n)\rVert_2 = \delta_i(n) \quad(1\le i\le d).
\end{equation}
In particular, the restriction of $\tr{\!M}_{[0,n)} $ to the hyperplane $\ba_1(n)^\perp$ satisfies 
\begin{equation}\label{eq:LocOsed2bound2} \lVert {\tr{{\!M}_{[0,n)}}}|_{\ba_1(n)^{\perp}}\rVert_2=\delta_2(n).
\end{equation} 
\end{lemma}

\begin{proof}
By definition, we have  $\tr{\!M}_{[0,n)} \ba_i(n)=B(n)D(n)\be_i$ for $1\le i \le d$. Therefore, \eqref{eq:Mdeltainorm} follows from the definition of $D(n)$ and the orthogonality of $B(n)$, and \eqref{eq:LocOsed2bound2} is a consequence of~\eqref{eq:Mdeltainorm}.
\end{proof}

The local Pisot condition implies that $\delta_2(n)$ decreases exponentially fast,  and the growth condition $\lim_{n\to\infty} \frac{1}{n} \log \lVert M_n\rVert_2 = 0$ implies that $\lVert M_n\rVert_2$ increases subexponentially.
Hence, we may choose $\beta < 0$ and $C>0$ in a way that 
\begin{equation}\label{eq:d2Mn_betabound}
\delta_2(n)  \lVert M_n\rVert_2 \le C  e^{\beta n} \qquad(n\in \NN). 
\end{equation}
We will use this to prove that a vector $\by(n)$ orthogonal to $\ba_1(n)$ is nearly orthogonal to $\ba_1(n{+}k)$. More precisely, we establish the following result.

\begin{lemma} \label{l:korthogonal}
Let  $\bM = (M_n)_{n\in\ZZ} \in \cM_d^{\ZZ}$ be a sequence of nonnegative unimodular integer matrices satisfying the local Pisot condition and the growth condition   $\lim_{n\to\infty} \frac{1}{n} \log \lVert M_n\rVert = 0$.
Then there exists $K>0$ such that, for all $k,n \in \NN$ and for each unit vector $\by(n)$ in $\ba_1(n)^\perp$, we have 
\begin{equation} \label{e:korthogonal}
\lvert\langle \ba_1(n{+}k), \by(n) \rangle\rvert \le K \frac{e^{\beta n}}{\delta_1(n)}. 
\end{equation}
\end{lemma}

\begin{proof}
We start by the case $k=1$.
Let $\by(n)$ be a unit vector in $\ba_1(n)^\perp$. We  represent $\by(n)$ as $ \by(n)= \alpha(n) \ba_1(n{+}1) + \bz(n)$ with $\bz(n) \in \ba_1(n{+}1)^\perp$. To prove~\eqref{e:korthogonal}, we first want to make sure that $|\alpha(n)|$ is small. Since $\bz(n) \in \ba_1(n{+}1)^\perp$, the images of $\bz(n)$ and $\ba_1(n+1)$ under $\tr{\!M}_{[0,n+1)}$ are orthogonal, i.e., 
\[
\langle \tr{\!M}_{[0,n+1)} \bz(n), \tr{\!M}_{[0,n+1)}\ba_1(n{+}1) \rangle=0.
\]
This orthogonality implies that
\begin{equation}\label{eq:LocOsed2bound3}
\begin{aligned}
\lVert\tr{\!M}_{[0,n+1)} \by(n)\rVert_2 &= \lVert\alpha(n) \tr{\!M}_{[0,n+1)} \ba_1(n{+}1 )+ \tr{\!M}_{[0,n+1)} \bz(n)\rVert_2 \\
&\ge \lVert\alpha(n) \tr{\!M}_{[0,n+1)} \ba_1(n{+}1)\rVert_2 = |\alpha(n)| \, \delta_1(n{+}1),
\end{aligned}
\end{equation}
where the last equality is a consequence of \eqref{eq:Mdeltainorm}.
Combining \eqref{eq:LocOsed2bound2} and \eqref{eq:LocOsed2bound3}, we obtain
\[
|\alpha(n)| \le \frac{\lVert\tr{\!M}_{[0,n+1)} \by(n)\rVert_2}{\delta_1(n{+}1)}\le  \frac{\lVert\tr{\!M}_{n}\rVert_2\,\lVert\tr{\!M}_{[0,n)} \by(n)\rVert_2}{\delta_1(n{+}1)} \le 
 \lVert M_n\rVert_2 \, \frac{\delta_2(n)}{\delta_1(n{+}1)}
\]
and, hence, \eqref{eq:d2Mn_betabound} yields
\begin{equation}\label{eq:angleuy}
\begin{aligned}
|\langle \ba_1(n{+}1), \by(n) \rangle | & = |\langle \ba_1(n{+}1), \alpha(n) \ba_1(n{+}1) + \bz(n) \rangle | \\
& = |\alpha(n)| \le  \lVert M_n\rVert_2\frac{\delta_2(n)}{\delta_1(n{+}1)} \le  \frac{C e^{\beta n}}{\delta_1(n{+}1)}.
\end{aligned}
\end{equation}
Since, by positivity of the matrices $M_n$, $n\in\NN$, the sequence $(\delta_1(n) )_{n\in\NN}$ is nondecreasing, this proves the case $k=1$ of the lemma.

We now prove the general case. For $i,j\in \NN$, let $\beta_{i,j}$ be the angle between $\ba_1(n{+}i)$ and~$\ba_1(n{+}j)$, and let $\gamma_{i,j} = \frac{\pi}{2} - \beta_{i,j}$. The angle~$\gamma_{i,j}$ is the angle between the vector~$\ba_1(n{+}i)$ and the hyperplane~$\ba_1(n{+}j)^\perp$. This angle can be written as
\[
\gamma_{i,j} = \arccos\big(\max\big\{\langle \ba_1(n{+}i), \by(n{+}j) \rangle \,:\, \by(n{+}j) \in (\ba_1(n{+}j))^\perp,\, \lVert\by(n{+}j)\rVert_2 = 1\big\}\big)
\]
Combining this with~\eqref{eq:angleuy}, we gain $\cos\gamma_{i,i+1} \le  \frac{C e^{\beta n}}{\delta_1(n{+}1)}$. From the triangle inequality, we obtain
\begin{equation}\label{eq:anglesum}
\beta_{k,0} \le \sum_{j=0}^{k-1} \beta_{j+1,j}. 
\end{equation}
Using the well known equality $x\le 2\sin x$ for $x\in [0,\frac {\pi}2]$, we get 
\begin{equation}\label{eq:angleuy2}
\beta_{j+1,j} \le  2\sin\beta_{j+1,j} = 2\cos \gamma_{j+1,j} \le \frac{2C e^{\beta (n+j)}}{\delta_1(n{+}j{+}1)} \qquad (j \ge 0).
\end{equation}
Since $\sin x \le x$ for $x\ge 0$, we gain from the definition of~$\gamma_{i,j}$,  \eqref{eq:anglesum}, and \eqref{eq:angleuy2} that there is $K>0$ such that, for all $k,n \in \NN$ and each unit vector $\by(n)\in \ba_1(n)^\perp$, we have
\[
\begin{aligned}
|\langle \ba_1(n{+}k), \by(n) \rangle| & \le  \cos\gamma_{k,0} = \sin\beta_{k,0}
\le \beta_{k,0} \le \sum_{j=0}^{k-1}  \beta_{j+1,j} \\
& \le \sum_{j=0}^{k-1} \frac{2C e^{\beta (n+j)}}{\delta_1(n{+}j{+}1)} \le K \frac{e^{\beta n}}{\delta_1(n)}.
\end{aligned}
\]
In the last inequality, we have again used the fact that $(\delta_1(n))_{n\in\NN}$ is nondecreasing. 
Therefore, \eqref{e:korthogonal} holds.
\end{proof}

\begin{proof}[Proof of Proposition~\ref{p:localOseledets}]
Since Lemma~\ref{l:korthogonal} implies that $(\ba_1(n))_{n\in\NN}$ is a Cauchy sequence (this also follows directly from \eqref{eq:angleuy2}) and, hence, converges to a unit vector~$\bu$, we conclude from \eqref{e:korthogonal} that  
\begin{equation} \label{e:orthogonalU1}
\max \big\{|\langle \bu, \by(n) \rangle| \,:\, \by(n) \in \ba_1(n)^\perp,\, \lVert\by(n)\rVert_2 = 1\big\} \le K\, \frac{e^{\beta n}}{\delta_1(n)} \quad(n\in \NN). 
\end{equation}
We have 
\begin{equation} \label{e:orthogonalU1B}
\begin{aligned}
&\max \big\{|\langle \bu, \by(n) \rangle| \,:\, \by(n) \in \ba_1(n)^\perp,\, \lVert\by(n)\rVert_2 = 1\big\} \\
& \hspace{4em} = \max \{|\langle \ba_1(n), \by \rangle| \,:\, \by \in \bu^\perp,\, \lVert\by\rVert_2 = 1\}
\end{aligned} \quad(n\in \NN)
\end{equation}
because both quantities are equal to the sine of the angle between $\bu$ and~$\ba_1(n)$.

Let $\by \in \bu^\perp$ with $\lVert\by\rVert_2 = 1$. Write $\by = \eta(n) \ba_1(n) + \bz(n)$ with $\bz(n) \in \ba_1(n)^\perp$ for each $n \in \NN$. By \eqref{e:orthogonalU1} and \eqref{e:orthogonalU1B}, we have $|\eta(n)| \le K \frac{e^{\beta n}}{\delta_1(n)}$, and therefore $\lVert\bz(n)\rVert_2 \le 1 {+} |\eta(n)|$. Using \eqref{eq:Mdeltainorm}, \eqref{eq:LocOsed2bound2}, and \eqref{eq:d2Mn_betabound}, we get
\[
\begin{aligned}
\lVert\tr{\!M}_{[0,n)} \by\rVert_2 & \le \lVert \eta(n) \tr{\!M}_{[0,n)} \ba_1(n) \rVert_2 +  \lVert \tr{\!M}_{[0,n)} \bz(n) \rVert_2 \\
& \le K \frac{e^{\beta n}}{\delta_1(n)} \lVert \tr{\!M}_{[0,n)} \ba_1(n) \rVert_2 +  \lVert \tr{\!M}_{[0,n)} \bz(n) \rVert_2 \\
&\le K e^{\beta n} + \delta_2(n)(1{+}|\eta(n)|) \le (K{+}2C) e^{\beta n}
\end{aligned}
\]
for $n$ large, because $|\eta(n)|\le 1$ for $n$ large. This immediately yields
\[
\limsup_{n\to\infty} \tfrac{1}{n} \log \lVert\tr{\!M}_{[0,n)} \by\rVert_2 \le \beta \qquad \mbox{for all $\by \in \bu^\perp$ with $\lVert\by\rVert_2 = 1$},
\]
thus
\begin{equation} \label{e:exp2beta}
\limsup_{n\to\infty} \frac{1}{n} \log \lVert\tr{\!M}_{[0,n)} \vert_{\bu^\perp}\rVert_2 \le \beta < 0.
\end{equation}

It remains to prove that $\bu$ is a generalized right eigenvector of~$\bM$. 
Note that, for $\bb_1, \bb_2 \in \RR^d$, if $\bx$ stands for the orthogonal projection of $\bb_1$ to $\RR \bb_2$, then, by the definition of the scalar product, we get $|\langle \bb_1,\bb_2 \rangle|= \lVert \bx \rVert  \cdot \lVert \bb_2 \rVert$.  Using this fact, for $1 \le i \le d$ and $n\in\NN$, we have
\begin{equation}\label{eq:transposed}
\begin{aligned}
\mathrm{d}(M_{[0,n)}\be_i,\RR\bu) & = \lVert \pi_{\bu,\bu} M_{[0,n)}\be_i \rVert_2 \\
&= \bigg| \bigg\langle M_{[0,n)}\be_i, \frac{\pi_{\bu,\bu} M_{[0,n)}\be_i}{\lVert \pi_{\bu,\bu} M_{[0,n)}\be_i \rVert_2}\bigg\rangle  \bigg| \\
& \le \max\big\{ | \langle M_{[0,n)}\be_i, \bx \rangle | \,:\,  \bx \in \bu^\perp,\, \lVert \bx \rVert_2=1 \big\} \\ 
&= \max\big\{ | \langle \be_i, \tr{\!M}_{[0,n)} \bx \rangle | \,:\,  \bx \in \bu^\perp,\, \lVert \bx \rVert_2=1 \big\} \\
& \le  \lVert  \tr{\!M}_{[0,n)}|_{\bu^\perp} \rVert_2 .
\end{aligned}
\end{equation}
Together with \eqref{e:exp2beta}, this implies that $\bM$ converges exponentially to~$\bu$; see Definition~\ref{def:wsc}.
Since exponential convergence implies weak convergence, $\bu$~is a generalized right eigenvector of~$\bM$. This proves the proposition.
\end{proof}

We can now relate the local Pisot condition with exponential convergence to a generalized right eigenvector.

\begin{proposition} \label{prop:strongcv2sided}
Let  $\bM = (M_n)_{n\in\ZZ} \in \cM_d^{\ZZ}$ be a sequence of nonnegative unimodular integer matrices that satisfies the growth condition $\lim_{n\to\infty} \frac{1}{n} \log \lVert M_n\rVert =0$.
Then the following assertions are equivalent.
\begin{enumerate}[\upshape (i)]
\item \label{i:7131}
The sequence $\bM$ satisfies the local Pisot condition.
\item \label{i:7132}
The sequence $\bM$ admits a generalized right eigenvector~$\bu$ and converges exponentially to~$\bu$.
\item \label{i:7133}
There exists  a vector $\bu \in \RR^d\setminus\{\mathbf{0}\}$ such that $\limsup_{n\to\infty} \frac{1}{n} \log \lVert\tr{\!M}_{[0,n)} \vert_{\bu^\perp}\rVert < 0$. 
\end{enumerate}
If we replace the growth condition by  a similar condition holding in the past, namely  $\lim_{n\to-\infty} \frac{1}{n} \log \lVert M_n\rVert = 0$, then the analogous equivalences hold for the past.
\end{proposition}

\begin{proof}
We only prove the equivalences for the future. The according  equivalences for the past  follow along the same lines.
The implication (\ref{i:7131}) $\Rightarrow$ (\ref{i:7133}) follows from Proposition~\ref{p:localOseledets}, and (\ref{i:7133}) $\Rightarrow$ (\ref{i:7132}) is established in the last paragraph of the proof of Proposition~\ref{p:localOseledets}  (starting after \eqref{e:exp2beta}).

It remains to show that (\ref{i:7132}) implies~(\ref{i:7131}).  Let $n \in \NN$ be arbitrary but fixed. We recall that the image of the unit sphere under~$M_{[0,n)}$ is an ellipsoid~$E(n)$ with lengths of its semiaxes given by the singular values $\delta_1(n) \ge \cdots \ge \delta_d(n)$. 
Here, we write $\delta_i(n) = \delta_i(M_{[0,n)})$.  By~(\ref{i:7132}), $\mathrm{d}(M_{[0,n)}\be_i,\RR\bu) \le C_1 e^{\beta n}$ for some $C_1>0$ and $\beta <0$. The ellipsoid $E(n)$  is contained in a cylinder around~$\bu$ of radius $ C_1 e^{\beta n}$.
Thus the Euclidean length of $d{-}1$ of its semiaxes has to be bounded by $ C_2 e^{\beta n}$ for some $C_2>0$.  In particular, $\delta_2(n) \leq  C_2 e^{\beta n}$, which proves~(\ref{i:7131}).
\end{proof}

The next step toward the proof of Theorem~\ref{th:matrixPisot} is Proposition~\ref{prop:domev}. For a given sequence $\bM\in\cM_d^\ZZ$ satisfying the Pisot condition, we establish  a ``Pisot property''  for the growth of  eigenvalues of the products~$M_{[0,n)}$, namely an exponential  dominant behaviour for the largest  eigenvalue.  We need the following preparatory~result.

\begin{lemma}\label{lem:unwconv}
Assume that the primitive sequence $\bM = (M_n)_{n\in\ZZ} \in \cM_d^{\ZZ}$ admits a generalized right eigenvector~$\bu$ with $\lVert\bu\rVert=1$. Choose $m$ in a way that $M_{[0,m)}$ is positive. For $n \ge m$, let $\bu_1(n)$ be the positive dominant eigenvector of~$M_{[0,n)}$ satisfying $\lVert\bu_1(n)\rVert=1$. Then $\lim_{n\to\infty} \bu_1(n) = \bu$. 
\end{lemma}

\begin{proof}
Because $\bM$ admits a generalized right eigenvector~$\bu$, the sequence of vectors $\frac{M_{[0,n)}\be_i}{\lVert M_{[0,n)}\be_i\rVert}$  converges to~$\bu$ for all $i \in \{1,\dots, d\}$. Thus, by convexity, for each sequence of positive vectors $(\bx_n)_{n\in \NN}$, the sequence $\big(\frac{M_{[0,n)}\bx_n}{\lVert M_{[0,n)}\bx_n\rVert}\big)_{n\in \NN}$ converges to~$\bu$. 
In particular, $\big(\frac{M_{[0,n)}\bu_1(n)}{\lVert M_{[0,n)}\bu_1(n)\rVert}\big)_{n\ge m}$ converges to~$\bu$. Since the dominant eigenvector satisfies $\bu_1(n) = \frac{M_{[0,n)}\bu_1(n)}{\lVert M_{[0,n)}\bu_1(n)\rVert}$, the lemma follows.
\end{proof}

\begin{proposition}\label{prop:domev} 
Assume that  $\bM = (M_n)_{n\in\ZZ} \in \cM_d^{\ZZ}$ is primitive and satisfies the local Pisot condition. Choose $m \ge 0$ in a way that $M_{[0,m)}$ is a positive matrix. 
Then there exist $\alpha > 0$ and $C > 0$ such that, for $n\ge m$ we have
\begin{equation}\label{lamlam1}
\newlambda_1(n) \ge C\, e^{\alpha n}
\end{equation}
for the dominant eigenvalue $\newlambda_1(n)$ of~$M_{[0,n)}$ and 
\begin{equation}\label{lamlam}
|\newlambda(n)| \le C\, e^{-\alpha n} 
\end{equation}
for all other eigenvalues $\newlambda(n)$ of~$M_{[0,n)}$.
\end{proposition}

\begin{proof} 
Throughout this proof, we work with the Euclidean norm $\lVert\cdot\rVert_2$. For $n \ge m$, let $\bu_1(n)$ be the normalized positive dominant eigenvector of the positive matrix  $M_{[0,n)}$. 

We first prove \eqref{lamlam1}. Let $n \ge m$. We recall that the image of the unit sphere under~$M_{[0,n)}$ is an ellipsoid~$E(n)$ with lengths of its semiaxes given by the singular values $\delta_1(n) \ge \cdots \ge \delta_d(n)$. 
Here, we write $\delta_i(n) = \delta_i(M_{[0,n)})$. 
Let $\bw_1(n)$ be the positive unit vector in the direction of the  largest axis of~$E(n)$ (where positivity is due to the fact that $\bw_1(n)$ is an eigenvector for the dominant eigenvalue $\delta_1(n)^2$ of the positive and, hence, primitive, matrix $\tr{\!M}_{[0,n)} M_{[0,n)}$). 
The singular values $\delta_2(n), \dots, \delta_d(n)$ are the lengths of the semiaxes that are situated on the intersection of the ellipsoid $E(n)$ with the hyperplane $H(n) = \bw_1(n)^\perp$. 
This hyperplane contains no nonnegative element other than the origin. Since $M_{[0,n)}^{-1} H(n) = (\tr{\!M}_{[0,n)} \bw_1(n))^\perp$ and $\tr{\!M}_{[0,n)} \bw_1(n)$ is a positive vector, we gain that $M_{[0,n)}^{-1} H(n)$ contains no positive element. 
For any $\bx \in M_{[0,n)}^{-1}H(n)$, we have $\lVert M_{[0,n)} \bx\rVert_2 \le \delta_2(n) \lVert\bx\rVert_2$.

Let $\bw = \delta_1(n) M_{[0,n)}^{-1} \bw_1(n)$. 
Since $H(n)$ does not intersect the interior of the positive cone, we have $\RR \bu_1(n) \oplus M_{[0,n)}^{-1}H(n) = \RR^d$, hence, 
\begin{equation}\label{eq:vxy}
\bw = x\, \bu_1(n) + y\, \bu_2(n), 
\end{equation}
for some $\bu_2(n) \in M_{[0,n)}^{-1}H(n)$ of unit norm and some $x,y \in \RR$.
As $(\bu_1(n))_{n\ge m}$ converges to the (positive) generalized right eigenvector~$\bu$ of~$\bM$ by Lemma~\ref{lem:unwconv}, there is $m_2 \ge m$ and $\gamma >0$ such that the angle between $\bu_1(n)$ and the nonpositive vector $\bu_2(n)$ is bounded from below by $\gamma $ for $n \ge m_2$. 
From now, we assume that $n \ge m_2$.
Because $\bu_1(n), \bu_2(n)$, and $\bw$ are all unit vectors, this implies that $|x|,|y| \le \frac{1}{\sin\gamma}$. Multiplying \eqref{eq:vxy} by $M_{[0,n)}$ gives 
\[
\delta_1(n) \bw_1(n) = M_{[0,n)} \bw = x \newlambda_1(n) \bu_1(n) + y M_{[0,n)} \bu_2(n).
\] 
Observing that $\lVert M_{[0,n)} \bu_2(n)\rVert_2 \le \delta_2(n)$, taking norms we obtain
\begin{equation}\label{eq:delta1est}
\delta_1(n) = \lVert  M_{[0,n)} \bw  \rVert_2 \le \frac{\newlambda_1(n)+\delta_2(n)}{\sin\gamma}
\end{equation}
for $n \ge m_2$. 
By the local Pisot condition, we have $\delta_2(n) < e^{-\alpha n}$ and thus, by unimodularity of $M_{[0,n)}$, $\delta_1(n) > e^{\alpha n}$ for some $\alpha > 0$ and large enough~$n$.
This implies that $\newlambda_1(n) > (\sin\gamma) e^{\alpha n} - e^{-\alpha n}$, i.e., \eqref{lamlam1} holds for some $C > 0$.

We now turn to the proof of \eqref{lamlam}.
Let $n\ge m$. Suppose that $\newlambda(n)$ is real, and let $\bu(n)$ be a unit eigenvector associated to the eigenvalue $\newlambda(n)$. Since $\wedge^2 M_{[0,n)}\, (\bu_1(n) {\wedge} \bu(n)) = \newlambda_1(n) \newlambda(n)\, ( \bu_1(n) {\wedge} \bu(n))$, one has 
\begin{equation} \label{eq:wedge2bound}
\lVert\wedge^2 M_{[0,n)}\rVert_2 \ge \newlambda_1(n) |\newlambda(n)|.
\end{equation}
If $\newlambda(n)$ is not real, then there is a conjugate eigenvalue with a conjugate eigenvector; take the intersection of the 2-dimensional complex space generated by these eigenvectors with the real space: this space is invariant under $M_{[0,n)}$, and the restriction of  $M_{[0,n)}$ to it is a conformal map which changes norm by $|\newlambda(n)|$ which shows that \eqref{eq:wedge2bound} also holds in this case.

Combining \eqref{eq:wedge2bound} with \eqref{eq:delta1est} we have
\begin{equation}\label{eq:lambda2est}
\newlambda_1(n) |\newlambda(n)| \le \lVert\wedge^2 M_{[0,n)}\rVert_2 =\delta_1(n)\delta_2(n)
\le \frac{\newlambda_1(n)+\delta_2(n)}{\sin\gamma}\delta_2(n).
\end{equation}
Thus, since $\newlambda_1(n) \geq  C e^{\alpha n}$  by \eqref{lamlam1} and $\delta_2(n) < e^{-\alpha n}$, for $n$ large enough, we have $\newlambda_1(n) \ge \delta_2(n)$. Therefore, dividing \eqref{eq:lambda2est} by $\newlambda_1(n)$ we gain
\[
|\newlambda(n)| < \frac{2}{\sin\gamma} e^{-\alpha n},
\]
which implies that \eqref{lamlam} holds for some constant~$C$.
\end{proof}

It follows from the previous result that if $\bM = (M_n)$ satisfies the  Pisot condition,  all matrices~$M_{[0,n)}$ are Pisot matrices for $n$ large enough  (recall that a Pisot matrix is characterized by the behavior of its eigenvalues,  namely  all its eigenvalues except the dominant eigenvalue are strictly smaller than~1 in modulus). The converse does not hold according to Example~\ref{ex:wm} below.

\subsection{Consequences of the Pisot condition}\label{sec:LyapunovPisot3}
We now state further consequences of the local Pisot condition, with the following important irreducibility property, derived from Proposition~\ref{prop:domev}. This irreducibility property will be crucial in proving a criterion for rational independence of the entries of the generalized right eigenvector of~$\bM$; see Corollary~\ref{cor:indepNew}, and also  Theorem~\ref{theo:sufcondPisotS} in the metric framework. 
The following definition goes back to \cite[Section~2.2]{BST:19} for one-sided sequences.

\begin{definition}[Algebraic irreducibility]\label{def:AI}\indx{irreducible!algebraically}
A sequence $\bM = (M_n)_{n\in\ZZ} \in \cM_d^{\ZZ}$ is called \emph{algebraically irreducible (in the future)} if for each $m \in \ZZ$, there is $n_0 \in \NN$ such that the characteristic polynomial of $M_{[m,n)}$ is irreducible for each $n \ge n_0$.
\end{definition}

We now relate the Pisot condition to algebraic irreducibility.

\begin{corollary} \label{bst8.7} 
If a primitive sequence $\bM = (M_n)_{n\in\ZZ} \in \cM_d^{\ZZ}$  satisfies the local Pisot condition, then $\bM$ is algebraically irreducible.
\end{corollary}

\begin{proof} 
By the shift invariance of the Pisot condition, it suffices to show that for all large enough~$n$, the characteristic polynomial of~$M_{[0,n)}$ is irreducible. By Proposition~\ref{prop:domev}, there is $n_0 \in \NN$ so that each eigenvalue of~$M_{[0,n)}$ which is not dominant is strictly smaller than~$1$ for each $n \ge n_0$. If the characteristic polynomial of~$M_{[0,n)}$ were reducible, then the factor which does not have the dominant eigenvalue as a root would be an integral polynomial whose roots all have modulus strictly less than~$1$, which is impossible.  (Note that it is important here that the matrices are invertible, so that $0$ cannot occur as an eigenvalue.)
\end{proof}

The following result is a variant of \cite[Lemma~4.2]{BST:19}. 

\begin{corollary} \label{cor:indepNew} 
If a primitive sequence $\bM = (M_n)_{n\in\ZZ} \in \cM_d^{\ZZ}$ satisfies the local Pisot condition and the growth condition $\lim_{n\to\infty} \frac{1}{n} \log \lVert M_n\rVert = 0$, then the coordinates of any generalized right eigenvector~$\bu$ are rationally independent. 
\end{corollary}

\begin{proof}
Proposition~\ref{prop:strongcv2sided} implies that $\bM$ converges strongly to $\bu$.
Suppose that $\bu$ has rationally dependent coordinates.  Then $\langle \bx, \bu  \rangle= 0$ for some fixed $\bx \in \ZZ^d \setminus \{\mathbf{0}\}$. Thus, by strong convergence to $\bu$, \eqref{eq:strongconv} implies that there is a constant $C>0$ such that
\footnote{Here we use the same property of the scalar product as in the proof of Proposition~\ref{prop:strongcv2sided}.}
\[
\begin{aligned}
\langle \tr{\!M}_{[0,n)}\, \bx,  \be_i  \rangle & = \langle \bx,  M_{[0,n)}\, \be_i \rangle = \langle \pi_{\bu,\bu}\bx,  M_{[0,n)}\, \be_i \rangle \\
&\le \lVert \pi_{\bu,\bu}\bx\rVert_2\cdot  \lVert\pi_{\bu,\bu} M_{[0,n)}\, \be_i\rVert_2 =  \lVert\pi_{\bu,\bu}\bx \rVert_2 \cdot\mathrm{d}(M_{[0,n)}\, \be_i, \RR\bu)<C
\end{aligned}
\]
for all $i \in  \{1,\ldots,d\}$ and all $n\in \NN$.  Therefore $\lVert\tr{\!M}_{[0,n)}\, \bx\rVert_\infty<C$, and because $\tr{\!M}_{[0,n)}\, \bx \in \ZZ^d$, there is $k \in \NN$ such that $\tr{\!M}_{[0,\ell)}\, \bx = \tr{\!M}_{[0,k)}\, \bx$ for infinitely many $\ell > k$.  Being a product of regular matrices, the matrix $M_{[0,k)}$ is regular. 
Thus $\tr{\!M}_{[0,k)}\, \bx \neq \mathbf{0}$ is an eigenvector of~$\tr{\!M}_{[k,\ell)}$ to the eigenvalue~$1$, and, hence, $M_{[k,\ell)}$ has reducible characteristic polynomial for large~$\ell$. Because this contradicts Corollary~\ref{bst8.7}, the result follows.
\end{proof}

Versions for the past can be formulated for Proposition~\ref{prop:domev} and for Corollaries~\ref{bst8.7} and~\ref{cor:indepNew}, with analogous proofs. 

We sum up the most important results of the present section and of Section~\ref{sec:LyapunovPisot2} in the following theorem, which emphasizes the strength of the Pisot condition; it is formulated as Theorem~\nameref{t:A} in the introduction.

\begin{theorem}\label{th:matrixPisot}
For $d\ge 2$, let $\bM = (M_n)_{n\in\ZZ} \in \cM_d^{\ZZ}$ be a primitive sequence of nonnegative unimodular integer matrices satisfying the local Pisot condition and the growth condition $\lim_{n\to\infty} \frac{1}{n} \log \lVert M_n\rVert = 0$. Then $\bM$ is algebraically irreducible and admits a generalized right eigenvector with rationally independent coordinates to which it converges exponentially. 

An analogous result holds for the past.
\end{theorem}

We illustrate our theory by an example that comes from the theory of multidimensional continued fraction algorithms.

\begin{example}\label{ex:AR1}
We consider the set $\cM_{\rAR} = \{M_{\rAR,1},M_{\rAR,2},M_{\rAR,3}\}$ of \emph{Arnoux--Rauzy matrices}\indx{Arnoux--Rauzy!matrix}\indx{matrix!Arnoux--Rauzy}\notx{AxR}{$\rAR$}{object related to the Arnoux--Rauzy algorithm} (see \cite{Arnoux-Rauzy:91,AD:19}),  that generalize the case of Sturmian matrices from \eqref{eq:sturmmat} (see also Example~\ref{ex:farey1}), defined as 
\[
M_{\rAR,1} = \begin{pmatrix}
1& 1 &  1\\
0&1&0\\
0&0&1
\end{pmatrix}, \quad
M_{\rAR,2} = \begin{pmatrix}
1& 0 &  0\\
1&1&1\\
0&0&1
\end{pmatrix}, \quad
M_{\rAR,3} = \begin{pmatrix}
1& 0 &  0\\
0&1&0\\
1&1&1
\end{pmatrix}.
\]
Choose a sequence $\bM = (M_n)_{n \in \ZZ}\in  \cM_{\rAR}^{\ZZ}$ for which  there exists a constant $h > 0$ such that  $\{M_n,\ldots,M_{n+h-1}\} = \cM_{\rAR}$ for all $n \in \ZZ$. Such a sequence is said to have \emph{strong partial quotients}\indx{partial quotient!strong} bounded by~$h$.
The sequence $\bM$ is primitive because each block of length at least~$h$ is positive. Because $\cM_{\rAR}$ contains only finitely many matrices, the growth condition holds for~$\bM$.
By the proof of \cite[Theorem~1]{Delecroix-Hejda-Steiner},  there exists a constant $C > 0$, such that for each $n \in \NN$, there exists a hyperplane $\bw(n)^\perp $ such that 
\[
\lVert{\bM}_{[0,n)} \vert_{\bw(n)^\perp}\rVert \leq  C\, \Big( \frac{2^h-3}{2^h-1}\Big)^{n/h}.
\]
Hence, the hyperplane $\bw(n)^\perp $ is exponentially contracted by $M_{[0,n)}$, i.e.,  
\[
\limsup_{n\to\infty} \frac{1}{n} \log \lVert\tr{\!M}_{[0,n)} \vert_{\bw(n)^\perp}\rVert < 0.
\]
By the definition of singular values this implies that
\[
\limsup_{n\to\infty} \frac{1}{n} \log \delta_2(M_{[0,n)}) \leq \limsup_{n\to\infty} \frac{1}{n} \log \lVert\tr{\!M}_{[0,n)} \vert_{\bw(n)^\perp}\rVert < 0.
\]
Thus the Pisot condition holds for~$\bM$ and we may apply Theorem~\ref{th:matrixPisot}. This theorem shows that the sequence~$\bM$ is algebraically irreducible and admits a generalized right eigenvector with rationally independent coordinates to which it converges exponentially.
Note that rational independence also follows from \cite{Andrieu:21}.

Further examples involving Arnoux--Rauzy matrices are given in \cite{Berthe-Cassaigne-Steiner}. In particular, in this paper generalizations of Arnoux--Rauzy matrices in~$\cM_d$ are defined for arbitrary $d\ge 3$. Our theory can be applied to this more general situation as well. 
\end{example}

We provide further examples for sequences that satisfy the conditions of Theorem~\ref{th:matrixPisot} in Propositions~\ref{prop:brunpisot:2} and~\ref{prop:brunpisot:1}.

\subsection{Eventually Anosov $\cM$-adic mapping families} \label{subsec:anosovM} We now relate a linear mapping family to a sequence of matrices; see Section~\ref{subsec:mapf} for the  general notion of a mapping family. Linear mapping families are reminiscent of multidimensional continued fraction algorithms and will be related to these algorithms in Section~\ref{subsec:MCFmappingfamily}.

\begin{definition}[$\cM$-adic mapping family] \label{def:Madicmapping}\indx{mapping family!M@$\cM$-adic}
Fix $d \ge 2$ and  $\cM\subset\cM_d$, and consider a sequence $\bM \in \cM^{\ZZ}$. For each $n \in \ZZ$, let $\TT_n$ be a copy of the $d$-dimensional torus $\RR^d/\ZZ^d$ equipped with the Euclidean metric inherited from~$\RR^d$. 
Moreover, regard the inverse of the matrix~$M_n$ as the automorphism\footnote{This is possible since $M_n$ is a unimodular integer matrix.} $M_n^{-1}: \TT_n \to \TT_{n+1}$ of $\RR^d/\ZZ^d$  given by left multiplication by~$M_n^{-1}$. 
Then we set $\TT = \coprod_{n\in\ZZ}\TT_n$ and define $f_{\bM}: \TT \to \TT$ by $f_{\bM}(\bx) = M_n^{-1} \bx$ for $\bx \in \TT_n$. \notx{fM}{$(\TT,f_{\bM})$}{$\cM$-adic mapping family}
We call $(\TT,f_{\bM})$ the \emph{$\cM$-adic mapping family} associated to~$\bM$. 
It can be written out as
\[  
\cdots \xrightarrow{M_{-2}^{-1}} \TT_{-1} \xrightarrow{M_{-1}^{-1}}  \TT_0 \xrightarrow{M_{0}^{-1}} \TT_1 \xrightarrow{M_{1}^{-1}}  \cdots . 
\]
\end{definition}

Let $\cM\subset\cM_d$ and $\bM \in \cM^{\ZZ}$ be given. Since the mapping $M_n^{-1}: \TT_n\to \TT_{n+1}$ of the associated $\cM$-adic mapping family $(\TT,f_{\bM})$ is linear for each $n\in\ZZ$, the derivative of this mapping $M_n^{-1}$ at a point $\bp \in \TT_n$ is  the linear mapping $M_n^{-1}:\,\RR^d\to \RR^d$ (note that $T_{\bp} \TT_n \cong \RR^d$ and $T_{f_{\bM}(\bp)} \TT_{n+1} \cong \RR^d$). 

As  in Section~\ref{subsec:mapf}, assume that there exists an $f_{\bM}$-invariant splitting\indx{splitting!invariant} $G^s \oplus G^u$ of the tangent space~$T\TT$ of~$\TT$, and set $G^s_{\bp} = G^s \cap T_{\bp}\TT$, $G^u_{\bp} = G^u \cap T_{\bp}\TT$ for $\bp \in \TT$. Thus the derivative~$D(f_{\bM})$ and, hence, also the spaces $G_{\bp}^s$ and~$G_{\bp}^u$  depend only on the integer~$n$ satisfying $\bp \in \TT_n$, and we may define (by some abuse of notation) $G_n^s = G_{\bp}^s$\notx{Gs}{$G^s,G_p^s,G_n^s$}{stable subspace of a hyperbolic splitting} and $G_n^u = G_{\bp}^u$\notx{Gu}{$G^u,G_p^u,G_n^u$}{unstable subspace of a hyperbolic splitting} for $\bp \in \TT_n$, $n \in \ZZ$.  By $f_{\bM}$-invariance of the splitting, one has  $M_n^{-1}G_n^s =G_{n+1}^s$ and $M_n^{-1}G_n^u = G_{n+1}^u$ for all~$n$.

Summing up, in this case, the conditions required for $(\TT,f_{\bM})$ to be eventually Anosov can be written as follows. 
For some $n \in \ZZ$ (and, hence, for each $n \in \ZZ$), we have 
\begin{enumerate}[\upshape (i)]
\itemsep.5ex
\item
$\lim\limits_{k\to+\infty} \sup \{ \lVert M_{[n,k)}^{-1}\, \bx \rVert / \lVert \bx \rVert  \,:\, \bx \in G_n^s \setminus \{\mathbf{0}\} \} = 0$, 
\item
$\lim\limits_{k\to+\infty} \inf \{ \lVert M_{[n,k)}^{-1}\, \bx \rVert / \lVert \bx \rVert \,:\, \bx \in G_n^u \setminus \{\mathbf{0}\} \} = +\infty$, 
\item 
$\lim\limits_{k\to-\infty} \inf\{ \lVert M_{[k,n)}\, \bx \rVert / \lVert \bx \rVert \,:\, \bx \in G_n^s \setminus \{\mathbf{0}\} \}= +\infty$,
\item 
$\lim\limits_{k\to-\infty} \sup\{ \lVert M_{[k,n)}\,\bx \rVert / \lVert \bx \rVert \,:\, \bx \in G_n^u \setminus \{\mathbf{0}\} \} = 0$.
\end{enumerate}

We shall now prove that an $\cM$-adic mapping familiy is eventually Anosov if the corresponding sequence of matrices~$\bM$ admits strong convergence in the future and in the past (see Definition~\ref{def:wsc} for the definition of strong convergence). 
In this case, the families of generalized right and left eigenvectors~$\bu_n$ and~$\bv_n$ from \eqref{eq:unvn} are used to construct the splitting of the tangent space  for which the $\cM$-adic mapping family is eventually Anosov. 

\begin{proposition} \label{th:anosov} 
For $d \ge 2$ and $\cM\subset  \cM_d$, let $\bM \in \cM^{\ZZ}$ be a bi-infinite sequence of $d{\times}d$ unimodular nonnegative integer matrices. 
Assume that $\bM$ is two-sided primitive and admits strong convergence to the generalized right eigenvector $\bu$ in the future and to the generalized left eigenvector~$\bv$ in the past. Then the $\cM$-adic mapping family $(\TT,f_{\bM})$ associated to~$\bM$ is eventually Anosov for the splitting 
\begin{equation}\label{eq:220splitting}
G^s \oplus G^u = \coprod_{n\in\ZZ} G_n^s \oplus G_n^u = \coprod_{n\in\ZZ} \RR \bu_n \oplus \bv_n^\perp,
\end{equation}
where $\bu_n$ and~$\bv_n$ are as in~\eqref{eq:unvn} for all $n \in \ZZ$.
\end{proposition}

\begin{proof}
For each $n \in \ZZ$, the vectors $\bu_n$ and~$\bv_n$ are a right and left eigenvector of the two-sided primitive sequence $\Sigma^n \bM = (M_{k+n})_{k\in\ZZ}$, respectively. Therefore, these vectors have only positive entries.
We first prove that
\begin{align} 
\lim_{k\to+\infty} M_{[n,k)}^{-1}\, \bx  & = \mathbf{0} \qquad \mbox{for all}\ \bx \in \RR \bu_n, \label{eq:anosovthm1} \\
\lim_{k\to+\infty} \lVert M_{[n,k)}^{-1}\, \bx \rVert & = +\infty \quad \mbox{for all}\ \bx \in  \bv_n^\perp \setminus \{\mathbf{0}\}.
\label{eq:anosovthm2}
\end{align} 
These equations imply that $(\TT,f_{\bM})$ is eventually Anosov in the future with the splitting \eqref{eq:220splitting}. By~\eqref{eq:unvn} it suffices to prove \eqref{eq:anosovthm1} and \eqref{eq:anosovthm2} for $n = 0$. 

To prove~\eqref{eq:anosovthm1} for $n = 0$, assume that there exists $C > 0$ such that  $\lVert \bu_k\rVert_\infty$ is larger than~$C$ for infinitely many positive integers~$k$.  Primitivity and positivity of the coordinates of $\bu_k$ implies that the coordinates of $M_{[0,k)} \bu_k$ are unbounded in~$k$, which cannot be true because $\bu = M_{[0,k)} \bu_k$ holds by~\eqref{eq:unvn}. Thus $\lim_{k\to+\infty} \lVert M_{[0,k)}^{-1 }\bu \rVert_\infty = \lim_{k\to+\infty} \lVert \bu_k \rVert_\infty = 0$, which yields~\eqref{eq:anosovthm1} for~$n = 0$.

To prove~\eqref{eq:anosovthm2} for $n = 0$, let $\bx \in  \bv^\perp \setminus \{\mathbf{0}\}$. Assume that there exists $C > 0$ such that $\lVert M_{[0,k)}^{-1} \bx \lVert \leq C$ for infinitely many $k \ge 0$. 
Because $\bu \not\in \bv^\perp$, the strong convergence of~$\bM$ to~$\bu$ yields $\lim_{k\to\infty} \pi_{\bu,\bv} M_{[0,k)} \be_i = \mathbf{0}$ for all $i \in  \{1,\ldots,d\}$, thus 
\[
\lim_{k\to\infty} \max_{\lVert\by\rVert\le C} \lVert \pi_{\bu,\bv} M_{[0,k)} \by \rVert =  0.
\] 
Since  $ \lVert M_{[0,k)}^{-1} \bx \rVert \leq C$ for infinitely many~$k$, this implies that 
\[
\lVert \pi_{\bu,\bv} \bx \rVert = \liminf_{k\to\infty} \lVert \pi_{\bu,\bv} M_{[0,k)}  M_{[0,k)}^{-1} \bx \rVert \le \liminf_{k\to\infty} \max_{\lVert\by\rVert\le C} \lVert \pi_{\bu,\bv} M_{[0,k)} \by \rVert =  0
\]
and, hence, $\pi_{\bu,\bv} \bx = \mathbf{0}$.
This contradicts the fact that $\bx \in \bv^\perp \setminus \{\mathbf{0}\}$. This contradiction proves \eqref{eq:anosovthm2} for $n = 0$.

Next, we prove that the mapping family $(\TT,f_{\bM})$ is eventually Anosov in the past with the splitting \eqref{eq:220splitting}. 
This amounts to showing that
\begin{align} 
\lim_{k\to+\infty} \lVert M_{[-k,n)} \bx \rVert & = +\infty \quad \mbox{for all}\ \bx \in \RR \bu_n \setminus \{\mathbf{0}\}, \label{eq:anosovpast1} \\
\lim_{k\to+\infty} M_{[-k,n)} \bx  & = \mathbf{0} \qquad \mbox{for all}\ \bx \in  \bv_n^\perp. \label{eq:anosovpast2}
\end{align} 
To prove these equations we may again assume w.l.o.g.\ that $n=0$. Equation \eqref{eq:anosovpast1} follows immediately from primitivity in the past and by the fact that $\bu=\bu_0$ has positive coordinates.

To prove \eqref{eq:anosovpast2}, we use that $\bM$ converges strongly to~$\bv$ in the past, which implies that  
\begin{equation}\label{eq:vstrongNew}
\lim_{k\to +\infty} \mathrm{d}(\tr{\!M}_{[-k,0)}\,\be_i, \RR \bv) = 0\quad \hbox{for all}\ i\in  \{1,\ldots,d\}.
\end{equation}
Let $\bx \in \bv^\perp \setminus \{\mathbf{0}\}$ (the case $\bx = \mathbf{0}$ is trivial). 
Suppose that $M_{[-k,0)} \bx \in \bv_k^\perp$ does not converge to~$\mathbf{0}$ for $k \to \infty$. 
Then there exists $i \in \{1,\ldots,d\}$ such that the $i$-th coefficient of $M_{[-k,0)} \bx$ is greater than some constant $C > 0$ in modulus, for infinitely many~$k$. 
Combining this with~\eqref{eq:vstrongNew}, we may choose $k\in\NN$ in a way that there is $\gamma_k > 0$ with $\lVert \tr{\!M}_{[-k,0)}\, \be_i - \gamma_k \bv \rVert_2 < \frac{C}{\lVert\bx\rVert_2}$ and $|\langle \bx, \tr{\!M}_{[-k,0)}\, \be_i \rangle| > C$.
But then, because $\bx \in \bv^\perp \setminus \{\mathbf{0}\}$, we have
\[
\begin{aligned}
|\langle \bx, \tr{\!M}_{[-k,0)}\, \be_i\rangle| &= |\langle \bx, \tr{\!M}_{[-k,0)}\, \be_i - \gamma_k \bv \rangle + \langle \bx, \gamma_k \bv \rangle|
= |\langle \bx, \tr{\!M}_{[-k,0)}\, \be_i - \gamma_k \bv \rangle| \\
&\le  \lVert\bx\rVert_2 \cdot  \lVert \tr{\!M}_{[-k,0)}\, \be_i - \gamma_k \bv\rVert_2 < C,
\end{aligned}
\]
a contradiction. This proves \eqref{eq:anosovpast2} for $n = 0$.
\end{proof}

\begin{remark}
If $\bM$ converges strongly to~$\bu$ in the future and only weakly to~$\bv$ in the past, then we can only prove that $(\TT,f_{\bM})$ is eventually Anosov in the future.
The existence of generalized right and left eigenvectors is needed to define the splitting (directly for the generalized right eigenvector, and by duality for the generalized left eigenvector); weak convergence is enough to prove the desired limits 
\eqref{eq:anosovthm1} and \eqref{eq:anosovpast1} on the lines~$\RR \bu_n$, but strong convergence is needed to obtain the \eqref{eq:anosovthm2} and \eqref{eq:anosovpast2}
limits on the complementary hyperplanes~$\bv_n^\perp$.
\end{remark}

The following theorem, which is stated as Theorem~\nameref{t:B} in the introduction, is an immediate consequence of Proposition~\ref{th:anosov} and Theorem~\ref{th:matrixPisot}; see also Proposition~\ref{prop:strongcv2sided}.

\begin{theorem}\label{cor:anosov} 
Assume that $\bM = (M_n)_{n\in\ZZ} \in \cM_d^{\ZZ}$ is two-sided primitive and satisfies the two-sided local Pisot condition and the two-sided growth condition $\lim_{n\to\pm \infty} \frac{1}{n} \log \rVert M_n\lVert = 0$. 
Then $\bM$ admits strong convergence to a generalized right eigenvector $\bu$ in the future and to a generalized left eigenvector~$\bv$ in the past. The mapping family $(\TT,f_{\bM})$ associated to $\bM$ is eventually Anosov for the splitting 
\[
G^s \oplus G^u = \coprod_{n\in\ZZ} G_n^s \oplus G_n^u = \coprod_{n\in\ZZ} \RR \bu_n \oplus \bv_n^\perp,
\]
where $\bu_n$ and~$\bv_n$ are as in~\eqref{eq:unvn} for all $n \in \ZZ$.
\end{theorem}

We refer to Example \ref{ex:AR1} and to Propositions~\ref{prop:brunpisot:2} and~\ref{prop:brunpisot:1} for examples of sequences $\bM\in \cM_d^{\ZZ}$ that satisfy the conditions of this corollary.

\section{Metric theory for bi-infinite sequences of matrices} \label{sec:metricmat}
In Section~\ref{sec:matrices}, we have set up a theory for a single sequence of matrices~$\bM\in\cM_d^{\ZZ}$, $d\ge 2$. Our aim in this section is to establish a metric theory that is valid for almost all sequences of matrices in a subset~$D$ of $\cM_d^{\ZZ}$. In Section~\ref{sec:oseledets} we define Lyapunov exponents for an invertible dynamical system, state the Oseledets Multiplicative Ergodic Theorem, and provide a metric version of the Pisot condition. Section~\ref {sec:induceddyn} states the definition of an induced dynamical system and shows the effect of inducing on Lyapunov exponents. This will be needed later in the context of multidimensional continued fraction algorithms. In Section~\ref{subsec:cocyclemat} we define $\cM$-adic cocycles  for shifts of sequences of matrices. We also clarify how the generic Pisot condition on a shift of sequences of matrices is related to the local Pisot condition on each single sequence of matrices contained in this shift. Finally, in Section~\ref{sec:hyperb} we provide a metric version of the main results of Section~\ref{sec:matrices} for shifts of sequences of matrices. In particular, we show the effect of the Pisot condition on the Oseledets splitting of such a shift.

\subsection{Oseledets splitting and the Pisot condition}\label{sec:oseledets}
In the present section, we study linear cocycles on ergodic dynamical systems. For convenience, we recall the definition of these objects.
First, a dynamical system $(X,F,\nu)$ with invariant measure~$\nu$ is \emph{ergodic}\indx{ergodic} if $F^{-1}(B) = B$ implies that $\nu(B) = 0$ or $\nu(X \setminus B) = 0$.
Two dynamical systems $(X,F,\nu)$ and $(X',F',\nu')$ are \emph{measurably conjugate}\indx{conjugacy!measurable} if there are measurable subsets $X_1 \subset  X$, $X'_1 \subset  X'$ of measure~$1$ and a bijection $g: X_1 \to X'_1$ with $g$ and~$g^{-1}$ being measurable, such that $g \circ F(x)= F' \circ g(x)$ for every $x \in X_1$, and $\nu' \circ g(U)= \nu(U)$ for every measurable set $U \subset X_1$.
The conjugacy is said \emph{topological}\indx{conjugacy!topological}  if  $X_1=X$, $X'_1=X'$ and $g$ is a homeomorphism.

\begin{definition}[{Linear cocycle; cf.\ e.g.\ \cite[Definition~1.1.1 and Section~2.1.2]{Arnold98}}]
\label{def:cocy}
Let $(X,F,\nu)$ be an invertible ergodic dynamical system.
Define a mapping $A: X \to \GL(d,\ZZ)$. Then
\begin{equation}\label{bcocycleAllg}  
A^{(n)}(x) = 
\begin{cases}
A(F^{n-1}(x)) \cdots A(F(x)) A(x) & \text{if}\ n > 0, \\
\mathrm{Id} & \text{if}\ n = 0, \\
A(F^{n}(x))^{-1} \cdots A(F^{-1}(x))^{-1} & \text{if}\ n < 0,
 \end{cases}
\end{equation}
\notx{An}{$A(\cdot),A^{(n)}(\cdot)$}{\hspace{.5em}linear cocycle}which is a function from $\ZZ \times X$ to $\GL(d,\ZZ)$,
is called a \emph{linear cocycle on $(X,F,\nu)$}\indx{cocycle!linear}. To keep things simple, we call $A$ a linear cocycle on $(X,F,\nu)$.
\end{definition}

A linear cocycle~$A$ is defined analogously for a noninvertible ergodic dynamical system $(X,F,\nu)$. In this case, $A^{(n)}$ makes sense only for $n\ge 0$.

Oseledets' Multiplicative Ergodic Theorem\indx{Oseledets!multiplicative ergodic theorem} is a powerful tool in the study of linear cocycles. 
Recommended references to this theorem are for instance ~\cite{Arnold98} and \cite{Viana:book}. In the case of invertible dynamical systems (for instance, a two-sided shift is invertible), Oseledets' Theorem yields a splitting of the space (see \cite[Theorem~3.4.11]{Arnold98}) reflecting the growth behavior of the cocycle. For noninvertible dynamical systems like for instance, one-sided shifts, it only provides a filtration; see \cite[Theorem~3.4.1]{Arnold98}. In this section, we state results around Oseledets' Theorem for general invertible ergodic dynamical systems in a way that suits our purposes and that will be needed later in various special cases. The Furstenberg Kesten theorem is stated in the following proposition.

\begin{proposition}[{cf.~\cite[Theorem~3.3.3 and~3.3.10]{Arnold98}}]\label{prop:genly}
Let $(X,F,\nu)$ be an invertible ergodic dynamical system with linear cocycle $A: X \to \GL(d,\ZZ)$. Assume that $A$ is log-integrable\indx{log-integrable}, in the sense that\footnote{Note that $\lVert\cdot\rVert_\infty$ is a matrix norm here. Because $A$ has determinant $\pm1$, by \cite[Remark~3.5]{BouLa:85} we do not need to consider $\int_X \log \max\{\lVert A(x)\rVert_{\infty},\lVert A(x)^{-1}\rVert_{\infty}\}d\nu(x) < \infty$ here.}
\begin{equation}\label{eq:logint} 
\int_X \log \lVert A(x)\rVert_{\infty} d\nu(x) < \infty. 
\end{equation}
Then the quantities~$\vartheta_1,\ldots,\vartheta_d$, which are recursively defined by
\begin{equation} \label{eq:LyGenericBothSides}
\vartheta_1 + \cdots +\vartheta_k = {\lim_{n\to\infty}} \frac{1}{n}\log \lVert\wedge^k A^{(n)}(x)\rVert \qquad (1 \le k \le d) \notx{theta}{$\vartheta_i,(\Theta_i,d_i)$}{Lyapunov exponent}
\end{equation}
exist and do not depend on~$x$ for almost all $x \in X$.
\end{proposition}

\begin{definition}[Lyapunov exponents and Lyapunov spectrum for generic sequences of matrices]
\label{def:genly}\indx{Lyapunov!generic exponent}\indx{Lyapunov!generic spectrum}
Let $(X,F,\nu)$ be an invertible ergodic dynamical system with log-integrable linear cocycle $A: X \to \GL(d,\ZZ)$. The numbers $\vartheta_1 \ge \cdots \ge \vartheta_d$ in \eqref{eq:LyGenericBothSides} are called the \emph{(generic) Lyapunov exponents} of the system $(X,F,\nu)$ for the linear cocycle~$A$.

Let $\Theta_1 > \Theta_2 > \cdots > \Theta_p$ be the different elements in the sequence $(\vartheta_i)_{1\le i\le d}$ of Lyapunov exponents ordered according to $\vartheta_1 \ge \vartheta_2 \ge \cdots \ge \vartheta_d$ and denote by~$d_i$ ($1\le i \le p$) the number of occurences of $\Theta_i$ in the sequence $(\vartheta_i)_{1\le i\le d}$. Then $\{(\Theta_i,d_i) : 1 \le i \le p\}$ is called the \emph{(generic) Lyapunov spectrum} of $(X,F,\nu)$.
\end{definition}

\begin{remark}\label{rem:oneLy}\label{fn03}
Note that only positive integers $n$ enter the definition of $\vartheta_1,\ldots, \vartheta_d$ in \eqref{eq:LyGenericBothSides}. Nevertheless, the Lyapunov exponents $\vartheta_1,\ldots, \vartheta_d$ of Definition~\ref{def:genly} are ``two-sided'' Lyapunov exponents and $\{(\Theta_i,d_i) : 1 \le i \le p\}$ is the ``two-sided'' Lyapunov spectrum because the matrices on the negative side generically have the same behaviour as the ones on the positive side by $F$-invariance of $\nu$.  In particular, we see from \cite[Theorem~3.3.3 and~3.3.10]{Arnold98} that for almost all $x\in X$ we have
\[
\vartheta_k = \lim_{n\to\infty} \frac1n \log \delta_k(A^{(n)}(x))\qquad (1\le k \le d),
\]
and if we go to the negative direction we gain
\begin{equation}\label{eq:LyapInverseNegative}
\vartheta_{d+1-k} = \lim_{n\to-\infty} \frac1n \log \delta_k(A^{(n)}(x))\qquad (1\le k \le d)
\end{equation}
almost all $x\in X$.
\end{remark}

The following proposition contains a version of Oseledets' Multiplicative Ergodic Theorem.

\begin{proposition}[{cf.~\cite[Theorem~3.4.11]{Arnold98}}]\label{prop:genly2}
Let $(X,F,\nu)$ be an invertible ergodic dynamical system with Lyapunov spectrum $\{(\Theta_i,d_i) : 1 \le i \le p\}$ and log-integrable linear cocycle~$A: X \to \GL(d,\ZZ)$ in the sense of (\ref{eq:logint}). Then, for $\nu$-a.e.\  $x \in X$, there is an \emph{Oseledets~splitting}\indx{splitting!Oseledets}\indx{splitting!invariant}\indx{Oseledets!splitting}
\begin{equation}\label{eq:oseledetsSplitting}
\RR^d = G^1(x) \oplus \cdots \oplus G^p(x)
\end{equation}
that is dynamically characterized by
\begin{equation}\label{ospaces}
{\lim_{n\to\pm\infty}} \frac{1}{n}\log \lVert A^{(n)}(x)\bz\rVert  = \Theta_i \ \Longleftrightarrow\  \bz \in G^i(x)\setminus \{\mathbf{0}\} \qquad (1 \le i \le p).
\end{equation}
The Oseledets splitting in \eqref{eq:oseledetsSplitting} satisfies $\dim (G^i(x)) = d_i$ and the invariance property $A^{(n)}(x)G^i(x) = G^i(F^nx)$, for all $i \in \{1,\dots,p\}$ and $n\in\ZZ$.
\end{proposition}

We need the following property of the Lyapunov exponents of a dynamical system. Note that this property concerns invertible dynamical systems and, hence, is two-sided;  cf.\ Remark~\ref{rem:oneLy}.

\begin{definition}[Generic Pisot condition for a cocycle]\label{def:gengenPisot}
\indx{Pisot!condition!generic}
Let $(X,F,\nu)$ be an invertible ergodic dynamical system with log-integrable linear cocycle $A: X \to \GL(d,\ZZ)$. We say that $(X,F,\nu)$ satisfies the \emph{(generic) Pisot condition} for~$A$ if its generic Lyapunov exponents $\vartheta_1 \ge \vartheta_2 \ge \cdots \ge \vartheta_d$ satisfy $\vartheta_1 > 0 > \vartheta_2$ for~$A$.  
\end{definition}

When we say that $(D,\Sigma,\nu)$ satisfies the generic Pisot condition for~$A$, then the log-integrability of~$A$ is implied (because otherwise the Lyapunov exponents are not defined). 

\subsection{Induced dynamical systems}\label{sec:induceddyn}
We will later consider induced rotations. For this reason we recall the following general definition for induced maps (also called  \emph{first return maps})\indx{return map}.

\begin{definition}[Induced dynamical system and induced map; {see for instance~\cite[Definition~8.11]{Hawkins:21}}]  \label{def:induction}\indx{induction}\indx{dynamical system!induced}
Let  a measurable dynamical system $(X,F,\nu)$ with $\nu(X)<\infty$ and a subset $Y \subset X$ be given. For each $x \in Y$, define the first return time\footnote{The return time exists $\nu$-a.e.\ by the Poincar\'e recurrence theorem.}
by $n(x) = \min\{m \ge 1 : F^{m}(x) \in Y\}$.  The map $G:\, Y \to Y$, $x \mapsto F^{n(x)}(x)$, is called the \emph{induced map} of $F$ on $Y$. The resulting dynamical system $(Y,G,\nu|_Y)$ is called the \emph{induced dynamical system} of $(X,F,\nu)$ on~$Y$. 

If $A$ is a linear coycle of  $(X,F,\nu)$ then $B: Y \to \GL(d,\ZZ)$, $y \mapsto A^{(n(y))}(y)$ is the linear cocycle of $(Y,G,\nu)$ associated to $A$.
\end{definition}

For ergodic invertible dynamical systems, the Lyapunov exponents behave quite regularly under induction. In particular, we have the following result.

\begin{lemma}[see {\cite[Proposition~4.18]{Viana:book}}]\label{lem:viana819}
Let $(X,F,\nu)$ with $\nu(X)<\infty$ be an invertible ergodic dynamical system with linear cocycle $A$
and let $Y \subset X$. Let $(Y,G,\nu)$ be the induced dynamical system of $(X,F,\nu)$ on~$Y$ and let $B$ be the linear cocycle of $(Y,G,\nu)$ associated to $A$.

If $\vartheta_1> \ldots > \vartheta_d$ are the Lyapunov exponents of $A$, then there is a constant $c\ge 1$ such that $c\vartheta_1> \ldots > c\vartheta_d$ are the Lyapunov exponents of $B$.
\end{lemma}

\begin{proof}
This immediately follows from \cite[Proposition~4.18]{Viana:book}. Just note that the function $c(x)$ occurring in this result is a constant in our case, because we assumed ergodicity.
\end{proof}

\subsection{Definition of $\cM$-adic cocycles} \label{subsec:cocyclemat}
Recall that $\cM_d$, $d \ge 2$, is the set of unimodular nonnegative integer $d{\times}d$ matrices. Equip the set~$\cM_d^{\ZZ}$ of two-sided infinite sequences of matrices in~$\cM_d$ with the product topology of the discrete topology on $\cM_d$, and let $D \subset \cM_d^{\ZZ} $ be a shift invariant set equipped with the subspace topology. Then $(D,\Sigma,\nu)$ stands for the measure-theoretical dynamical system which is defined by the shift $\Sigma$ on $D$ endowed with an ergodic  $\Sigma$-invariant Borel measure~$\nu$. The set $D$ is not assumed to be closed. To the ergodic dynamical system $(D,\Sigma,\nu)$, we relate two linear cocycles.

\begin{definition}[$\cM$-adic cocycles]\label{def:cocycle}
Let $\cM\subset \cM_d$, and let $(D,\Sigma,\nu)$ be an ergodic dynamical system with $D \subset \cM^{\ZZ}$. We define the following cocycles.

The linear cocycle $A_{\mathrm{tr}}: D \to \GL(d,\ZZ)$, $(M_n)_{n\in\ZZ} \mapsto \tr{\!M}_0$  is called \emph{($\cM$-adic) transpose cocycle}\indx{cocycle!transpose}. \notx{Atr}{$A_{\mathrm{tr}}$}{transpose cocycle}

The linear cocycle $A_{\mathrm{inv}}: D \to \GL(d,\ZZ)$, $(M_n)_{n\in\ZZ} \mapsto M_0^{-1}$  is called \emph{($\cM$-adic) inverse cocycle}\indx{cocycle!inverse}.
\notx{Atinv}{$A_{\mathrm{inv}}$}{inverse cocycle}
\end{definition}

\begin{remark}
One can now define the vector bundle morphism as the skew product ($\bM=(M_n)_{n\in\ZZ}$)
\[
\Phi_{\mathrm{tr}}:\, D \times \TT^d \to D \times \TT^d, \quad (\bM,\bx) \mapsto (\Sigma\bM, \tr{\!M}_0 \bx). 
\]
This mapping constitutes a \emph{random dynamical system}\indx{dynamical system!random} in the sense of \cite[Definition~1.1.1]{Arnold98} with \emph{base transformation}~$\Sigma$ and \emph{generator}~$A_{\mathrm{tr}}$. It can be defined for~$A_{\mathrm{inv}}$ analogously.
\end{remark}

For the iterates $A_{\mathrm{tr}}^{(n)}$ of~$A_{\mathrm{tr}}$, we have, for $n \in \ZZ$, that 
\begin{equation}\label{bcocycle}  
A_{\mathrm{tr}}^{(n)}(\bM) = \begin{cases} 
A_{\mathrm{tr}}(\Sigma^{n-1}\bM) \cdots A_{\mathrm{tr}}(\Sigma\bM)A_{\mathrm{tr}}(\bM) = \tr{\!M}_{[0,n)} & \text{if}\ n > 0, \\
\mathrm{Id} & \text{if}\ n = 0, \\
A_{\mathrm{tr}}(\Sigma^{n}\bM)^{-1} \cdots A_{\mathrm{tr}}(\Sigma^{-1}\bM)^{-1} = \tr{\!M}_{[n,0)}^{-1} & \text{if}\ n < 0. 
\end{cases}
\end{equation}
Iterates of~$A_{\mathrm{inv}}$ can be viewed in an analogous way by $A_{\mathrm{inv}}^{(n)}(\bM) = M_{[0,n)}^{-1}$ for $n \ge 0$ and  $A_{\mathrm{inv}}^{(n)}(\bM) = M_{[n,0)}$ for $n < 0$.

\begin{remark}\label{rem:invtr}
The generic Lyapunov exponents of the system $(D,\Sigma,\nu)$ for $A_{\mathrm{tr}}$ are $\vartheta_1,\ldots, \vartheta_d$ if and only if the  generic Lyapunov exponents of the system $(D,\Sigma,\nu)$ for $A_{\mathrm{inv}}$ are $-\vartheta_d,\ldots, -\vartheta_1$. This follows immediately from the definition of these two cocycles.
\end{remark}

\begin{definition}[Generic Pisot condition for sequences of matrices]\label{def:gengenPisot2}
\indx{Pisot!condition!generic}
The ergodic dynamical system $(D,\Sigma,\nu)$ satisfies the \emph{(generic) Pisot condition} (without qualifying the cocycle) if it satisfies the generic Pisot condition for  the linear cocycle~$A_{\mathrm{tr}}$.
\end{definition}

The lemma below relates the generic Pisot condition and the local Pisot conditions; see Definitions~\ref{def:genPisot} and~\ref{def:strongPisot}.

\begin{lemma}\label{lem:GloLocPisot}
For $d\ge 2$, let $(D,\Sigma,\nu)$, with $D \subset \cM_d^{\ZZ}$, be an ergodic dynamical system. 
\begin{enumerate}[\upshape (i)]
\item \label{it:glp1}
If $(D,\Sigma,\nu)$ satisfies the log-integrability condition \eqref{eq:logint}, then for $\nu$-almost all $\bM = (M_n)_{n\in\ZZ}\in D$ the growth condition   $\lim_{n\to\infty} \frac{1}{n} \log \lVert M_n\rVert = 0$ holds.
\item \label{it:glp2}
If $(D,\Sigma,\nu)$ satisfies the generic Pisot condition, i.e., $\vartheta_1 > 0 > \vartheta_2 \ge \cdots \ge \vartheta_d$ for the transpose cocycle~$A_{\mathrm{tr}}$, then, for $\nu$-almost all $\bM = (M_n)_{n\in\ZZ} \in D$, the growth condition, the two-sided strong local Pisot condition, and  the  two-sided local Pisot condition hold.
\end{enumerate}
\end{lemma}

\begin{proof}
To prove~(\ref{it:glp1}), we define a \emph{suspension flow}\indx{suspension flow}~$T_t$ over the shift $(D,\Sigma,\nu)$. To this end, let 
\[
\tau: D \to \RR_{\ge0};\quad (M_n)_{n\in\ZZ} \mapsto \log \lVert M_0\rVert = \log \lVert \tr{\!M}_0\rVert
\]
be the \emph{roof function} and
\[
D_\tau=\{ (\bM,x) \in D \times \RR_{\ge0} \,:\, 0\le x < \tau(\bM) \}.
\]
Now we define the suspension flow by 
\[
T_t: D_\tau \to D_\tau;\quad (\bM,x) \mapsto (\bM,x + t), 
\]
where $(\bM,x+\tau(\bM))$ is identified with $(\Sigma\bM,x)$. 
Let $\lambda$ be the Lebesgue measure on~$\RR_{\ge0}$.
Since $\tau \in L^1(\nu)$ holds by log-integrability, the measure
\[
\bar\nu = \frac{(\nu\times \lambda)|_{D_\tau}}{\int \tau d\nu}
\]
is well defined, and the flow $(D_\tau,T_t,\bar\nu)$ is ergodic; see \cite[Lemma~9.24~(2)]{EinsiedlerWard:11}.

We now follow \cite[Section~5.2]{DelHL}. For each $\delta>0$ small enough, we have $U = D \times [0,\delta]\subset D_\tau$, and by Birkhoff's ergodic theorem we gain
\begin{equation}\label{eq:genrec}
\lim_{t\to \infty} \frac{\lambda(\{s\in[0,t] : T_s(\bM,x) \in U \})}{t}= \bar\nu(U)
\end{equation}
for a.e.\ $(\bM,x) \in D_\tau$. 
Denote the $n$-th return time of~$T_t$ to $D {\times} \{0\}$ by $t_n = t_n(\bM,x)$. Then, \eqref{eq:genrec} implies that for a.e.\ $(\bM,x)\in D_\tau$ we have $\lim_{n\to \infty} {t_n}/{n} < \infty$ 
and, hence, for each $\varepsilon > 0$, there is $N \in \NN$ such that $t_{n+1} - t_n < \varepsilon n$ for each $n \ge N$. Thus, for a.e.\ $(\bM,x)\in D_\tau$, we have $\lim_{n\to\infty} \frac{1}{n} \log\lVert M_n\rVert = \lim_{n\to\infty} \frac{t_{n+1}-t_n}{n} = 0$ and, hence, the growth condition holds for a.e.\ $\bM\in D$; a~different proof of this fact, which is based on the Borel--Cantelli Lemma, is contained in \cite[Lemma~3.4.3~(i)]{Arnold98}. 

It remains to prove~(\ref{it:glp2}). Because the generic Pisot condition contains the log-integrability condition, the growth condition holds for a.e.\ $\bM\in D$ by (i). Thus, in view of Lemma~\ref{lem:strongweak}, we have to show that a.e.\ $\bM\in D$ satisfies the two-sided strong local Pisot condition.

As for the future, because $(D,\Sigma,\nu)$ satisfies the Pisot condition for the cocycle~$A_{\mathrm{tr}}$, by \eqref{eq:LyGenericBothSides} we have for a.e.\ $\bM \in D$ that
\[
\begin{aligned}
\vartheta_1 + \cdots +\vartheta_k & = \lim_{n\to\infty} \frac{1}{n}\log \lVert\wedge^k A^{(n)}(\bM)\rVert = \lim_{n\to\infty} \frac{1}{n} \log \lVert\wedge^k \, \tr{\!M}_{[0,n)} \rVert \\
& = \lim_{n\to\infty} \frac{1}{n}\log \lVert\wedge^k M_{[0,n)}\rVert
\end{aligned}
\]
holds for all $k \in \{1,\dots,d\}$, with $\vartheta_1 > 0 > \vartheta_2$.  Thus a.e.\ $\bM \in D$ satisfies the strong local Pisot condition in the future. 

For the past, we argue as follows. By Proposition~\ref{prop:genly2}, the Pisot condition of $(D,\Sigma,\nu)$ for the cocycle~$A_{\mathrm{tr}}$ implies that for a.e.\ $\bM \in D$ there is a hyperplane~$H$ in $\RR^d$ such that
\[
\lim_{n\to -\infty} \frac{1}{n} \log \lVert \tr{\!M}_{[n,0)}^{-1}\bx \rVert \le \vartheta_2 < 0
\]
holds for each $\bx \in H\setminus\{\mathbf{0}\}$. Because $n$ tends to~$-\infty$, this implies that there exists $N\in \ZZ$ such that for each $\bx \in H$ with $\lVert \bx \rVert = 1$ and each $n \le N$ we have $\lVert \tr{\!M}_{[n,0)}^{-1}\bx \rVert \ge e^{\vartheta_2n/2}$. As each $\bx\in H$ with $\lVert\bx\rVert=1$ is a linear combination of the standard basis vectors with coefficients bounded by $1$ in modulus, $N$~is uniform for all $\bx \in H$ with $\lVert \bx \rVert = 1$. Thus, by the definition of the $(d{-}1)$-st singular value, $\delta_{d-1}(\tr{\!M}_{[n,0)}^{-1}) \ge e^{\vartheta_2n/2}$ holds for $n \le N$, and we get 
\begin{equation}\label{eq:nehlylocglo}
\begin{aligned}
0 > \frac{\vartheta_2}{2} & \ge  \lim_{n\to -\infty} \frac{1}{n} \log \delta_{d-1}(\tr{\!M}_{[n,0)}^{-1})  \\
& = -\lim_{n\to \infty} \frac{1}{n} \log \delta_{d-1}(\tr{\!M}_{[-n,0)}^{-1})  \\
& = \lim_{n\to \infty} \frac{1}{n} \log \delta_{2}(\tr{\!M}_{[-n,0)});
\end{aligned}
\end{equation}
note that the a.e.\ existence of the limit in the first line of \eqref{eq:nehlylocglo} follows from \eqref{eq:LyGenericBothSides}.
Thus a.e.\ $\bM \in D$ satisfies the strong local Pisot condition in the past as well. 
\end{proof}

\begin{remark}\label{rem:GloLocPisot}
For the local Pisot condition in Definition~\ref{def:genPisot}, we worked with a limit superior instead of a limit. For the generic Pisot condition in Definition~\ref{def:gengenPisot2} we use limits and, hence, Lyapunov exponents, because the limits in \eqref{eq:LyGenericBothSides} exist a.e.\ in view of Proposition~\ref{prop:genly} (indeed, these limits even exist in the past). This means that the generic Pisot condition of an invertible dynamical system, although defined only in the future, is always two-sided. So, generically, the Pisot condition in the future coincides with the two-sided one. 
Thus Definition~\ref{def:gengenPisot2} is the appropriate generic version of the two-sided Pisot condition in Definition~\ref{def:genPisot}. 

Moreover, we did not include the growth condition $\lim_{n\to\infty} \frac{1}{n} \log \lVert M_n\rVert = 0$ in the definition of the local Pisot condition because the local Pisot condition can hold even if the growth condition fails. On the contrary, the log-integrability, which is the generic counterpart of the growth condition, is needed for stating the generic Pisot condition.
\end{remark}

\subsection{Hyperbolic splitting for sequences of matrices}\label{sec:hyperb}
By Remark~\ref{rem:invtr}, the  Pisot condition for~$A_{\mathrm{tr}}$ guarantees the existence of an invariant Oseledets splitting into a stable space of dimension~$1$ and an unstable space of codimension~$1$ for the cocycle~$A_{\mathrm{inv}}$. This is one of the assertions of the following theorem which summarizes important consequences of the generic Pisot condition. It forms a metric version of Theorem~\ref{th:matrixPisot} and Corollary~\ref{cor:anosov}. In its statement, for a dynamical system $(D,\Sigma,\nu)$, with $D\subset \cM_d^\ZZ$, we call a set of the form
\[
[Z_0,\ldots,Z_{\ell-1}] = \{(M_n)_{n\in\ZZ} \in  D\,:\, M_0=Z_0, \ldots,M_\ell=Z_3\} \notx{0cylinder}{$[\ldots]$}{cylinder set}
\] 
a \emph{cylinder set}\indx{cylinder set}. Here, $\ell\in\NN$ and $Z_0,\ldots, Z_{\ell-1} \in \cM_d$.

\begin{theorem}\label{thm:oseledetsmat}
For $d\ge 2$, let $(D,\Sigma,\nu)$ with $D \subset \cM_d^{\ZZ}$ be an ergodic dynamical system which contains a cylinder set $[Z_0,\ldots,Z_{\ell-1}]$ of positive measure corresponding to a positive matrix $Z_{[0,\ell)}$.
If $(D,\Sigma,\nu)$ satisfies the Pisot condition for the Lyapunov exponents of the cocycle~$A_{\mathrm{tr}}$, then, for $\nu$-almost all $\bM \in D$,  the following assertions hold.
\begin{enumerate}[\upshape (i)]
\item
The sequence~$\bM$ is two-sided primitive and algebraically irreducible (in the future).
\item
The sequence~$\bM$ converges exponentially to a generalized right eigenvector $\bu$ in the future and to the generalized left eigenvector~$\bv$ in the past.
\item
The mapping family $(\TT,f_{\bM})$ associated to~$\bM$ is eventually Anosov for the hyperbolic splitting\indx{splitting!hyperbolic}   
\begin{equation}\label{eq:hyperbolic}
G^s(\bM) \oplus G^u(\bM) = \coprod_{n\in\ZZ} G_n^s(\bM) \oplus G_n^u(\bM) \end{equation}
with 
\[
G^s_n(\bM) = \RR\bu_n \quad\hbox{and}\quad G^u_n(\bM)= \bv_n^\perp,
\]
where $ \bu_n$ and $\bv_n$ are the generalized right and left eigenvectors of $\Sigma^n\bM$ defined in \eqref{eq:unvn}.
\item
There is a relation between the Anosov splitting of $(\TT,f_{\bM})$ and the Oseledets splittings \eqref{eq:oseledetsSplitting} of~$\bM$ for the cocycles $A_{\mathrm{inv}}$ and~$A_{\mathrm{tr}}$. In particular, for each $n \in \ZZ$, we have (setting $G^i = G_{\mathrm{inv}}^p$ for the inverse and  $G^i = G_{\mathrm{tr}}^p$ for the transpose splitting) 
\begin{equation}\label{eq:osele2}
\begin{aligned}
G^s_n(\bM) & = G_{\mathrm{inv}}^p(\Sigma^n\bM) = \RR\bu_n, \\
G^u_n(\bM) & = G_{\mathrm{inv}}^1(\Sigma^n\bM) \oplus \cdots \oplus G_{\mathrm{inv}}^{p-1}(\Sigma^n\bM) = \bv_n^\perp, \\
G^s_n(\bM)^\perp & = G_{\mathrm{tr}}^2(\Sigma^n\bM) \oplus \cdots \oplus G_{\mathrm{tr}}^{p}(\Sigma^n\bM) = \bu_n^\perp, \\
G^u_n(\bM)^\perp & = G_{\mathrm{tr}}^1(\Sigma^n\bM) = \RR\bv_n.
\end{aligned}
\end{equation}
\end{enumerate}
\end{theorem}

\begin{proof}
From Lemma~\ref{lem:GloLocPisot}, a.e.\ $\bM \in D$ satisfies the two-sided  (local) Pisot condition and the growth condition. Thus we may apply our results from Section~\ref{sec:LyapunovPisot2}  and  ~\ref{sec:LyapunovPisot3} almost everywhere in $D$.

We start with Assertion~(i). By assumption, there exists a cylinder set of positive measure in~$D$ corresponding to a  positive  matrix. This implies two-sided primitivity for $\nu$-a.e.\ sequence $\bM \in D$ by ergodicity of~$D$ and the Poincar\'e recurrence theorem. Since a.e.\ algebraic irreducibility is a consequence of Corollary~\ref{bst8.7}, Assertion (i) follows. Assertion~(ii) follows from Proposition~\ref{prop:strongcv2sided}, and Corollary~\ref{cor:anosov} can be applied to get~(iii).  

To prove~(iv), we have to characterize the Oseledets splitting \eqref{ospaces} for the inverse and the transpose cocycle. Let $\{(\Theta_i,d_i) : 1 \le i \le p\}$ be the Lyapunov spectrum for~$A_{\mathrm{tr}}$. By the definition of Lyapunov exponents in \eqref{eq:LyGenericBothSides}, $A_{\mathrm{inv}}$~has the Lyapunov exponents $-\vartheta_d \ge \dots \ge -\vartheta_2 > 0 > -\vartheta_1$ and, hence, Lyapunov spectrum $\{(-\Theta_i,d_i) : 1 \le i \le p\}$.
We have that $\bx \in G_{\mathrm{inv}}^p(\bM)$ if and only if $\lVert A_{\mathrm{inv}}^{(n)}(\bM)\bx\rVert $ grows like $e^{{-\Theta_1}n}$, with $\Theta_1 = \vartheta_1 > 0$. By Corollary~\ref{cor:anosov}, see in particular \eqref{eq:anosovthm1} and \eqref{eq:anosovpast1}, we have 
\[
\lim_{n\to+\infty} M_{[0,n)}^{-1}\, \bx  = \mathbf{0}, \quad \lim_{n\to-\infty} \lVert M_{[n,0)}\bx  \lVert = +\infty, \quad \mbox{for all}\ \bx \in  \RR\bu \setminus \{\mathbf{0}\}, 
\]
hence, $G_{\mathrm{inv}}^p(\bM) = \RR\bu$.
The same reasoning applies to~$\bv^\perp$. Again by Corollary~\ref{cor:anosov} this time we refer also to \eqref{eq:anosovthm2} and \eqref{eq:anosovpast2}, we gain
\[ 
\lim_{n\to+\infty} \lVert M_{[0,n)}^{-1}\, \bx \lVert = +\infty, \quad \lim_{n\to-\infty}  M_{[n,0)}\bx = \mathbf{0}, \quad  \mbox{for all}\ \bx \in \bv^\perp \setminus \{\mathbf{0}\}, 
\]
therefore $G_{\mathrm{inv}}^1(\bM) \oplus \cdots \oplus G_{\mathrm{inv}}^{p-1}(\bM) = \bv^\perp$. Now, the first two lines of~(iv) follow from the invariance property of the Oseledets spectrum stated in Propostion~\ref{prop:genly2} and the definition of $\bu_n$ and~$\bv_n$ in \eqref{eq:unvn}. The other two lines follow because 
\[
G_{\mathrm{tr}}^i(\Sigma^n\bM) = \bigg( \prod_{j\neq i} G_{\mathrm{inv}}^{p-j+1}(\Sigma^n\bM) \bigg)^\perp \qquad (1\le i\le p,\, n\in\NN).
\]
To see this, pick $\bx \in G_{\mathrm{inv}}^{p-i+1}(\bM)$ and $\by \in \big( \prod_{j\neq i} G_{\mathrm{inv}}^{p-j+1}(\bM) \big)^\perp$ satisfying $\langle \bx,\by\rangle\not=\mathbf{0}$, which is possible because $\prod_{j=1}^d G_{\mathrm{inv}}^{j}(\bM)=\RR^d$. Since
\[
\lVert M_{[0,n)}^{-1} \bx \rVert_2 \cdot \lVert \tr{\!M}_{[0,n)} \by\rVert_2  \ge  \langle M_{[0,n)}^{-1} \bx, \tr{\!M}_{[0,n)} \by \rangle = \langle \bx,\by \rangle \neq 0,
\] 
we gain ${\lim_{n\to\pm\infty}} \frac{1}{n} \log \lVert \tr{\!M}_{[0,n)}\by\rVert \ge \vartheta_i$.
Starting with $i=1$, one recursively checks that this inequality actually has to be an equality.
\end{proof}

\begin{example}\label{ex:AR2}
For the set $\cM_{\rAR} = \{M_{\rAR,1},M_{\rAR,2}, M_{\rAR,3}\}$ of Arnoux--Rauzy matrices defined in Example~\ref{ex:AR1}, consider the full shift $(\cM_{\rAR}^{\ZZ},\Sigma,\nu)$, where $\nu$ is an ergodic shift invariant probability measure with full support. Because $\cM_{\rAR}$ is finite, the log-integrability condition is trivially satisfied. Moreover, by \cite[Proposition~6.4]{BST:23} (see also its proof), the shift $(\cM_{\rAR}^{\ZZ},\Sigma,\nu)$ satisfies the generic Pisot condition. Finally, because $\nu$ has full support, the cylinder set
\[
\begin{aligned}
& [M_{\rAR,1},M_{\rAR,2},M_{\rAR,3}] \\
& \hspace{3em} = \{(M_n)_{n\in\ZZ} \in  \cM_{\rAR}^{\ZZ}\,:\, M_0 = M_{\rAR,1}, M_1 = M_{\rAR,2}, M_2 = M_{\rAR,3}\}
\end{aligned}
\]
has positive measure. Because $M_{\rAR,1} M_{\rAR,2} M_{\rAR,3}$ is a positive matrix, $(\cM_{\rAR}^{\ZZ},\Sigma,\nu)$ satisfies all conditions of Theorem~\ref{thm:oseledetsmat}.

Thus $\nu$-a.e.\ sequence $\bM\in \cM_{\rAR}^{\ZZ}$ is two-sided primitive and algebraically irreducible. It converges exponentially to a left and right eigenvector in the future and in the past, respectively. Moreover, the mapping family $(\TT,f_{\cM})$ associated to it is eventually Anosov w.r.t.\ the splittings defined in Theorem~\ref{thm:oseledetsmat}~(iii) and~(iv).
\end{example}

We will see in Proposition~\ref{prop:Brun23Pisot} that the Brun continued fraction algorithm gives rise to sequences of matrices that satisfy the conditions of Theorem~\ref{thm:oseledetsmat} for $d\in\{3,4\}$.

\chapter{Multidimensional continued fraction algorithms}\label{sec:cf}

In this chapter we develop the parts of the theory of multidimensional continued fraction algorithms that are relevant for us. In particular, Section~\ref{sec:cfgentheory} contains the definition of a multidimensional continued fraction algorithm, presents the construction of geometric natural extensions for these algorithms, and shows how they are related to mapping families. In Section~\ref{sec:Brun} we provide a detailed treatment of different versions of the Brun continued fraction algorithm, calculate their natural extensions and invariant measures, and show that they satisfy the Pisot condition.

\section[General theory of continued fraction algorithms]{General theory of multidimensional continued fraction algorithms} \label{sec:cfgentheory}

As mentioned in the introduction, one of the motivations for studying sequences of integer matrices comes from the theory of multidimensional continued fraction algorithms. In the present section, we show how continued fraction algorithms fit in the theory developed for sequences of matrices developed in Chapter~\ref{chapter:matrix}. In particular, in Section~\ref{subsec:cfdef} we provide all necessary definitions around multidimensional continued fraction algorithms and illustrate them for the classical continued fraction algorithm and for the Arnoux--Rauzy algorithm. In Section~\ref{sec:natex}, we discuss geometric natural extensions of continued fraction algorithms and show, for the example of the classical continued fraction algorithm, how they can be used to calculate an invariant measure for an algorithm that is equivalent to the Lebesgue measure. Section~\ref{sec:prpcppp} gives some definitions in the context of multidimensional continued fraction algorithms that will be needed later. For instance,  Lyapunov exponents and a version of the Pisot condition for such an algorithm are provided. In Section~\ref{subsec:MCFmappingfamily}, we show how multidimensional continued fraction algorithms can be related to $\cM$-adic mapping families. This enables us to apply the theory of Chapter~\ref{chapter:matrix} to multidimensional continued fraction algorithms.

\subsection{Unimodular algorithms}\label{subsec:cfdef}
There exists a variety of formalisms for defining multidimensional continued fraction algorithms (see for instance \cite{Lagarias:93,Schweiger:00,KLM:07,AL18,Fougeron:20}). For our purposes, it turns out to be convenient to define a multidimensional continued fraction algorithm on a projective space.  In this setting, a continued fraction algorithm ``produces'' sequences of matrices, which perfectly matches with our studies in Chapter~\ref{chapter:matrix}. In the sequel, the $d$-dimensional projective space is denoted as~$\PP^{d}$\notx{Pd}{$\PP^d, \PP^d_{\ge0}$}{projective $d$-space, nonnegative cone}, and $\PP^{d}_{\ge0}$ is the nonnegative cone of~$\PP^{d}$.

\begin{definition}[Continued fraction algorithm] \label{def:cf} 
\indx{continued fraction algorithm!multidimensional}
Let $d\ge 2$. Let 
\[
A:\, X \to \GL(d,\ZZ) \notx{An}{$A(\cdot),A^{(n)}(\cdot)$}{\hspace{.5em}linear cocycle}
\]
be a map from a subset~$X$ of the $(d{-}1)$-dimensional projective space~$\PP^{d-1}$ to the set $ \GL(d,\ZZ)$ of unimodular integer $d {\times} d$ matrices such that $\tr{\!A}(\bx)^{-1} \bx \in X$ for all $\bx \in X$, and let 
\begin{equation}\label{eq:CFalgo}
F:\, X \to X, \quad \bx \mapsto \tr{\!A}(\bx)^{-1} \bx. \notx{Fc}{$F(\cdot)$}{continued fraction map}
\end{equation}
Then $(X,F,A)$\notx{XFA}{$(X,F,A[,\nu])$}{\hspace{.75em}continued fraction algorithm} is called a \emph{$(d{-}1)$-dimensional continued fraction algorithm}.

If the range of~$A$ is a finite set of matrices, then we say that the continued fraction algorithm is \emph{additive}\indx{continued fraction algorithm!additive}; otherwise, it is called \emph{multiplicative}\indx{continued fraction algorithm!multiplicative}.
If all $A(\bx)$, $\bx \in X$, have nonnegative entries, then the algorithm is called \emph{positive}\indx{continued fraction algorithm!positive}. 
\end{definition}

The use of the transpose inverse $\tr{\!A}(\bx)^{-1}$ in the definition of a multidimensional continued fraction algorithm $(X,F,A)$ may look cumbersome at first sight. The inversion is used in \eqref{eq:CFalgo} because the \emph{partial quotient matrices}\indx{partial quotient!matrix} $\tr{\!A}(F^n\bx)$, $n\ge 0$, are more important in our work than their inverses. Indeed, for each $\bx \in X$, iterative application of the algorithm $(X,F,A)$ generates a sequence of partial quotient matrices $(\tr{\!A}(F^n\bx))_{n\ge0}$. Taking products of the partial quotient matrices leads to the \emph{convergent matrices}\indx{convergent!matrix} $\tr{\!A}^{(n)}(\bx)=\tr{\!A}(\bx)\, \tr{\!A}(F\bx) \cdots \tr{\!A}(F^{n-1}\bx)$, $n\ge 0$ that are used for simultaneous Diophantine approximation of $\bx=[x_1:\dots : x_d]\in X$. 
Indeed, the column vectors of the convergent matrices
\[
\tr{\!A}^{(n)}(\bx) = \begin{pmatrix}p_{1,1}^{(n)} & p_{1,2}^{(n)} & \cdots & p_{d,1}^{(n)} \\ \vdots & \vdots & \ddots & \vdots \\ p_{d-1,1}^{(n)} & p_{d-1,2}^{(n)} & \cdots & p_{d-1,d}^{(n)} \\ q_1^{(n)} & q_2^{(n)} & \cdots & q_d^{(n)}\end{pmatrix} \qquad(n\ge 0)
\]
allow us to define $d$ sequences of \emph{rational convergents}\indx{convergent!rational} $(p^{(n)}_{i,j}/q_j^{n)})_{1\le i<d}$, one for each $j\in\{1,\ldots,d\}$, that are supposed to converge  in~$\RR^{d-1}$ to $(\frac{x_1}{x_d},\dots,\frac{x_{d-1}}{x_d})$; recall that $\bx=[x_1:\dots : x_d]$. Here, we have chosen to  use   the last line  for  the  denominators of convergents. This is  arbitrary but it corresponds to a situation met for   classical
 algorithms  such as the Brun algorithm discussed in Section \ref{sec:Brun} below, where the  integers $q_i^{(n)}$  in the last line are the largest entries for each column.   

The transposition used in \eqref{eq:CFalgo} comes from the fact that we are interested in the \emph{linear cocycle}\indx{cocycle!linear} of the algorithm $(X,F,A)$, which is given by~$A$. Indeed
\[
A^{(n)}(\bx) = A(F^{n-1}\bx) \cdots A(F\bx) \, A(\bx),
\]
\notx{An}{$A(\cdot),A^{(n)}(\cdot)$}{\hspace{.5em}linear cocycle}satisfies the \emph{cocycle property} $A^{(m+n)}(\bx) = A^{(m)}(F^n\bx)\, A^{(n)}(\bx)$; see also Section~\ref{subsec:cocyclemat} and Definition~\ref{def:cocycle}. In other words, $A(\bx)$ and $\tr{\!A}(\bx)^{-1}$ are covariant, while ${A}(\bx)^{-1}$ and $\tr{\!A}(\bx)$ are contravariant. 
 
We will need the following notions of cylinder sets and follower sets.

\begin{definition}[Cylinder set and follower set for multidimensional continued fraction algorithms]
\label{def:MFCcyl}\indx{cylinder set}
The \emph{cylinder sets} of a multidimensional continued fraction algorithm $(X,F,A)$ are given by
\[
X_{[0,n)}(\bx) = \{\by \in X \,:\, A(\by) = A(\bx), A(F\by) = A(F \bx), \dots, A(F^{n-1}\by) = A(F^{n-1} \bx)\}, \notx{X}{$X_{[0,n)}(\cdot),X_M$}{\hspace{1em}cylinder set of a continued fraction algorithm}
\]
with $X_{[0,0)}(\bx) = X$.
Moreover, for $M\in A(X)$ we write
\[
X_M = \{\by \in X \,:\, A(\by) = M\} = A^{-1}(M)
\]
for the cylinder set associated to a given matrix~$M$. 

A \emph{follower set}\indx{follower set} is a set of the form $F^nX_{[0,n)}(\bx)$, where $\bx \in X$, $n \in \NN$, and $X_{[0,n)}(\bx)$ is a cylinder set of $(X,F,A)$. 
\end{definition}

Instead of working in the homogeneous space~$\PP^{d-1}$, one can choose a norm $\lVert\cdot\rVert$ on $\RR^d$ (various norms can be convenient for the normalization; see for instance Section~\ref{sec:Brun}) and select for each element of~$\PP^{d-1}$ a representative in 
\[
\Delta^{d-1} = \{\bx \in \RR^d \,:\, \lVert\bx\rVert = 1\}. 
\]
In other words, via the corresponding  (injective) map\footnote{We will mostly stay in the nonnegative cone $\PP_{\geq 0}^{d-1}$ of $\PP^{d-1}$, where the restriction of $\chi: \PP_{\geq 0}^{d-1} \rightarrow \chi(\PP_{\geq 0}^{d-1})$   can be regarded as an affine chart; of course, we need more than one chart to obtain an affine atlas for $\PP^{d-1}$.} 
\begin{equation}\label{eq:chidef}
\chi: \PP^{d-1} \to \Delta^{d-1}, \notx{chi}{$\chi(\cdot),\hat{\chi}(\cdot)$}{affine chart on projective space} 
\end{equation} 
which assigns to each $\bw \in \PP^{d-1}$ a representative with norm~$1$, we can equip~$X$ with affine coordinates. Therefore, setting $\Delta = \chi(X)$, in the coordinates~$\chi$ the continued fraction map~$F$ is given by the map $T = \chi \circ F \circ \chi^{-1}\vert_\Delta$, which can be written out as
\begin{equation}\label{eq:CFT}
T:\, \Delta \to \Delta, \quad \bx \mapsto \frac{\tr{\!A}(\bx)^{-1} \bx}{\lVert\tr{\!A}(\bx)^{-1} \bx\rVert}, \notx{T}{$T(\cdot)$}{affine continued fraction map}
\end{equation}
and $(\Delta,T,A)$\notx{Delta}{$(\Delta,T,A[,\mu])$}{\hspace{.75em}affine version of a continued fraction algorithm} is the \emph{affine version}\indx{continued fraction algorithm!affine version} of the continued fraction algorithm $(X,F,A)$. In our applications, $X$~is always small enough to guarantee that the restriction of~$\chi$ to~$X$  is bijective from~$X$ onto $\Delta = \chi(X)$.

A multidimensional continued fraction algorithm $(X,F,A)$ as in Definition~\ref{def:cf} as well as its affine version \eqref{eq:CFT} is called \emph{linear simplex-splitting}\indx{continued fraction algorithm!simplex-splitting}  in \cite[Section~2]{Lagarias:93}. This kind of algorithms contain classical algorithms like the algorithms of Brun~\cite{Brun19,Brun20,BRUN}, Jacobi--Perron \cite{Heine1868,Perron:07,Bernstein:71,Schweiger:73}, and Selmer~\cite{Selmer:61}. 
As the following example shows, also the classical continued fraction algorithm fits into the scheme of  Definition~\ref{def:cf}.

\begin{example}[Classical continued fraction algorithm]\label{ex:classicalCF}
To define the \emph{classical continued fraction algorithm}\indx{continued fraction algorithm!classical} we set $d=2$ and $X_{\rG} = \{[x:1]\,:\, 0< x < 1\} \subset \PP^1_{\ge0}$, and define the cocycle~$A_{\rG}$ by
\begin{equation}\label{eq:classicalcocycle}
A_{\rG}: X \to \GL(2,\ZZ); \quad \begin{bmatrix}x \\ 1\end{bmatrix} \mapsto \begin{pmatrix} 
0&1\\1&a
\end{pmatrix} \quad\hbox{with } a=\left\lfloor\frac1{x} \right\rfloor.
\end{equation}
The mapping~$F_{\rG}$ is therefore given by 
\[
F_{\rG}: X_{\rG} \to X_{\rG}; \quad \begin{bmatrix} x\\ 1 \end{bmatrix} \mapsto \begin{pmatrix} 
-a&1\\1&0
\end{pmatrix} \begin{bmatrix} x\\ 1 \end{bmatrix} =  
\begin{bmatrix} 1-ax\\ x \\  \end{bmatrix} =  
\begin{bmatrix}\frac{1-ax}{x} \\ 1\end{bmatrix}.
\]
If we choose the maximum norm $\lVert\cdot\rVert_\infty$ to define $\Delta^1$ we get 
\[
\Delta^1 = \chi(X_{\rG}) = \{(x,1) \,:\, 0<x< 1\}. 
\]
According to \eqref{eq:CFT}, the affine version of the classical continued fraction map is given by
\[
T_{\rG}:  \Delta^1 \to \Delta^1; \quad \begin{pmatrix}x \\ 1 \end{pmatrix} = \begin{pmatrix}\frac{1}{x}-a  \\ 1\end{pmatrix}.
\]
Because there is no information in the second coordinate in the definition of~$T_{\rG}$, we may identify $\Delta^1$ with $(0,1)$. Because $a=\lfloor\frac1{x}\rfloor$ by the definition of~$A_{\rG}$ we gain
\[
T_{\rG}:\, (0,1) \to (0,1), \quad x \mapsto \frac1x - \left\lfloor\frac1x\right\rfloor,
\]
which is exactly the well-known Gauss map that we already mentioned in \eqref{eq:GaussMapIntro}.\footnote{Since we are mainly interested in metric results, we didn't take care of the fact that $0$ is a possible image of $T_{\rG}$.} Note that, since $A_{\rG}$ can take infinitely many values, the classical continued fraction algorithm is a multiplicative algorithm.
\end{example} 

We note that the Farey map \eqref {eq:FareyIntro} can be defined in terms of Definition~\ref{def:cf} as an additive continued fraction algorithm. The range of its cocycle equals the finite set $\{M_{\rF,1},M_{\rF,2}\}$, where $M_{\rF,1}$ and~$M_{\rF,2}$ are the Farey matrices\indx{Farey!matrix} introduced in Example~\ref{ex:farey1}.
The following example contains a well-known generalization of the Farey algorithm to higher dimensions.
 
\begin{example}[Arnoux--Rauzy continued fraction algorithm]\label{ex:AR3} \indx{Arnoux--Rauzy!algorithm}\indx{continued fraction algorithm!Arnoux--Rauzy}
To define the \emph{Arnoux--Rauzy continued fraction algorithm} for $d=3$, recall the set of Arnoux--Rauzy matrices $\cM_{\rAR}=\{M_{\rAR,1},M_{\rAR,2},M_{\rAR,3}\}$ from Example~\ref{ex:AR1}. The cocyle $A_{\rAR}: \PP^2_{\ge0}  \to \GL(3,\ZZ)$\notx{AxR}{$\rAR$}{object related to the Arnoux--Rauzy algorithm} for this algorithm is given by
\begin{equation*}\label{eq:ARcocycle}
A_{\rAR}([x_1:x_2:x_3]) = 
M_{\rAR,a}^{-1} \quad \hbox{if } x_a=\max\{x_1,x_2,x_3\};\;  a\in\{1,2,3\}.
\end{equation*}
Using this cocycle, we define a mapping $F: \PP^2_{\ge0} \to \PP^2$ by  $F(\bx)= \tr{\!A}_{\rAR}(\bx)^{-1} \bx$, which can be written explicitly as
\[
F: \PP^2_{\ge0} \to \PP^2, \quad  [x_1:x_2:x_3] \mapsto
\begin{cases}
[x_1{-}x_2{-}x_3:x_2:x_3] & \mbox{ if }x_1=\max\{x_1,x_2,x_3\}, \\
[x_1:x_2{-}x_1{-}x_3:x_3] & \mbox{ if }x_2=\max\{x_1,x_2,x_3\}, \\
[x_1:x_2:x_3{-}x_1{-}x_2] & \mbox{ if }x_3=\max\{x_1,x_2,x_3\}.
\end{cases}
\]
In general, the image of~$F$ is not contained in the positive cone of~$\PP^2$ and, hence, $F$ cannot always be iterated. We therefore restrict~$F$ to the set
\[
X_{\rAR} = \big\{\bx \in \PP^2_{\ge0} \,:\,  F^n(\bx) \in \PP^2_{\ge0}\ \mbox{for all}\ n \in\NN\big\},
\]
which is a projective version of the \emph{Rauzy gasket}\indx{Rauzy!gasket}, see~\cite{ArnouxStarosta:13}. This set allows us to define the \emph{Arnoux--Rauzy map}\indx{Arnoux--Rauzy!map}
\begin{equation} \label{eq:ARmap}
F_{\rAR}:\, X_{\rAR} \to X_{\rAR}, \quad \bx \mapsto F(\bx),
\end{equation}
and, hence, the \emph{Arnoux--Rauzy continued fraction algorithm} $(X_{\rAR},F_{\rAR},A_{\rAR})$. 
If we choose the norm $\lVert\cdot\rVert_1$ in the definition of the chart $\chi$ in \eqref{eq:chidef} to define $\Delta_{\rAR}$ we get the \empty{Rauzy gasket}
$\Delta_{\rAR}=\chi(X_{\rAR})$. According to \eqref{eq:ARmap}, the affine version of the Arnoux--Rauzy map $T_{\rAR}: \Delta_{\rAR} \to \Delta_{\rAR}$ is given by
\[
T_{\rAR}(x_1,x_2,x_3) =
\begin{cases}
\frac{1}{x_1}(x_1{-}x_2{-}x_3,x_2,x_3) & \hbox{ if }x_1=\max\{x_1,x_2,x_3\}, \\
\frac{1}{x_2}(x_1,x_2{-}x_1{-}x_3,x_3) & \hbox{ if }x_2=\max\{x_1,x_2,x_3\}, \\
\frac{1}{x_3}(x_1,x_2,x_3{-}x_1{-}x_2) & \hbox{ if }x_3=\max\{x_1,x_2,x_3\}.
\end{cases}
\]
Since $A_{\rAR}$ can take only 3 different values, the Arnoux--Rauzy algorithm is an additive continued fraction algorithm. It can be defined in an analogous way for each $d\ge 3$, see for instance \cite[Section~6.3]{BST:23}.
\end{example}

\subsection{Natural extension and invariant measure}\label{sec:natex}
Our main goal in this section is to construct an invariant measure for a multidimensional continued fraction algorithm $(X,F,A)$. Because invariant measures are easier to find for invertible dynamical system, we first extend $(X,F,A)$ to an invertible dynamical system, the so-called natural extension $(\hX,\hF,\hA)$ of $(X,F,A)$. By constructing an invariant measure $\hnu$ on some geometric version of $(\hX,\hF,\hA)$, we can then obtain the desired invariant measure $\nu$ of $(X,F,A)$ as a pushed forward measure of $\hnu$. For technical reasons we will mainly work in the affine version $(\Delta,T,A)$ of $(X,F,A)$. The section ends with a series of important definitions related to multidimensional continued fraction algorithms like the Pisot condition, positive range, and periodic Pisot point.

A multidimensional continued fraction algorithm $(X,F,A)$ is usually not one-to-one, hence, it is not invertible. To produce an invertible version of the algorithm, we use the  well-known concept of inverse limit. Let $F: X \to X$ be a surjective map. Then we can always find a set~$\hX$, an invertible map $\hF: \hX \to \hX$, and a surjective map $\pi: \hX \to X$ such that 
\begin{equation}\label{eq:natural}
\pi \circ \hF = F \circ \pi. 
\end{equation}
The system $(\hX,\hF)$ is called the \emph{inverse limit}\indx{inverse limit} of $(X,F)$. It can be constructed as 
\begin{equation}\label{eq:hatX}
\hX = \big\{(\bx_n)_{n\in\ZZ} \in X^{\ZZ} \,:\, F(\bx_n) = \bx_{n+1}\ \mbox{for all}\ n \in \ZZ\big\},
\end{equation}
with $\hF((\bx_n)_{n\in\ZZ}) = (\bx_{n+1})_{n\in\ZZ}$ and $\pi((\bx_n)_{n\in\ZZ}) = \bx_0$.
If, furthermore, $\mu$~is an invariant measure for the dynamical system $(X,F)$, then one can easily build an invariant measure $\hmu$ for $(\hX,\hF)$ such that $\mu = \pi_*\hmu$.\footnote{For a function $f$ and a measure~$\mu$, the pushforward measure\indx{measure!pushforward} $f_* \mu$ is defined by setting $f_* \mu (B) := \mu(f^{-1} (B))$ for all measurable sets~$B$.} The triple $(\hX,\hF,\hmu)$ is a \emph{natural extension}\indx{natural extension} of the measurable dynamical system $(X,F,\mu)$ in the sense of Rokhlin, see e.g.\ \cite[Section~2.4.2]{VianaOlivieira}. 
The cocycle~$A$ then extends to a cocycle $\hA: \hX \to \GL(d,\ZZ)$, with $\hA = A \circ \pi$, and $(\hX,\hF,\hA)$ is called a \emph{natural extension} of the multidimensional continued fraction algorithm $(X,F,A)$.

According to Definition~\ref{def:cf}, a multidimensional continued fraction algorithm $(X,F,A)$ is defined on a subset~$X$ of the projective space~$\PP^{d-1}$. While there is no canonical Lebesgue measure in a projective space, there is a Lebesgue measure class: the class of measures that are absolutely continuous with respect to the Lebesgue measure in any local affine coordinate system. The problem of finding an ergodic invariant measure of $(X,F,A)$ in the Lebesgue measure class (i.e., a \emph{Gauss measure}\indx{Gauss!measure}) is difficult. As mentioned before, it is convenient to look for an invariant measure $\hnu$ in a geometric version of the natural extension of $(X,F,A)$ and then get the desired invariant measure as a pushforward measure of $\hnu$. Moreover, it is more convenient to do these constructions for the affine version of the algorithm $(X,F,A)$.

Let $d\ge 2$, let $\lVert\cdot\rVert$  be a norm on~$\RR^d$, and let $\Delta^{d-1} = \{\bx \in \RR^d : \lVert\bx\rVert = 1\}$. Equip $(X,F,A)$ with affine coordinates using the map $\chi: \PP^{d-1}\to\Delta^{d-1}$, introduced above,  which assigns to each $\bw \in \PP^{d-1}$ a representative with norm~$1$, and consider the affine version $(\Delta,T,A)$ of $(X,F,A)$ with $\Delta = \chi(X) \subset \Delta^{d-1}$ and $T$ as in \eqref{eq:CFT}. 

Our starting point for a geometric version of a natural extension of $(\Delta,T,A)$ is the map
\begin{equation}\label{eq:hFdef2}
\hT:\, \Delta^{d-1} \times\Delta^{d-1}\to \Delta^{d-1} \times\Delta^{d-1}, \quad (\bx, \by) \mapsto \left(\frac{\tr{\!A}(\bx)^{-1}(\bx) \bx}{\lVert \tr{\!A}(\bx)^{-1}(\bx)\bx \rVert}, 
\frac{{A}(\bx)\by}{\lVert{A}(\bx)\by\rVert}\right).
\end{equation} 
Our task is now twofold. We have to find a suitable subset $\hDelta$ of  $\Delta^{d-1} \times\Delta^{d-1}$ on which $\hT$ is bijective, and an invariant measure for $\hT$ on $\hDelta$. 

We first define a measure $\hmu$ on $\Delta^{d-1} \times\Delta^{d-1}$ which is locally preserved by $\hT$. We start by a map $h: \Delta^{d-1} \to \Delta^{d-1}$ that is locally defined by $h(\bx) = \frac{M\bx}{\lVert M \bx\rVert}$ for some invertible matrix $M$. An easy computation shows that the Jacobian of~$h$ at~$\bx$ is $\frac{|\det M|}{\lVert M\bx\rVert^d}$; for more details on this computation, see e.g.\  Veech~\cite[Proposition~5.2]{Veech:78}.\footnote{Note that Veech uses $\lVert \cdot\rVert_1$, however, his proof goes through for any other norm.} This implies that the Jacobian of the map $\hT$ at $(\bx,\by)\in \Delta^{d-1} \times\Delta^{d-1}$ equals  ${\lVert \tr{\!A}(\bx)^{-1}(\bx) \bx \rVert}^{-d} {\lVert {A}(\bx) \by \rVert}^{-d}$, which motivates the following definition. Let $\lambda$\notx{lambda}{$\lambda,\lambda_{\bone}$}{Lebesgue measure} be the (piecewise)\footnote{If the unit ball for the norm $\lVert\cdot\rVert$ is for instance a cube or an octahedron (according to the dimension), then we use the Lebesgue measure on each surface; if it is a sphere, then we use the spherical Lebesgue measure.} Lebesgue measure on~$\Delta^{d-1}$ and define $\hmu$ on $\Delta^{d-1} \times\Delta^{d-1}$ by $\mathrm{d} \hmu(\bx,\by) = \frac{\mathrm{d}\lambda(\bx)\cdot \mathrm{d}\lambda{(\by)}}{\langle \bx,\by \rangle^d}$. We obviously have $\langle \bx,\by \rangle=\langle \tr{\!A}(\bx)^{-1} \bx,A(\bx) \by \rangle$. Thus the computation of the Jacobian of~$\hT$ shows that 
\begin{equation}\label{eq:YhTa}
\hmu(\hT(Y)) = \hmu(Y), \quad\text{if } Y \subset \chi(A^{-1}(M)) \times\Delta^{d-1} \text{ for some } M\in \GL(d,\ZZ), 
\end{equation}
because $\hT$ is obviously injective on such a set $Y$. Thus $\hT$ preserves the measure $\hmu$ on subsets of $\Delta^{d-1} \times\Delta^{d-1}$ on which $\chi\circ A$ is constant. If we allow for general subsets of $\Delta^{d-1} \times\Delta^{d-1}$, \eqref{eq:YhTa} implies that
\begin{equation}\label{eq:YhT}
\hmu(\hT(Y)) \le \hmu(Y)  \quad\text{if }  Y \subset \Delta^{d-1} \times\Delta^{d-1}.
\end{equation}

As for the region of bijectivity for $\hT$, it turns out that, when the map~$T$  satisfies some mild contractivity conditions and  using the Banach fixed-point argument from~\cite[Theorem~1.1]{ArnouxSchmidt:19}, it is possible to construct a set~$\hDelta\subset\Delta^{d-1} \times\Delta^{d-1}$ on which $\hT$ is bijective. 
If the set $\hDelta$ exists, then it is unique by \cite[Theorem~1.2]{ArnouxSchmidt:19}  under some mild conditions on~$T$. Sometimes (see e.g.\ Section~\ref{sec:Brun} for the Brun continued fraction algorithm) we obtain a candidate for $\hDelta$ by guessing. In this case one has to prove bijectivity of $\hT$ on $\hDelta$ directly and the following lemma, which is an immediate consequence of \eqref{eq:YhT}, is useful (see also \cite[Theorem~4.7]{ArnouxSchmidt:19}).

\begin{lemma}\label{lem:hTsurbi}
If $\hT$ is surjective on $\hDelta$ then it is bijective on $\hDelta$.
\end{lemma}

From now on, we regard $\hT$ as a function on $\hDelta$ and $\hmu$ as a measure on $\hDelta$. 
Because $\hT$ is bijective on $\hDelta$, it follows from \eqref{eq:YhTa} that in \eqref{eq:YhT} we have equality for each $Y\subset \hDelta$. Again by bijectivity of $\hT$ this implies that $\hT_* \hmu = \hmu$, i.e., $\hmu$ is an invariant measure for $\hT$ that is equivalent to the piecewise Lebesgue measure $\lambda\times\lambda$ on~$\hDelta$.
Using the natural projection 
\[
\pi:\, \hDelta \to \Delta, \quad (\bx,\by) \mapsto \bx,
\]
we see that our construction implies that $(\hDelta, \hT,\hA,\hmu)$\notx{Delta}{${(\hat{\Delta},\hat{T},\hat{A},\hat{\mu})}$}{\hspace{.25em}natural extension of $(\Delta,T,A,\mu)$},  with $\hA = A \circ \pi$, is an invertible measurable dynamical system which is a geometric version of the natural extension of $(\Delta,T,A,\mu)$ with $\mu = \pi_*\hmu$.\notx{Delta}{$(\Delta,T,A[,\mu])$}{\hspace{.75em}affine version of a continued fraction algorithm}

If $\hmu(\hDelta) > 0$ and if $\hmu$ (or, equivalently, $\mu$; see~\cite[Chapter~10, \S~4]{CFS:82}) is ergodic, then the measure~$\mu$ is a Gauss measure of the affine version $(\Delta,T,A,\mu)$ of the algorithm. There is no general method to ensure that $\hmu(\hDelta) > 0$, and one has to check this directly for each algorithm. The situation is better for ergodicity. Usually, ergodicity of $\hmu$ is proved by verifying a R\'enyi type condition, see e.g.\ \cite{Renyi:57} or \cite[Chapter~3]{Schweiger:00}. Indeed, it suffices to show that there exists an ergodic invariant measure which is equivalent to the Lebesgue measure: If we exhibit such a measure, each invariant set has either zero or full Lebesgue measure and, hence, either zero or full measure~$\hmu$. Thus $\hmu$ and, hence, $\mu$, is ergodic. For the Brun continued fraction algorithm, a proof of ergodicity is detailed in \cite[Chapter~3]{Schweiger:00}. This also shows that sometimes, before a R\'enyi type condition can be applied, one has to consider certain accelerations of~$T$, so-called \emph{jump transformations}\indx{jump transformation}; see~\cite{Schweiger:1975} for a theory of these transformations (we will use a jump transformation in Section~\ref{subsec:BrunMul} to define the multiplicative version of the Brun continued fraction algorithm).  
 
It remains to obtain the ergodic invariant measure $\nu$ in the Lebesgue measure class for the multidimensional continued fraction algorithm $(X,F,A)$. Via the inverse of the  bijective chart  given by the restriction of~$\chi$ from~$X$ onto $\Delta = \chi(X)$, we can transfer this dynamical system back to~$\PP^{d-1}$ and end up with the system $(X,F,A,\nu)$.\notx{XFA}{$(X,F,A[,\nu])$}{\hspace{.75em}continued fraction algorithm} The measure~$\nu$, which is defined by $\chi_*\nu = \mu$, is the desired Gauss measure for $(X,F,A,\nu)$. In our examples we will only provide the Gauss measure $\mu$ for the projective version of the algorithm, because it is the more natural object and $\nu$ can easily be computed from $\nu$. 
 
By transferring $(\hDelta,\hT,\hA,\hmu)$ back to $\PP^{d-1} {\times} \PP^{d-1}$ via the bijection\footnote{Again we assume that $\hDelta$ is not too large, otherwise we need more than one chart.} $\hchi = \chi {\times} \chi$\notx{chi}{$\chi(\cdot),\hat{\chi}(\cdot)$}{affine chart on projective space}, we see that $(\hX,\hF,\hA,\hnu)$\notx{XFA}{$(\hat{X},\hat{F},\hat{A},\hat{\nu})$}{\hspace{.25em}natural extension of $(X,F,A,\nu)$} is a natural extension of $(X,F,A,\nu)$ with  $\hX = \hchi^{-1}(\hDelta)$ and $\hF = \chi^{-1} \circ \hT \circ \chi$. Note that $\hF$ can be written explicitly as 
\begin{equation}\label{eq:hatFdef}
\hF:\, \hX \to \hX, \quad (\bx, \by) \mapsto (\tr{\!A}(\bx)^{-1} \bx, A(\bx)\by).
\end{equation}

\begin{remark}\label{rem:nu>0}
Of course, this construction is meaningful only if $\hDelta$ and, hence~$\hX$, can be chosen in a way that $\hmu(\hDelta) > 0$, and equivalently $\hnu(\hX) > 0$, because otherwise $\mu(\Delta) = 0$. As mentioned above, no general method is known at the moment to achieve this positivity.  

If $\hnu(\hX) > 0$ and $\hnu$ is absolutely continuous with respect to the Lebesgue measure, we get that $\hX$ has positive Lebesgue measure.  Hence, all the statements  that hold a.e.\ hold Lebesgue a.e.\ in this case. This turns out to be particularly relevant, since the set of parameters for which the mapping families for multidimensional continued fraction algorithms will be defined in Section~\ref{subsec:MCFmappingfamily} is precisely~$\hX$.
\end{remark}

We illustrate our constructions briefly for the classical continued fraction algorithm discussed in Example~\ref{ex:classicalCF}.

\begin{example}[Gauss measure for the classical continued fraction algorithm]\label{ex:classicalCF2}
Let $(\Delta^1,T_{\rG},A_{\rG})$ be the affine version of the classical continued fraction algorithm defined in Example~\ref{ex:classicalCF} and use the notation of that example. By the definition of the cocycle~$A_{\rG}$ provided in \eqref{eq:classicalcocycle}, we gain from~\eqref{eq:hFdef2} (recall that we use the norm $\lVert\cdot\rVert_\infty$ for the classical continued fraction algorithm) that the mapping
\[
\hT_{\rG}:\, \Delta^1 {\times} \Delta^1 \to \Delta^1 {\times} \Delta^1, \quad \bigg(\binom{x}{1}, \binom{y}{1}\bigg) \mapsto 
\bigg(\binom{\frac1{x}-a}{1}, \binom{\frac{1}{a+y}}{1}\bigg) \quad \mbox{with}\ a=\Big\lfloor\frac1{x}\Big\rfloor
\]
is the starting point in our quest for a natural extension of~$T_{\rG}$. Because the second coordinates of the occurring vectors are irrelevant, we can regard~$\hT_{\rG}$ as a mapping from $(0,1)^2$ to itself defined by
\begin{equation}\label{eq:NEclassicalCF} 
\hT_{\rG}:\, (0,1)^2 \to (0,1)^2, \quad (x,y) \mapsto \big(\tfrac{1}{x}-a, \tfrac{1}{a+y}\big) \quad \mbox{with}\ a = \big\lfloor\tfrac1{x}\big\rfloor.
\end{equation}
With the methods developed in \cite{ArnouxSchmidt:19} (or directly from \cite{TI:81}) one can see that $\hT_{\rG}$ is bijective on $(0,1)^2$, thus $\hX_{\rG} = (0,1)^2$ and $((0,1)^2,\hT_{\rG}, \hA_{\rG})$ is a geometric natural extension of the classical continued fraction algorithm. According to the above construction, the invariant measure of the mapping~$\hT_{\rG}$ is given by $\mathrm{d}\hmu_{\rG}(\bx,\by)= \frac{\mathrm{d}x\mathrm{d}y}{(1+xy)^2}$. Positivity of $\hmu_{\rG}(\hX)$ is easy to see in this case. Pushing forward $\hmu_{\rG}$ to $(0,1)$ (i.e., integrating away $y$ from $0$ to $1$) and renormalizing yields the well-known Gauss measure $\mathrm{d}\mu_{\rG}(x)=\frac{1}{\log 2}\frac1{1+x}$ on $(0,1)$. Ergodicity of this Gauss measure can be proved by standard methods, see for instance~\cite[Theorem~3.7]{EinsiedlerWard:11}, or by using the strategy described above. 

According to \eqref{eq:hatFdef}, the natural extension of~$F_{\rG}$ is given by ($a= \lfloor1/x\rfloor$)
\[
\hF_{\rG}: \PP_{>0}^1 {\times} \PP_{>0}^1  \to \PP_{>0}^1 {\times} \PP_{>0}^1; \; 
\left(\begin{bmatrix}x\\1\end{bmatrix},\begin{bmatrix}y\\1\end{bmatrix}\right) \mapsto \left(\begin{pmatrix} -a & 1 \\ 1&0 \end{pmatrix}\begin{bmatrix}x\\1\end{bmatrix}, \begin{pmatrix} 0 & 1 \\ 1&a \end{pmatrix}\begin{bmatrix}y\\1\end{bmatrix} \right)
\]
It is easy to carry over the invariant measure to the projective version of the classical continued fraction algorithm with help of the chart~$\chi$. We refrain from giving the details here. 

The geometric natural extension~$\hT_{\rG}$ given in \eqref{eq:NEclassicalCF} goes back to~\cite{Nakada-Ito-Tanaka:77}.
\end{example}

\subsection{Positive range, Pisot condition, and periodic Pisot point}\label{sec:prpcppp}
The following property of a multidimensional continued fraction algorithm will be of importance for us as well.

\begin{definition}[Positive range] \label{def:prFC}\indx{range!positive}
Let $(X,F,A,\nu)$ be a multidimensional continued fraction algorithm
An element $\bx \in X$ is said  to have \emph{positive range} in $(X,F,A,\nu)$ if 
\[
\inf_{n\in\NN} \nu(F^n X_{[0,n)}(\bx)) > 0.
\]
\end{definition}

Note that measurability of the cylinder sets of $(X,F,A,\nu)$ holds because $A$ is measurable by assumption.

In all the classical algorithms we are aware of, almost every $\bx \in X$ has positive range, and we even have the stronger (global) \emph{finite range property}\indx{range!finite} (cf.~\cite{ItoYuri:87}) stating that the set of follower sets
\[
\cF = \{F^n X_{[0,n)}(\bx) \,:\, \bx \in X,\, n \in \NN\}
\]
is finite, where sets differing only on a set of $\nu$-measure zero are identified. Some of them are even \emph{full}\indx{continued fraction algorithm!full} in the sense that $\cF=\{X\}$ (again up to sets of $\nu$-measure zero). Since our theory works in the much weaker setting of  positive range, we develop it under this condition, This could turn out to be useful in other situations.

By the $F$-invariance of~$\nu$, the finite range property obviously implies positive range for a.e.\ $\bx \in X$ if we suppose that all cylinders satisfy $\nu(X_{[0,n)}(\bx)) > 0$; this will be the case for the algorithms considered in this monograph. If $(X,F,A,\nu)$ has the finite range property and $\bigcap_{n\in\NN}  X_{[0,n)}(\bx) = \{\bx\}$ for almost all $\bx \in X$, i.e., the set of cylinders $\{X(\bx) : \bx \in X\}$  is a generating partition, then $\{U \cap X(\bx) : U \in \cF,\, \bx \in X\}$ forms a (measurable countable)  {generating Markov partition} of $(X,F)$; see e.g.\  \cite[Theorem~10.1]{Yuri:95}. Most of the classical continued fraction algorithms (like Brun, Selmer, and Jacobi--Perron) are designed in a way that this Markov partition property holds. This Markov condition is not to be confused with the Markov partition we are studying here. Indeed, according to Definition~\ref{def:cf}, a multidimensional continued fraction algorithm generates sequences of matrices. If $D \subset \cM_d^\NN$ is the set of all possible sequences of matrices, then this ``global'' Markov partition states that the set~$D$ is a sofic shift; see again \cite[Theorem~10.1]{Yuri:95}. The Markov partitions we are considering in this monograph are ``local'' in the sense that they concern a single sequence of matrices~$\bM$ taken from the natural extension $\widehat{D}$ of~$D$. Regarding such a sequence as a sequence of toral automorphisms, our theorems associate nonstationary Markov partitions to almost all sequences $\bM \in \widehat{D}$.

In certain cases, a classical construction leads to a dual algorithm of a given multidimensional continued fraction algorithm $(X,F,A)$ (see e.g.~\cite[Definition~27]{Schweiger:00}): Indeed, if $(X,F,A)$ is a full algorithm, the image $\hF(\hX)$ admits a product structure in the sense that, for all $M$ contained in the image $A(X)$ of the cocycle, 
\[
\hF((X_M \times \PP^{d-1}) \cap \hX) = (X \times X^*_M) 
\]
with the cylinder set~$X_M$ from Definition~\ref{def:MFCcyl} and some $X_M^* \subset \PP^{d-1}$. This implies that the algorithm is full and enables us to define the \emph{dual algorithm}\indx{continued fraction algorithm!dual} $(X^*,F^*,A^*)$\notx{XFAstar}{$(X^*\hspace{-.25em},F^*\hspace{-.25em},A^*)$}{\hspace{.25em}dual continued fraction algorithm} on $X^* = \bigcup_{M\in A(X)} X_M^*$ (up to measure~0) by
\[
\begin{aligned}
A^*:\ & X^* \to \GL(d,\ZZ), & \by & \mapsto \tr{\!M} \quad \mbox{if}\ \by \in X^*_M, \\
F^*:\ & X^* \to X^*, & \by & \mapsto \tr{\!A}^*(\by)^{-1} \by,
\end{aligned}
\]
Thus $(X^*,F^*,A^*)$ is a multidimensional continued fraction algorithm. Moreover, we have, for almost all $(\bx, \by) \in \hX$,
\[
\hF^{-1}(\bx, \by) = (A^*(\by) \bx, \tr{\!A}^*(\by)^{-1}\by).
\]
In particular, $X\times X_M^* = \hF(X_M {\times} X^*)$ and, hence, $X_M^*$ is a cylinder set for the dual algorithm. 

The following definition concerns one of our main objects of interest.

\begin{definition}[Pisot condition for multidimensional continued fractions]
\indx{Pisot!condition!generic}
\label{def:MCF_Pisot}
Let $(X,F,A,\nu)$ be an ergodic continued fraction algorithm as defined in Definition~\ref{def:cf} and assume that the log-integrability condition 
\begin{equation}\label{eq:logintMCF}
\int_X  \log\lVert A \rVert_{\infty}  d{\nu} < \infty
\end{equation}
holds. We say that $(X,F,A,\nu)$ satisfies the Pisot condition if the (one-sided) Lyapunov exponents $\vartheta_1 \ge \cdots \ge \vartheta_d$ of~$A$ satisfy $\vartheta_1  > 0 > \vartheta_2$. (Note that for the definition of Lyapunov exponents only the future is relevant but observe Remark~\ref{rem:oneLy}.)
\end{definition}

From Example~\ref{ex:AR2}, we see that the Arnoux--Rauzy algorithm $(X_{\rAR},F_{\rAR},A_{\rAR})$ defined in Example~\ref{ex:AR3} satisfies the Pisot condition.

\begin{remark}[Lyapunov exponents and Pisot condition of the natural extension] \label{rem:1vs2}
If $(X,F,A,\nu)$ satisfies the Pisot condition then its natural extension $(\hX,\hF,\hA,\hnu)$, where $\nu = \pi_*\hnu$, satisfies the Pisot condition from Definition~\ref{def:gengenPisot}. We follow \cite[p.~99]{Viana:book}. For $f \in L_1(\nu)$  and $\widehat{f} = f \circ \pi_0: \hX \rightarrow \RR$ one therefore has $\int \widehat{f} d\hnu =\int f d{\nu}$. Thus, in this setting, the two-sided log-integrability condition \eqref{eq:logint} for $(\hX,\hF,\hA,\hnu)$ reads 
\[
\int_{\hX} \log \lVert \hA\rVert_{\infty}  \mathrm{d}\hnu= \int_{X} \log\lVert A\rVert_{\infty} \mathrm{d}\nu < \infty
\] 
and, hence, is the same as \eqref{eq:logintMCF}. 
We conclude that the Lyapunov exponents of $(\hX,\hF,\hA,\hnu)$ are equal to those of $(X,F,A,\nu)$.\footnote{This is true even in the general setting of Section~\ref{sec:oseledets} but we only need it for continued fraction algorithms.}
\end{remark}

We will need in our main results that any set $E \subset X$ with $\nu(E) > 0$ included in the follower set $F^n X_{[0,n)}(\bx)$ leads to an intersection $F^{-n} B \cap X_{[0,n)}(\bx)$ with positive measure. 
To this end, we will often assume that  
\begin{equation} \label{e:abscontinuity}
\nu(E) = 0 \quad \mbox{implies that} \quad \nu \circ F(E) = 0 \quad \mbox{for each measurable set}\ E.
\end{equation}
We will use the notation $\nu\circ F \ll \nu$ to indicate that \eqref{e:abscontinuity} holds. \notx{0ll}{$\ll$}{relation between measures similar to absolute continuity}
Although $\nu \circ F$ is usually not additive and therefore not a measure, we use this notation because \eqref{e:abscontinuity} is reminiscent of absolute continuity. 

Finally, the following notion of \emph{periodic Pisot point} will be of importance later.

\begin{definition}[Periodic Pisot point] \label{d:PisotpointFC}\indx{Pisot!periodic point}
For a multidimensional continued fraction algorithm $(X,F,A)$ with natural extension $(\hX,\hF,\hA)$, we say that $(\bx_0,\by_0) \in \hX$ is a \emph{periodic Pisot point} if there is $k \ge 1$ such that $\hF^k(\bx_0,\by_0)=(\bx_0,\by_0)$ and $A^{(k)}(\bx_0)$ is a Pisot matrix. 
\end{definition}

In other words, a periodic Pisot point is an element of $\hX$, on which the algorithm produces a periodic orbit $\bM\in \cM^\ZZ$ that converges exponentially. 

\begin{example}[The golden ratio gives rise to a periodic Pisot point] \label{ex:Pisotperiodic}
Let $\phi=\frac{1+\sqrt{5}}{2}$ be the golden ratio. If we apply the natural extension $(\hX,\hF,\hA)$ of the classical continued fraction algorithm $(X,F,A)$ constructed in Example~\ref{ex:classicalCF2} to the element $([\phi^{-1},1], [\phi^{-1},1]) \in \hX$, we get 
\[
\hF\bigg(\begin{bmatrix}\phi^{-1}\\1\end{bmatrix}, \begin{bmatrix}\phi^{-1}\\1\end{bmatrix} \bigg) =
\bigg(\begin{pmatrix}-1 & 1 \\1&0\end{pmatrix} \begin{bmatrix}\phi^{-1}\\1\end{bmatrix}, \begin{pmatrix}0 & 1 \\ 1&1\end{pmatrix} \begin{bmatrix}\phi^{-1}\\1\end{bmatrix} \bigg) =
\bigg(\begin{bmatrix}\phi^{-2}\\ \phi^{-1}\end{bmatrix}, \begin{bmatrix}1\\\phi\end{bmatrix} \bigg)
\]
and because $[1,\phi] = [\phi^{-1},1] = [\phi^{-2}, \phi^{-1}]$, the invariance property of Definition~\ref{d:PisotpointFC} is satisfied with $k=1$. Since 
\[
A\bigg(\begin{bmatrix}\phi^{-1}\\1\end{bmatrix}\bigg) = \begin{pmatrix} 0 & 1 \\ 1&1 \end{pmatrix}
\]
is a Pisot matrix (its eigenvalues are $\phi$ and $-1/\phi$), the element $([\phi^{-1},1], [\phi^{-1},1])$ is a periodic Pisot point.
\end{example}

\subsection{Multidimensional continued fraction algorithms and mapping families}
\label{subsec:MCFmappingfamily}
Let $(X,F,A)$ be a positive multidimensional continued fraction algorithm with natural extension $(\hX,\hF,\hA)$. We now show how this algorithm is related to the concepts defined in Chapter~\ref{chapter:matrix}. Each orbit of the natural extension $(\hX,\hF,\hA)$ produces a bi-infinite sequence of partial quotient matrices via the mapping
\begin{equation}\label{eq:bpsidef}
\bpsi:\, \hX \to \cM_d^{\ZZ}, \quad (\bx,\by) \mapsto(\tr{\!\hA}(\hF^n(\bx,\by)))_{n\in\ZZ}. \notx{psib}{$\bpsi(\cdot)$}{map returning sequence of partial quotient matrices}
\end{equation} 
This leads to the shift $(\bpsi(\hX),\Sigma)$, for which the diagram 
\begin{equation}\label{eq:diagpsis}
\begin{tikzcd}
\hX\arrow[r, "\hF"]\arrow[d,"\bpsi"] & \hX \arrow[d, "\bpsi"] \\
\bpsi(\hX) \arrow[r, "\Sigma"]& \bpsi(\hX) 
\end{tikzcd}
\end{equation}
commutes. Set $D = \bpsi(\hX)$. 

This commutative diagram can be used to carry over the notions of convergence defined for sequences of matrices in Definition~\ref{def:wsc} to the natural extension $(\hX,\hF,\hA)$ of a multidimensional continued fraction algorithm. More precisely, $(\hX,\hF,\hA)$ is \emph{weakly}\indx{continued fraction algorithm!convergent}, and \emph{exponentially convergent} (in the past or in the future) at $(\bx,\by)$ if the same is true for $\bpsi(\bx,\by)$ in the sense of Definition~\ref{def:wsc}. Also, the local Pisot condition (see Definition~\ref{def:genPisot}) can be carried over  to $(\hX,\hF,\hA)$ in this way. The local Pisot condition implies exponential convergence under mild conditions; see Proposition~\ref{prop:strongcv2sided}. 

Recall that $\PP^{d-1}_{\ge0}$  stands for the \emph{nonnegative cone} of~$\PP^{d-1}$. It turns out that the domain~$\hX$ of the natural extension $(\hX,\hF,\hA)$ of $(X,F,A)$ is often contained in the direct product $\PP^{d-1}_{\ge0} {\times} \PP^{d-1}_{\ge0}$ of nonnegative cones. Thus the following  relation of $(\hX,\hF,\hA)$ to the generalized right and left eigenvectors of $\bpsi(\bx,\by)$ is often useful.

\begin{lemma}\label{lem:MCF_eigen}
Assume that $(X,F,A)$ is a positive multiplicative continued fraction algorithm whose natural extension $(\hX,\hF,\hA)$ satisfies $\hX \subset X {\times} \PP^{d-1}_{\ge0}$ and is two-sided weakly convergent at $(\bx,\by)$. Then  $(\bx,\by) \in \PP_{\ge0}^d {\times} \PP_{\ge0}^d$ and any representative of $(\bx,\by)$ in $\RR_{\ge0}^d {\times} \RR_{\ge0^d}$ is a pair $(\bu,\bv)$, where $\bu$ and $\bv$ are generalized right and left eigenvectors of $\bpsi(\bx,\by)$.
\end{lemma}

\begin{proof}
Setting $(M_n)_{n\in\ZZ} = (\tr{\!\hA}(\hF^n(\bx ,\by)))_{n\in\ZZ}$, the diagram \eqref{eq:diagpsis} yields together with the definition of~$\hF$ in \eqref{eq:hatFdef} that 
\[
\hF^{n}(\bx,\by) = \begin{cases} ((M_{[0,n)})^{-1}\bx, \tr{(M_{[0,n)})}\by), & \mbox{if}\ n \ge 0, \\
(M_{[n,0)}\bx, \tr{(M_{[n,0)})}^{-1}\by), & \mbox{if}\ n < 0.\end{cases}
\]
Thus, setting $\hF^{n}(\bx,\by)=(\bx_n,\by_n)$, we have
\begin{equation}\label{eq:xyhatxy}
\bx =  \bigcap_{n\ge0}  M_{[0,n)}\bx_n  \quad\text{and}\quad
\by = \bigcap_{n<0} \,\tr\! M_{[n,0)}\by_n.
\end{equation}
For an element $\ba \in \PP^{d-1}_{\ge0}$ let $\tilde \ba$ be some fixed nonnegative representative of $\tilde\ba$ in $\RR^d$. Then, by two-sided weak convergence of $(M_n)_{n\in\ZZ}$, \eqref{eq:xyhatxy} implies that
\begin{equation}\label{eq:txtyeqivuv}
\begin{aligned}
\RR_{\ge0} \tilde\bx &= \bigcap_{n\ge0}  M_{[0,n)}\tilde\bx_n  = \bigcap_{n\ge0}  M_{[0,n)} \RR_{\ge0}^d \quad\text{and}\quad \\
\RR_{\ge0} \tilde\by &= \bigcap_{n<0} \,\tr\! M_{[n,0)}\tilde\by_n =\bigcap_{n<0} \,\tr\! M_{[n,0)} \RR_{\ge0}^d .
\end{aligned}
\end{equation}
On the other hand, by Definition~\ref{def:gea}, a generalized right eigenvector $\bu$ and a generalized left eigenvector $\bv$ of $(M_n)_{n\in\ZZ}$ satisfy
\begin{equation}\label{eq:txtyeqivuv2}
\RR_{\ge0}\bu =  \bigcap_{n\ge 0}  M_{[0,n)}\RR_{\ge0}^d  \quad\text{and}\quad
\RR_{\ge0}\bv =  \bigcap_{n<0} \tr{(M_{[n,0)})} \RR_{\ge0}^d.
\end{equation}
By comparing \eqref{eq:txtyeqivuv} with \eqref{eq:txtyeqivuv2}, we see that $(\bu,\bv)$ is a representative of $(\bx,\by)$ in $\RR_{\ge0}^d {\times} \RR_{\ge0}^d$.
\end{proof}

A~positive multidimensional continued fraction algorithm $(X,F,A)$ produces one-sided sequences of nonegative matrices to which, in case of weak convergence, we can associate a generalized right eigenvector. To get a left eigenvector as well, we need to work with the natural extension of the algorithm, as shown by Lemma~\ref{lem:MCF_eigen}. 

We turn to the metric theory. The following result gives a criterion of $\bpsi$ being a measurable conjugacy.

\begin{proposition}\label{prop:conjMCFD}
Let $(X,F,A,\nu)$ be a positive multidimensional continued fraction algorithm. If 
its natural extension $(\hX,\hF,\hA,\hnu)$ satisfies $\hX \subset X {\times} \PP^{d-1}_{\ge0}$ and is two-sided weakly convergent a.e., then the mapping~$\bpsi$ in \eqref{eq:diagpsis} is injective and, hence, $(\hX,\hF,\hA,\hnu)$ is measurably conjugate to the shift $(\bpsi(\hX),\Sigma,\bpsi_*\nu)$. 
\end{proposition}

\begin{proof}
The mapping $\bpsi$ is obviously measurable. Assume that it is not a.e.\ injective. 
Then a.e.\ weak convergence of $(\hX,\hF,\hA,\hnu)$ implies the existence of distinct elements $(\bx_1,\by_1),(\bx_2,\by_2)\in\hX$ at which $(\hX,\hF,\hA,\hnu)$ is weakly convergent, such that $\bpsi(\bx_1,\by_1) = \bpsi(\bx_2,\by_2)$.  Hence, by Lemma~\ref{lem:MCF_eigen}, the sequence $\bpsi(\bx_1,\by_1) = \bpsi(\bx_2,\by_2)$ would have two different pairs of right and left eigenvectors 
$(\bu,\bv)$ and $(\bu',\bv')$, where $\bu$ and $\bu'$ or $\bv$ and $\bv'$ are not collinear. This is a contradiction to two sided weak convergence.
\end{proof}

Under the conditions of Proposition~\ref{prop:conjMCFD}, the map which associates to each sequence~$\bM\in D=\bpsi(\hX)$ the pair of generalized eigenvectors $(\bu,\bv)$, is a.e.\ injective as well (in order to make this map well-defined, we need to normalize $\bu$ and $\bv$ in some way). Thus it allows to view the domain~$\hX$ as a set of pairs of generalized eigenvectors. Note that the conditions of Proposition~\ref{prop:conjMCFD} conditions are realized for most multidimensional continued fraction algorithms, see e.g.\ for the Brun algorithm described below in Section~\ref{sec:Brun}.

We also get that $(X,F,A,\nu)$ satisfies the Pisot condition from Definition~\ref{def:MCF_Pisot} if $(\bpsi(\hX),\Sigma,\bpsi_*\nu)$ satisfies the generic Pisot condition for~$A_{\mathrm{tr}}$ from Definition~\ref{def:gengenPisot}.  

Each sequence of matrices $\bM \in D$ produced by $(\hX,\hF,\hA)$ can be used to define the $\cM$-adic mapping family $(\TT,f_\bM)$ which performs the multidimensional continued fraction algorithm~$\hF$ for a suitable element of the domain~$\hX$. In the cocycle fiber $\{\bM\} {\times} \TT^d$, the orbits of~$\Phi_{\mathrm{tr}}$ are the $\cM$-adic mapping families of Definition~\ref{def:Madicmapping}. Thus $D$ parametrizes the set of mapping families and the shift on $D$ performs a renormalization procedure. 

For $\bM = (M_n)_{n\in\ZZ} \in D= \bpsi(\hX)$,  the linear cocycle~$A_{\mathrm{tr}}$ defined in Section~\ref{subsec:cocyclemat} acts on generalized eigenvectors by
\[
(\bu,\bv) \mapsto  (\tr{\!A}_{\mathrm{tr}}^{(n)}(\bM)^{-1}\bu, A_{\mathrm{tr}}^{(n)}(\bM)\bv) = (M_{[0,n)}^{-1}\bu, \tr{(\!M}_{[0,n)} \bv) = (\bu_n,\bv_n) \quad (n >0). 
\]
The vectors $\bu_n$ and~$\bv_n$ will be used in the sequel as the vectors whose entries are the measures of the bases and heights of the \emph{Rauzy boxes} (see Definition~\ref{def:rauzybox}), respectively. As we will see in Section~\ref{sec:markov}, these Rauzy boxes define nonstationary Markov partitions for $\cM$-adic mapping families and, hence, for multidimensional continued fraction algorithms.

\section{Brun continued fraction algorithm} \label{sec:Brun} 
In this section, we want to apply the theory we developed in Sections~\ref{sec:matrices}--\ref{sec:cfgentheory} to a multidimensional continued fraction algorithm that goes back to Viggo Brun. We start with the unordered version~$F_{\rU}$ of the Brun continued fraction algorithm in Section~\ref{subsec:BrunUnOrd}. We then consider its ordered version~$F_{\rB}$ in Section~\ref{subsec:BrunOrd}.  Both versions are additive, they hence rely on a finite set of matrices. The accelerated version~$F_{\rM}$ of the ordered Brun algorithm is handled in Section~\ref{subsec:BrunMul}; it is a multiplicative algorithm, i.e., it relies 
on an infinite set of matrices. Finally,  Section~\ref{sec:BrunPisot} investigates the relation between Brun's algorithm and the generic Pisot condition. A~general description for the Brun continued fraction algorithm is as follows: \emph{For a given vector it produces a new vector by subtracting its second largest coordinate from its largest one}. 

\subsection{Unordered Brun continued fraction algorithm} \label{subsec:BrunUnOrd} 
We first discuss the unordered Brun continued fraction algorithm $(X_{\rU},F_{\rU},A_{\rU})$, which is an additive multidimensional continued fraction algorithm. Since we will use it in later examples in the 2- and 3-dimensional case, we present this algorithm in arbitrary dimensions; for the 2-dimensional case we also refer to \cite{AL18}, some aspects of the higher-dimensional case are discussed in~\cite{Delecroix-Hejda-Steiner}. The algorithm $(X_{\rU},F_{\rU},A_{\rU})$ consists in subtracting the second largest coordinate of a given vector from the largest one, without reordering the coordinates. 

In the sequel we will use the notation $\cA=\{1,\ldots,d\}$.
We first define the sets and matrices needed for the definition of these algorithm. For $d\ge 3$, we define the subset
\[
X_\rU = \{ [w_1:w_2:\cdots:w_d] \in \PP_{\ge0}^{d-1} \,:\, w_i \neq w_j\  \mbox{for}\ i \neq j  \}.
\notx{U}{$\rU$}{object related to the unordered Brun algorithm}
\]
If, for all $i,j \in \cA$ with $i \neq j$, we set
\begin{equation}\label{eq:XunBrun}
X_{\rU,ij} = \{ [w_1:w_2:\cdots:w_d] \in \PP_{\ge0}^{d-1}: w_k \le  w_j \le w_i\ \text{for}\ k \in \cA \setminus \{i,j\} \}.
\end{equation}
then the collection $\{X_{\rU,ij}\colon i,j \in \cA,\, i \neq j\}$ is, apart from overlaps at the boundaries, a partition of $X_\rU$.

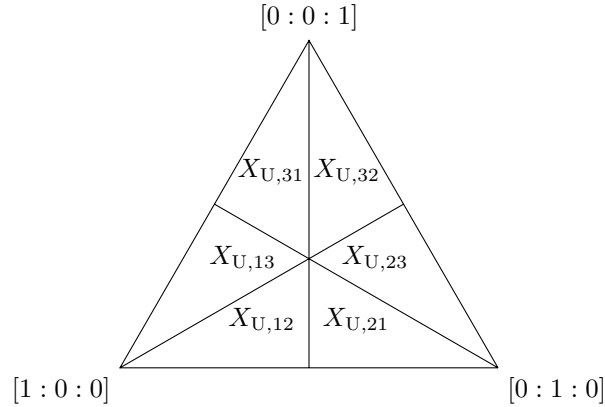
\begin{figure}[ht]
\begin{tikzpicture}[scale=2.5]
\coordinate [label={below right:$[0:1:0]$}] (A) at (1,0);
\coordinate [label={above :$[0:0:1]$}] (B) at (0, 1.732);
\coordinate [label={below left:$[1:0:0]$}] (C) at (-1, 0);
\draw (A) -- node[above] {} (B) -- node[right] {} (C) -- node[below] {} (A) (B)--(0,0) (-.5,1.732/2)--(A) (.5,1.732/2)--(C) ;
\draw (0.2,1.732*3/5)node[]{$X_{\rU,32}$} (0.35,1.732*1/3)node[]{$X_{\rU,23}$} (-0.35,1.732*1/3)node[]{$X_{\rU,13}$} (-0.25,1.732*1/7)node[]{$X_{\rU,12}$} (0.25,1.732*1/7)node[]{$X_{\rU,21}$} (-0.2,1.732*3/5)node[]{$X_{\rU,31}$};
\end{tikzpicture}
\caption{The topological partition $X_{\rU} = \bigcup_{i,j\in \{1,2,3\},\,i\neq j} X_{\rU,ij}$ for the unordered Brun algorithm in two dimensions.} \label{fig:brun}
\end{figure}
We also define the $(d{-}1)d$ different matrices $M_{\rU,ij}$, $i,j\in \cA$, $i \neq j$,  
\begin{equation}\label{eq:unorderedBrunMatrices}
M_{\rU,ij} = I_d + P_{ij} \qquad (i,j\in \cA,\, i \neq j),
\end{equation}
where $I_d$ is the $d{\times}d$ identity matrix and $P_{ij}$ is the $d{\times}d$ matrix having entries zero except a~$1$ at position $(i,j)$. We call the matrices $M_{\rU,ij}$ the \emph{unordered Brun matrices}\indx{matrix!Brun!unordered}\indx{Brun!matrix!unordered}.

Using this data, we can define our algorithm as follows.

\begin{definition}[Unordered Brun algorithm]
\indx{Brun!algorithm, map!unordered}\indx{continued fraction algorithm!Brun!unordered}
For $d \ge 3$ let
\[
A_{{\rU}}:\, X_{\rU} \to \GL(d,\ZZ);\quad \bx \mapsto \tr{\!M}_{\rU,ij} \qquad \text{if}\ \bx \in X_{\rU,ij} \qquad (i,j\in \cA,\, i \neq j)
\]
with $X_{\rU,ij}$ as in \eqref{eq:XunBrun} and $M_{\rU,ij}$ as in \eqref{eq:unorderedBrunMatrices}.
This defines the \emph{unordered Brun continued fraction map}
\begin{equation} \label{CFalgoBU}
F_{\rU}:\, X_{\rU} \to X_{\rU}; \quad \bx \mapsto \tr{\!A}_{\rU}(\bx)^{-1} \bx.
\end{equation}
Then the multidimensional continued fraction algorithm $(X_{{\rU}},F_{{\rU}},A_{{\rU}})$ is called the $(d{-}1)$-dimensional \emph{unordered Brun continued fraction algorithm} or the \emph{unordered Brun algorithm}, for short.
\end{definition}

The map~$F_{\rU}$ lives on $X_{\rU} = \bigcup_{i,j\in \cA,\, i\neq j} X_{\rU,ij}$; see Figure~\ref{fig:brun} for the case $d=3$. It is defined as in \eqref{eq:CFalgo} by the cocycle\footnote{Note that the definition is ambigous in the small overlaps of the sets $X_{\rU,ij}$. This can easily be avoided by redefining $X_{\rU,ij}$ on their boundaries. Since we are mainly interested in metric properties of the algorithm, we do not care how $A_{\rU}$ is defined on these overlaps; this also holds for other algorithms.}
Note that $X_{\rU,ij}=X_{^t\!{M}_{\rU,{ij}}}$ are exactly the cylinder sets of length one of the unordered Brun algorithm; see Definition~\ref{def:MFCcyl}. This algorithm is \emph{symmetric} in the sense that $F_{\rU}(\zeta \bx) = \zeta F_{\rU}(\bx)$, where the permutation $\zeta \in \mathfrak{S}_d$\notx{sdy}{$\mathfrak{S}_d$}{symmetric group} acts on $\bx \in \RR^d$ naturally by $\zeta(x_1,\dots,x_d) = (x_{\zeta(1)},\dots,x_{\zeta(d)})$. Moreover, since $F_{\rU}^n(X_{[0,n)}(\bx))$ is a union of sets of the form $X_{\rU,ij}$ for each $n\in \NN$ and each $\bx\in X_{\rU}$, the unordered Brun algorithm has \emph{finite range}.  

According to \eqref{eq:CFT}, the \emph{affine version} of~$(X_{\rU},F_{\rU},A_{\rU})$ is given by
$(\Delta_{\rU},T_{\rU},A_{\rU})$, where $\Delta_{\rU} = \Delta^{d-1} \cap [0,1]^d$ and
\begin{equation}\label{eq:CFTBU}
T_{\rU}:\, \Delta_{\rU} \to \Delta_{\rU}, \quad \bx \mapsto \frac{\tr{\!A}(\bx)^{-1} \bx}{\lVert\tr{\!A}(\bx)^{-1} \bx\rVert};
\end{equation}
recall that $\Delta_{\rU}=\chi(X_{\rU})$, with $\chi$ being the chart defined in~\eqref{eq:chidef}.

Our first goal is to construct a natural extension and a Gauss measure for the unordered Brun algorithm. To do this, we proceed by the strategy outlined in Section~\ref{sec:natex}. 
In the following lemma we set up the natural extension of $(\Delta_{\rU},T_{\rU},A_{\rU})$.

\begin{lemma}\label{eq:natexBU}
For $i,j \in \cA$, $i \neq j$, let 
\[
\begin{aligned}
\Delta_{ij} &= \{ (x_1,x_2,\dots,x_d) \in\Delta_{\rU} \,:\, x_k \le x_j \le  x_i\ \text{for all}\ k \in \cA \setminus \{i,j\} \}, \\
\Delta_{ij}^* &= \{ (y_1,y_2,\dots,y_d) \in\Delta_{\rU} \,:\, y_k \le y_i\ \text{for all}\ k\in \cA \setminus \{i,j\} \}.
\end{aligned}
\]
Then for 
\begin{equation}\label{eq:DeltaUDef}
\hDelta_{\rU} = \bigcup_{i,j\in \cA,\, i \neq j} \Delta_{ij} \times \Delta_{ij}^*,
\end{equation}
the mapping
\[
\hT_{\rU}:\,  \hDelta_{\rU} \to \hDelta_{\rU},\quad (\bx, \by) \mapsto \left(\frac{\tr{\!A_{\rU}}^{-1}(\bx)\bx}{\lVert \tr{\!A_{\rU}}^{-1} (\bx)\bx \rVert} , 
\frac{A_{\rU}(\bx)\by}{\lVert  A_{\rU}(\bx)\by \rVert}\right)
\]
is bijective almost everywhere w.r.t.\ the Lebesgue measure on $\hDelta_{\rU}$. 
\end{lemma}

\begin{proof}
In view of Lemma~\ref{lem:hTsurbi}, it suffices to prove  the surjectivity of~$\hT$. If $\bx = (x_1,\dots,x_d)\in \Delta_{ij}$ and $\by = (y_1,\dots,y_d) \in \Delta_{ij}^*$ holds for some $i,j \in \cA$, $i \neq j$, then
\[
\begin{aligned}
\tr{\!A_{\rU}}(\bx)^{-1} \bx &= M_{\rU,ij}^{-1} \bx = (x_1,\dots,x_{i-1},x_i-x_j,x_{i+1},\dots,x_d) \quad \text{and}  \\
A_{\rU}(\bx) \by &= \tr{\!M}_{\rU,ij} \by = (y_1,\dots,y_{j-1},y_j+y_i,y_{j+1},\dots,y_d).
\end{aligned}
\]
This immediately implies that 
\begin{equation}\label{eq:Deltatrafo}
\tr{\!A_{\rU}}(\Delta_{ij})^{-1} \Delta_{ij} \subseteq \Delta_{ji}^* \quad\text{and}\quad
A_{\rU}(\Delta_{ij})\Delta_{ij}^* \subseteq \Delta_{ji}
\end{equation}
(here and in the sequel we write $A_{\rU}(\Delta_{ij})$ for the matrix $A_{\rU}(\bx)$ with $\bx \in \Delta_{ij}$; this is a slight abuse of notation).
Hence, since $\tr{\!M_{U,ij}}=M_{U,ji}$,  one has 
\[
A_{\rU}(\Delta_{ij})^{-1} \Delta_{ji} \subseteq \Delta_{ij}^* \quad\text{and}\quad
\tr{\!A_{\rU}}(\Delta_{ij})\Delta_{ji}^* \subseteq \Delta_{ij}.
\]
However, together with \eqref{eq:Deltatrafo}, this yields
\[
\begin{aligned}
\Delta_{ij} & = \tr{\!A}_{\rU}(\Delta_{ij})\, \tr{\!A}_{\rU}(\Delta_{ij})^{-1} \Delta_{ij} \subseteq \tr{\!A}_{\rU}(\Delta_{ij})\, \Delta_{ji}^* \subseteq \Delta_{ij}  \quad\text{and}\quad \\
\Delta_{ij}^* & = A_{\rU}(\Delta_{ij})^{-1}\, A_{\rU}(\Delta_{ij}) \Delta_{ij}^* \subseteq  A_{\rU}(\Delta_{ij})^{-1} \Delta_{ji} \subseteq \Delta_{ij}^*.
\end{aligned}
\]
Thus in \eqref{eq:Deltatrafo} we can replace ``$\subseteq$'' by ``$=$'' to get
\[
\tr{\!A}_{\rU}(\Delta_{ij})^{-1} \Delta_{ij} = \Delta_{ji}^* \quad \mbox{and} \quad
A_{\rU}(\Delta_{ij}) \Delta_{ij}^* = \Delta_{ji}.
\]
We therefore gain $\hT (\hDelta_{\rU}) = \bigcup_{i,j\in \cA,\, i \neq j} \Delta_{ij}^* {\times} \Delta_{ij}$. 
It easily follows from the definitions that $\Delta_{ij}^*= \Delta_{ji} \cup\bigcup_{k\in\cA,k\not=i}\Delta_{ik}$ (with measure disjoint unions) for all $i,j\in \cA$, $i\not=j$. This implies that $\Delta_{k\ell} \subset \Delta_{ij}^*$ if and only if $\Delta_{ij} \subset \Delta_{k\ell}^*$ $(i,j,k,\ell\in \cA,\, i \neq j,\, k \neq \ell)$. Using these facts we immediately gain
\[
\bigcup_{i,j\in \cA,\, i \neq j} \Delta_{ij} \times \Delta_{ij}^* = \bigcup_{i,j\in \cA,\, i \neq j} \Delta_{ij}^* \times \Delta_{ij},
\]
and surjectivity of~$\hT$ follows.
\end{proof}

Recall from Section~\ref{sec:natex} that $\mathrm{d} \hmu(\bx,\by) = \frac{\mathrm{d}\lambda(\bx)\cdot \mathrm{d}\lambda{(\by)}}{\langle \bx,\by \rangle^d}$ is a measure on $\hDelta_{\rU}$.
By Lemma~\ref{eq:natexBU}, $(\hDelta_{\rU},\hT_{\rU},\hA_{\rU}, \hmu)$  is a natural extension of $(\Delta_{\rU},T_{\rU},A_{\rU},\mu_{\rU})$, where $\mu_{\rU} = \pi_*\hmu$ with $\pi: \hDelta\to\Delta$, $(\bx,\by) \mapsto \bx$ is the natural projection. This can easily be transferred to~$\hX$ via $\chi {\times} \chi$, and we exhibit a natural extension $(\hX_{\rU},\hF_{\rU},\hA_{\rU},\hnu)$ of $(X_{\rU},F_{\rU},A_{\rU},\nu_{\rU})$. We have that $\hDelta_{\rU}$ has positive Lebesgue measure by \eqref{eq:DeltaUDef}, and, hence, $\hnu_{\rU}(\hX_{\rU}) > 0$ (see also Remark~\ref{rem:nu>0}). 

Our next aim is to give the invariant measure~$\mu$ explicitly.

\begin{lemma}\label{lem:unorderedinvariant}
If we use the norm $\lVert\cdot\rVert_\infty$ in the definition of~$\Delta$, then the invariant measure~$\mu_{\rU}$ of the affine version of the unordered Brun algorithm $(T_{\rU},\Delta_{\rU},A_{\rU},\mu_{\rU})$ is absolutely continuous w.r.t.\ the (piecewise) Lebesgue measure on~$\Delta_{\rU}$ and has density
\[
\frac{1}{x_1\cdots x_{i-1}x_{i+1}\cdots x_d} \sum_{S\subset\cA\setminus\{i,j\}} \frac{(-1)^{|S|}}{x_i+\sum_{k\in S}x_k}
\]
for $(x_1,\dots,x_d) \in \Delta_{ij}$. This can be transfered to an invariant measure~$\nu_{\rU}$ of~$F_{\rU}$ via the bijection $\chi: X_{\rU} \to \Delta_{\rU}$. The measures $\mu_{\rU}$ and~$\nu_{\rU}$ are ergodic.
\end{lemma}

\begin{proof}
To derive the required density, we have to integrate the density $\frac1{\langle \bx,\by \rangle^d}$ of the invariant measure $\hmu$ of $\hT$ over $\by$. Thus we have to compute
\[
I = \int_{\Delta_{ij}^*} \frac{\mathrm{d}\lambda{(\by)}}{\langle \bx,\by \rangle^d}.
\]
W.l.o.g., we can consider the case $(i,j) = (1,2)$. From the definition of $\Delta_{12}^*$ we see that a vector $(y_1,\ldots,y_d) \in \Delta_{12}^*$ satisfies $y_1=1$ or $y_2=1$. According to these two alternatives, we can split $I$ into a sum and gain
\[
\begin{aligned}
I & = \underbrace{\int_{y_2=0}^1 \hspace{-1em} \cdots \hspace{-.1em} \int_{y_d=0}^1 \frac{\mathrm{d}y_2\cdots \mathrm{d}y_d}{(x_1+x_2y_2+\cdots+x_dy_d)^d}}_{\displaystyle I^{(1)}_d(x_1)} \\
& \quad + \underbrace{\int_{y_1=0}^1 \int_{y_3=0}^{y_1} \hspace{-1em} \cdots \hspace{-.1em} \int_{y_d=0}^{y_1} \frac{\mathrm{d}y_1\mathrm{d}y_3\cdots \mathrm{d}y_d}{(x_1y_1+x_2+x_3y_3+\cdots+x_dy_d)^d}}_{\displaystyle I^{(2)}_d(x_1)}.
\end{aligned}
\]
By evaluating the innermost integral of~$I^{(\ell)}_d(x_1)$, $\ell \in \{1,2\}$, we gain
\[
I_d^{(\ell)}(x_1) = \frac{1}{(d-1)x_d} (I_{d-1}^{(\ell)}(x_1)-I_{d-1}^{(\ell)}(x_1+x_d)) \qquad (\ell \in \{1,2\}).
\]
Iterating this procedure for $d{-}2$ times yields by induction that
\begin{equation}\label{eq:id12}
I_d^{(\ell)}(x_1) = \frac{1}{(d{-}1)!\,x_3\cdots x_d} \sum_{S\subset\cA\setminus\{1,2\}} (-1)^{|S|} I_2^{(\ell)}\bigg(x_1+ \sum_{k\in S} x_k\bigg)
\quad (\ell\in\{1,2\}).
\end{equation}
Because 
\[
\begin{aligned}
I_2^{(1)}(z) &= \int_{y_2=0}^1 \frac{\mathrm{d}y_2}{(z+x_2y_2)^2}=\frac{1}{z(x_2+z)}, \\
I_2^{(2)}(z) &= \int_{y_1=0}^1 \frac{\mathrm{d}y_1}{(z y_1+x_2)^2}=\frac{1}{x_2(x_2+z)},
\end{aligned}
\]
and $\frac{1}{z(x_2+z)} + \frac{1}{x_2(x_2+z)} = \frac{1}{x_2z}$, the result on the density follows by inserting this in \eqref{eq:id12}. Note that the constant factor $1/(d{-}1)!$ is immaterial because the invariant measure is only defined up to a constant.

Ergodicity of $\mu_{\rU}$ and~$\nu_{\rU}$ is established in \cite[Proposition~28]{AL18}.
\end{proof}

Let $\bpsi$ be defined as in \eqref{eq:bpsidef}.
Because the collection $\{X_{\rU,ij} \colon i,j\in\cA,\, i\not=j\}$ form a measurable partition of $X_{\rU}$, the diagram \eqref{eq:diagpsis} in Section~\ref{subsec:MCFmappingfamily} defines a measurable conjugacy between $(\hX_{\rU},\hF_{\rU},\hA_{\rU},\hnu)$ and the shift $(\bpsi(\hX),\Sigma, \bpsi_*\hnu)$. It is known that the sequences in $\bpsi(\hX)$ are characterized by a Markov condition, more precisely, we have
\begin{equation}\label{eq:unBrunMarkov}
\begin{aligned}
\bpsi(\hX_{\rU}) & = \{(M_n)_{n\in\ZZ}\,:\, (M_n,M_{n+1}) = (M_{\rU,ij},M_{\rU,ij})  \text{ for } i,j \in \cA,\, i \neq j,  \\
& \hspace{2em} \text{or}\ (M_n,M_{n+1}) = (M_{\rU,ij},M_{\rU,jk}) \text{ for } i,j,k \in \cA, \, i\ne j \ne k\};
\end{aligned}
\end{equation}
see\ \cite[Section~3]{Delecroix-Hejda-Steiner}. Via this conjugacy, we attach to each $(\bx,\by)\in \hX_{\rU}$ the mapping family $(\TT,f_\bM)$ with $\bM=\bpsi(\bx,\by)$.

As the following result shows, the 2-dimensional unordered Brun algorithm furnishes examples for the theory we developed in Chapters~\ref{sec:matrices} and~\ref{sec:metricmat}. Recall that $\chi: \PP^{d-1} \to \Delta^{d-1}$ assigns to each element of~$\PP^{d-1}$ a representative in~$\Delta^{d-1}$. Note that the set of parameters in $X_{\rU}$ described in Proposition~\ref{prop:brunpisot:2} has zero measure. Indeed, this statement holds for concrete sequences of matrices. For according metric results we refer to Section~\ref{sec:BrunPisot}.

\begin{proposition}\label{prop:brunpisot:2}
Let $(\hX_{\rU},\hF_{\rU},\hA_{\rU},\hnu_{\rU})$ be the natural extension of the 2-dimensional ordered Brun algorithm and let $\bpsi$ be as in \eqref{eq:bpsidef}. Suppose that for $\bpsi(\bx,\by) = (M_n)_{n\in\ZZ}$ there exists $h \in \NN$ such that 
\begin{equation}\label{eq:BrunAllMatrices}
\{M_n,\dots,M_{n+h}\} = \{M_{\rU,ij}\,:\, i,j\in \{1,2,3\},\, i \neq j\}
\end{equation}
for each $n \in \ZZ.$\footnote{Note that this condition is what we called ``bounded strong partial quotients'' in Example~\ref{ex:AR1}.}
Then the following assertions hold.
\begin{itemize}
\item $\bpsi(\bx,\by)$ satisfies the two-sided Pisot condition.
\item $\bpsi(\bx,\by)$ is algebraically irreducible.
\item $\bpsi(\bx,\by)$ converges strongly to the generalized right eigenvector $\bu = \chi(\bx)$ in the future and to the generalized left eigenvector $\bv = \chi(\by)$ in the past. The vector $\chi(\bx)$ has rationally independent coordinates.
\item The mapping family $(\TT,f_{\bpsi(\bx,\by)})$ associated to $(\bx,\by)$ is eventually Anosov for the splitting
\[
\coprod_{n\in\ZZ} \RR \bu_n \oplus \bv_n^\perp.
\]
\end{itemize}
\end{proposition}

\begin{proof}
Primitivity in the future and in the past holds because \eqref{eq:BrunAllMatrices} implies that $M_{[n,n+h)}$ is positive for each $n \in \ZZ$. It follows from the proof of \cite[Theorem~2]{Delecroix-Hejda-Steiner} that $\bM$ satisfies the two-sided Pisot condition. Since there are only finitely many (namely, six) unordered Brun matrices, the growth condition $\lim_{n\to\pm\infty}\frac{1}{n} \log \lVert M_n\rVert = 0$ is again trivially satisfied, and Proposition~\ref{prop:domev}, Theorem~\ref{th:matrixPisot}, Corollary~\ref{cor:anosov}, and Lemma~\ref{lem:MCF_eigen} imply the lemma.
\end{proof}

\subsection{Ordered Brun continued fraction algorithm} \label{subsec:BrunOrd} 
Let us now turn to the ordered version of the Brun continued fraction algorithm. This is the classical version of this algorithm. Indeed, the 2-dimensional case of the ordered Brun algorithm goes back to~\cite{Brun19,Brun20,BRUN}. Let $d\ge 3$ be fixed. We obtain the ordered algorithm from the unordered one by taking the quotient by the symmetric group~$\mathfrak{S}_d$ of $d$ elements, so the ordered algorithm is a $d!$-to-$1$ factor of the unordered one. This fact will allow us to immediately derive its invariant measure form the one of the unordered version.

For $d\ge 3$, we define 
\begin{equation}\label{eq:BOX}
X_{\rB} = \{ [w_1:w_2:\cdots:w_d] \in \PP_{\ge0}^{d-1} \,:\, w_1 \le w_2 \le \cdots \le w_d \},
\notx{B}{$\rB$}{object related to the ordered Brun algorithm}
\end{equation}
and its affine version will be defined on 
\[
\Delta_{\rB} := \{(x_1,\ldots x_{d-1},1) \in [0,1]^d \,:\, x_1 \le x_2 \le \cdots \le x_{d-1}\}.
\]

Let $\mathrm{ord}: \PP^{d-1} \to \PP^{d-1}$ be the \emph{coordinate ordering map} that assigns to each $\bw \in \PP^{d-1}$ the element $\mathrm{ord}(\bw)$ of~$\PP^{d-1}$\notx{ord}{$\mathrm{ord}(\cdot)$}{coordinate ordering map} that emerges from~$\bw$ by ordering its coordinates ascendingly. It will cause no confusion that we use the same name also for the coordinate ordering map on~$\RR^d$. Because $\zeta \circ F_{\rU} = F_{\rU} \circ \zeta$ and $\zeta \circ T_{\rU} = T_{\rU} \circ \zeta$ holds for each $\zeta \in \mathfrak{S}_d$, the maps $F_{\rB}$ and~$T_{\rB}$ are well-defined by the following commutative diagrams.
\begin{equation}\label{eq:diagBrun}
\begin{tikzcd}
X_{\rU} \arrow[r, "F_{\rU}"]\arrow[d,"\mathrm{ord}"] &X_{\rU} \arrow[d, "\mathrm{ord}"] \\
X_{\rB} \arrow[r, "F_{\rB}"]& X_{\rB}
\end{tikzcd}
\qquad\qquad
\begin{tikzcd}
\Delta_{\rU} \arrow[r, "T_{\rU}"]\arrow[d,"\mathrm{ord}"] &\Delta_{\rU} \arrow[d, "\mathrm{ord}"] \\
\Delta_{\rB} \arrow[r, "T_{\rB}"]& \Delta_{\rB}
\end{tikzcd}
\end{equation}
Note that $F_{\rB}$ subtracts the second largest coordinate from the largest one and reorders the result, i.e., $F_{\rB}$ performs the operation
\[
[w_1 : \cdots : w_d] \mapsto \mathrm{ord}[w_1 : \cdots : w_{d-1} : w_d{-}w_{d-1}].
\]
Since $w_d{-}w_{d-1}$ is the only element that is out of order after the subtraction, this element has to be inserted in the appropriate position. Because there are $d$ available positions, there are $d$ different \emph{ordered Brun matrices}\indx{Brun!matrix!ordered}\indx{matrix!Brun!ordered} that correspond to the definition of the mapping~$F_{\rB}$. These $d{\times}d$ matrices are defined by
\begin{equation}\label{eq:orderedbrunmatrices}
\begin{aligned}
M_{\rB,k} & = 
\begin{pmatrix} 
1&&&&&&&\\
&\ddots&&&&&&\\
&&1&&&&&\\
&&&0&1&&&\\
&&&&\ddots&\ddots&&\\
&&&&&0&1&\\
&&&1&&&1&
\end{pmatrix} \text{$\leftarrow$ $k$-th row} \qquad(k\in \cA\setminus\{d\})
\\
&\hskip2.5cm \uparrow \\
&\hskip2.2cm \text{$k$-th column}
\end{aligned}
\end{equation}
and 
\begin{equation}\label{eq:orderedbrunmatrices2}
M_{\rB,d}= 
\begin{pmatrix} 
1&&&\\
&\ddots&&\\
&&1&\\
&&1&1
\end{pmatrix}.
\end{equation}
For $d=3$, one gets for instance that
\begin{equation}\label{eq:brunmatrices}
M_{\rB,1} = \begin{pmatrix}0&1&0\\0&0&1\\1&0&1 \end{pmatrix},\quad
M_{\rB,2} = \begin{pmatrix}1&0&0\\0&0&1\\0&1&1 \end{pmatrix},\quad
M_{\rB,3} = \begin{pmatrix}1&0&0\\0&1&0\\0&1&1 \end{pmatrix}.
\end{equation}
One checks immediately that the $d$ sets 
\begin{equation}\label{eq:XpartOrdered}
X_{\rB,k} = M_{\rB,k} X_{\rB} \subset X_{\rB} \qquad (k\in \cA)
\end{equation}
form a partition of~$X_{\rB}$ up to a set of measure zero; see Figure~\ref{fig:brun2} for the case $d=3$. 

\begin{figure}[ht]
\begin{tikzpicture}[scale=3.5]
\coordinate [label={below :$[1:1:1]$}] (D) at (0,1.732/3);
\coordinate [label={above:$[0:0:1]$}] (E) at (0,1.732);
\coordinate [label={right:$[0:1:1]$}] (F) at (.5,1.732/2);
\coordinate [label={left:$[1:1:2]$}] (G) at (0,1.732/2);
\coordinate [label={right:$[0:1:2]$}] (H) at (1/3,2/3*1.732);
\draw (D)--(E)--(F)--(D) (G)--(F) (G)--(H);
\node at (.15,1.2){$X_{\rB,1}$};
\node at (.125,.75){$X_{\rB,3}$};
\node at (.275,.95){$X_{\rB,2}$};
\end{tikzpicture}
\caption{The topological partition $X_{\rB} = X_{\rB,1} \cup X_{\rB,2} \cup X_{\rB,3}$ for the 2-dimensional ordered Brun algorithm.} \label{fig:brun2}
\end{figure}
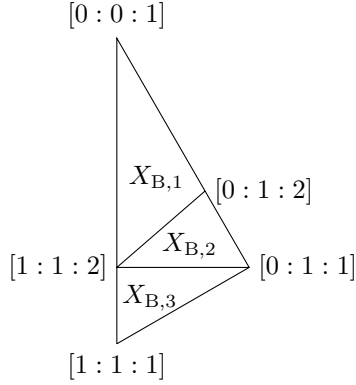

By these considerations, the mapping $F_{\rB}$ defined in \eqref{eq:diagBrun} fits into the general definition of a multidimensional continued fraction algorithm given in Definition~\ref{def:cf}. In particular, we may rewrite $F_{\rB}$ by using the cocycle $A_{\rB}$ as in the following definition.

\begin{definition}[Ordered Brun algorithm]
\indx{Brun!algorithm, map!ordered}\indx{continued fraction algorithm!Brun!ordered}
For $d \ge 3$ let\footnote{As before, there is ambiguity in the definition of $A_{\rB}$ on the overlap of the sets $X_{\rB,k}$ and we can define $A_{\rB}$ on these overlaps by any of the possible choices.}
\[
A_{\rB}:\, X_{\rB} \to \GL(d,\ZZ), \quad \bx \mapsto  \tr{\!M}_{\rB,k} \qquad \text{if}\ \bx \in X_{\rB,k} \qquad (k \in \cA).
\notx{B}{$\rB$}{object related to the ordered Brun algorithm}
\]
with $X_{\rB,k} $ as in \eqref{eq:XpartOrdered} and $M_{\rB,ij}$ as in \eqref{eq:orderedbrunmatrices} and \eqref{eq:orderedbrunmatrices2}. This defines the \emph{ordered Brun continued fraction~map}
\begin{equation} \label{CFalgoB}
F_{\rB}:\, X_{\rB} \to X_{\rB}; \quad \bx \mapsto \tr{\!A}_{\rB}(\bx)^{-1} \bx.
\end{equation}
Then the multidimensional continued fraction algorithm $(X_{{\rB}},F_{{\rB}},A_{{\rB}})$ is called the $(d{-}1)$-dimensional \emph{ordered Brun continued fraction algorithm} or the \emph{ordered Brun algorithm}, for short.
\end{definition}

This algorithm is additive because it is defined by a finite family of matrices. Again, $X_{\rB,k} = X_{\tr{\!M}_{\rB,k}}$, $k \in \cA$, are the cylinder sets of the ordered Brun algorithm; see Definition~\ref{def:MFCcyl}. Moreover, because $F_{\rB}(X_{\rB,k})=X_{\rB}$ holds for each $k\in\cA$, the ordered Brun algorithm is a \emph{full} algorithm.

The \emph{affine version} $(\Delta_{\rB},T_{\rB},A_{\rB})$ of the ordered Brun algorithm can now be defined according to \eqref{eq:CFT}.  In particular, since we used the norm $\lVert\cdot\rVert_\infty$ in the definition of~$\Delta_{\rB}$, the chart $\chi: X_{\rB} \to \Delta^{d-1}$ yields that 
\[
T_{\rB}\big(\tfrac{w_1}{w_d},\dots,\tfrac{w_{d-1}}{w_d}\big) = \big(\tfrac{w_1'}{w_d'},\dots,\tfrac{w_{d-1}'}{w_d'}\big) \ \Longleftrightarrow\
F_{\rB}(w_1,\dots,w_d) = (w_1',\dots,w_d').
\]
If we identify $\Delta_{\rB}$ with $\widetilde{\Delta}_2= \{(x_1,x_2) \in [0,1]^2 \,:\, x_1\le x_2\}$ in the 2-dimensional case $d=3$, the mapping $T_{\rB}: \widetilde{\Delta}_2 \to \widetilde{\Delta}_2$ is defined explicitly by 
\begin{equation}\label{eq:brunmap}
T_{\rB}:\, (x_1,x_2) \mapsto 
\begin{cases}
\big(\frac{x_1}{1-x_2},\frac{x_2}{1-x_2}\big) & \hbox{if}\ x_2 \le \frac{1}{2}, \\[.5ex]
\big(\frac{x_1}{x_2},\frac{1-x_2}{x_2}\big) & \hbox{if}\ \frac{1}{2} \le x_2 \le 1{-}x_1, \\[.5ex]
\big(\frac{1-x_2}{x_2},\frac{x_1}{x_2}\big) & \hbox{if}\ 1{-}x_1 \le x_2.
\end{cases}
\end{equation}
This is the classical continued fraction algorithm of Brun going back to \cite{Brun19,Brun20,BRUN}.

Let us define a set~$\hX_{\rB}$ for the geometric version of the natural extension of the ordered Brun algorithm. Let 
\begin{equation}\label{eq:BOXstar}
X_{\rB}^* = \{ [w_1 : \cdots : w_d] \in \PP_{\ge0}^{d-1} \,:\,  w_1,\dots,w_{d-2} \le w_d\} 
\end{equation}
and define the subsets 
\[
\begin{aligned}
X _{\rB,k}^* & = \{ [w_1 : \cdots : w_d] \in \PP_{\ge0}^{d-1} \,:\,
w_1,\dots,w_{d-1} \le w_k \le w_d\} \quad (1 \le k \le d{-}2), \\
X_{\rB,d-1}^* & = \{ [w_1 : \cdots : w_d] \in \PP_{\ge0}^{d-1} \,:\, w_1,\dots,w_{d-2} \le w_{d-1} \le w_d\}, \\
X_{\rB,d}^* & = \{ [w_1 : \cdots : w_d] \in \PP_{\ge0}^{d-1} \,:\, w_1,\dots, w_{d-2} \le w_d\le  w_{d-1}\}.
\end{aligned}
\]
For each $k\in \cA$, one readily checks that  $X_{\rB,k}^* = \tr{\!M}_{\rB,k} X_{\rB}^*$. Moreover, the collection $\{X_{\rB,k}^* : k\in \cA\}$ forms a topological partition of~$X_{\rB}^*$, and the map~$\hF_{\rB}$ defined  on $X_{\rB} {\times} X_{\rB}^*$ by $\hF_{\rB}(\bx,\by) = (\tr{\!A}_{\rB}(\bx)^{-1}\bx,A_{\rB}(\bx)\by)$ sends $X_{\rB,k} {\times} X_{\rB}^*$ to $X_{\rB} {\times} X_{\rB,k}^*$, so $\hF_{\rB}$ is a one-to-one map from $\hX_{\rB} = X_{\rB} {\times} X_{\rB}^*$ to itself. 
Thus, with the appropriate measure~$\hnu_{\rB}$ chosen according to Section~\ref{sec:natex}, we have found a natural extension $(\hX_{\rB},\hF_{\rB},\hA_{\rB},\hnu_{\rB})$ of $(X_{\rB},F_{\rB},A_{\rB},\nu_{\rB})$. 

However, we do not need this natural extension in order to get the Gauss measure~$\nu_{\rB}$ for the ordered Brun algorithm. Indeed, from Lemma~\ref{lem:unorderedinvariant}, we immediately obtain the density of the invariant measure of the ordered Brun algorithm; compare this with the equivalent formula for the density in \cite{Arnoux-Nogueira}.

\begin{lemma}\label{lem:orderedinvariant}
The invariant measure~$\mu_{\rB}$ of the affine version of the ordered Brun algorithm $(T_{\rB},\Delta_{\rB},A_{\rB},\mu_{\rB})$ is absolutely continuous w.r.t.\ the (piecewise) Lebesgue measure on~$\Delta_{\rB}$ and has density\footnote{We checked for small values of $d$ that this formula coincides with the one in \cite[p.~661]{Arnoux-Nogueira}.}
\begin{equation}\label{eq:orderedDensityBrun}
\frac{1}{x_1\cdots x_{d-1}}\sum_{S \subset\{1,\dots,d-2\}} \frac{(-1)^{|S|}}{1+\sum_{k\in S} x_k}.
\end{equation}
Via the bijection $\chi: X_{\rB} \to \Delta_{\rB}$, this provides an invariant measure~$\nu_{\rB}$ of~$F_{\rB}$. The measures $\mu_{\rB}$ and $\nu_{\rB}$ are ergodic.
\end{lemma}

\begin{proof}
This is an immediate consequence of \eqref{eq:diagBrun} and Lemma~\ref{lem:unorderedinvariant}. Indeed, we pushforward the invariant measure of the unordered case via $\mu_{\rB} = \mathrm{ord}_* \mu_{\rU}$. The result follows because $\zeta \circ T_{\rU} = T_{\rU}\circ \zeta$ for each $\zeta \in \mathfrak{S}_d$ and, hence, the measure~$\mu_{\rU}$ is invariant under the action of~$\mathfrak{S}_d$ on~$\Delta_{\rU}$. 

Ergodicity of the measures $\mu_{\rB}$ and $\nu_{\rB}$ follows from \cite[Theorem~21]{Schweiger:00}.
\end{proof}

We get the invariant measure of the classical case as a corollary.

\begin{corollary}[{cf.~\cite[p.~646]{Arnoux-Nogueira}}]
The classical 2-dimensional Brun continued fraction algorithm $T_{\rB}: \widetilde{\Delta}_2 \to \widetilde{\Delta}_2$ defined in \eqref{eq:brunmap} has an invariant probability measure~$\mu_{\rB,2}$ that is absolutely continuous w.r.t.\ the Lebesgue measure and has density
\[
\frac{12}{\pi^2x_2(1+x_1)}.
\]
\end{corollary}

\begin{proof}
To gain the result, put $d = 3$ in \eqref{eq:orderedDensityBrun}. The renormalization factor~$\frac{12}{\pi^2}$ is a consequence of $\int_{1\le x_1\le x_2\le 1} \frac{\mathrm{d}x_1\mathrm{d}x_2}{x_2(1+x_1)}=\frac{\pi^2}{12}$.
\end{proof}

Because the sets $X_{\rB,k}$ form a partition of~$X_{\rB}$, the diagram in \eqref{eq:diagpsis} defines a  (measurable) conjugacy between the multidimensional continued fraction algorithm $(\hX_{\rB},\hF_{\rB},\hA_{\rB},\hnu)$ and the shift $(\{M_{\rB,1},\dots,M_{\rB,d}\}^{\ZZ},\Sigma, \bpsi_*\hnu)$. Indeed, it follows from the definition, that $\bpsi(\hX) = (M_n)_{n\in\ZZ}\in \{M_{\rB,1},\dots,M_{\rB,d}\}^{\ZZ}$. Thus $\bpsi(\hX)$ is the full shift which is equivalent to the fact that the algorithm is full. 

Also this algorithm serves as an example for the theory we developed in Chapters~\ref{sec:matrices} and~\ref{sec:metricmat}. Indeed, we gain the following result.

\begin{proposition}\label{prop:brunpisot:1}
Let $(\hX_{\rB},\hF_{\rB},\hA_{\rB}, \hnu_{\rB})$ be the natural extension of the 2-dimensional ordered Brun algorithm and let $\bpsi$ be as in \eqref{eq:bpsidef}. Let $\tilde{M}$ be a product of the 2-dimensional Brun ordered matrices from \eqref{eq:brunmatrices} that contains~$M_{\rB,3}$. Suppose that for  $\bpsi(\bx,\by) = (M_n)_{n\in\ZZ}$ there exists $h \in \NN$ such that $M_{[n,n+h)} = \tilde{M}^3$ occurs with bounded gaps in~$\bM$.  Then the following assertions hold.
\begin{itemize}
\item $\bpsi(\bx,\by)$ satisfies the two-sided Pisot condition.
\item $\bpsi(\bx,\by)$ is algebraically irreducible.
\item $\bpsi(\bx,\by)$ converges strongly to the generalized right eigenvector $\bu=\chi(\bx)$ in the future and to the generalized left eigenvector $\bv = \chi(\by)$ in the past. The vector $\chi(\bx)$ has rationally independent coordinates.
\item The mapping family $(\TT,f_{\bpsi(\bx,\by)})$ associated to $(\bx,\by)$ is eventually Anosov for the splitting
\[
\coprod_{n\in\ZZ} \RR \bu_n \oplus \bv_n^\perp.
\]
\end{itemize}
\end{proposition}

\begin{proof}
It follows from \cite[Section~5]{AD:19} that $\bM$ is two-sided primitive and satisfies the two-sided Pisot condition. Since there are only finitely many (namely, three) ordered Brun matrices, the growth condition $\lim_{n\to \pm\infty}\frac{1}{n} \log \lVert M_n\rVert = 0$ is trivially satisfied. Thus, Proposition~\ref{prop:domev}, Theorem~\ref{th:matrixPisot}, Corollary~\ref{cor:anosov}, and Lemma~\ref{lem:MCF_eigen} imply the lemma.
\end{proof}

For according metric results we again refer to Section~\ref{sec:BrunPisot}.

\subsection{Multiplicative acceleration of the Brun continued fraction algorithm} \label{subsec:BrunMul}
We now study an accelerated version of the ordered Brun algorithm.\footnote{The unordered Brun algorithm can be accelerated in a similar way but we do not give the details of this acceleration.} This algorithm, which is often called modified Jacobi--Perron algorithm, 
was first studied in~\cite{POD77} and further examined for instance in \cite{FUKE96,Schratzberger:98,MEESTER,Schratzberger:01,Hardcastle:02,HK:02}. Because it is a multiplicative algorithm, it is also known as multiplicative Brun algorithm.

Assume that $d \ge 3$ and let $X_{\rB}$ and $X_{\rB,k}$, $k\in\cA$, be given as in \eqref{eq:BOX} and \eqref{eq:XpartOrdered}, respectively. Set $X_{\rM} = X_{\rB}$ and $\Delta_{\rM} = \Delta_{\rB}$. \notx{M}{$\rM$}{object related to the multiplicative Brun algorithm}
We now use the \emph{jump transformation}\indx{jump transformation} (in the sense of~\cite{Schweiger:1975})
\begin{equation}\label{eq:BrunJumpMult}
j:X_{\rM} \to \NN;\quad\bx \mapsto \min\{m \ge 1 \,:\, F_{\rB}^{m-1}(\bx) \not\in X_{\rB,d}\}
\end{equation}
on the ordered Brun algorithm $(X_{\rB},F_{\rB},A_{\rB})$, i.e., we define the mappings (cf.~e.g.\ \cite{BST21})
\begin{equation}\label{eq:FMTM}
\begin{aligned}
F_{\rM}:\, X_{\rM} \to X_{\rM}, \quad \bx \mapsto F_{\rB}^{j(\bx)}(\bx), \\
T_{\rM}:\, \Delta_{\rM} \to \Delta_{\rM}, \quad \bx \mapsto T_{\rB}^{j(\bx)}(\bx).
\end{aligned}
\end{equation}
The expected value of $j$, which is the expected value of the partial quotients of the multiplicative Brun algorithm considered in \cite[Theorem~1~(d)]{BLV:18}, is finite by this theorem.  This implies that $j(\bx)$ is a well-defined positive integer for Lebesgue a.e.\ $\bx\in X_M$ and, hence, the mappings $F_{\rM}$ and $T_{\rM}$ are defined Lebesgue a.e.\ on their respective domains. Observe that the notion of a jump transformation is  reminiscent from the notion  of induction from Definition~\ref{def:induction}.

Note that $F_{\rM}$ is related to the dynamical system let $(X_{\rB}\setminus X_{\rB,d}, G, \nu_{\rB}|_{X_{\rB}\setminus X_{\rB,d}})$ obtained by inducing $(X_{\rB},F_{\rB},\nu_{\rB})$ on $X_{\rB}\setminus X_{\rB,d}$ in the sense of Definition~\ref{def:induction}. From the definition of $j$ we see that the diagram 
\begin{equation}\label{eq:diagXbXm}
\begin{tikzcd}
X_{\rB}\setminus X_{\rB,d}\arrow[r, "G"]\arrow[d,"F_{\rB}"] &X_{\rB}\setminus X_{\rB,d} \arrow[d, "F_{\rB}"] \\
X_{\rM} \arrow[r, "F_{\rM}"]& X_{\rM}
\end{tikzcd}
\end{equation}
commutes. 

For $1 \le k < d$ and $m \ge 1$, define the \emph{multiplicative Brun matrices}\indx{Brun!matrix!multiplicative}\indx{matrix!Brun!multiplicative} (the matrices $M_{\rB,k}$ are defined in \eqref{eq:orderedbrunmatrices} for $k\in\cA\setminus \{d\}$ and in \eqref{eq:orderedbrunmatrices2} for $k=d$)
\begin{equation}\label{eq:MMkm}
\begin{aligned}
M_{\rM,k,m} = M_{\rB,d}^{m-1} M_{\rB,k}
& = 
\begin{pmatrix} 
1&&&&&&&\\
&\ddots&&&&&&\\
&&1&&&&&\\
&&&0&1&&&\\
&&&&\ddots&\ddots&&\\
&&&&&0&1&\\
&&&1&&&m&
\end{pmatrix} \text{$\leftarrow$ $k$-th row} \\
&\hskip2.5cm \uparrow \\
&\hskip2.2cm \text{$k$-th column}
\end{aligned}
\end{equation}
and the sets 
\begin{equation}\label{eq:XMkm}
X_{\rM,k,m} = \{ [w_1 : \cdots : w_d] \,:\, [w_1 : \cdots : w_{d-1} : w_d-(m{-}1)w_{d-1}] \in X_{\rB,k} \}.
\end{equation}
These sets form a measurable partition of~$X_{\rM}$. 

By these considerations, the mapping $F_{\rM}$ defined in \eqref{eq:FMTM} fits into the general definition of a multidimensional continued fraction algorithm given in Definition~\ref{def:cf}. In particular, we may rewrite $F_{\rM}$ (Lebesgue a.e.) by using the cocycle $A_{\rM}$ as follows.

\begin{definition}[Modified Jacobi--Perron or multiplicative Brun algorithm]
\indx{Brun!algorithm, map!multiplicative}\indx{continued fraction algorithm!Brun!multiplicative}\indx{continued fraction algorithm!Jacobi--Perron!modified}\indx{Jacobi--Perron!algorithm!modified}
For $d \ge 3$ let\footnote{As before, there is ambiguity in the definition of $A_{\rB}$ on the overlap of the sets $X_{\rB,k}$ and we can define $A_{\rB}$ on these overlaps by any of the possible choices.}
\[
A_{\rM}:\, X_{\rM}  \to \GL(d,\ZZ), \quad \bx \mapsto \tr{\!M}_{\rM,k,m} \quad \text{if}\ \bx \in X_{\rM,k,m} \qquad (1 \le k < d,\, m \ge 1).
\]
with $X_{\rM,k,m} $ as in \eqref{eq:XMkm} and $M_{\rM,k,m}$ as in \eqref{eq:MMkm}. This defines the \emph{modified Jacobi--Perron~map}
\begin{equation} \label{CFalgoBM}
F_{\rM}:\, X_{\rM} \to X_{\rM}; \quad \bx \mapsto \tr{\!A}_{\rM}(\bx)^{-1} \bx.
\end{equation}
Then the multidimensional continued fraction algorithm $(X_{{\rB}},F_{{\rB}},A_{{\rB}})$ is called the $(d{-}1)$-dimensional \emph{modified Jacobi--Perron algorithm} or the \emph{multiplicative Brun algorithm}.
\end{definition}

Again, $X_{M,k,m} = X_{\tr{\!M}_{\rM,k,m}}$, $1 \le k < d$, $m \ge 1$, are the cylinder sets of the multiplicative Brun algorithm; see Definition~\ref{def:MFCcyl}. Since they are all mapped to $X_{\rM}$ by $F_{\rM}$, it is a \emph{full} algorithm. 

If we define~$\Delta_{\rM}$ using the norm $\lVert\cdot\rVert_\infty$, which we will do in this section, we get the following explicit representation of the affine version of the modified Jacobi--Perron algorithm (see also \cite{HK:02}; we suppress the last coordinate $x_d = 1$ and set $x_0 = 0$):
\[
\begin{aligned}
T_{\rM}(x_1,\ldots,x_{d-1})  = &
\Big( \frac{x_1}{x_{d-1}}, \dots, \frac{x_{k-1}}{x_{d-1}}, \frac{1}{x_{d-1}} - \Big\lfloor \frac{1}{x_{d-1}} \Big\rfloor, \frac{x_{k}}{x_{d-1}}, \dots, \frac{x_{d-2}}{x_{d-1}} \Big) \\
& \hbox{if} \quad\frac{x_{k-1}}{x_{d-1}} \le \frac{1}{x_{d-1}} - \Big\lfloor \frac{1}{x_{d-1}} \Big\rfloor \le \frac{x_{k}}{x_{d-1}} \qquad (1 \le k < d).
\end{aligned}
\]
It is not hard to see that $F_{\rM}$ is bijective on $\hX_{\rM} = X_{\rM} {\times} X_{\rM}^*$ with
\begin{equation}\label{eq:xmstar}
X_{\rM}^* = \{[w_1 : \cdots : w_d] \,:\, w_1,\dots,w_{d-1} \le w_d\}.
\end{equation}
Thus, using the method developed in Section~\ref{sec:natex}, we get a natural extension $(\hX_{\rM},\hF_{\rM},\hA_{\rM},\hnu_{\rM})$ of $(X_{\rM},F_{\rM},A_{\rM},\nu_{\rM})$. Again, we are able to come up with an explicit formula for the invariant measure; see \cite[p.~661]{Arnoux-Nogueira} for a different expression for the density in \eqref{eq:multDensityBrun}.

\begin{lemma}\label{lem:multinvariant}
The invariant measure~$\mu_{\rM}$ of the affine version of the modified Jacobi--Perron algorithm $(T_{\rM},\Delta_{\rM},A_{\rM},\mu_{\rM})$ is absolutely continuous w.r.t.\ the (piecewise) Lebesgue measure on~$\Delta_{\rM}$ and has density
\begin{equation}\label{eq:multDensityBrun}
\frac{1}{x_1\cdots x_{d-1}} \sum_{S \subset\{1,\dots,d-1\}} \frac{(-1)^{|S|}}{1+\sum_{k\in S}x_k}.
\end{equation}
Via the bijection $\chi: X_{\rM} \to \Delta_{\rM}$, this provides an invariant measure~$\nu_{\rM}$ of~$F_{\rM}$. The measures $\mu_{\rM}$ and $\nu_{\rM}$ are ergodic.
\end{lemma}

\begin{remark}\label{rem:AN93}
In \cite{Arnoux-Nogueira}, the density of the measure~$\mu_{\rM}$ was represented by the formula 
\begin{equation}\label{eq:multDensityBrun2}
\sum_{\sigma \in \mathfrak{S}_{d-1}} \prod_{i=1}^{d-1}  \frac{1}{1+x_{\sigma(1)}+\cdots+x_{\sigma(i)}},
\end{equation}
where $\mathfrak{S}_d$ is the group of permutations on $n$ elements. One can show that the representations \eqref{eq:multDensityBrun} and \eqref{eq:multDensityBrun2} are equal.\footnote{Note that, a priori, the density of $\mu_{\rM}$ is only determined up to a constant factor.} However, the representation in \eqref{eq:multDensityBrun} is simpler in the sense that the number of subsets of a set with $d{-}1$ elements is much smaller than $\mathfrak{S}_{d-1}$ for large~$d$.
\end{remark}

\begin{proof}
By using affine coordinates we see from \eqref{eq:xmstar} that the natural extension of~$T_{\rM}$ is defined on $\Delta_{\rM} {\times} [0,1]^{d-1} {\times} \{1\}$. According to Section~\ref{sec:natex}, to derive the required density we have to compute 
\[
I_d=\int_{0\le y_1,\ldots,y_{d-1}\le 1} \frac{\mathrm{d}\lambda{(\by)}}{\langle \bx,\by \rangle^d}.
\]
Using the evaluation of~$I_d^{(1)}(x_1)$ in the proof of Lemma~\ref{lem:unorderedinvariant} and partial fraction decomposition, we gain
\[
\begin{aligned}
I_d & = \int_{y_1=0}^1 \cdots \int_{y_{d-1}=0}^1 \frac{\mathrm{d}y_1\cdots \mathrm{d}y_{d-1}}{(x_1y_1+\cdots + x_{d-1}y_{d-1}+1)^d}  \\
& = \frac{1}{(d{-}1)!\,x_1\cdots x_{d-2}} \sum_{S\subset\{1,\dots,d-2\}} \frac{(-1)^{|S|}}{\big(1+x_{d-1}+\sum_{k\in S}x_k\big)\big(1+\sum_{k\in S}x_k\big)} \\
& = \frac{1}{(d{-}1)!\,x_1\cdots x_{d-1}}\sum_{S\subset\{1,\dots,d-1\}} \frac{(-1)^{|S|}}{1+\sum_{k\in S}x_k}.
\end{aligned}
\]
Since the constant factor $(d{-}1)!^{-1}$ is immaterial because the invariant measure is only defined up to a constant, the result follows.

Ergodicity of the measures $\mu_{\rM}$ and $\nu_{\rM}$ follows from \cite[Theorem~21]{Schweiger:00}.
\end{proof}

Let $\cM_{\rM} = \{M_{\rM,k,m} : 1 \le k < d,\, m \ge 1\}$.
Again the sets $X_{\rM,k,m}$, $1 \le k < d$, $m \ge 1$, form a partition of~$X_{\rM}$. Thus the diagram \eqref{eq:diagpsis} in Section~\ref{subsec:MCFmappingfamily} defines a measurable conjugacy between $(\hX_{\rM},\hF_{\rM},\hA_{\rM},\hnu)$ and $(\cM_{\rM}^{\ZZ},\Sigma,\bpsi_*\hnu)$. Indeed, as in the additive ordered case, we have $\bpsi(\hX) = \cM_{\rM}^{\ZZ}$. 

We did not work out concrete examples of our theory from Chapters~\ref{sec:matrices} and~\ref{sec:metricmat} for this multiplicative version of Brun's algorithm. However, we will see in Section~\ref{sec:BrunPisot} that the 2- and 3-dimensional case of all versions of Brun's algorithm covered so far furnish examples for the metric theory developed in that chapters.

\subsection{Pisot condition for the Brun continued fraction algorithm}\label{sec:BrunPisot}
In this section, we provide results on the relation between Brun's algorithm and the generic Pisot condition of Definition~\ref{def:MCF_Pisot}. First we show that it suffices to establish the generic Pisot condition for one of the three versions of Brun's continued fraction algorithm.

\begin{proposition}\label{prop:BrunEquivalentPisot}
If one of the $(d{-}1)$-dimensional continued fraction algorithms $(X_{\rU},F_{\rU},A_{\rU},\nu_{\rU})$, $(X_{\rB},F_{\rB},A_{\rB},\nu_{\rB})$, and $(X_{\rM},F_{\rM},A_{\rM},\nu_{\rM})$ satisfies the Pisot condition for a given $d \ge 3$, then the other two satisfy the Pisot condition for this~$d$ as well.
\end{proposition}

\begin{proof}
Let $d \ge 3$ be fixed. By \eqref{eq:diagBrun}, it is immediate that $(X_{\rU},F_{\rU},A_{\rU},\nu_{\rU})$ satisfies the Pisot condition if and only if $(X_{\rB},F_{\rB},A_{\rB},\nu_{\rB})$ satisfies it. It remains to deal with $(X_{\rM},F_{\rM},A_{\rM},\nu_{\rM})$, which is a multiplicative acceleration of  $(X_{\rB},F_{\rB},A_{\rB},\nu_{\rB})$ that is obtained by the jump transformation~$j$ defined in~\eqref{eq:BrunJumpMult}. 
In view of Remark~\ref{rem:1vs2}, it suffices to show that $(\hX_{\rB},\hF_{\rB},\hA_{\rB},\hnu_{\rB})$ satisfies the Pisot condition if and only if $(\hX_{\rM},\hF_{\rM},\hA_{\rM},\hnu_{\rM})$ does. Let 
\[
\hY=\{(\bx,\by)\in \hX_{\rB}\colon \bx\in X_{\rB}\setminus X_{\rB,d} \}
\]
and let $(\hY, \hG, \hB,\hnu_{\rB}|_{\hY})$ be the dynamical system that we obtain by inducing the system $(\hX_{\rB},\hF_{\rB},\hA_{\rB},\hnu_{\rB})$ on $\hY$ in the sense of Definition~\ref{def:induction}. From the definition of~$j$, we see that (compare \eqref{eq:diagXbXm})
\[
\begin{tikzcd}
\hY\arrow[r, "\hG"]\arrow[d,"\hF_{\rB}"] &\hY \arrow[d, "\hF_{\rB}"] \\
\hX_{\rM} \arrow[r, "\hF_{\rM}"]& \hX_{\rM}
\end{tikzcd}
\]
defines a measurable conjugacy. Thus we may apply Lemma~\ref{lem:viana819} to see that the Lyapunov exponents of $(\hX_{\rM},\hF_{\rM},\hA_{\rM},\hnu_{\rM})$ are obtained from the Lyapunov exponents  from  $(\hX_{\rB},\hF_{\rB},\hA_{\rB},\hnu_{\rB})$ by multiplication by a constant $c\ge 1$. In fact, this constant~$c$ is the expected value of $j$ which is finite by \cite[Theorem~1~(d)]{BLV:18}. Thus $(X_{\rM},F_{\rM},A_{\rM},\nu_{\rM})$ satisfies the Pisot condition if and only if the same is true for $(X_{\rB},F_{\rB},A_{\rB},\nu_{\rB})$.
\end{proof}

This result allows us to show that each version of the Brun continued fraction algorithm satisfies  the Pisot condition in low dimensions~$d$.

\begin{proposition}\label{prop:Brun23Pisot}
For $d \in \{3,4\}$, each of the $(d{-}1)$-dimensional continued fraction algorithms $(X_{\rU},F_{\rU},A_{\rU},\nu_{\rU})$, $(X_{\rB},F_{\rB},A_{\rB},\nu_{\rB})$, and $(X_{\rM},F_{\rM},A_{\rM},\nu_{\rM})$ satisfies the Pisot condition. In other words, each of these algorithms is a.e.\ exponentially convergent in the future and in the past. Moreover for each of these algorithms there is a cylinder of positive measure corresponding to a positive matrix.
\end{proposition}

\begin{proof}
For $d = 3$, it is proved in \cite{AD:19} that $(X_{\rB},F_{\rB},A_{\rB},\nu_{\rB})$ satisfies the Pisot condition; a proof for $(X_{\rM},F_{\rM},A_{\rM},\nu_{\rM})$ is contained in \cite{Schratzberger:98}. For $d = 4$, the Pisot condition for $(X_{\rM},F_{\rM},A_{\rM},\nu_{\rM})$ is proved in \cite{Schratzberger:01}; see also \cite{HK00,Hardcastle:02}. Thus the result on the Pisot condition follows from Proposition~\ref{prop:BrunEquivalentPisot}.
It is easy to see that there exist finite blocks that are positive. The corresponding cylinders have positive measure by equivalence of the invariant measure with the Lebesgue measure.
\end{proof}

Proposition~\ref{prop:Brun23Pisot} shows that the conditions of Theorem~\ref{thm:oseledetsmat} are satisfied for the Brun algorithms $(X_{\rU},F_{\rU},A_{\rU},\nu_{\rU})$, $(X_{\rB},F_{\rB},A_{\rB},\nu_{\rB})$, and $(X_{\rM},F_{\rM},A_{\rM},\nu_{\rM})$ if $d\in\{3,4\}$. Thus the conclusions of this theorem hols for a.e.\ sequence of matrices produced by these algorithms.

\chapter{Sequences of substitutions and Rauzy fractals}
\label{chapter:substitution}
We now consider bi-infinite sequences of substitutions and their associated dynamical systems.  These sequences form a refinement of the sequences of nonnegative integer matrices that we studied so far. Indeed, one can obtain a sequence of substitutions by superimposing a combinatorial structure on a sequence of matrices. In Section~\ref{sec:subs}, we deal with single bi-infinite sequences of substitutions. Section~\ref{sec:rauzy} is devoted to Rauzy fractals and to Rauzy boxes, which are defined as suspensions of Rauzy fractals. These objects are used to give a geometric interpretation of  a sequence of substitutions. Finally, the metric viewpoint for shift invariant sets of bi-infinite sequences of substitutions will be treated in Section~\ref{sec:metricS}.

\section{Bi-infinite sequences of substitutions}\label{sec:subs}
This section is the analog of Section~\ref{sec:matrices} for sequences of substitutions. In particular, we define \emph{$\cS$-adic shifts} by using sequences of substitutions. After setting up the basic terminology in Section~\ref{subsec:lang}, we see in Section~\ref{sec:combprop} that most of the properties we defined for sequences of matrices naturally carry over to sequences of substitutions by looking at \emph{incidence matrices}. However, there is a new property that makes sense only in the context of substitutions, namely \emph{balance}. As we will see in Section~\ref{subsec:crit}, besides the Pisot condition, also balance can be used to formulate a criterion for strong convergence in the setting of sequences of substitutions. In Section~\ref{sec:spectr-prop-sub} we will establish some spectral properties of $\cS$-adic shifts, and Section~\ref{sec:SadicMF} is devoted to $\cS$-adic mapping families. In Section~\ref{subsec:realization} we show how we can associate $\cS$-adic shifts to multidimensional continued fraction algorithms. This is of great importance later because the extra structure furnished by sequences of substitutions is needed in order to define the atoms of nonstationary Markov partitions, that we will associate with multidimensional continued fraction algorithms. Finally, in Section\ref{sec:subsSadicBrun} we relate $\cS$-adic shifts to the Brun algorithm.

\subsection{Languages and $\cS$-adic shifts}\label{subsec:lang}
For $d \ge 2$ let $\cA = \{1,2,\dots,d\}$\notx{alphabet}{$\cA$}{finite alphabet} be a finite alphabet, and let $\cA^*$\notx{alphabet}{$\cA^*$}{set of finite words}  be the set of finite words over~$\cA$, which is a free monoid with respect to concatenation of words. Moreover, we denote by~$\cA^{\ZZ}$\notx{alphabet}{$\cA^{\ZZ}$}{set of bi-infinite sequences} the set of two-sided infinite sequences over~$\cA$, equipped with the product topology of the discrete topology on~$\cA$. 
For a word $w \in \cA^*$, we denote by $|w|$\notx{0length}{$\lvert\cdot\rvert$}{length of a word} the number of letters of~$w$, i.e., the \emph{length}\indx{length} of $w$, and by~$|w|_a$\notx{0lengtha}{$\lvert\cdot\rvert_a$}{number of occurrences of a letter in a word}, $a \in \cA$, the number of occurrences of the letter $a$ in~$w$. A~word~$v$ is a \emph{factor}\indx{factor} of a word~$w$ if there exist words $p,s$ such that $w = pvs$. Moreover, if $p$ is the empty word, then $v$ is a \emph{prefix}\indx{prefix} of~$w$, which is denoted by $v \preceq w$.\notx{0prefix}{$\preceq,\prec$}{prefix relation} We write $v \prec w$ when $v \preceq w$ and $v \neq w$. 

A~\emph{substitution}\indx{substitution}~$\sigma$\notx{sigma}{$\sigma,\sigma_n$}{substitution} on the alphabet~$\cA$ is an endomorphism of the free monoid~$\cA^*$ that is \emph{nonerasing} in the sense that $|\sigma(a)| \ge 1$ holds for each $a\in\cA$. 
With~$\sigma$, we associate its \emph{incidence matrix}\indx{substitution!incidence matrix}\indx{matrix!incidence}\notx{Msigma}{$M_\sigma$}{incidence matrix of a substitution} $M_\sigma = (|\sigma(b)|_a)_{a,b\in\cA} \in \NN^{d\times d}$. This matrix is the abelianized version of~$\sigma$ in the following sense. Let  
\[
\bl:\, \cA^*\to \NN^d, \quad w\mapsto \tr{(|w|_1,\ldots,|w|_d)}
\]
\notx{l}{$\bl(\cdot)$}{abelianization map}be the \emph{abelianization map}\indx{abelianization}, then $\bl(\sigma(w)) = M_\sigma \bl(w)$ holds for all $w \in \cA^*$. A~substitution is called \emph{unimodular}\indx{substitution!unimodular} if $|{\det M_\sigma}| = 1$, \emph{primitive}\indx{primitive!substitution}\indx{substitution!primitive} if $M_{\sigma}$ is primitive, and it is called \emph{Pisot}\indx{Pisot!substitution}\indx{substitution!Pisot} if $M_\sigma$ is a  matrix whose characteristic polynomial is the minimal polynomial of a Pisot number.\footnote{Some authors call such a substitution \emph{irreducible Pisot}, and define a Pisot substitution as a substitution whose matrix is primitive, and has a Perron eigenvalue (i.e., a unique largest eigenvalue) which is a Pisot number.} Pisot substitutions are extensively studied in the literature; see e.g.\ \cite{Rauzy:82,Arnoux-Ito:01,Fog02,AkiBBLS,AD:19} and the references given there. 

Set
\[
\cS_d = \big\{\sigma \,:\, \sigma\ \text{is a unimodular substitution over the alphabet}\ \cA = \{1,\dots,d\}\big\}.
\]
\notx{Sd}{$\cS_d$}{set of unimodular substitutions} and let $\bsigma = (\sigma_n)_{n\in\ZZ} \in \cS_d^{\ZZ}$\notx{sigmab}{$\bsigma$}{sequence of substitutions} be a sequence of substitutions over the alphabet~$\cA$.
In accordance with our convention for matrices, compositions of consecutive substitutions  will be written as
\[
\sigma_{[m,n)} = \sigma_m \sigma_{m+1} \cdots \sigma_{n-1} \quad \text{for}\ m,n \in \ZZ\ \text{with}\ m \le n;
\]
\notx{subs}{$\sigma_{[m,n)}$}{composition of substitutions}here, $\sigma_{[m,m)}$ is the identity substitution, which  maps each letter of $\cA$ to itself.
Note that the incidence matrix of $\sigma_{[m,n)}$ satisfies
\[
M_{\sigma_{[m,n)}} = M_{\sigma_m} M_{\sigma_{m+1}} \cdots M_{\sigma_{n-1}}.
\]
Many properties of~$\bsigma$ depend only on its \emph{sequence of incidence matrices} 
\[
\bM_{\bsigma} = (M_{\sigma_n})_{n\in\ZZ}. \notx{Msigman}{$\bM_{\bsigma}$}{sequence of incidence matrices}
\]
In particular, $\bsigma$ is called \emph{primitive}\indx{primitive!sequence!of substitutions} if $\bM_{\bsigma}$ is primitive (see Definition~\ref{def:geaS} below). 

In order to define the $\cS$-adic shift associated to $\bsigma$, we set, for $n \in \ZZ$,
\begin{equation}\label{eq:langDef}
\cL_{\bsigma}^{(n)} = \{w \in \cA^* \,:\, \hbox{$w$ is a factor of $\sigma_{[n,m)}(a)$ for some $a \in \cA$, $m>n$}\}.
\end{equation}
We call $\cL_{\bsigma}^{(n)}$\notx{La}{$\cL_{\bsigma}, \cL_{\bsigma}^{(n)}$}{$\cS$-adic language} the \emph{language of level~$n$}\indx{language!level $n$} associated to~$\bsigma$, and we define $\cL_{\bsigma} := \cL_{\bsigma}^{(0)}$.
We recall that $\Sigma: \cA^{\ZZ} \to \cA^{\ZZ}$, $(\omega_n)_{n\in\ZZ}\mapsto (\omega_{n+1})_{n\in\ZZ}$, denotes the \emph{shift operator}. 

\begin{definition}[$\cS$-adic shift]\label{def:sadicshift}
\indx{shift!S@$\cS$-adic}\indx{S@$\cS$-adic!shift}
Let $\bsigma \in \cS^{\ZZ}$ with $\cS \subset \cS_d$, $d \ge 2$.
We define
\[
X _{\bsigma}^{(n)} = \{\omega \in \cA^{\ZZ} \,:\, \hbox{each factor of $\omega$ is an element of $\cL_{\bsigma}^{(n)}$}\} \qquad (n \in \ZZ). \notx{Xsigma}{$X _{\bsigma}, X _{\bsigma}^{(n)}$}{$\cS$-adic shift}
\]
The symbolic dynamical system $(X_{\bsigma}^{(n)},\Sigma)$ is called the \emph{(two-sided) $\cS$-adic shift of level~$n$ associated to $\bsigma$}. We let $X_{\bsigma}:=X^{(0)}_{\bsigma}$, and call $(X_{\bsigma},\Sigma)$ the \emph{$\cS$-adic shift associated to $\bsigma$}. 
\end{definition}

In Section~\ref{sec:matrices}, we considered sequences of matrices taken from a set $\cM \subset \cM_d$. Choose a set $\cS \subset \cS_d$ of substitutions such that $\mathcal{M} = \{M_\sigma : \sigma \in \cS\}$ and $M_\sigma \ne M_\tau$ for distinct $\sigma, \tau \in \cS$. This yields a natural injection $\varrho:\cM \to \cS_d$ with $\cS=\varrho(\cM)$. The mapping $\varrho$ is called a \emph{substitution assignment}\indx{substitution!assignment}; see Definition~\ref{d:realization}). Note that, given~$\bM$, the choice of~$\varrho$ is in general not unique and there is no ``standard'' choice for such a set~$\cS$. Nevertheless, given such a substitution assignment $\varrho$, we can relate with each sequence $(M_n)_{n\in\ZZ}$ a unique sequence $(\sigma_n)_{n\in\ZZ}$. In view of \eqref{eq:diagpsis}, a substitution selection associates with each orbit of the natural extension of a continued fraction algorithm a sequence of substitutions with values in~$\cS$ and, hence, an $\cS$-adic shift; we will make this precise in Section~\ref{subsec:realization}. For now we just illustrate this by a famous example.

\begin{definition}[Sturmian sequences]\label{def:sturm}
A bi-infinite sequence $\omega = (\omega_n)_{n\in\ZZ} \in \{1,2\}^{\ZZ}$  is called a \emph{Sturmian sequence}\indx{Sturmian!sequence} if there exist $\alpha \in (0,1) \setminus \QQ$ and $x \in \RR$ such that, for each $n \in \ZZ$, we have
\[ 
\omega_n = i \quad \mbox{if and only if} \quad  \fr_\alpha^n(x) = n \alpha + x  \in I_i\  \pmod{\ZZ},
\]
where 
\[
\fr_\alpha:\, \TT^1 \to \TT^1, \quad x \mapsto x + \alpha, \notx{rot}{$\fr_{\balpha}, \fr_{\bx}$}{toral rotation}
\]
with either $I_1 = [0,1{-}\alpha)$, $I_2=[1{-}\alpha, 1)$, or $I_1=(0,1{-}\alpha]$, $I_2 = (1{-}\alpha, 1]$. In this sense, a Sturmian sequence is a \emph{coding} of the rotation $(\TT^1,\fr_\alpha)$ w.r.t.\ the two-interval partition $\{I_1,I_2\}$.
\end{definition}

An irrational rotation~$\fr_\alpha$ leading to a Sturmian sequence is illustrated in Figure~\ref{fig:Rotationsequence}.

\begin{figure}[ht] 
\begin{tikzpicture}[scale=1.5,radius=1]
\draw[red,very thick] (0:1)arc[start angle=0,end angle=263]--(263:1);
\draw[blue,very thick] (263:1)arc[start angle=263,end angle=360]--(0:1);
\draw (0:.9)--(0:1.1) (263:.9)--(263:1.1);
\node at (311.5:1.2){$I_1$};
\node at (131.5:1.2){$I_2$};
\fill (-10:1) circle (1pt);
\node at (-10:1.4){$\fr_\alpha^{-1}(z)$};
\fill (87:1) circle (1pt);
\node at (87:1.15){$z$};
\fill (184:1) circle (1pt);
\node at (184:1.35){$\fr_\alpha(z)$};
\fill (281:1) circle (1pt);
\node at (281:1.2){$\fr_\alpha^2(z)$};
\fill (18:1) circle (1pt);
\node at (18:1.35){$\fr_\alpha^3(z)$};
\end{tikzpicture}
\caption{Sturmian sequences as codings of the rotation~$\fr_\alpha$ on~$\RR/\ZZ$ (with $\alpha = \frac{1}{1+e}$), represented on the unit circle}
\label{fig:Rotationsequence} 
\end{figure}
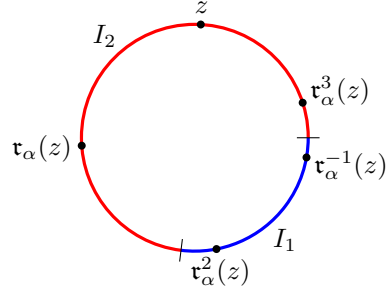

Sturmian sequences are well-known and have many interesting properties like low subword complexity;  in fact they have exactly $n+1$  subwords of length $n$ for each~$n$; see e.g.~\cite[Chapter~6]{Fog02} or \cite{Berstel&Seebold:2002}. 
For details on the correspondence between Sturmian sequences, the classical continued fraction algorithm, and sequences of low subword complexity, we also refer to \cite{Arnoux-Rauzy:91} or  \cite[Section~3.2]{thuswaldner2019boldsymbolsadic}. 

\begin{example}[Sturmian sequences and $\cS$-adic shifts]\label{ex:sturm}
We continue Examples~\ref{ex:classicalMappingFamily} and~\ref{ex:NMexpl2}.
For integers $a \ge 1$, define the matrices and substitutions 
\[
M_{\rG,a} = \begin{pmatrix}0&1\\1&a\end{pmatrix}, \quad \sigma_{\rG,a}: \begin{cases}1 \mapsto 2,\\ 2 \mapsto 1\underbrace{2\cdots2}_{a \text{ \scriptsize times}}.\end{cases}
\notx{G}{$\rG$}{object related to the classical continued fraction algorithm}
\]
Let $\cM_{\rG} = \{M_{\rG,a} : a \ge 1\}$ and $\cS_{\rG} = \{\sigma_{\rG,a} : a \ge 1\}$.
Then $M_{\rG,a}$ is the incidence matrix of~$\sigma_{\rG,a}$, hence the substitution selection
\[
\varrho:\ \cM_{\rG} \to \cS_{\rG}, \quad M_{\rG,a} \mapsto \sigma_{\rG,a},
\]
is well-defined. 
Let $x,y \in [0,1] \setminus \QQ$ be irrational numbers with continued fraction expansions $x = [0;a_0,a_1,\dots]$, $y = [0;a_{-1},a_{-2},\dots]$.
Then the natural extension~$\hF$ of the classical continued fraction algorithm associates to $\big(\genfrac[]{0pt}{}{x}{1}, \genfrac[]{0pt}{}{y}{1}\big) \in \PP_{>0}^1 {\times} \PP_{>0}^1$ the sequence of matrices $(M'_{a_n})_{n\in\ZZ}$; 
see Examples~\ref{ex:classicalCF} and~\ref{ex:classicalCF2}.
The set~$\cM$ is the range of the cocycle~$A$ of the classical continued fraction algorithm defined in~\eqref{eq:classicalcocycle}. 
The substitution selection~$\varrho$ then associates the sequence of substitutions $\bsigma = (\sigma_{\rG,a_n})_{n\in\ZZ}$ to $\bM = (M_{\rG,a_n})_{n\in\ZZ}$ and, hence, an $\cS$-adic shift $(X_{\bsigma},\Sigma)$. 
It was observed already in \cite{Morse&Hedlund:1940,Coven&Hedlund:1973} that the elements of~$X_{\bsigma}$ are \emph{Sturmian sequences}, which are codings of the rotation~$\fr_\alpha$ on $\RR/\ZZ$ by $\alpha = \frac{x}{1+x}$ with respect to two intervals $I_1$ and~$I_2$ of length~$\alpha$ and $1{-}\alpha$ respectively; see Figure~\ref{fig:Rotationsequence}.
This can be rephrased by saying that the shift $(X_{\bsigma},\Sigma)$ is mesurably  conjugate to the irrational rotation $(\RR/\ZZ,\fr_\alpha)$; this conjugacy is stated in a more general context in Theorem~\ref{t:tilingpds} below. 

Note that $M_{\rG,a_1} M_{\rG,a_2} = M_{\rF,2}^{a_1} M_{\rF,1}^{a_2}$ and $\sigma_{\rG,a_1} \circ \sigma_{\rG,a_2} = \sigma_{\rF,2}^{a_1} \circ \sigma_{\rF,1}^{a_2}$ with the matrices~$M_{\rF,i}$ from \eqref{eq:sturmmat} and the substitutions $\sigma_{\rF,i}$ from~\eqref{eq:sturmsubs}.
\end{example}

For surveys on results about $\cS$-adic shifts we refer to \cite{Fer:96,DL:12,Durand-Leroy-Richomme:13,Berthe-Delecroix,thuswaldner2019boldsymbolsadic}. In \cite{AMS:14,BST:19,BST:23,Fogg:24}, relations between $\cS$-adic shifts to geometry and arithmetic have been investigated, and in \cite{BSTY,BPRS} recognizability results of these shifts have been proved in a quite general context. Recall that a shift $(X,\Sigma)$ is \emph{minimal}\indx{shift!minimal} if $(X,\Sigma)$ has no nontrivial closed shift invariant subset. In the same way as the one-sided infinite case (see for instance \cite[Proposition~3.5.3~(iii)]{thuswaldner2019boldsymbolsadic}), one can show that  the $\cS$-adic shift associated to a primitive sequence of substitutions is minimal. Primitivity and minimality are in fact equivalent. To be more precise, if the sequence $\bsigma$ of substitutions is everywhere growing (in the sense of \cite[Definition~3.1]{Berthe-Delecroix}), but not primitive, it results from the definition of primitivity that some letter occurs with unbounded gaps in some element of~$X_{\bsigma}$. This contradicts minimality of $(X_{\bsigma},\Sigma)$. Primitivity is a necessary condition for obtaining meaningful symbolic dynamics in our context. 

\subsection{Combinatorial properties and generalized eigenvectors}\label{sec:combprop}
In this section, we revisit the properties stated in Section~\ref{sec:matrices} for sequences of matrices and carry them over to sequences of substitutions. Besides that, we define notions that are specific to sequences of substitutions. These are mainly two-sided versions of properties defined in \cite{BST:19,BST:23}.

One of the main notion that depends on~$\bsigma$ and not only on~$\bM_{\bsigma}$ is the notion of \emph{balance} (for letters), which is a property that we desire for the languages~$\cL_{\bsigma}^{(n)}$ associated to a sequence~$\bsigma$ of substitutions.
We stress the  fact that this notion is specific  to   the symbolic setting and its importance is a reason for superimposing sequences of substitutions on sequences of matrices. 

\begin{definition}[Balance]\label{def:balance}
\indx{balanced}
Let $C \in \NN$. 
A~ set of words $\cL \subset \cA^*$ is called \emph{$C$-balanced} if
\[  
|v|_a - |w|_a \le C \quad \mbox{for all factors $v,w$ of $\cL$ with $|v| = |w|$ and for all $a \in \cA$}. 
\]
An element of $\cA^{\ZZ}$ is called \emph{$C$-balanced} if the set of its factors is $C$-balanced. 
A~shift $(X,\Sigma)$ is called \emph{$C$-balanced} if the set of all factors of its elements is $C$-balanced. If one of these objects is $C$-balanced for some unspecified~$C$, then we just say that it is \emph{balanced}.
\end{definition}

The language of a Pisot substitution~$\sigma$, i.e., the language $\mathcal{L}_{\bsigma}$ for the constant sequence $\bsigma = (\sigma)$, is balanced; see e.g.~\cite{Adamczewski:03,Adamczewski:04}.  

\begin{definition}[Properties of sequences of substitutions inherited from their incidence matrices]\label{def:geaS}
\indx{convergence!weak}\indx{convergence!strong}\indx{convergence!exponential}\indx{eigenvector!generalized}\indx{irreducible!algebraically}\indx{Pisot!condition}
Let $\bsigma = (\sigma_n)_{n\in\ZZ}\in \cS_d^{\ZZ}$, $d \ge 2$, be a sequence of substitutions and $\bM_{\bsigma} = (M_{\sigma_n})_{n\in\ZZ}$ the associated sequence of incidence matrices. If a certain property holds for $\bM_{\bsigma}$ then we say that it also holds for $\bsigma$. In particular, this applies to primitivity, the existence of generalized eigenvectors, weak, strong, and exponential convergence, the Pisot condition, and to algebraic irreducibility.
\end{definition}

The following notion of letter frequency is related to balance.
\begin{definition}[Letter frequency]\indx{letter frequency}
For a sequence $(\omega_n)_{n\in\ZZ} \in \cA^{\ZZ}$, the \emph{frequency} of a letter $a \in \cA$ is 
\[ 
f_a = \lim_{m,n\to\infty} \frac{|\omega_{-m} \cdots \omega_n|_a}{n+m+1}, 
\]
if this limit exists. The vector $\mathbf{f}= \tr{(f_1,\dots,f_d)}$ is then called \emph{letter frequency vector}. 
\end{definition}

Let $\bsigma \in \cS_d^{\ZZ}$ be primitive. As in the case of a primitive substitution, if the generalized right eigenvector~$\bu$ exists, then it is proportional to the letter frequency vector for any infinite word in~$X_{\bsigma}$. In fact, the existence of this generalized right eigenvector provides the existence of the frequency vector, and even in most cases unique ergodicity of the symbolic system~$X_{\bsigma}$; see e.g.\ \cite{Fisher:09,BKMS12,Berthe-Delecroix}. In other words, the existence of well-defined frequencies for the letters is related  to weak convergence. 

But we also need balance, which has to do with strong convergence.
Indeed, balance (as defined in Definition~\ref{def:balance}) implies several important properties in terms of letter frequencies. 
It implies first that if $(X,\Sigma)$ is a minimal balanced shift (defined over the alphabet~$\cA$), then letters admit frequencies, and moreover, the (letter) \emph{symbolic discrepancy}\indx{symbolic discrepancy} 
\begin{equation}\label{eq:discrepancy}
\begin{aligned}
\mathrm{Discr}(X,\Sigma) & = \sup \big\{|\omega_{-n} \cdots \omega_n|_a - (2n{+}1) f_a \,:\, (\omega_k)_{k\in\ZZ} \in X,\, n \ge 1,\, a \in \cA\big\} \\
& = \sup \big\{\lVert \bl(w)-|w|\,\mathbf{f}  \rVert_\infty \,:\, \mbox{$w$ is a factor of an element of~$X$}\big\}
\end{aligned}
\end{equation}
is bounded, where $\mathbf{f}$ stands for the  letter frequency vector.

We can  even characterize balance geometrically by using the projection~$\pi_{\bu,  \bone}$  (defined in \eqref{eq:projections}) as follows, with $\bone= {}^t(1,\ldots,1)$\notx{1}{$\bone$}{diagonal vector} being the ``diagonal vector'' in~$\RR^d$.
Note that, for $\bu \in \RR_{\ge0}^d$ with $\lVert\bu\rVert_1 = 1$, we have $$\pi_{\bu,\bone}\,\bl(w) = \bl(w)-|w|\,\bu,$$ which is related to local discrepancy when $\bu=\mathbf{f}$. A minimal shift  $(X,\Sigma)$  such that all sequences have   frequency vector $\mathbf{f}$
 is balanced if and only if $\lVert\pi_{\bu,\bone}\,\bl(w)\rVert$ is bounded for factors $w$ of elements of~$X$  (see \cite[Proposition 4.10]{BST:23} or   \cite[Proposition~7]{Adamczewski:03}). This geometric vision is crucial for  establishing the boundedness of Rauzy fractals; see  Lemma~\ref{lem:RFbalanced}.  It turns out that balance also implies that the symbolic system is a finite extension of a toral rotation; see e.g.\ \cite[Proposition~2.9]{BC:19}.

\begin{example}[Sturmian sequences and balance]\label{ex:sturmbalance}
We continue Example \ref{ex:sturm}. Consider  a Sturmian sequence $\omega=(\omega_n)_n$ coding the rotation~$\fr_{\alpha}$, with $\alpha$ irrational, according to Definition~\ref{def:sturm}.
By equidistribution of the sequence $(n \alpha)_n \bmod 1$,  the Sturmian sequence~$\omega$ has letter frequencies, given by the lengths $\alpha$ and $1{-}\alpha$ of the coding intervals. One can easily deduce from Definition~\ref{def:sturm} that Sturmian sequences are $1$-balanced.
Moreover, aperiodic $1$-balanced sequences are Sturmian; see \cite[Chapter~6]{Fog02}.
\end{example}

\subsection{Pisot condition and strong convergence} \label{subsec:crit}
The condition we are after now is  strong convergence. We will need strong convergence (in the future) of $\bsigma\in \cS_d^{\ZZ}$ to $\bu$ in order to have the eventually Anosov property for the mapping family associated to~$\bsigma$; see Section~\ref{sec:SadicMF} below. Strong convergence is related to balance (see Definition \ref{def:balance} above), and will allow us to define the $\cS$-adic Rauzy fractals (see Section~\ref{sec:rauzy}), which will be important to construct explicit Markov partitions for the mapping families. The generalized left eigenvector $\bv$ plays an important role in the construction of Markov partitions.

In the stationary case $\bsigma = (\sigma)$, the fact that $\sigma$ is a Pisot substitution is  equivalent to strong convergence; this can be seen by looking at the eigenvalues and eigenvectors of the incidence matrix~$M_\sigma$. We now consider  the general case of  an $\cS$-adic  system satisfying   the  Pisot condition and we give sufficient conditions for the strong convergence of~$\bsigma$ (in the future). This statement is the substitutive form of Theorem~\nameref{t:A} stated in the introduction; see also Theorem~\ref{th:matrixPisot}.

\begin{theorem} \label{theo:sufcondPisotS}
Let $\bsigma \in \cS^{\ZZ}$, with $d \ge 2$, be a primitive sequence of substitutions with sequence of incidence matrices $\bM_{\bsigma} = (M_{\sigma_n})_{n\in\ZZ}$ that satisfies the local Pisot condition and the growth condition $\lim_{n\to\infty} \frac{1}{n} \log \lVert M_{\sigma_n}\rVert = 0$.  Then the follwing assertions hold.
\begin{enumerate}[\upshape (i)]
\item \label{i:scP2} The sequence $\bsigma$ converges exponentially and, hence, strongly to~$\bu$.
\item \label{i:scP1}
The coordinates of the generalized right eigenvector~$\bu$ of $\bsigma$ are rationally independent,   
\item \label{i:scP3}
The language $\cL_{\bsigma}$ is a balanced.
\end{enumerate}
\end{theorem} 

\begin{proof}[Proof of Theorem~\ref{theo:sufcondPisotS}]
Since $\bsigma$ satisfies the Pisot condition and the growth condition $\lim_{n\to\infty} \frac{1}{n} \log \lVert M_{\sigma_n}\rVert = 0$, assertion~(\ref{i:scP2}) is an immediate consequence of Proposition~\ref{prop:strongcv2sided}.  Moreover, Corollary~\ref{bst8.7} yields algebraic irreducibility of $\bsigma$, and we may apply Corollary~\ref{cor:indepNew} to get assertion\eqref{i:scP1}. 

It remains to prove assertion~(\ref{i:scP3}). If we show that there is a constant $C > 0$ such that, for each $w \in \cL_{\bsigma}$, the vector~$\bl(w)$ is at distance at most~$C$ from the ray~$\RR\bu$, balance follows. 
Let $\delta_1(n) \ge \cdots \ge \delta_d(n)$ be the singular values of $M_{\sigma_{[0,n)}}$. By the Pisot condition and the growth condition, we may choose $\beta > 0$ and $C_1 > 0$  in a way that 
\begin{equation} \label{eq:d2Mn_betabd2}
\delta_2(n) \le C_1  e^{-\beta n} \quad \mbox{for all}\ n \in \NN. 
\end{equation}
By the growth condition there is $C_2>0$ such that 
\[
\lVert M_{\sigma_n} \rVert_{\infty}  \le C_2\, e^{\frac\beta2 n} \quad \mbox{for all}\ n \in \NN. 
\]
Let $w \in \cL_{\bsigma}$, i.e., $w$~is a factor of $\sigma_{[0,n)}(a)$ for some $n \in \NN$, $a \in \cA$.
Then we can write $w = p_0 \sigma_0(w_1) s_0$ for some word $w_1 \in \cL_{\bsigma}^{(1)}$, some factors $p_0$ of $\sigma_0(a_0)$ and $s_0$ of $\sigma_0(b_0)$, $a_0, b_0 \in \cA$.
Recursively, we obtain
\begin{equation} \label{eq:lsigmarep}
w = p_0\, \sigma_0(p_1 \cdots \sigma_{n-1}(s_n p_n) \cdots s_1) s_0  
\end{equation}
for some words $p_j,s_j \in \cA^*$, $0 \le j \le n$, that are factors of words~$\sigma_j(a)$, $a \in \cA$.
By the growth condition, we have
\begin{equation}\label{eq:pnsnest}
\max\{|p_n|,|s_n|\} \le \max\{|\sigma_n(a)| \,:\, a \in \cA\}  = \lVert M_{\sigma_n}\rVert_{\infty}  \le C_2\, e^{\frac\beta2 n}.
\end{equation}
Then \eqref{eq:lsigmarep}, \eqref{eq:transposed}, \eqref{eq:pnsnest}, and \eqref{eq:d2Mn_betabd2} imply that
\[
\begin{aligned}
\mathrm{d}(\bl(w),\RR\bu) &\le \sum_{n=0}^N \mathrm{d}\big(M_{\sigma_{[0,n)}} \bl(p_ns_n), \RR\bu\big) \le 2 \sum_{n=0}^\infty \lVert \tr{\!M}_{\sigma_{[0,n)}}|_{\bu^\perp} \rVert_2 \cdot  \max\{|p_n|,|s_n|\} \\ 
& \le 2 C_2 \sum_{n=0}^\infty \lVert \tr{\!M}_{\sigma_{[0,n)}}|_{\bu^\perp}\rVert_2 \cdot e^{\frac\beta2 n}
\le  2 C_1 C_2 \sum_{n=0} ^\infty  \exp\Big({-}\frac{\beta}{2} n\Big) = \frac{2C_1C_2}{1+\beta/2}
\end{aligned}
\]
(with $\mathrm{d}(\bl(w),\RR\bu)$ measured by the 2-norm). 
This implies that $\cL_{\bsigma}$ is balanced.
\end{proof}

Another condition assuring strong convergence and rational independence of the coordinates of the generalized right eigenvector is contained in the following proposition.

\begin{proposition} \label{propo:sufcondPisotS}
Let $\bsigma \in \cS^{\ZZ}$, with $d \ge 2$, be a primitive sequence of substitutions that is algebraically irreducible and whose language~$\cL_{\bsigma}$ is balanced. Then the coordinates of its generalized right eigenvector~$\bu$ are rationally independent, and $\bsigma$ converges strongly to~$\bu$.
\end{proposition} 

\begin{proof}
If the language~$\cL_{\bsigma}$ is balanced, then the existence of a frequency vector of~$\cL_{\bsigma}$ and thus a generalized right eigenvector $\bu$ is proved e.g.\ in \cite{PoirierSteiner}. Moreover, balance implies that the distance of $M_{\sigma_{[0,n)}} \be_a$ from~$\RR \bu$ is bounded for all $a \in \cA$ (see \cite[Proposition~4.10]{BST:23}), thus algebraic irreducibility of~$\bsigma$ implies rational independence of~$\bu$ as in the proof of Corollary~\ref{cor:indepNew}. Now, the strong convergence to~$\bu$ is the content of \cite[Lemma~5.11]{BST:23}, the proof of which is quite technical. 
\end{proof}

\begin{remark}
We have seen in Proposition~\ref{prop:fur} that primitivity and eventual recurrence play a fundamental role in order to have the existence of generalized eigenvectors.  Theorem~\ref{theo:sufcondPisotS}  shows that the Pisot condition (together with the quite weak growth condition, which is even empty for finite~$\cS$) is much stronger.

Converses of the implications in Theorem~\ref{theo:sufcondPisotS} and Proposition~\ref{propo:sufcondPisotS} are false. Indeed, there exist substitutions with balanced fixed points that do not satisfy the Pisot condition (see \cite[Example~2.6]{BC:19}), and strong convergence does not give in general further information on the speed of convergence, which is exponential in the Pisot case. In general, the proof of the balance property uses the Dumont--Thomas decomposition of a prefix of the fixed point into words of the form $\sigma_0\sigma_1\cdots \sigma_n(a)$ (see \cite{Dumont-Thomas}), and the distance of the corresponding abelianized vector to the generalized eigenline~$\RR\bu$. We can in this way derive a bound for balance in terms of a series; see the proof of  the third assertion of Theorem~\ref{theo:sufcondPisotS}. If the strong convergence is too slow, this series does not converge and the word is not balanced. On the other hand, the Pisot condition implies exponential convergence, but a slower convergence suffices for balance. 
\end{remark}

As an application of Theorem~\ref{theo:sufcondPisotS}, we provide the following example. It shows that a sequence $(\sigma_n)_{n\in\ZZ}\in\cM^\ZZ$ may not have the local Pisot property even if the substitutions $\sigma_{[0,n)}$, $n\in\NN$, are Pisot substitutions. 

\begin{example} \label{ex:wm}
Define the set $\cS_{\rAR} =\{\sigma_{\rAR,i} : i \in \cA\}$ of \emph{Arnoux--Rauzy substitutions}\indx{Arnoux--Rauzy!substitution}\indx{substitution!Arnoux--Rauzy}\notx{AxR}{$\rAR$}{object related to the Arnoux--Rauzy algorithm} over the alphabet~$\cA$ by
\[
\sigma_{\rAR,i}:\, i \mapsto i,\ j \mapsto ij\ \mbox{for}\ j \in \cA \setminus \{i\}\qquad (i\in\cA).
\]
An \emph{Arnoux--Rauzy  directive sequence} is a sequence $\bsigma \in \cS_{\rAR}^{\ZZ}$  for which each substitution $\alpha_i$ occurs infinitely often. These sequences were introduced in \cite{Arnoux-Rauzy:91}.
By  \cite[Example 5]{Arnoux-Ito:01} all finite products of Arnoux--Rauzy substitutions are primitive Pisot substitutions
as soon as they contain each of the  three substitutions. 
By \cite[Theorem 2.4]{Cassaigne-Ferenczi-Zamboni:00},\footnote{The paper \cite{Cassaigne-Ferenczi-Zamboni:00} considers one-sided directive sequences; they can easily be prolongated on the left as two-sided Arnoux--Rauzy directive sequences because the balance property only depends on the future.} there exists a primitive Arnoux--Rauzy directive sequence $\bsigma$ for which $(X_{\bsigma},\Sigma)$ is  not balanced;  it satisfies the growth condition  since it is defined on a finite set of substitutions; however it does not satisfy the local Pisot condition by Theorem~\ref{theo:sufcondPisotS}~(\ref{i:scP3}), even though all the substitutions
$\sigma_{[0,n]}$ are Pisot substitutions for $n$ large enough.
\end{example}

\subsection{Spectral properties of $\cS$-adic shifts} \label{sec:spectr-prop-sub}
In this section we discuss spectral properties of $\cS$-adic shifts that follow as a consequence of the Pisot condition. Before stating the according results, we recall some definitions from ergodic theory.

We say that a directive sequence~$\bsigma$ has \emph{pure discrete spectrum}\indx{spectrum!pure discrete} if the system $(X_{\bsigma},\Sigma)$ is \emph{uniquely ergodic}\indx{ergodic!uniquely} (i.e., it has a unique shift invariant probability measure~$\mu$), minimal, and has pure discrete measurable spectrum (i.e., the measurable eigenfunctions of the \emph{Koopman operator}\indx{Koopman operator} $U_T: L^2(X_{\bsigma},\Sigma,\mu) \to L^2(X_{\bsigma},\Sigma,\mu)$, $f \mapsto f \circ \Sigma$, span $L^2(X_{\bsigma},\Sigma,\mu))$. A~complex number $\newlambda$ is a \emph{continuous eigenvalue}\indx{eigenvalue!continuous} of $(X,T)$ if there exists a nonzero continuous function $f: X \to \CC$ such that $f \circ T = \newlambda f$. 
A \emph{weak topological mixing}\indx{mixing!weak topological} shift~$X$ is one that has no nonconstant continuous (with respect to the topology) eigenfunctions of the shift operator. The Pisot condition implies the existence of a topological factor.

The following lemma shows that our standard assumptions on~$\bsigma$ imply minimality and unique ergodicity of the shift $(X_{\bsigma},\Sigma)$.

\begin{lemma} \label{l:uniquelyergodic}
Let $\bsigma \in \cS_d^{\ZZ}$, with $d \ge 2$, be a primitive sequence of unimodular substitutions that admits a generalized right eigenvector. 
Then the shift $(X_{\bsigma},\Sigma)$ is minimal and uniquely ergodic. 
\end{lemma}

\begin{proof}
If $\bsigma$ is primitive, then $(X_{\bsigma},\Sigma)$ is minimal by  \cite[Theorem~5.2]{Berthe-Delecroix} and everywhere growing in the sense of \cite[Definition~3.1]{Berthe-Delecroix}.
Then the existence of a generalized right eigenvector yields unique ergodicity by \cite[Theorem~5.7]{Berthe-Delecroix}.
\end{proof}

We also need the following classical consequence of a theorem of Gottschalk and Hedlund~\cite{GotHed:55} which can be found for instance in \cite{BC:19}.

\begin{lemma}[{see e.g.~\cite[Proposition~2.8]{BC:19}}]\label{lemma:GotHed}
Let $(X,T,\mu)$ be a minimal and uniquely ergodic subshift of~$\cA^{\ZZ}$. If $X$ is balanced on letters, then $\exp(2\pi i \mu([a]))$ is a continuous eigenvalue of $(X,T,\mu)$ for each $a\in\cA$.
\end{lemma}

\begin{theorem}\label{th:subsSpectral}
Let $\bsigma = (\sigma_n)_{n\in\ZZ} \in \cS_d^{\ZZ}$, with $d \ge 2$, be a primitive sequence of unimodular substitutions with incidence matrices $(M_{\sigma_n})_{n\in\ZZ}$ satisfying the Pisot condition and the growth condition $\lim_{n\to\infty} \frac{1}{n} \log \lVert M_{\sigma_n} \rVert = 0$. Then the uniquely ergodic $\cS$-adic shift $(X_{\bsigma},\Sigma,\mu)$ is not weakly mixing. In particular, $\exp(2\pi i \mu([a]))$  is a continuous eigenvalue for each letter $a\in\cA$. Moreover, the shift admits a minimal rotation on~$\TT^{d-1}$ as a topological factor.
\end{theorem}

\begin{proof}
It follows from the Pisot condition, which is shift invariant, that $\Sigma^n \bsigma$ converges weakly to $\RR \bu_n$ for some $\bu_n \in \RR_{\ge0}^d\setminus\{\mathbf{0}\}$ for each $n \in \ZZ$. Therefore, Lemma~\ref{l:uniquelyergodic} implies that $(X_{\bsigma},\Sigma)$ satisfies the conditions of Lemma~\ref{lemma:GotHed}, and the assertion on the eigenvalues follows from this lemma. The existence of these eigenvalues implies that $(X_{\bsigma},\Sigma)$ is not topologically weakly mixing. 
Let $\mu$ denote its unique invariant measure.
Since, by \cite[Theorem~5.7]{Berthe-Delecroix}, the vector $(\mu([a]))_{a\in\cA}$ is a generalized right eigenvector of~$\bsigma$, Corollary~\ref{cor:indepNew} implies that the numbers~$\mu([a])$, $a \in \cA$, are rationally independent and, since $\sum_{a\in\cA} \mu([a]) = 1$, the vector $(\mu([a]))_{a\in\cA\setminus\{d\}}$ induces a minimal rotation on~$\TT^{d-1}$. Then, by a classical result (see e.g.\ \cite[Exercise~1.6.3]{Fog02}), the shift admits a minimal rotation on~$\TT^{d-1}$ as a topological factor.
\end{proof}

\subsection{Eventually Anosov $\cS$-adic mapping families}\label{sec:SadicMF}
We now revisit the results stated for sequences of matrices in  Section~\ref{subsec:anosovM} and we now  relate a mapping family to a sequence of substitutions. 
We consider again the set~$\cS_d$ of substitutions associated to unimodular matrices.

\begin{definition}[$\cS$-adic mapping family]\label{def:SadicmappingS}
\indx{mapping family!S@$\cS$-adic}\indx{S@$\cS$-adic!mapping family}
Let $\cS \subset \cS_d$, with $d \ge 2$, be a family of unimodular substitutions and let $\bsigma = (\sigma_n)_{n\in\ZZ}\in \cS^{\ZZ}$. For each $n \in \ZZ$, let $\TT_n$  be a copy of the $d$-dimensional torus $\RR^d/\ZZ^d$ equipped with the standard metric inherited from~$\RR^d$. Moreover, regard the inverse of the incidence matrix of $\sigma_n$ as the automorphism\footnote{This is possible since $M_{\sigma_n}$ is an integer matrix of determinant $\pm 1$ by unimodularity of~$\sigma_n$.} $M_{\sigma_n}^{-1}: \TT_n \to \TT_{n+1}$ of $\RR^d/\ZZ^d$  given by left multiplication by~$M_{\sigma_n}^{-1}$. Then we set $\TT = \coprod_{n\in\ZZ} \TT_n$ and define $f_{\bsigma}: \TT \to \TT$ by $f_{\bsigma}\bx = M_{\sigma_n}^{-1} \bx$ for $\bx \in \TT_n$. \notx{fsigma}{$(\TT,f_{\bsigma})$}{$\cS$-adic mapping family}
We call $(\TT,f_{\bsigma})$ the \emph{$\cS$-adic mapping family} associated to~$\bsigma$. It can be written out as
\[  
\cdots \xrightarrow{M_{\sigma_{-2}}^{-1}} \TT_{-1} \xrightarrow{M_{\sigma_{-1}}^{-1}}  \TT_0 \xrightarrow{M_{\sigma_0}^{-1}} \TT_1 \xrightarrow{M_{\sigma_1}^{-1}}  \cdots . 
\]
\end{definition}

We shall now state the $\cS$-adic version of Theorem~\ref{cor:anosov}, i.e., $\cS$-adic mapping families are eventually Anosov if the corresponding sequences of substitutions $\bsigma$ admit strong convergence to the generalized eigenvectors (see Definition~\ref{def:geaS}) in the future and in the past. In the following result, also an $\cS$-adic version of Proposition~\ref{th:anosov} is included.

\begin{theorem}\label{th:anosovS}
Let $\bsigma \in \cS_d^{\ZZ}$, with $d \ge 2$, be a two-sided primitive sequence of 
unimodular substitution that admits strong convergence to the generalized right eigenvector $\bu$ in the future and to the generalized left eigenvector~$\bv$ in the past. Then the mapping family $(\TT,f_{\bsigma})$ associated to~$\bsigma$ is eventually Anosov  for the splitting 
\[
G^s \oplus G^u = \coprod_{n\in\ZZ} G_n^s \oplus G_n^u = \coprod_{n\in\ZZ}\RR \bu_n \oplus \bv_n^\perp,
\]
where $\bu_n$ and $\bv_n$ are defined in \eqref{eq:unvn}.

In particular, this holds when $\bsigma = (\sigma_n)_{n\in\ZZ}$  is two-sided primitive, satisfies the two-sided Pisot condition and the growth condition $\lim_{n\to\pm \infty} \frac{1}{n} \log \lVert M_{\sigma_n}\rVert = 0$. 
\end{theorem}

\subsection{Substitutive realizations} \label{subsec:realization} 
Our goal is to set up symbolic realizations of  multidimensional continued fraction algorithms and sequences of matrices. We follow here \cite[Section~2.3]{BST:23}, and we provide a two-sided version  of the notion of a substitutive realization. Consider the natural extension $(\hX,\hF,\hA)$ of a multidimensional continued fraction algorithm $(X,F,A)$; see Chapter~\ref{sec:cf} for the definition. We associate with each $(\bx, \by) \in \hX$ a bi-infinite sequence of substitutions $\bsigma = (\sigma_n)_{n\in\ZZ} \in  \cS_d^{\ZZ}$ with generalized  right and left eigenvectors $\bx$ and~$\by$.
We regard the partial quotient matrices $\tr{\!\hA}(\hF^n(\bx,\by))$ as incidence matrices of substitutions, i.e., for each $n \in \ZZ$ we choose $\sigma_n$ with incidence matrix $M_{\sigma_n} = \tr{\!\hA}(\hF^n(\bx,\by))$. 

\begin{definition}[Substitutive realization] \label{d:realization}
\indx{substitution!assignment}
\indx{substitutive!realization}
\indx{substitution!selection}
\indx{substitution!selection!faithful}
Let $\cM \subset \cM_d$, $d \ge 2$, be a set of matrices. We first call a map $\varrho: \cM \to \cS_d$\notx{roz}{$\varrho$}{substitution assignment} a \emph{substitution assignment on $\cM$} if the incidence matrix of $\varrho(M)$ is~$M$ for each $M \in \cM$. Often we will have to apply such a substitution assignment to a sequence of matrices. In this case it is to be understood that $\varrho$ is applied component-wise. 

We then call a map $\varphi: \hX \to \cS_d$\notx{phi}{$\varphi$}{substitution selection} a \emph{substitution selection} for a positive $(d{-}1)$-dimensional continued fraction algorithm $(X,F,A)$ with natural extension $(\hX,\hF,\hA)$  if the incidence matrix of $\varphi(\bx,\by)$ is equal to $\tr{\!\hA}(\bx,\by)$  for all $(\bx,\by) \in \hX$. The corresponding \emph{substitutive realization} of $(X,F,A)$ is the map 
\[
\bphi:\, \hX \to \cS_d^{\ZZ}, \quad (\bx,\by) \mapsto (\varphi(\hF^n(\bx,\by)))_{n\in\ZZ}, \notx{phib}{$\bphi$}{substitutive realization}
\]
together with the shift $(\bphi(\hX),\Sigma)$.
For any $(\bx,\by) \in \hX$, the sequence $\bphi(\bx,\by)$ is called an \emph{$\cS$-adic expansion} of~$(\bx,\by)$. Because $\bphi(\bx,\by)\in\cS_d^\ZZ$, the shift $(X_{\bphi(\bx,\by)},\Sigma)$ is an $\cS$-adic dynamical system which we call the \emph{$\cS$-adic dynamical system of~$(\bx,\by)$ w.r.t.\ $( \hX, \hF,\varphi)$}.

If $\varphi(\bx,\by) = \varphi(\bx',\by')$ for all $(\bx,\by), (\bx', \by') \in \hX$ with $\hA(\bx,\by) = \hA(\bx',\by')$, then $\varphi$ is called a \emph{faithful substitution selection}, and $\bphi$ is a \emph{faithful substitutive realization}. 
\end{definition}

If $\varphi$ is a faithful substitution selection, then it may be defined by a substitution assignment $\varrho$ on $\cM = \tr{\!\hA}(\hX)$, i.e., for a suitable $\varrho$ we have $\varphi(\bx,\by)=\varrho \circ \tr{\!\hA}(\bx,\by)$.

The definition of a substitutive realization implies that the diagram
\begin{equation}\label{eq:diagphis}
\begin{tikzcd}
\hX\arrow[r, "F"]\arrow[d,"\bphi"] & \hX \arrow[d, "\bphi"] \\
\bphi(\hX) \arrow[r, "\Sigma"]& \bphi(\hX) 
\end{tikzcd}
\end{equation}
commutes.
If $(X,F,A)$ converges weakly in the future and in the past at $(\bx, \by)$ for $\nu$-almost all $(\bx,\by) \in \hX$ (w.r.t.\ a measure~$\nu$ having the properties determined in Section~\ref{sec:cfgentheory}), then the dynamical system $(\hX,\hF,\hA)$ is measurably conjugate to its substitutive realization, which we write as 
\begin{equation} \label{eq:phi}
(\hX,\hF,\hA,\nu) \overset{\bphi}{\cong} (\bphi(\hX ),\Sigma,\bphi_*\nu).
\end{equation}
From Lemma~\ref{lem:MCF_eigen}, we immediately obtain the following correspondence between $(\bx,\by)$ and the generalized right and left eigenvectors of $\bphi(\bx,\by)$.

\begin{lemma}\label{lem:MCF_eigenS}
Assume that $\hX \subset \PP_{\ge0}^{d-1} {\times} \PP_{\ge0}^{d-1}$ and that $(\bx,\by) \in (\hX,\hF,\hA)$ is weakly convergent. If $\bu$ and~$\bv$ are generalized right and left eigenvectors of $\bphi(\bx,\by)$, then $(\bu,\bv)$ is a representative of $(\bx,\by)$ in $\RR^d {\times} \RR^d$.
\end{lemma}

Analogously to the matrix setting in Section~\ref{subsec:MCFmappingfamily}, each sequence of substitutions $\bphi(\bx,\by)$, with $(\bx,\by) \in \hX$, can be used to define the $\cS$-adic mapping family $(\TT,f_{\bphi(\bx,\by)})$, which performs the multidimensional continued fraction algorithm~$\hF$ for a suitable element of the domain~$\hX$. Thus $\hX$ parametrizes the set of mapping families and the shift $(\bphi(\hX ),\Sigma,\bphi_*\nu)$ forms a renormalization procedure. 

These substitutive realizations will be crucial for the construction of nonstationary Markov partitions for multidimensional continued fraction algorithms; see Chapter~\ref{chapter:markov}.

\subsection{Substitutive realizations of the Brun algorithm}\label{sec:subsSadicBrun}
In this section, we illustrate the results of this section by examples related to the version of the Brun continued fraction algorithm studied in Section~\ref{sec:Brun}. We do this in three examples,  namely,  the ordered, unordered and multiplicative versions of the  Brun algorithm.

\begin{example}[Substitutive realization of the unordered Brun algorithm]\label{ex:UnBrunSubsRe}
\indx{Brun!substitution!unordered}\indx{substitution!Brun!unordered}
For $i,j \in \cA$, $i \neq j$, we define the \emph{unordered Brun substitutions} (see also~\cite{Delecroix-Hejda-Steiner}; we take the mirror images of the images of letters under the substitutions for better comparability to the ordered case) by
\begin{equation}\label{eq:sigmauij}
\sigma_{\rU,ij} :\, \begin{cases} j \mapsto ji,  \\ k \mapsto k&(k\in \cA\setminus\{j\}).
\notx{U}{$\rU$}{object related to the unordered Brun algorithm}
\end{cases} 
\end{equation}
These substitutions are chosen in such a way that for all $i,j \in \cA$, $i \neq j$, the incidence matrix of $\sigma_{\rU,ij}$ equals the matrix~$M_{\rU,ij}$ used in \eqref{subsec:BrunOrd} to define the ordered Brun algorithm. We now use Definition~\ref{d:realization} to set up a substitutive realization for the ordered Brun algorithm. We start with the definition of a  substitution assignment. To this end let $\cM_{\rU}$ and~$\cS_{\rU}$ be the set of unordered Brun matrices and unordered Brun substitutions, respectively, and set
\[
\varrho_{\rU}:\, \cM_{\rU} \to \cS_{\rU}, \quad M_{\rU,ij} \mapsto \sigma_{\rU,ij} \quad (i,j\in \cA,\, i \neq j).
\]
With $X_{\rU,ij}$ defined as in \eqref{eq:XunBrun}, we can now define the substitution selection of the unordered Brun algorithm by
\[
\varphi_{\rU}:\, \hX_{\rU} \to \cS_{\rU}; \quad
(\bx,\by) \mapsto \varrho(\tr{\!A}(\bx,\by)).
\]
This gives the \emph{substitutive realization of the unordered Brun algorithm $(X_{\rU},F_{\rU},A)$}  
\[
\bphi_{\rU}: \hX_{\rU} \to \cS_{\rU}^{\ZZ}, \quad (\bx,\by) \mapsto (\varphi_{\rU}(F_{\rU}^n(\bx,\by))).
\]
From \eqref{eq:unBrunMarkov}, we see immediately that
\begin{equation}\label{eq:unBrunMarkovSub}
\begin{aligned}
\bphi(\hX) & = \big\{(\sigma_n)_{n\in\ZZ} \,:\, (\sigma_n,\sigma_{n+1}) = (\sigma_{\rU,ij},\sigma_{\rU,ij})\ \text{for}\ i,j \in \cA,\, i \neq j\\
& \hspace{3em} \text{or}\ (\sigma_n,\sigma_{n+1}) = (\sigma_{\rU,ij},\sigma_{\rU,jk})\ \text{for}\ i,j,k \in \cA, \, i \ne j \hbox{ and } j \ne k\big\}.
\end{aligned}
\end{equation}
Since the ordered Brun algorithm $(X_{\rU},F_{\rU},A_U)$ converges weakly almost everywhere (see Section~\ref{sec:Brun}), we gain from \eqref{eq:phi} the  measurable conjugacy
\[
(\hX_{\rU},\hF_{\rU},\hA_{\rU},\hnu_{\rU}) \overset{\bphi_{\rU}}{\cong} (\bphi(\hX_{\rU}),\Sigma,\bphi_*\hnu_{\rU}).
\]
We can now prove the following proposition, which is in analogy with Proposition~\ref{prop:brunpisot:2} apart from the assertion about balance.
Condition \eqref{eq:BrunAllMatrices2} means that the expansion has bounded strong partial quotients in the sense defined in Example~\ref{ex:AR1}.

\begin{proposition}\label{prop:brunpisotsubsUN}
Let $(X_{\rU},F_{\rU},A_{\rU},\nu_{\rU})$ be the natural extension of the 2-dimensional ordered Brun algorithm. Suppose that for $\bphi_{\rU}(\bx,\by) = (\sigma_n)_{n\in\ZZ}$ there exists $h\in \NN$ such that 
\begin{equation}\label{eq:BrunAllMatrices2}
\{\sigma_n,\ldots, \sigma_{n+h}\} = \{\sigma_{\rU,ij}\,:\, i,j\in \{1,2,3\},\, i \neq j\}
\end{equation}
holds for each $n\in\ZZ$. Then the following assertions hold.
\begin{itemize}
\item $\bphi(\bx,\by)$ satisfies the two-sided Pisot condition.
\item $\bphi(\bx,\by)$ is algebraically irreducible.
\item The language $\cL_{\bphi(\bx,\by)}$ is balanced.
\item $\bphi(\bx,\by)$ converges strongly to the generalized right eigenvector $\bu=\chi(\bx)$ in the future and to the generalized left eigenvector $\bv=\chi(\by)$ in the past. The vector $\chi(\bx)$ has rationally independent coordinates.
\item The mapping family $(\TT,f_{\bpsi(\bx,\by)})$ associated to $(\bx,\by)$ is eventually Anosov for the splitting
\[
\coprod_{n\in\ZZ} \RR \bu_n \oplus \bv_n^\perp.
\]
\end{itemize}
\end{proposition}
\begin{proof}
The conditions of Theorems~\ref{th:anosovS} and~\ref{theo:sufcondPisotS} are satisfied  by reasoning as in the proof of Proposition~\ref{prop:brunpisot:2}. Thus the result follows from these theorems.
\end{proof}
\end{example}

\begin{example}[Substitutive realization of the ordered Brun algorithm]
\label{ex:OrdBrunSubsRe}\indx{Brun!substitution!ordered}\indx{substitution!Brun!ordered}
For $k\in \cA\setminus\{d\}$ we define the \emph{ordered Brun substitutions} 
\begin{equation}\label{eq:brunsubsallg}
\begin{aligned}
\sigma_{\rB,k} : &
\begin{cases}i \mapsto i, & (1 \le i < k)  \\ k \mapsto d, \\ i \mapsto i{-}1, &(k<i<d) \\ d \mapsto (d{-}1)d, \end{cases}  \\
\sigma_{\rB,d} : &
\begin{cases}i \mapsto i, & (i \neq d{-}1)  \\ d{-}1 \mapsto (d{-}1)d. \end{cases} 
\end{aligned}
\notx{B}{$\rB$}{object related to the ordered Brun algorithm}
\end{equation}
For $d=3$, they read 
\begin{equation}\label{eq:brunsubs}
\sigma_{\rB,1} :
\begin{cases} 1 \mapsto 3, \\ 2 \mapsto 1, \\ 3 \mapsto 23, \end{cases} \quad
\sigma_{\rB,2} : 
\begin{cases} 1 \mapsto 1, \\ 2 \mapsto 3, \\ 3 \mapsto 23, \end{cases} \quad
\sigma_{\rB,3} :
 \begin{cases} 1 \mapsto 1, \\ 2 \mapsto 23, \\ 3 \mapsto 3. \end{cases}
\end{equation}
These substitutions are chosen in a way that that for each $k \in\cA$ the incidence matrix of $\sigma_{\rB,k}$ equals the matrix $M_{\rB,k}$ used in \eqref{subsec:BrunOrd} to define the ordered Brun algorithm. For this reason, the sequences $\bsigma \in \{\sigma_{\rB,1},\dots,\sigma_{\rB,d}\}^{\ZZ}$ can be regarded as ``combinatorial'' analogs of the sequences of matrices generated by the ordered Brun algorithm. In the same way as we did for the unordered Brun algorithm in Example~\ref{ex:UnBrunSubsRe}, we can now set up a substitutive realization~$\bphi_{\rB}$ for $(X_{\rB},F_{\rB},A_{\rB},\nu_{\rB})$. This leads to the measurable conjugacy
\[
(\hX_{\rB},\hF_{\rB},\hA_{\rB},\hnu_{\rB}) \overset{\bphi_{\rB}}{\cong} (\bphi(\hX_{\rB}),\Sigma,\bphi_*\hnu_{\rB}).
\]

We can now prove a substitutive version of Proposition~\ref{prop:brunpisot:1} leading to the same assertions as Proposition~\ref{prop:brunpisotsubsUN}, also for the ordered Brun algorithm. Since this is completely analogous to the unordered case covered in Example~\ref{ex:UnBrunSubsRe}, we omit the details.
\end{example}

\begin{example}[Substitutive realization of the multiplicative Brun algorithm]
\label{ex:MultBrunSubsRe}\indx{Brun!substitution!multiplicative}\indx{substitution!Brun!multiplicative}
For $1 \le k < d$ and $m \ge 1$, we define the \emph{multiplicative Brun substitutions} 
\begin{equation}\label{eq:brunsubsallgMult}
\sigma_{\rM,k,m}:\, 
\begin{cases} i \mapsto i, & (1 \le i<k)  \\ k \mapsto d, \\ i \mapsto i{-}1, & (k<i<d) \\ d \mapsto (d{-}1)\underbrace{d\cdots d}_{\text{$m$ times}}, \end{cases} 
\notx{M}{$\rM$}{object related to the multiplicative Brun algorithm}
\end{equation}
and set $\cS_{\rM} = \{\sigma_{\rM,k,m} : 1 \le k < d,\, m\ge 1\}$.
In the same way as before, we get a measurable conjugacy to a shift on a substitutive realization~$\bphi_{\rM}$:
\begin{equation} \label{eq:phiMult}
(\hX_{\rM},\hF_{\rM},\hA_{\rM},\hnu_{\rM}) \overset{\bphi_{\rM}}{\cong} (\cS_{\rM}^{\ZZ},\Sigma,\bphi_*\hnu_{\rM}).
\end{equation}
In this case, we established no combinatorial conditions that would lead to an analog of Proposition~\ref{prop:brunpisotsubsUN}. Besides that, this algorithm is multiplicative and, hence, we need infinitely many substitutions in its substitutive realization. Thus, the balance assertion of Theorem~\ref{theo:sufcondPisotS} is not true. We will see in Section~\ref{sec:metricS}, that in our metric theory we can settle the problem of balance also for infinitely many substitutions.
\end{example}

Note that there are several possible substitutive realizations of the ordered and unordered Brun algorithms.
We have chosen realizations in a way that the two versions of the algorithm provide the same $\cS$-adic shifts. 
However, since the unordered version of the Brun algorithm is given by a larger set of substitutions than the ordered one, there is also more flexibility in the definition of the substitutive realization, i.e., the unordered case provides a larger class of $\cS$-adic shifts.

\section{Rauzy fractals and Rauzy boxes}\label{sec:rauzy}
We now recall the key definition of a Rauzy fractal in the $\cS$-adic setting. 
These objects were originally defined in \cite{BST:19,BST:23,Fogg:24}.  A~Rauzy fractal for a sequence~$\bsigma$ of substitutions is obtained as the closure of the projection of abelianizations of certain words taken from  the language $\cL_{\bsigma}^{(n)}$, $n \in \NN$. Under suitable hypotheses, a Rauzy fractal allows us to build a geometric model of the symbolic system $(X_{\bsigma},\Sigma)$.
The first hypothesis we need is weak convergence of the continued fraction defined in terms of~$\bsigma$. This condition is satisfied by most of the interesting algorithms. The second condition requires that the language of the symbolic system is balanced. This is a much stronger condition, which can be seen as a generalization of the Pisot condition. 

We first define Rauzy fractals and their suspensions, Rauzy boxes, in Section~\ref{subsec:defRF}.
We then prove tiling properties of these sets in Section~\ref{subsec:tilings}, and in Section~\ref{subsec:dspectrum} show that tiling properties imply pure discrete spectrum of the underlying $\cS$-adic dynamical systems. Finally, in Section~\ref{sec:suff-cond-tilings}, we provide a combinatorial condition that plays an important role in establishing tiling properties of Rauzy fractals. All this will be needed later on when we use them to construct nonstationary Markov partitions for $\cS$-adic mapping families.  

\subsection{Basic properties of Rauzy fractals and Rauzy boxes} \label{subsec:defRF}
  
For $d \ge 2$, let $\bsigma=(\sigma_n)_{n\in\ZZ} \in \cS_d^{\ZZ}$, be a sequence of unimodular substitutions over the alphabet $\cA = \{1,\dots,d\}$.
Throughout the section, we assume that $\bsigma$ is primitive (in the future), has a generalized right eigenvector~$\bu$ and a generalized left eigenvector~$\bv$ with $\langle \bu, \bv \rangle \ne 0$, and we use the abbreviations
\begin{equation}\label{eq:abbrproj}
\pi_n = \pi_{\bu_n,\bv_n} \quad\hbox{and}\quad
\tilde{\pi}_n = \tilde{\pi}_{\bu_n,\bv_n} \notx{pen}{$\pi_n,\tilde{\pi}_n$}{projections w.r.t.\ generalized eigenvectors}
\end{equation}
for the projections along $\RR\bu_n$ to $\bv_n^\perp$ and along $\bv_n^\perp$ to $\RR\bu_n$ respectively, with $\bu_n$ and $\bv_n$ as in~\eqref{eq:unvn}.
Note that $\pi_n$ is well defined because $\bv_n \in \RR_{\ge0}^d \setminus \{\mathbf{0}\}$ and $\bu_n \in \RR_{>0}^d$ by primitivity. 

We define $\cS$-adic Rauzy fractals as in \cite[Section~2.4, Definition~2.7]{BST:23}. 
An alternative definition of a Rauzy fractal is provided in~\cite[Section~2.9]{BST:19}, which uses prefixes of an an infinite \emph{limit word} of~$\bsigma$ (leading to a \emph{broken line} as depicted in Figure~\ref{fig:brokenline}) that is defined in terms of~$\cL_{\bsigma}$. As we do not want to use limit words, we do not pursue this approach.

Here, we mostly consider Rauzy fractals in the hyperplane~$\bv_n^\perp$, but we can project them of course to other hyperplanes~$\bw^\perp$, see Section~\ref{subsec:dspectrum} below.
Recall that $v \preceq w$ if $v$ is a prefix of~$w$ and that $\bl(w)$ is the abelianized vector of~$w$.

\begin{definition}[$\cS$-adic Rauzy fractal]\label{def:rauzy}
\indx{Rauzy!fractal}\indx{S@$\cS$-adic!Rauzy!fractal}
\indx{subtile}
Let $\bsigma = (\sigma_n)_{n\in\ZZ} \in \cS_d^{\ZZ}$, with $d \ge 2$, be a primitive sequence of unimodular substitutions over the alphabet $\cA$ that admits generalized right and left eigenvectors.
For $n \in \ZZ$, the \emph{Rauzy fractal}~$\cR_n$ (of level~$n$) associated to~$\bsigma$ is defined as 
\[
\cR_n = \overline{\{\pi_n \bl(p) \,:\, p\preceq \sigma_{[n,m)}(b)\ \mbox{for infinitely many}\ m \geq n,\, b \in \cA\}}, \notx{Rauzy}{$\cR_n,\cR_n(\cdot)$}{Rauzy fractal, subtile}
\]
and, for each $w \in \cA^*$, a~\emph{subtile}~$\cR_n(w)$ is defined as 
\begin{equation} \label{eq:RFactor}
\cR_n(w) = \overline{\{\pi_n \bl(p) \,:\, p\,w \preceq \sigma_{[n,m)}(b)\ \mbox{for infinitely many}\ m \geq n,\, b \in \cA\}}.
\end{equation}
\end{definition} 

We clearly have $\cR_n = \bigcup_{a\in\cA} \cR_n(a)$ for all $n \in \ZZ$.
The next lemma states that  the Rauzy fractals are bounded if and only if the corresponding language is balanced. 

\begin{lemma}[{\cite[Lemmas~4.1 and~5.4]{BST:19}}]\label{lem:RFbalanced}
Assume that $\bsigma \in \cS_d^{\ZZ}$ is primitive and has generalized right and left eigenvectors.
Then $\cR_n$ is bounded if and only if $\cL_{\bsigma}^{(n)}$ is balanced. 
In particular, if $\cL_{\bsigma}^{(n)}$ is $C$-balanced, then $\cR_n \subset \pi_n([-C, C]^d \cap \bone^\perp)$.
\end{lemma}

In the stationary setting, the Rauzy fractals are the attractors of a \emph{graph-directed iterated function system} (GIFS for short) associated in a natural way with the substitution (see \cite[Equation~(4.2)]{SirventWang02}). They are thus described by finitely many set equations defined in terms of this GIFS. 
In the nonstationary setting, however, we have an infinite sequence of such set equations. 

\begin{lemma}[{see~\cite[Proposition~5.6]{BST:19}, \cite[Lemma~4.4]{BST:23}}] \label{lem:seteq}
Assume that $\bsigma \in \cS_d^{\ZZ}$ is primitive and has generalized right and left eigenvectors.
Then
\begin{equation} \label{e:setequationkl}
\cR_n(a) = \bigcup_{\substack{p \in\cA^*,\,b\in\cA:\\pa\preceq\sigma_n(b)}} \big(\pi_n \bl(p) +  M_{\sigma_n} \cR_{n+1}(b)\big) \qquad(n \in \ZZ,\, a \in \cA).
\end{equation}
\end{lemma}

Lemma~\ref{lem:seteq} is about sets in~$\bv_n^\perp$. There is a dual statement for the projection on the line~$\RR\bu_n$, and in that case the sequence $\bsigma$ yields a set equation for line segments. To state it, we use the following notation for line segments. For a vector $\bx \in \ZZ^d$, the (closed) line segment connecting $\bx$ and the origin is written as~$\llbracket \bx \rrbracket$\notx{0line}{$\llbracket \cdot \rrbracket$}{line segment}.

\begin{lemma}\label{lem:brokenprojection}
Assume that $\bsigma \in \cS_d^{\ZZ}$ is primitive and has generalized right and left eigenvectors.
Then
\begin{equation} \label{e:dualsetequation}
\tilde{\pi}_n \llbracket \bl(\sigma_n(b)) \rrbracket =
\bigcup_{\substack{p\in\cA^*,\,a\in\cA:\\pa\preceq\sigma_n(b)}} \tilde{\pi}_n \big(\bl(p) + \llbracket \be_a\rrbracket\big) \qquad(n \in \ZZ,\, b\in \cA),
\end{equation}
where the union is disjoint up to the endpoints of the line segments. 
\end{lemma}

\begin{proof}
The line segments $\bl(p) + \llbracket \be_a \rrbracket$ with $pa \preceq \sigma_n(b)$ form a broken line from $\mathbf{0}$ to $\bl(\sigma_n(b))$.
Since $\bu \in \RR_{>0}^d$ and $\bv \in \RR_{\ge0}^d \setminus \{\mathbf{0}\}$, $\tilde{\pi}_n$~projects these line segments to successive line segments of $\RR \bu$ between $\mathbf{0}$ and $\tilde{\pi}_n \bl(\sigma_n(b))$, and their union is $\llbracket \tilde{\pi}_n \bl(\sigma_n(b)) \rrbracket = \tilde{\pi}_n \llbracket \bl(\sigma_n(b)) \rrbracket$.
\end{proof}

Note that, while the union in \eqref{e:dualsetequation} is disjoint in measure, some work and additional conditions are needed to show that \eqref{e:setequationkl} is disjoint up to sets of measure zero; in \cite[Proposition~7.3]{BST:19}, this is proved under the so-called PRICE property, see Section~\ref{sec:suff-cond-tilings}.

Suspending the subtile $\cR_n(a)$ with the line segment $\tilde{\pi}_n \llbracket \be_a \rrbracket$, we obtain Rauzy boxes. 

\begin{definition}[Rauzy box]\label{def:rauzybox}
\indx{Rauzy!box}\indx{S@$\cS$-adic!Rauzy!box}
Let $\bsigma \in \cS_d^{\ZZ}$, with $d\ge 2$, be a primitive sequence of unimodular substitutions over the alphabet $\cA$ that admits generalized right and left eigenvectors. 
For each $n \in \ZZ$, the \emph{Rauzy box} (of level~$n$) is
\[
\hR_n = \bigcup_{a\in\cA} \hR_n(a), \notx{Rauzybox}{$\hat{\cR}_n,\hat{\cR}_n(\cdot)$}{Rauzy box}
\]
with cylinders~$\hR_n(a)$ that are defined as the Minkowski sums
\begin{equation}  \label{e:Rhat}
\hR_n(a) = \tilde{\pi}_n\,\llbracket \be_a\rrbracket - \cR_n(a) \qquad(n \in \ZZ,\, a\in\cA).
\end{equation}
\end{definition} 

Note that the minus sign in the definition of Rauzy boxes is due to the fact that Rauzy fractals will be regarded as duals of the intervals. Rauzy boxes are studied in \cite{BST:19}; they can also be defined for one-sided sequences, by replacing the generalized left eigenvector~$\bv_n$ by an arbitary vector $\bw \in \RR_\ge^d \setminus \{\mathbf{0}\}$.
In the present paper, Rauzy boxes have several applications. 
Indeed, each member of the partition by Rauzy boxes has a pair of horizontal and vertical transverses which allows  the expression of Property~M from Definition~\ref{def:M} in an easy way; see the proof  of Theorem~\ref{thm:MarkovCoarse}. 
Moreover, the tiling property for Rauzy fractals is equivalent to the tiling property for Rauzy boxes; see Proposition~\ref{p:cChC}, working modulo~$\ZZ^d$ will provide a cut-and-stack process; see Figure~\ref{fig:restack1} and Section~\ref{subsec:MarkovRB}. Note that Rauzy boxes generalize classical $L$-shaped constructions for fundamental domains for 2-dimensional lattices such as described in \cite{Arnoux:94}; they are also known as Klotz constructions, see e.g.\ \cite{Kramer:88}.

\subsection{Tilings by Rauzy fractals and Rauzy boxes, and restacking} \label{subsec:tilings}
Let $\bsigma \in \cS_d^{\ZZ}$, with $d\ge 2$, be a primitive sequence of unimodular substitutions over the alphabet~$\cA$ with generalized right and left eigenvectors $\bu$ and~$\bv$ respectively. 
In this section, we recall fundamental tiling properties of the $\cS$-adic Rauzy boxes and Rauzy fractals.
Our main interest in this paper lies in the collection
\begin{equation} \label{e:hC}
\hC_n = \{\bx + \hR_n(a)  \,:\, \bx \in \ZZ^d,\, a \in \cA\} \qquad (n \in \ZZ) \notx{Cb}{$\hat{\cC}_n,\hat{\cC}_n^\gen$}{collection of Rauzy boxes}
\end{equation}
of Rauzy boxes. 
Recall that a \emph{tiling}\indx{tiling} of~$\RR^d$ is a collection of closed sets (so-called ``tiles'') that covers~$\RR^d$ in a way that the intersection of any two distinct tiles has Lebesgue measure~$0$.
If $\hC_n$ forms a tiling of~$\RR^d$, then $\hR_n$ is a fundamental domain of $\RR^d/\ZZ^d$. 
The following \emph{tiling condition} will be important for us.

\begin{definition}[Tiling condition] \label{def:tilingcond}\indx{tiling!condition}
Let $\bsigma \in \cS_d^{\ZZ}$, with $d \ge 2$.
If $\bsigma$ is primitive, admits generalized right and left eigenvectors, and the collection~$\hC_n$ tiles~$\RR^d$ for all $n \in \ZZ$, then we say that the \emph{tiling condition} holds for~$\bsigma$.
\end{definition}

We will see in Proposition~\ref{p:cChC} below that the tiling property of~$\hC_n$ is equivalent to a tiling property of the collection of Rauzy fractals \begin{equation} \label{e:defcCn}
\cC_n = \{ \pi_n \bx+  \cR_n(a) \,:\, \bx \in \ZZ^d,\, a \in \cA,\, 0 \le \langle \bv_n,\bx\rangle < \langle \bv_n,\be_a\rangle\}.
\notx{C}{$\cC_n,\cC_n^{\bw}$}{collection of Rauzy fractals, in~$\bw^\perp$}
\end{equation}
The elements of $\cC_n$ are subsets of the hyperplane~$\bv_n^\perp$, and a tiling of a hyperplane~$\bw^\perp$ is defined in the same way as a tiling of~$\RR^d$ but with the intersections of tiles having zero $(d{-}1)$-dimensional Lebesgue measure. 
The set of pairs $\{ (\bx,a) \in \ZZ^d{\times}\cA : 0 \le \langle \bv_n,\bx\rangle < \langle \bv_n,\be_a\rangle\}$ is called the \emph{discrete hyperplane} approximating the hyperplane~$\bv_n^\perp$. 
This set is aperiodic whenever the coordinates of~$\bv_n$ are rationally independent.
Note that $0 \le \langle \bv_n,\bx\rangle \le \langle \bv_n,\be_a\rangle$ is equivalent to $\tilde{\pi}_n \bx \in \tilde{\pi}_n \llbracket \be_a\rrbracket$.

In the tiling properties, we do not assume boundedness of the Rauzy fractals. However, if we assume in addition that the language~$\cL_{\bsigma}$ is balanced, Lemma~\ref{lem:RFbalanced} yields boundedness, hence compactness of the Rauzy fractals~$\cR_n$ and of the Rauzy boxes~$\hR_n$.
Although later we will require balancedness in some of our results, we did not include it in the definition of the tiling condition because there exist interesting ``unbalanced'' examples, whose Rauzy fractals deserve to be studied; see for instance \cite{Cassaigne-Ferenczi-Zamboni:00,Andrieu:21}.

The following proposition gives a condition for the tiling property of~$\hC_n$ being independent of~$n$. 
For the proof, we define subtiles of $\cR_n(a)$ and~$\hR_n(a)$ according to~\eqref{e:setequationkl}.
Let
\begin{equation} \label{e:En}
E_n = \big\{ (a,p,b) \in \cA \times \cA^* \times \cA \,:\, pa \preceq \sigma_n(b) \big\} \qquad (n \in \ZZ) \notx{En}{$E_n$}{set of edges of a Bratteli diagram or nonstationary edge shift, prefixes of a substitution}
\end{equation}
be the set of prefixes of the substitution~$\sigma_n$ and define, for $n \in \ZZ$, $(a,p,b) \in E_n$,
\begin{equation} \label{e:Re}
\cR_n(a,p,b) = \pi_n\bl(p) + M_{\sigma_n} \cR_{n+1}(b), \quad
\hR_n(a,p,b) = \tilde{\pi}_n\llbracket \be_a\rrbracket - \cR_n(a,p,b). \notx{Rauzy}{$\cR_n,\cR_n(\cdot)$}{Rauzy fractal, subtile} \notx{Rauzybox}{$\hat{\cR}_n,\hat{\cR}_n(\cdot)$}{Rauzy box}
\end{equation}
By \eqref{e:setequationkl}, we have, for each $n \in \ZZ$, $a \in \cA$,
\begin{equation} \label{e:Rapb}
\bigcup_{\substack{p\in\cA^*,\,b\in\cA:\\(a,p,b)\in E_n}} \cR_n(a,p,b) = \cR_n(a) \quad \mbox{and} \quad \bigcup_{\substack{p\in\cA^*,\,b\in\cA:\\(a,p,b)\in E_n}} \hR_n(a,p,b) = \hR_n(a).
\end{equation} 
We set
\begin{equation} \label{e:Cgen}
\hC_n^\gen = \big\{\bx + \hR_n(a,p,b) \,:\, \bx \in \ZZ^d,\, (a,p,b) \in E_n \big\} \notx{Cb}{$\hat{\cC}_n,\hat{\cC}_n^\gen$}{collection of Rauzy boxes}
\end{equation}
because this collection will define generating Markov partitions in Section~\ref{subsec:refined}.

\begin{proposition} \label{p:ntilingbox}
Assume that $\bsigma \in \cS_d^{\ZZ}$ is primitive and has generalized right and left eigenvectors. For each $n \in \ZZ$, the following are equivalent.
\begin{enumerate}[\upshape (i)]
\itemsep.5ex
\item
$\hC_{n+1}$ forms a tiling of~$\RR^d$,
\item
$\hC_n^\gen$ forms a tiling of~$\RR^d$,
\item
$\hC_n$ forms a tiling of~$\RR^d$ and the union in \eqref{e:setequationkl} is disjoint in measure for all $a \in \cA$.
\end{enumerate}
\end{proposition}

\begin{proof}
Let $n \in \ZZ$.
By Lemma~\ref{lem:brokenprojection} and since $M_{\sigma_n} \tilde{\pi}_{n+1} =  \tilde{\pi}_n M_{\sigma_n}$ (cf.\ \cite[Lemma~5.2]{BST:19}), we have, for each $b \in \cA$, the topological partitions
\begin{align}
M_{\sigma_n} \tilde{\pi}_{n+1} \llbracket \be_b \rrbracket & = \tilde{\pi}_n M_{\sigma_n} \llbracket M_{\sigma_n} \be_b \rrbracket = \tilde{\pi}_n \llbracket \bl(\sigma_n(b)) \rrbracket =
\hspace{-1em} \bigcup_{\substack{p\in\cA^*,\,a\in\cA:\\(a,p,b)\in E_n}} \hspace{-1em} \tilde{\pi}_n \big(\bl(p) {+} \llbracket \be_a\rrbracket\big), \label{e:Mneb} \\
M_{\sigma_n} \hR_{n+1}(b) & = \tilde{\pi}_n \llbracket \bl(\sigma_n(b))\rrbracket - M_{\sigma_n} \cR_{n+1}(b) \nonumber \\
& = \bigcup_{\substack{p\in\cA^*,\,a\in\cA:\\(a,p,b)\in E_n}} \big(\bl(p) {+} \underbrace{\tilde{\pi}_n \llbracket \be_a\rrbracket {-} \pi_n \bl(p) {-} M_{\sigma_n} \cR_{n+1}(b)}_{\displaystyle\hR_n(a,p,b)}\big). \label{e:MnRn}
\end{align}

Since $M_{\sigma_n}$ is unimodular, the collection~$\hC_{n+1}$ forms a tiling if and only if $\{\bx {+} M_{\sigma_n} \cR_{n+1}(b) : \bx \in \ZZ^d,\, b \in \cA\}$ forms a tiling, and this is equivalent to $\hC_n^\gen$ being a tiling because the union in \eqref{e:MnRn} is disjoint in measure for each $b \in \cA$.
The collection~$\hC_n^\gen$ forms a tiling if and only if $\hC_n$ forms a tiling and the union in \eqref{e:Rapb} is disjoint in measure for all $a \in \cA$, i.e., the union in \eqref{e:setequationkl} is disjoint in measure for all $a \in \cA$.
\end{proof}

The proof of Proposition~\ref{p:ntilingbox} is of particular interest because it uses a restacking procedure \indx{restacking} that will be of importance for the study of Markov partitions in Chapter~\ref{chapter:markov}.
This is illustrated in Figure~\ref{fig:restack1} for a sequence of Brun substitutions.
More precisely, Figure~\ref{fig:restack1}~(i) shows the partition $\{\hR_{n+1}(b) : b \in \cA\}$ of the Rauzy box~$\hR_{n+1}$.
Multiplying by the incidence matrix~$M_{\sigma_n}$, we obtain~(ii), and subdividing according to \eqref{e:MnRn} gives~(iii). (In our example, only one cylinder is subdivided into two subcylinders.)
Figure~\ref{fig:restack1}~(iv) is obtained from~(iii) by restacking, i.e., by subtracting the integer vector $\bl(p)$ from the subcylinder $\bl(p) {+} \hR_n(a,p,b)$, for each $(a,p,b) \in E_n$. (In our example, only one vector $\bl(p)$ is nonzero and thus only one subcylinder is restacked.) 
Finally, merging the subcylinders $\hR_n(a,p,b)$ with same~$a$ gives the partition $\{\hR_n(a) : a \in \cA\}$ of the Rauzy box~$\hR_n$, which is depicted in Figure~\ref{fig:restack1}~(v).
Note that the generalized right and left eigenvectors $\bu_n$ and~$\bv_n$, which define the hyperplane $\bv_n^\perp$ containing~$\cR_n$ and the direction~$\bu_n$ of the ``sides'' of the Rauzy box~$\hR_n$, are different from $\bu_{n+1}$ and~$\bv_{n+1}$: $\bu_n = M_{\sigma_n} \bu_{n+1}$, $\bv_n = \tr{\!M}_{\sigma_n} \bv_{n+1}$. 
However, both Rauzy boxes $\hR_n$ and~$\hR_{n+1}$ as well as $M_{\sigma_n} \hR_{n+1}$ are fundamental domains of $\TT^d = \RR^d/\ZZ^d$. 

\begin{figure}[ht]
\hspace{-2cm} \includegraphics{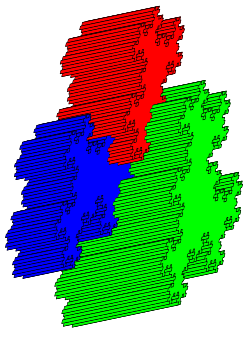} 
\includegraphics{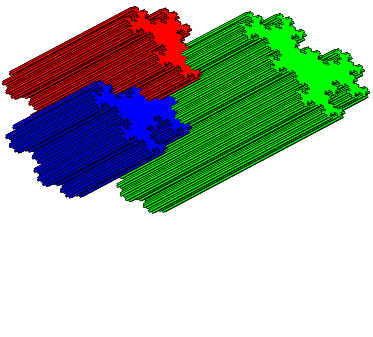} \\
\hspace{-8.5cm} (i) \\
\vspace{-2.5cm} (ii) \\
\vspace{-.75cm} \hspace{6cm} \includegraphics{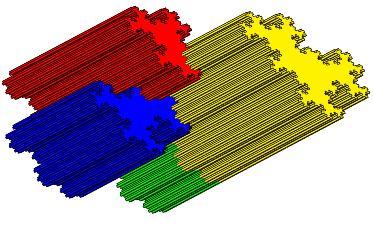} \\
\vspace{-.25cm} \hspace{6cm} (iii) \\
\vspace{-.75cm} \hspace{-3.5cm}\includegraphics{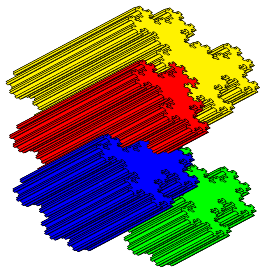}
\includegraphics{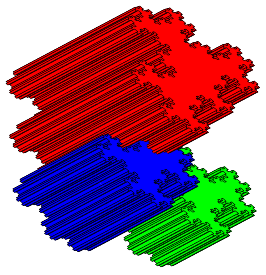} \\
\hspace{-3.5cm} (iv) \hspace{4cm} (v) \\
\caption{The restacking process described after Proposition~\ref{p:ntilingbox}, illustrated by the unordered Brun substitution $\sigma_n = \sigma_{\rU,12}$, with Rauzy boxes given by the periodic sequence of substitutions $\bsigma = (\sigma_{\rU,12}, \sigma_{\rU,23}, \sigma_{\rU,31})^{\ZZ}$. 
We see $\bigcup_{b\in\cA} \hR_{n+1}(b)$ in~(i), $\bigcup_{b\in\cA} M_{\sigma_n} \hR_{n+1}(b)$ in~(ii), $\bigcup_{(a,p,b)\in E_n} (\bl(p){+}\hR_n(a,p,b))$ in~(iii), $\bigcup_{(a,p,b)\in E_n} \hR_n(a,p,b)$ in~(iv), and $\bigcup_{a\in\cA} \hR_n(a)$ in~(v). Since $\sigma_n(1) = 1$, $\sigma_n(2) = 21$, $\sigma_n(3) = 3$, we have $E_n = \{(1,\epsilon,1), (2,\epsilon,2),$ $(1,2,2), (3,\epsilon,3)\}$, where $\epsilon$ denotes the empty word.}
 \label{fig:restack1}
\end{figure}

We next prove that the tiling properties of $\cC_n$ and $\hC_n$ are equivalent.
This result is proved in \cite[Proposition~7.6]{BST:19} for compact tiles and goes back to~\cite{Ito-Rao:06} in the substitutive setting.

\begin{proposition} \label{p:cChC}
Assume that $\bsigma \in \cS_d^{\ZZ}$ is primitive and has generalized right and left eigenvectors. For each $n \in \ZZ$, the collection~$\cC_n$ forms a tiling of~$\bv_n^\perp$ if and only if $\hC_n$ forms a tiling of~$\RR^d$.
\end{proposition}

\begin{proof}
For all $\bx,\bz \in \ZZ^d$, $a \in \cA$, we have
\[
(\bx {+} \tilde{\pi} \llbracket\be_a\rrbracket) \cap (\bz {+} \bv_n^\perp) = 
\begin{cases} \{\pi_n \bx {+} \tilde{\pi}_n \bz\} & \mbox{if}\ \langle \bv_n,\bx\rangle \le \langle \bv_n,\bz\rangle \le \langle \bv_n,\bx{+}\be_a\rangle, \\ \emptyset & \mbox{otherwise.} \end{cases}
\]
Since the second inequality in the defintion of~$\cC_n$ in~\eqref{e:defcCn} is strict, we slightly modify the definition of the (collection of) Rauzy boxes by setting
\[
\hC'_n = \{\bx{+}\hR'_n(a) : \bx\in\ZZ^d,\, a\in\cA\}, \quad \mbox{with} \quad \hR'_n(a) = \hR_n(a) \setminus (\tilde{\pi}_n\be_a{-}\cR_n(a)).
\]
Then we have
\[
(\bx {+} \hR'_n(a)) \cap (\bz {+} \bv_n^\perp) = \begin{cases}\tilde{\pi}_n \bz {+} \pi_n \bx {-} \cR_n(a) & \mbox{if}\ 0 \le \langle \bv_n,\bz{-}\bx\rangle < \langle \bv_n,\be_a\rangle, \\ \emptyset & \mbox{otherwise.} \end{cases}
\]
Since $\tilde{\pi}_n \bz {+} \pi_n \bx {+} \cR_n(a) = \bz - (\pi_n (\bz {-} \bx) {+} \cR_n(a))$, this means that the intersection of $\hC'_n$ with $\bz {+} \bv_n^\perp$ is the translation of~$-\cC_n$ by~$\bz$.

Assume first that $\cC_n$ is not a tiling of~$\bv_n^\perp$.  
If $(\pi_n\bx{-}\cR_n(a)) \cap (\pi_n\by{-}\cR_n(b))$ has positive measure for distinct $(-\bx,a), (-\by,b)$ in the discrete hyperplane approximating~$\bv_n^\perp$, then we use the fact that the line segment $(\tilde{\pi}_n (\bx {+} \llbracket \be_a \rrbracket)) \cap (\tilde{\pi}_n (\by {+} \llbracket \be_b \rrbracket))$ has positive length (since $\max\{\langle \bv_n, \bx \rangle, \langle \bv_n, \by \rangle\} \le 0 < \min\{\langle \bv_n, \bx {+} \be_a \rangle, \langle \bv_n, \by {+} \be_b \rangle\}$) to obtain that $(\bx {+} \hR_n(a)) \cap (\by {+}\hR_n(b))$ has positive measure.
If $\cC_n$ does not cover a point $\bz \in \bv_n^\perp$, then $\hC'_n$ does not cover~$\bz$ and thus there exists $\varepsilon > 0$ such that $\hC_n$ does not cover $\bz {+} \varepsilon \bu_n$. 
In both cases, the collection $\hC_n$ is not a tiling of~$\RR^d$.

Assume now that $\hC_n$ is not a tiling of~$\RR^d$.  
Assume that $(\bx{+}\hR'_n(a)) \cap (\by{+}\hR'_n(b))$ has positive measure for distinct $(\bx,a), (\by,b) \in \ZZ^d {\times} \cA$.
The set $\{\langle \bv_n,\bz \rangle : \bz \in \ZZ^d\}$ is either dense in~$\RR$ or equal to $c \ZZ$ for some $c \in \RR$.
In both cases, we have some $\bz \in \ZZ^d$ such that $\max\{\langle \bv_n, \bx \rangle, \langle \bv_n, \by \rangle\} \le \langle \bv_n, \bz \rangle < \min\{\langle \bv_n, \bx {+} \be_a \rangle, \langle \bv_n, \by {+} \be_b \rangle\}$, thus $(\pi_n (\bz {-} \bx) {+ }\cR_n(a)) \cap (\pi_n (\bz {-} \by) {+} \cR_n(b))$ has positive measure and both $(\bz {-} \bx, a)$, $(\bz {-} \bx, b)$ are in the discrete hyperplane approximating~$\bv_n^\perp$.
Similarly, if $\hC_n$ does not cover~$\RR^d$, then $\hC'_n$ does not cover $\bz {+} \bv_n^\perp$ for some $\bz \in \ZZ^d$ and thus~$\cC_n$ does not cover~$\bv_n^\perp$. 
In both cases, the collection $\cC_n$ is not a tiling of~$\bv_n^\perp$.
\end{proof}

\subsection{Tiling implies pure discrete spectrum} \label{subsec:dspectrum}
Now, we show that the tiling condition implies that $(X_{\bsigma}^{(n)},\Sigma)$ is measurably conjugate to a rotation on the torus~$\TT^{d-1}$ and, hence, that it has pure discrete spectrum.

More precisely, we use that the projections of the Rauzy fractals~$\cR_n(a)$ to the hyperplane~$\bone^\perp$ tile this hyperplane periodically.
More generally, for a sequence $\bsigma \in \cS_d^\ZZ$ admitting a generalized right eigenvector and, for $n \in \ZZ$ and $\bw \in \RR_{>0}^d$, we consider the collection of tiles
\[
\cC_n^{\bw} = \{\pi_{\bu_n,\bw}\bx + \cR_n^\bw(a) \,:\, \bx \in \ZZ^d,\, a \in \cA,\, 0 \le \langle \bw,\bx\rangle < \langle \bw,\be_a\rangle\} \notx{C}{$\cC_n,\cC_n^{\bw}$}{collection of Rauzy fractals, in~$\bw^\perp$}
\]
in~$\bw^\perp$, where 
\[
\cR_n^\bw(a) = \overline{\{\pi_{\bu_n,\bw} \bl(p) \,:\, pa \preceq \sigma_{[n,m)}(b)\ \mbox{for infinitely many}\ m \geq n,\, b \in \cA\}}. \notx{Rauzy}{$\cR_n^{\bw}, \cR_n^{\bw}(\cdot)$}{Rauzy fractal in $\bw^\perp$}
\]
If $\bw$ is the generalized left eigenvector~$\bv_n$, then we recover the collection~$\cC_n$ of~\eqref{e:defcCn}.
Choosing $\bw = \bone$ gives a periodic collection of tiles because, for $\bx \in \ZZ^d$, $a \in \cA$, $0 \le \langle \bone,\bx\rangle < \langle \bone,\be_a\rangle$ is equivalent to $\langle \bone,\bx\rangle  = 0$, and $\ZZ^d \cap \bone^\perp$ is a lattice of~$\bone^\perp$ spanned by the vectors $\be_b{-}\be_a$, $b \in \cA \setminus \{a\}$. 
We also set $\cR_n^\bw = \bigcup_{a\in\cA} \cR_n^\bw(a)$. Moreover, we define
\[
\hC_n^\bw = \{\bx + \hR_n^\bw(a)  \,:\, \bx \in \ZZ^d,\, a \in \cA\}, \notx{Cb}{$\hat{\cC}_n^{\bw}$}{collection of Rauzy boxes w.r.t.~$\bw$} \quad \mbox{with} \quad \hR_n^\bw(a) = \tilde{\pi}_{\bu_n,\bw}\,\llbracket \be_a\rrbracket - \cR_n^\bw(a), \notx{Rauzyboxw}{$\hat{\cR}_n^\bw, \hat{\cR}_n^\bw(\cdot)$}{Rauzy box w.r.t.~$\bw$}
\]
and we set $\hR_n^\bw = \bigcup_{a\in\cA} \hR_n^\bw(a)$.
We will see in Proposition~\ref{p:tilingw} that the tiling property of~$\cC_n^\bw$ does not depend~on~$\bw$. 
In particular, $\cC_n^\bone$ is a tiling if and only if $\cC_n$ is a tiling (and $\bsigma$ admits a generalized left eigenvector).

For $\balpha \in \RR^{d-1}$, define the rotation
\[
\fr_{\balpha}:\, \TT^{d-1} \to \TT^{d-1},\ \bz \mapsto \bz + \balpha. \notx{rot}{$\fr_{\balpha}, \fr_{\bx}$}{toral rotation}
\]

\begin{theorem} \label{t:tilingpds}
Let $\bsigma \in \cS_d^{\ZZ}$, with $d \ge 2$, be a primitive sequence of unimodular substitutions that admits a generalized right eigenvector~$\bu$, and let $n \in \ZZ$. 
If $\cC_n^{\bone}$ forms a tiling of~$\bone^\perp$, then the uniquely ergodic shift $(X_{\bsigma}^{(n)},\Sigma)$ is measurably conjugate to the minimal rotation $(\TT^{d-1}, \fr_{(\alpha_1,\dots,\alpha_{d-1})})$, with $(\alpha_1,\dots,\alpha_{d-1},1{-}\sum_{i=1}^{d-1}\alpha_i) \in \RR_{\ge0} \bu_n $, and  $(X_{\bsigma}^{(n)},\Sigma)$ has pure discrete spectrum. 
\end{theorem}

\begin{proof}
We follow partially \cite[Section~8]{BST:19}.
By Lemma~\ref{l:uniquelyergodic}, $(X_{\bsigma}^{(n)}, \Sigma)$ is minimal and uniquely ergodic. 

Next we establish the measurable conjugacy of $(X_{\bsigma}^{(n)},\Sigma,\mu)$ to the domain exchange\indx{domain exchange} $(\cR_n,\mathfrak{h}_n,\lambda_{\bone})$ with
\begin{equation} \label{e:Hn}
\mathfrak{h}_n:\ \cR_n^\bone \to \cR_n^\bone, \quad \bz \mapsto \bz + \pi_{\bu_n,\bone}\, \be_a \quad \mbox{if}\ \bx \in \cR_n^\bone(a).
\end{equation}
and $\lambda_{\bone}$\notx{lambda}{$\lambda,\lambda_{\bone}$}{Lebesgue measure} the Lebesgue measure on~$\bone^\perp$, normalized such that the lattice $\ZZ^d \cap \bone^\perp$ has covolume~$1$. 
The mapping~$\mathfrak{h}_n$ is well defined a.e.\ on~$\cR_n^\bone$ because $\cC_n^\bone$ forms a tiling. 
By the definition of~$\cR_n^\bone(a)$, we have that, up to measure zero,
\[
\mathfrak{h}_n(\cR_n^\bone(a)) = \overline{\{\pi_{\bu_n,\bone} \bl(pa) \,:\, pa \preceq \sigma_{[n,m)}(b)\ \mbox{for infinitely many}\ m \geq n,\, b \in \cA\}}.
\]
Therefore, $\mathfrak{h}_n$ is a piecewise isometry with $\bigcup_{a\in\cA} \mathfrak{h}_n(\cR_n^\bone(a)) = \cR_n^\bone$, hence bijective up to measure zero, and its invariant measure is~$\lambda_{\bone}$. 
The representation map 
\begin{equation}\label{eq:psin}
\psi_n:\, \cR_n^\bone \to X_{\bsigma}^{(n)}, \ \bz \mapsto (\omega_k)_{k\in\ZZ} \ \mbox{such that}\ \mathfrak{h}_n^k(\bz) \in \cR_n^\bone(\omega_k)\ \mbox{for all}\ k \in \ZZ
\end{equation}
is well defined outside the set 
\[
D_n = \bigcup_{a,b\in\cA,\by\in\ZZ^d} \big(\pi_{\bu_n,\bone}\by + (\cR_n^\bone(a) \cap \cR_n^\bone(b))\big),
\]
which has zero Lebesgue measure because $\cC_n^\bone$ is a tiling.
On $\cR_n^\bone \setminus D_n$, we obviously have $\psi_n \circ \mathfrak{h}_n = \Sigma \circ \psi_n$.
Therefore, $\lambda_{\bone} \circ \psi_n^{-1}$ is an invariant measure of $(X_{\bsigma}^{(n)}, \Sigma)$ and thus equal to the unique invariant measure~$\mu$. 

Since $\pi_{\bu_n,\bone} \be_a \equiv \pi_{\bu_n,\bone} \be_b \bmod{\ZZ^d \cap \bone^\perp}$ for all $a,b \in \cA$, this implies that the domain exchange $(\cR_n^\bone,\mathfrak{h}_n,\lambda_{\bone})$ is measurably conjugate to the rotation by $\pi_{\bu_n,\bone} \be_d$ on $\bone^\perp / (\ZZ^d \cap \bone^\perp)$.
Since $\pi_{\bu_n,\bone} \be_d = \be_d {-} \bu_n/\lVert\bu_n\rVert_1 = (-\alpha_1,\dots,-\alpha_{d-1},\sum_{i=1}^{d-1}\alpha_i)$, this rotation is conjugate to the rotation by $(\alpha_1,\dots,\alpha_{d-1})$ on~$\TT^{d-1}$, via the map $(x_1,\dots,x_d) \mapsto (-x_1,\dots,-x_{d-1})$. 

Moreover, $\bu_n$ has rationally independent coordinates because otherwise the Rauzy fractal~$\cR_n^\bone$ would lie in a countable union of $(d{-}2)$-dimensional subspaces of the $(d{-}1)$-dimensional space~$\bone^\perp$, contradicting the tiling property. 
Therefore, the rotation is minimal. 
\end{proof}

It is conjectured that pure discrete spectrum always holds under the conditions of Proposition~\ref{prop:sadic1}. This hard and unsolved problem is known as the \emph{$\cS$-adic Pisot conjecture}\indx{Pisot!S@$\cS$-adic conjecture}\indx{S@$\cS$-adic!Pisot conjecture}; see Conjecture~\ref{c:SadicPisotconj}.

The tiling property for~$\cC_n^\bone$ can be replaced by that for~$\cC_n$ due to the following proposition.
Moreover, this proposition shows that the tiling property of~$\hC_n$ depends only on the future $(\sigma_n,\sigma_{n+1},\dots)$ of the sequence $\bsigma = (\sigma_n)_{n\in\ZZ}$. 

\begin{proposition} \label{p:tilingw}
Let $\bsigma \in \cS_d^{\ZZ}$, with $d \ge 2$, be a primitive sequence of unimodular substitutions that admits a generalized right eigenvector, let $n \in \ZZ$ and $\bw \in \RR_{>0}^d$.
Then the following are equivalent.
\begin{enumerate}[\upshape (i)]
\itemsep.5ex
\item \label{it:w1}
$\cC_n^\bw$ forms a tiling of $\bw^\perp$. 
\item \label{it:w2}
$\hC_n^\bw$ forms a tiling of $\RR^d$. 
\item \label{it:w3}
$\cC_n^\bone$ forms a tiling of $\bone^\perp$. 
\end{enumerate}
\end{proposition}

\begin{proof}
First note that the projections $\pi_{\bu_n,\bw}$ are well defined since \mbox{$\langle\bw,\bu_n\rangle > 0$}.
Then the equivalence between (\ref{it:w1}) and~(\ref{it:w2}) follows directly from Proposition~\ref{p:cChC} and its proof; it suffices to replace $\bv_n$ by~$\bw$. 
Therefore, we only have to show that the tiling property of~$\hC_n^\bw$ does not depend on~$\bw$ (for $\bw \in \RR_{>0}^d$).

Let $\bz \in \cR_n^\bone$.
Since $\cR_n^\bone = \bigcup_{a\in\cA} (\pi_{\bu_n,\bone}\bl(a){+}\cR_n^\bone(a))$ (see the proof of Theorem~\ref{t:tilingpds}), we can find a sequence $(\omega_k)_{k\in\ZZ} \in \cA^\ZZ$ such~that 
\[
\bz + \pi_{\bu_n,\bone}\, \by_k \in \cR_n^\bone(\omega_k) \ \mbox{for all}\ k \in \ZZ, \quad \mbox{with}\ 
\by_k = \begin{cases}\bl(\omega_0\cdots\omega_{k-1}) & \mbox{if}\ k \ge 0, \\ -\bl(\omega_k\cdots\omega_{-1}) & \mbox{if}\ k < 0.\end{cases}
\] 
Since $\cR_n^\bw = \pi_{\bu_n,\bw} \cR_n^\bone$, we get that $\pi_{\bu_n,\bw}(\bz{+}\by_k) \in \cR_n^\bw(\omega_k)$, thus
\[
\begin{gathered}
\tilde{\pi}_{\bu_n,\bw}(\by_k {+} \llbracket\be_{\omega_k}\rrbracket) - \pi_{\bu_n,\bw} \bz \subset \by_k + \hR_n^\bw(\omega_k) \quad \mbox{for all}\ k \in \ZZ, \\
\RR \bu_n - \bz= \bigcup_{k\in\ZZ} \big(\tilde{\pi}_{\bu_n,\bw}(\by_k {+} \llbracket\be_{\omega_k}\rrbracket) - \pi_{\bu_n,\bw} \bz \big) \subset \bigcup_{k\in\ZZ} \big(\by_k + \hR_n^\bw(\omega_k)\big). 
\end{gathered}
\]
Note that changing~$\bw$ does not change the decomposition of the line, only the length of the line segments. 
We obtain that $\ZZ^d {+} \hR_n^\bw = \ZZ^d {+} \RR \bu_n {-} \cR_n^\bone$. 
Since this set does not depend on~$\bw$, the covering property of $\hC_n^\bw$ does not depend on~$\bw$. 

Assume now that $\hC_n^\bw$ is a covering. 
Then $\hC_n^\bw$ is not a tiling if and only if there is a subset of positive measure of one tile that has an intersection of positive measure with another tile of~$\hC_n^\bw$. 
By $\ZZ^d$-invariance, we can assume that this set is contained in $\RR \bu_n {-} \cR_n^\bone$.
Therefore, we consider w.l.o.g.\ a subset $\tilde{\pi}_{\bu_n,\bw} (\bx {+} \llbracket \be_a\rrbracket) {-} \pi_{\bu_n,\bw} Z$ of $\bx {+} \hR_n^\bw(a)$ with $(\bx,a) \in \ZZ^d {\times} \cA$, $Z \subset \cR_n^\bone(a)$ and $\lambda_{\bone}(Z) > 0$.
Then, we can, similarly to the previous paragraph, find a sequence $(\omega_k)_{k\in\ZZ} \in \cA^\ZZ$ such~that 
\[
\mathcal{\lambda}_{\bone} \big((\pi_{\bu_n,\bone} \by_k {+} Z) \cap \cR_n^\bone(\omega_k)\big) > 0\ \mbox{for all}\ k \in \ZZ, \ \mbox{with}\ 
\by_k = \begin{cases}\bl(\omega_0\cdots\omega_{k-1}) & \hspace{-.5em}\mbox{if}\ k \ge 0, \\ -\bl(\omega_k\cdots\omega_{-1}) & \hspace{-.5em}\mbox{if}\ k < 0,\end{cases}
\]
and we have, for all $k \in \ZZ$,
\[
\tilde{\pi}_{\bu_n,\bw}(\by_k {+} \llbracket\be_{\omega_k}\rrbracket) - \pi_{\bu_n,\bw}\big(Z \cap (\cR_n^\bone(\omega_k) {-} \pi_{\bu_n,\bone} \by_k)\big) \subset \by_k + \hR_n^\bw(\omega_k).
\]
Since $\by_k {+} \hR_n^\bw(\omega_k) \subset \by_k {+} \llbracket\be_{\omega_k}\rrbracket {+} \bw^\perp$ for all $k \in \ZZ$, the intersection of two distinct tiles in $\{\by_k {+} \hR_n^\bw(\omega_k) : k \in \ZZ\}$ has zero measure. 
Therefore, we can assume w.l.o.g.\ that $(\bx,a) \notin \{(\by_k,\omega_k) : k \in \ZZ\}$.
Now, changing $\bw$ does not change the property that $\tilde{\pi}_{\bu_n,\bw} (\bx {+} \llbracket \be_a\rrbracket) {-} \pi_{\bu_n,\bw} Z$ is a subset of $\bx {+} \hR_n^\bw(\ba)$ and has intersection of positive measure with $\by_k {+} \hR_n^\bw(\omega_k)$ for some $k \in \ZZ$ (with $k$ depending on~$\bw$); here, we use that the coordinates of $\bw$ and thus the lengths of $\tilde{\pi}_{\bu_n,\bw} \llbracket\be_a\rrbracket$, $a \in \cA$, are positive for all $\bw \in \RR_{>0}^d$.
This proves that the tiling property of~$\hC_n^\bw$ does not depend on~$\bw$; in particular, it holds for $\bw = \bone$ if and only if it holds for some~$\bw \in \RR_{>0}^d$. 
\end{proof}

A~similar result to Proposition~\ref{p:tilingw} can be found in \cite[Proposition~7.5]{BST:19} for $\bw \in \RR_{\ge0}^d \setminus \{\mathbf{0}\}$, with some additional conditions and a completely different proof. 
In our case, it is not difficult to see that the measure of the symmetric difference of $\hR_n^\bw$ and $\hR_n^{\bw'}$ is small when $\bw, \bw' \in \RR_{\ge0}^d \setminus \{\mathbf{0}\}$ are close.
This implies that $\hC_n^\bw$ is a tiling even for $\bw \in \RR_{\ge0}^d \setminus \{\mathbf{0}\}$ when $\cC_n^\bone$ is a tiling, but we do not know if the converse holds.

Theorem~\ref{t:tilingpds} and Proposition~\ref{p:tilingw} give the following corollary.

\begin{corollary} \label{c:tilingpds}
Let $\bsigma \in \cS_d^{\ZZ}$, with $d \ge 2$, be a primitive sequence of unimodular substitutions that admits generalized left and right eigenvectors.
If $\hC_n$ forms a tiling of~$\RR^d$, $n \in \ZZ$, then the uniquely ergodic shift $(X_{\bsigma}^{(n)},\Sigma)$ is measurably conjugate to the minimal rotation $(\TT^{d-1}, \fr_{\alpha_1,\dots,\alpha_{d-1})})$, with $(\alpha_1,\dots,\alpha_{d-1},1{-}\sum_{i=1}^{d-1}\alpha_i) \in \RR_{\ge0} \bu_n $, and  $(X_{\bsigma}^{(n)},\Sigma)$ has pure discrete spectrum.  
\end{corollary}

\subsection{Sufficient conditions for tilings and property PRICE} \label{sec:suff-cond-tilings}
There are several criteria for tiling that can be expressed in terms of so-called \emph{coincidence conditions} that exist in the usual Pisot framework; see e.g.\  \cite{AL11,AkiBBLS}.
In \cite{BST:19,BST:23}, sufficient conditions for the collections $\cC_n$ to be tilings are given.
These conditions require the so-called property PRICE, for which we introduce a two-sided version that implies the one-sided property PRICE of \cite[Definition~4.2]{BST:23} but is simpler to state. 
We use the notation 
\begin{equation}\label{eq_BC}
B_C^{(n)} = \{\bsigma \in \cS_d^{\ZZ} \,:\, \cL_{\bsigma}^{(n)}\ \mbox{is $C$-balanced}\} \qquad (n \in \ZZ). \notx{BC}{$B_C^{(n)}$}{set of balanced sequences of substitutions}
\end{equation} 
Note that 
\begin{equation}\label{eq:BCkn}
\Sigma^kB_C^{(n)}= B_C^{(n-k)} \qquad (k,n\in\ZZ).
\end{equation}

\begin{definition}[Two-sided property PRICE]\label{def:PRICE}\indx{PRICE}
Let $d\ge 2$ and fix $n\in\ZZ$. We say that a~sequence $\bsigma \,{=}\, (\sigma_m)_{m\in\ZZ} \,{\in}\, \cS_d^{\ZZ}$ has the \emph{two-sided property PRICE (at level~$n$)} if there are strictly increasing sequences of positive integers $(\ell_k)_{k\in\NN}$ and of integers $(r_k)_{k\in\NN}$ such that the following conditions hold.
\begin{itemize}
\labitem{(P)}{defP}
There exists $h \in \NN$ and a positive matrix~$M'$ such that $M_{\sigma_{[\ell_k-h,\ell_k)}} = M'$ for all $k \in \NN$.
\labitem{(R)}{defR}
We have $(\sigma_{r_k-\ell_k}, \dots,\sigma_{r_k+\ell_k-1}) = (\sigma_{n-\ell_k}, \dots,\sigma_{n+\ell_k-1})$ for all $k \in \NN$.
\labitem{(I)}{defI}
The sequence~$\bsigma$ is algebraically irreducible. 
\labitem{(C)}{defC}
There exists $C > 0$ such that $\bsigma \in B_C^{(r_k+\ell_k)}$ for all $k\in\NN$.
\labitem{(E)}{defE}
$\bsigma$ admits a generalized left eigenvector $\bv$.
\end{itemize}
\end{definition}

We recall some properties implied by PRICE that are proved in \cite{BST:19,BST:23} for the one-sided property PRICE defined in \cite[Definition~5.8]{BST:19}.

\begin{proposition} \label{prop:sadic1}
Let $n \in \ZZ$ and assume that $\bsigma \in \cS_d^{\ZZ}$ satisfies the two-sided property PRICE at level~$n$.
Then the following holds.
\begin{enumerate}[\upshape (i)]
\item \label{it:alllevelsPRICE}
The sequence~$\bsigma$ satisfies the two-sided property PRICE at all levels $m \ge n$. 
\item \label{it:righteigenvector}
The sequence~$\bsigma$ admits a generalized right eigenvector.
\item \label{it:onesidedPRICE}
The sequence $(\sigma_n,\sigma_{n+1},\dots)$ satisfies the one-sided property PRICE. 
\item \label{it:subdivision}
For each $a \in \cA$, the subtile~$\cR_n(a)$ of the Rauzy fractal is a compact set that is the closure of its interior, its boundary has zero Lebesgue measure, and the union in \eqref{e:setequationkl} is disjoint up to sets of measure zero.
\item \label{it:multipletiling}
The collection $\cC_n$ forms a multiple tiling of~$\bv_n^\perp$, i.e., there exists $c \ge 1$ such that each point of $\bv_n^\perp$ is contained in at least $c$ elements of~$\cC_n$ and a.e.\ point of $\bv_n^\perp$ is contained in exactly $c$ elements of~$\cC_n$.
\end{enumerate}
\end{proposition}

\begin{proof}
Since $\bsigma$ satisfies the property PRICE at level~$n$ if and only if $\Sigma^n \bsigma$ satisfies the property PRICE at level~$0$, we can assume w.l.o.g.\ that $n = 0$.

Then, for each $m \ge 0$, $\bsigma$ has the property PRICE at level~$m$ with sequences $(\ell_{k+m}{-}m)_{k\in\NN}$ and $(r_{k+m}{+}m)_{k\in\NN}$, hence (\ref{it:alllevelsPRICE}) holds.
Since $\bsigma$ is primitive and eventually recurrent by (P) and (R), it has a generalized right eigenvector~$\bu$ by Proposition~\ref{prop:fur}, i.e., (\ref{it:righteigenvector}) holds. 
We have 
\[
\bv = \lim_{k\to\infty} \frac{\tr{\!M}_{\sigma_{[-\ell_k,0)}}\bv}{\lVert\tr{\!M}_{\sigma_{[-\ell_k,0)}}\bv\rVert} = \lim_{k\to\infty} \frac{\tr{\!M}_{\sigma_{[r_k-\ell_k,r_k)}}\bv}{\lVert\tr{\!M}_{\sigma_{[r_k-\ell_k,r_k)}}\bv\rVert} = \lim_{k\to\infty} \frac{\tr{\!M}_{\sigma_{[0,r_k)}}\bv}{\lVert\tr{\!M}_{\sigma_{[0,r_k)}}\bv\rVert} = \lim_{k\to\infty} \frac{\bv_{r_k}}{\lVert\bv_{r_k}\rVert}
\]
when $\bv$ is the generalized left eigenvector of~$\bsigma$ with $\lVert\bv\rVert = 1$.
Therefore, we obtain that $(\sigma_0,\sigma_1,\dots)$ satisfies the one-sided property PRICE of \cite[Definition~5.8]{BST:19}, i.e., (\ref{it:onesidedPRICE}) holds. 
Now, (\ref{it:subdivision}) follows from Lemma~\ref{lem:RFbalanced} and 
\cite[Propositions~6.4, 7.3 and~7.5]{BST:19}.
Since our collection~$\cC_0$ is the collection~$\cC_{\bv}$ from \cite{BST:19}, (\ref{it:multipletiling}) follows from \cite[Proposition~7.5]{BST:19}.
\end{proof}

Combining Propositions~\ref{p:ntilingbox}, \ref{p:cChC} and~\ref{prop:sadic1}, we obtain the following result giving an important sufficient criterion for the tiling condition of Definition~\ref{def:tilingcond}.

\begin{proposition} \label{p:tilingcond}
Let $\bsigma \in \cS_d^{\ZZ}$, with $d \ge 2$, be a primitive sequence of unimodular substitutions.
If, for some $n \in \ZZ$, $\bsigma \in \cS_d^{\ZZ}$ satisfies the two-sided property PRICE at level~$n$ and $\cC_n$ forms a tiling, then $\bsigma$ satisfies the tiling condition.
\end{proposition}

It is a priori not clear how to decide for a given directive sequence~$\bsigma$ if the multiple tiling given by Proposition~\ref{prop:sadic1}~(\ref{it:multipletiling}) is actually a tiling. 
In the case where $\bsigma$ is a constant sequence given by the repetition of a single Pisot substitution~$\tau$, the tiling property can be checked e.g.\ with the balanced pair algorithm; see \cite[Section~6.1]{BST:23} for examples of application for substitutions associated to continued fractions.
For sequences~$\bsigma$ satisfying the property PRICE, the following (effective version of the) \emph{geometric coincidence condition} defined in \cite[Definition~4.6]{BST:23} is equivalent to the tiling property. 

\begin{definition}[Geometric coincidence condition]\indx{coincidence condition!geometric}
A~sequence of substitutions $\bsigma = (\sigma_n)_{n\in\ZZ} \in \cS_d^\ZZ$, with $d \ge 2$, satisfies the \emph{geometric coincidence condition} if it admits a generalized right eigenvector and there are $n \in \NN$, $\bz \in \bone^\perp$, $a \in \cA$, $C>0$, with the following property. 
We have $\bsigma \in B_C^{(n)}$ and, for each $(\by,b) \in \ZZ^d {\times} \cA$ satisfying $\lVert\pi_{\bu_n,\bone} M_{\sigma_{[0,n)}}^{-1} \by - \bz\rVert_\infty \le C$ and $0 \le \langle \bone, \by\rangle < |\sigma_{[0,n)}(b)|$, we have $\by = \bl(p)$ for some $p \in \cA^*$ such that $pa \preceq \sigma_{[0,n)}(b)$. 
\end{definition}

\begin{proposition} \label{l:coincidencetiling}
Let $\bsigma \in \cS_d^{\ZZ}$, with $d \ge 2$, be a primitive sequence of unimodular substitutions that satisfies the two-sided property PRICE at level~$0$.
Then $\bsigma$ satisfies the tiling condition if and only if it satisfies the geometric coincidence condition. 
\end{proposition}

\begin{proof}
This follows by combining \cite[Proposition~4.8]{BST:23} with Propositions~\ref{prop:sadic1} and~\ref{p:tilingcond}.
\end{proof}

In \cite{BST:23}, we used this result to infer the tiling condition for a sequence~$\bsigma$ containing large power of a substitution~$\tau$ from the tiling condition for the constant sequence of~$\tau$'s. 
More precisely, we have the following lemma.

\begin{lemma} \label{lem:sadic7.9crit}
Let $\tau$ be a unimodular Pisot substitution such that the substitutive dynamical system generated by $\tau$ has pure discrete spectrum. 
For each $C > 0$, there exists $m > 0$ with the following property: 
If $\bsigma = (\sigma_n)_{n\in\ZZ} \in \cS_d^{\ZZ}$ satisfies the two-sided property PRICE at level~$0$ and there are $\ell,n \in \NN$ such that $\sigma_{[n,n+\ell)} = \tau^m$ and $\bsigma \in B_C^{(n+\ell)}$, then $\bsigma$ satisfies the tiling condition.
\end{lemma}

\begin{proof}
This is an immediate consequence of \cite[Lemma~5.6]{BST:23}, since ``pure discrete spectrum'' is equivalent for substitutive systems to the ``geometric coincidence condition'' by \cite[Lemma~4.9]{BST:23}, the one-sided property PRICE follows from the two-sided one by Proposition~\ref{prop:sadic1}~(\ref{it:onesidedPRICE}), and the geometric coincidence condition (which is equivalent to the effective one by \cite[Proposition~4.8]{BST:23}) implies the tiling condition by Proposition~\ref{p:tilingcond}.
\end{proof}

The following result from \cite{BST:23} gives a way to construct substitutions with pure discrete spectrum. 

\begin{proposition}[{\cite[Proposition~5.9]{BST:23}}] \label{p:sigmatilde}
Let $M$ be a nonnegative unimodular Pisot matrix and $m \in \NN$ sufficiently large. 
Then there exists a substitution~$\sigma$ with incidence matrix~$M_\sigma = M^m$ that has pure discrete spectrum.\footnote{Again in the statement of \cite[Proposition~5.9]{BST:23} instead of ``pure discrete spectrum'' the ``geometric coincidence condition'' is assumed. However, by  \cite[Lemma~4.9]{BST:23} these two conditions are equivalent.}
\end{proposition}

\section{Metric theory for bi-infinite sequences of substitutions} \label{sec:metricS}
In this section, we revisit the metric results of Section~\ref{sec:metricmat} by considering sequences of substitutions instead of sequences of matrices, i.e., considering $\bsigma \in \cS_d^{\ZZ}$ instead of $\bM \in \cM_d^{\ZZ}$. In Section~\ref {subsec:cocyclesub} we introduce some terminology and relate the positive range property of a multidimensional continued fraction algorithm to an analogous property for subshifts of $\cS^{\ZZ}$. Section~\ref{subsec:balanceSub} is devoted to consequences of the Pisot condition and contains Theorem~\ref{thm:oseledetssub}, a substitutive version of Theorem~\ref{thm:oseledetsmat}, providing the consequences of the Pisot condition in the substitutive setting. Apart from an additional result on balance, its conclusions are analogous to the ones of Theorem~\ref{thm:oseledetsmat}. Theorem~\ref{thm:oseledetssub} is applied to the modified Jacobi--Perron algorithm to provide, among other results, strong convergence and an Anosov splitting for this algorithm. In Section~\ref{subsec:mcrtilingcond} we give metric conditions for tiling and pure discrete spectrum for subshifts of $\cS^{\ZZ}$. These conditions strongly rely on the generic Pisot condition and, as we will see in Section~\ref{subsec:mettilconpds}, can be slightly weakened if we accelerate the shift.  Finally, in Section~\ref{subsec:pdsmcf} we provide a metric result on the tiling condition and on pure discrete spectrum of the underlying $\cS$-adic dynamical systems in the framework of multidimensional continued fraction algorithms.

\subsection{Some terminology} \label{subsec:cocyclesub}
Let $d \ge 2$. 
Let $D$ be a $\Sigma$-invariant subset (where $\Sigma$ is the shift operator) of the space of bi-infinite sequences of substitutions $\cS_d^{\ZZ}$ and let $\nu$ be an ergodic $\Sigma$-invariant measure on~$D$. Then $(D,\Sigma,\nu)$ defines a measure-theoretic dynamical system (which must not be confused with the $\cS$-adic shift $(X_{\bsigma},\Sigma,\mu)$ defined in Section~\ref{subsec:lang}). 
Of particular importance for us will be the following cocycles defined by the incidence matrices of the substitutions, to be compared with the cocycles introduced in Section~\ref{subsec:cocyclemat} for sequences of matrices.
 
\begin{definition}[$\cS$-adic cocycles]
 \label{def:cocycleS}\indx{cocycle!transpose!S@$\cS$-adic}\indx{S@$\cS$-adic!cocycle!transpose}\indx{cocycle!inverse!S@$\cS$-adic}\indx{S@$\cS$-adic!cocycle!inverse}
Let $\cS \subset \cS_d$ for some $d \ge 2$.
Let $(D,\Sigma,\nu)$ be an ergodic measure-theoretic dynamical system with $D \subset \cS^{\ZZ}$.  
\begin{itemize}
\item 
The linear cocycle $A_{\mathrm{tr}}:\, D \to \GL(d,\ZZ)$, $(\sigma_n)_{n\in\ZZ} \mapsto \tr{\!M}_{\sigma_0}$, is called the \emph{($\cS$-adic) transpose cocycle}.
\notx{Atr}{$A_{\mathrm{tr}}$}{transpose cocycle}
\item 
The linear cocycle $A_{\mathrm{inv}}:\, D \to \GL(d,\ZZ)$, $(\sigma_n)_{n\in\ZZ} \mapsto M_{\sigma_0}^{-1}$, is called the \emph{($\cS$-adic) inverse cocycle}. 
\notx{Atinv}{$A_{\mathrm{inv}}$}{inverse cocycle}
\end{itemize}
\end{definition}

The iterates of $A_{\mathrm{tr}}$ and~$A_{\mathrm{inv}}$  are defined analogously to~\eqref{bcocycle}.

To state our main theorems, we need a few more definitions that are in analogy to some definitions provided in Chapters~\ref{chapter:matrix} and~\ref{sec:cf} in the context of sequences of matrices and multidimensional continued fraction algorithms, respectively.

The following \emph{periodic Pisot sequence} corresponds to a Pisot toral automorphism and a Pisot substitution, respectively. 

\begin{definition}[Periodic Pisot sequence] \label{d:Pisotpoint}\indx{Pisot!periodic sequence} 
A~sequence $(M_n) \in \cM_d^{\ZZ}$ [$(\sigma_n) \in \cS_d^{\ZZ}$] is called a \emph{periodic Pisot sequence} if there is $k \ge 1$ such that the sequence has period~$k$ and $M_{[0,k)}$ is a Pisot matrix [$\sigma_{[0,k)}$ is a Pisot substitution].
\end{definition}

Next, we propose an adaptation of the positive range property for one-sided sequences used in \cite[Section~2.5]{BST:23} to bi-infinite sequences.

\begin{definition}[Cylinder set, two-sided positive range] \label{def:pr}
\indx{cylinder set}
\indx{range!positive}
\indx{saturated}
Let $(D,\Sigma,\nu)$ be a dynamical system with $D \subset \cM_d^{\ZZ}$ $[D \subset \cS_d^{\ZZ}]$ and a shift invariant Borel probability measure~$\nu$. \emph{Cylinder sets of~$D$} are defined by\footnote{Omitting the dependence on~$D$ should cause no confusion.}  
 \[
[\omega_0,\dots,\omega_{n-1}] = \big\{(\upsilon_k)_{k\in\ZZ} \in D\,:\, (\upsilon_0,\dots,\upsilon_{n-1}) = (\omega_0,\dots,\omega_{n-1})\big\} \qquad (n \in \NN) \notx{0cylinder}{$[\ldots]$}{cylinder set}
\]
(and images under~$\Sigma^k$, $k \in \ZZ$, of these sets).

A~set $E \subset D$ is called \emph{left saturated} if $(\omega_k)_{k\in\ZZ} \in E$ implies that $(\upsilon_k)_{k\in\ZZ} \in E$ for all  $(\upsilon_k)_{k\in\ZZ} \in D$ such that $(\upsilon_0,\upsilon_1,\dots) = (\omega_0,\omega_1,\dots)$.

Moreover, we say that  $(\omega_k)_{k\in\ZZ} \in D$ has \emph{two-sided positive  range} in $(D,\Sigma,\nu)$ if there is $\varepsilon>0$ such that, for each left saturated set $F$ with $\nu(E) \ge 1{-}\varepsilon$, we have
\[
\nu(\Sigma^n[\omega_0,\dots,\omega_{n-1}] \cap E) > 0 \quad \mbox{for all}\ n \in \NN.
\]
\end{definition}

We recall that $D$ is not necessarily closed. However, cylinder sets of $(D,\Sigma,\nu)$ are measurable in $D$ because they are open sets in the subspace topology on~$D$. This is the reason why we assumed~$\nu$ to be a Borel measure. 
 
In the following lemma, we discuss the relation between the two-sided positive range property for sequences and the (one-sided) positive range property for continued fraction algorithms defined in Definition~\ref{def:prFC}. 

\begin{lemma} \label{l:positiverange}
Let $(X,F,A,\nu)$ be a positive multidimensional continued fraction algorithm with $\nu \circ F \ll \nu$, let $\bphi$ be a faithful substitutive realization of a natural extension $(\hX,\hF,\hA,\hnu)$ of $(X,F,A,\nu)$, and let $(\bx_0,\by_0) \in \hX$.
If $\bx_0$ has positive range in~$(X,F,A,\nu)$, then $\bphi(\bx_0,\by_0)$ has two-sided positive range in $(\bphi(\hX),\Sigma,\bphi_*\hnu)$.
\end{lemma}

\begin{proof} 
Let $0 < \varepsilon < \inf_{n\in\NN} \nu(F^n X_{[0,n)}(\bx_0))$ and consider a left saturated subset~$E$ of~$\bphi(\hX)$ with $\bphi_*\hnu(E) \ge 1{-}\varepsilon$.  
Let $E' = \pi \circ \bphi^{-1}(E)$ ; recall that $\pi : \hX \to X$ is defined by $\pi(\bx,\by) = \bx$. 
Then $\nu(E') = \hnu(\pi^{-1}(E')) \ge \hnu(\bphi^{-1}(E)) \ge 1{-}\varepsilon$. 
For each $n \in \NN$, we have thus $\nu(F^n X_{[0,n)}(\bx_0) \cap E') > 0$, hence $\nu(X_{[0,n)}(\bx_0) \cap F^{-n}E') > 0$ by \eqref{e:abscontinuity}.
We have
\[
\pi^{-1}\big(X_{[0,n)}(\bx_0) \cap F^{-n} E'\big) = \pi^{-1}(X_{[0,n)}(\bx_0)) \cap \hF^{-n}(\pi^{-1}(E'))
\]
since $(\bx,\by) = \hF^{-n} \hF^n(\bx,\by) = \hF^{-n}(F^n(\bx), A^{(n)}(\bx_0) \by)$ for $(\bx,\by) \in \pi^{-1}(X_{[0,n)}(\bx_0))$, thus $(\bx,\by) \in \hF^{-n}(\pi^{-1}(E'))$ if and only if $(\bx,\by) \in \pi^{-1}(F^{-n}(E'))$.
We obtain that
\[
\begin{aligned}
& \hnu\big(\hF^n(\pi^{-1}(X_{[0,n)}(\bx_0)) \cap \pi^{-1}(E')\big) = \hnu\big(\pi^{-1}(X_{[0,n)}(\bx_0)) \cap \hF^{-n}(\pi^{-1}(E'))\big) \\
& \quad = \hnu\big(\pi^{-1}(X_{[0,n)}(\bx) \cap F^{-n}E')\big)  = \nu\big(X_{[0,n)}(\bx_0) \cap F^{-n}E'\big) > 0.
\end{aligned}
\]
Writing $\bphi(\bx_0,\by_0) = (\sigma_n)_{n\in\ZZ}$, we have $\bphi(\hF^n\pi^{-1}(X_{[0,n)}(\bx_0))) = \Sigma^n[\sigma_0,\dots,\sigma_{n-1}]$.
Since $E$ is left saturated, we have $E = \bphi(\pi^{-1}(E'))$, thus
\[
\bphi_*\hnu(\Sigma^n [\sigma_0,\dots,\sigma_{n-1}] \cap E) = \hnu\big(\hF^n(\pi^{-1}(X_{[0,n)}(\bx_0)) \cap \pi^{-1}(E')\big) > 0
\]
for all $n \in \NN$, i.e., $\bphi(\bx_0,\by_0)$ has two-sided positive range in $(\bphi(\hX),\Sigma,\bphi_*\hnu)$.
\end{proof}

\subsection{Generic Pisot condition for bi-infinite sequences of substitutions}\label{subsec:balanceSub}
According to Definition~\ref{def:gengenPisot}, we consider the following Pisot condition for a  dynamical system generating sequences of substitutions.

\begin{definition}[Generic Pisot condition] \label{def:gengenPisot2subs}\indx{Pisot!condition!generic}
\indx{Pisot!condition!generic}
Let $d \ge 2$. An ergodic dynamical system $(D,\Sigma,\nu)$ with $D \subset \cS_d^{\ZZ}$ is said to satisfy the \emph{(generic) Pisot condition}  (without qualifying the cocycle) if it  satisfies the (generic) Pisot condition for  the linear cocycle~$A_{\mathrm{tr}}$.
\end{definition}

By Remark~\ref{rem:invtr}, the  Pisot condition (for~$A_{\mathrm{tr}}$) guarantees the existence of an invariant Oseledets splitting into a stable space of dimension~$1$ and an unstable space of codimension~$1$ for the cocycle~$A_{\mathrm{inv}}$. This is one of the assertions of the following theorem, which summarizes important consequences of the  Pisot condition; it is the substitution version of Theorem~\ref{thm:oseledetsmat}. Note however that the present formulation brings more information with the balance property.  

\begin{theorem}\label{thm:oseledetssub}
Let $(D,\Sigma,\nu)$ with $D \subset \cS_d^{\ZZ}$, $d \ge 2$,  be an ergodic dynamical system that has a cylinder $[\tau_0,\ldots,\tau_{\ell-1}]$ of positive measure corresponding to a substitution $\tau_{[0,\ell)}$ with positive incidence matrix.
If $(D,\Sigma,\nu)$ satisfies the Pisot condition, then the following assertions hold for $\nu$-almost every $\bsigma \in D$.
\begin{enumerate}[\upshape (i)]
\item \label{it:oseledets1}
The sequence $\bsigma$ is two-sided primitive and algebraically irreducible (in the future). 
\item \label{it:oseledets2}
The language $\cL_{\bsigma}$ is balanced.
\item \label{it:oseledets3}
The sequence $\bsigma$ converges strongly to the generalized right eigenvector $\bu$ in the future and to the generalized left eigenvector~$\bv$ in the past.
\item \label{it:oseledets4}
The mapping family $(\TT,f_{\bsigma})$ associated to $\bsigma$ is eventually Anosov for the hyperbolic splitting   
\begin{equation}\label{eq:hyperbolicsub}
G^s(\bsigma) \oplus G^u(\bsigma) = \coprod_{n\in\ZZ} G_n^s(\bsigma) \oplus G_n^u(\bsigma) \ \hbox{with}\ G^s_n(\bsigma) = \RR \bu_n \ \hbox{and}\ G^u_n(\bsigma) = \bv_n^\perp,
\end{equation}
where $\bu_n$ and $\bv_n$ are the generalized right and left eigenvectors defined in~\eqref{eq:unvn}.
\item \label{it:oseledets5}
There is a relation between the Anosov splitting of $(\TT,f_{\bsigma})$ and the Oseledets splitting \eqref{eq:oseledetsSplitting} of $\bsigma$ for the cocycles $A_{\mathrm{inv}}$ and~$A_{\mathrm{tr}}$. In particular, for each $n \in \ZZ$ we have 
\begin{equation}\label{eq:osele2subs}
\begin{aligned}
G^u_n(\bsigma) & = G_{\mathrm{inv}}^1(\Sigma^n\bsigma) \oplus \cdots \oplus G_{\mathrm{inv}}^{p-1}(\Sigma^n\bsigma) = \bv_n^\perp, \\
G^s_n(\bsigma) & = G_{\mathrm{inv}}^p(\Sigma^n\bsigma) = \RR \bu_n, \\
G^u_n(\bsigma)^\perp & = G_{\mathrm{tr}}^1(\Sigma^n\bsigma) = \RR \bv_n, \\
G^s_n(\bsigma)^\perp & = (G_{\mathrm{tr}}^2(\Sigma^n\bsigma) \oplus \cdots \oplus G_{\mathrm{tr}}^{p}(\Sigma^n\bsigma) = (\bu_n)^\perp. 
\end{aligned}
\end{equation}
\end{enumerate}
\end{theorem}

\begin{proof}
All items except (\ref{it:oseledets2}) depend only on the incidence matrices of~$\bsigma$ and, hence, are direct consequences of the according items of Theorem~\ref{thm:oseledetsmat}. Item (\ref{it:oseledets2}) follows from Lemma~\ref{lem:GloLocPisot}~(\ref{it:glp2}) and Theorem~\ref{theo:sufcondPisotS}~(\ref{i:scP3}).
\end{proof}

Oseledets' theorem only gives an ``almost everywhere'' result. Even if the Pisot condition holds, like for the Brun algorithm, there are sequences for which the balance condition does not hold (see~e.g.\ \cite{Cassaigne-Ferenczi-Zamboni:00,Delecroix-Hejda-Steiner}). These exceptions constitute a set of measure zero, which is nevertheless dense. In fact, if, for a sequence $\bsigma\in\cS_d^\ZZ$, there is a level $n$ at which the Language $\cL_{\bsigma}^{(n)}$ is not balanced, the same holds for all levels $n\in \ZZ$. This is true because, the incidence matrices $M_{\sigma_n}$ are invertible, and induce an automorphism of $\ZZ^d$, hence, the ``unbalance'' of the language $\cL_{\bsigma}^{(n)}$ propagates forward and backward on all levels.

The metric theory developed in Theorem~\ref{thm:oseledetssub} allows us to establish the following result for the modified Jacobi--Perron algorithm. 

\begin{lemma}\label{lem:MJP}
Let $(\hX_{\rM},\hF_{\rM},\hA_{\rM},\hnu_{\rM})$ be the natural extension of the two-di\-men\-sion\-al modified Jacobi--Perron algorithm, and let $\bphi_{\rM}$ be a substitutive realization of this algorithm. Then for $\nu_{\rM}$-almost every $(\bx,\by) \in \hX_{\rM}$ the following assertions~hold.
\begin{itemize}
\item $\bphi_{\rM}(\bx,\by)$ satisfies the two-sided Pisot condition.
\item $\bphi_{\rM}(\bx,\by)$ is algebraically irreducible.
\item The language $\cL_{\bphi_{\rM}(\bx,\by)}$ is balanced.
\item $\bphi_{\rM}(\bx,\by)$ converges strongly to a generalized right eigenvector $\bx$ in the future and to a generalized left eigenvector~$\by$ in the past. The vector~$\bx$ has rationally independent coordinates.
\item The mapping family $(\TT,f_{\bphi_{\rM}(\bx,\by)})$ associated to $\bphi_{\rM}(\bx,\by)$ is eventually Anosov for the splitting
\[
\coprod_{n\in\ZZ} \RR \bu_n \oplus \bv_n^\perp.
\]
\end{itemize}
\end{lemma}

\begin{proof}
The first assertion holds because by Proposition~\ref{prop:Brun23Pisot}  the modified Jacobi--Perron algorithm satisfies the generic Pisot condition. The remaining assertions are a consequence of Theorem~\ref{thm:oseledetssub}. 
\end{proof}

\subsection{Metric criteria for the tiling condition and pure discrete spectrum}\label{subsec:mcrtilingcond}
Our goal is to show that the Pisot condition implies property PRICE and, under some weak additional assumption, even the tiling condition from Definition~\ref{def:tilingcond} as well as pure discrete spectrum almost everywhere. In the proofs, we have to use some results and techniques from \cite[Section~5]{BST:23}.

The first lemma assures balance and is very similar \cite[Lemma~5.2]{BST:23} but uses weaker assumptions. 

\begin{lemma} \label{l:nuBC}
Let $(D, \Sigma, \nu)$, with $D \subset \cS_d^{\ZZ}$, $d \ge 2$, be an ergodic dynamical system that satisfies the Pisot condition and has a cylinder of positive measure corresponding to a substitution with positive incidence matrix. 
Then
\[\lim_{C\to\infty} \nu(D \cap B_C^{(n)}) = 1 \qquad (n \in \ZZ).
\] 
\end{lemma}

\begin{proof}
Theorem~\ref{thm:oseledetssub}~(\ref{it:oseledets2}) gives that $\nu(\bigcup_{C\in\NN} D \cap B_C^{(0)}) = 1$. 
Since $B_C^{(0)} \subset B_{C'}^{(0)}$ for $C < C'$ and measurability of $D \cap B_C^{(0)}$ is proved as in \cite[Lemma~5.2]{BST:23}, we obtain $\lim_{C\to\infty} \nu( D \cap B_C^{(0)}) = 1$. 
Finally, \eqref{eq:BCkn} yields $\lim_{C\to\infty} \nu( D \cap B_C^{(n)}) = 1$ for $n \in \ZZ$.
\end{proof}

The next result is a two-sided version of \cite[Proposition~5.3]{BST:23}, again removing superfluous assumptions. 
Together with Lemma~\ref{l:positivecylinder} below, it shows that the assumptions of Theorem~\ref{theo:metricmarkov} imply the two-sided property PRICE from Definition~\ref{def:PRICE} a.e.; the importance of this property for the tiling condition has been shown in Section~\ref{sec:suff-cond-tilings}.

\begin{proposition}\label{prop:combcond}
Let $(D,\Sigma,\nu)$, with $D \subset \cS_d^{\ZZ}$, $d \ge 2$, be an ergodic dynamical system that satisfies the Pisot condition and has a cylinder of positive measure corresponding to a substitution with positive incidence matrix.
Then $\nu$-almost every $\bsigma \in D$ is two-sided primitive and satisfies the two-sided property PRICE. 
\end{proposition}

\begin{proof}
We follow the proof of \cite[Proposition~5.3]{BST:23}. 
Let $[\tau_0,\dots,\tau_{h-1}]$ be a cylinder with $\nu([\tau_0,\dots,\tau_{h-1}]) > 0$ such that the incidence matrix $M_{\tau_{[0,h)}}$ is positive.
By Lemma~\ref{l:nuBC}, there is $C \in \NN$ such that $\nu(D \cap B_C^{(0)}) > 1 - \nu(\Sigma^h[\tau_0,\dots,\tau_{h-1}])$, thus $\nu(\Sigma^h[\tau_0,\dots,\tau_{h-1}] \cap  B_C^{(0)}) > 0$. 

By ergodicity of~$\nu$ together with the Poincar\'e Recurrence Theorem, we have for almost all $\bsigma = (\sigma_n)_{n\in\NN} \in D$ some $\ell_0(\bsigma) \ge h$ such that $\Sigma^{\ell_0(\bsigma)} \bsigma \in \Sigma^h[\tau_0,\dots,\tau_{h-1}] \cap B_C^{(0)}$, i.e., $(\sigma_0,\dots,\sigma_{\ell_0(\bsigma)-1})$ ends with $(\tau_0,\dots,\tau_{h-1})$ and $\bsigma \in B_C^{(\ell_0(\bsigma))}$. 
We will now extend $\ell_0(\bsigma)$ for almost all $\bsigma \in D$ to a sequence $(\ell_k(\bsigma))_{k\in\NN}$ such that
\begin{itemize}
\item
$(\sigma_0,\dots,\sigma_{\ell_{k+1}(\bsigma)-1})$ ends with $(\sigma_{-\ell_k(\bsigma)},\dots,\sigma_{\ell_k(\bsigma)-1})$ (and, a fortiori, with $(\tau_0,\ldots,\tau_{h-1})$),
\item
$\bsigma \in B_C^{(\ell_{k+1}(\bsigma))}$,
\item
$\ell_{k+1}(\bsigma) \ge 2\ell_k(\bsigma)$,
\end{itemize}
for all $k \in \NN$.
To this end, assume that $\ell_0(\bsigma),\ldots,\ell_k(\bsigma)$ are already defined for almost all $\bsigma \in D$.
Since there are countably many possibilites for $\ell_k(\bsigma)$ and $(\sigma_{-\ell_k(\bsigma)},\dots,\sigma_{\ell_k(\bsigma)-1})$, the set of those $\bsigma \in D$ such that
\[ \nu\big(\Sigma^{\ell_k(\bsigma)}[\sigma_{-\ell_k(\bsigma)},\dots,\sigma_{\ell_k(\bsigma)-1}] \cap B_C^{(\ell_k(\bsigma))}\big) = 0
\]
has zero measure.
We have thus $\nu\big(\Sigma^{\ell_k(\bsigma)}[\sigma_{-\ell_k(\bsigma)},\dots,\sigma_{\ell_k(\bsigma)-1}] \cap B_C^{(\ell_k(\bsigma))}\big) > 0$ for a.e.\ $\bsigma \in D$ and obtain (by the Poincar\'e Recurrence Theorem and ergodicity of~$\nu$) some $\ell_{k+1}(\bsigma)$ with the required properties. 
Recursively, we obtain a sequence $(\ell_k(\bsigma))_{k\in\NN}$ with the required properties for almost all $\bsigma \in D$.

Setting $r_k (\bsigma) = \ell_{k+1}(\bsigma) - \ell_k(\bsigma)$, we obtain that conditions (P), (R) and~(C) of Property PRICE hold for almost all $\bsigma \in D$.
By the Poincar\'e Recurrence Theorem and since $\tau_{[0,h)}$ has positive incidence matrix, two-sided primitivity holds and condition~(E) follows from Proposition~\ref{prop:fur}. 
From the Pisot condition and Corollary~\ref{bst8.7}, we obtain that almost all $\bsigma \in D$ are algebraically irreducible, i.e., also (I) holds a.e.\ and we are done. 
\end{proof}

The following lemma is needed to show that the assumptions of the theorems in Section~\ref{subsec:metricMP} imply the assumptions of Lemma~\ref{l:nuBC} and Proposition~\ref{prop:combcond}.

\begin{lemma} \label{l:positivecylinder}
Assume that $(D, \Sigma, \nu)$, with $D \subset \cS_d^{\ZZ}$, $d \ge 2$, has a periodic Pisot sequence with positive two-sided range. Then it has a cylinder of positive measure corresponding to a substitution with positive incidence matrix.
\end{lemma}

\begin{proof}
Let $\btau = (\tau_n)_{n\in\ZZ} \in D$ be a periodic Pisot sequence with two-sided positive range.
Then $\nu(\Sigma^n[\tau_0,\dots,\tau_{n-1}]) > 0$ for all $n \in \NN$ and there is 
$p \ge 1$ such that $\Sigma^p \btau = \btau$ and $\tau_{[0,p)}$ is a Pisot substitution.
Since $\tau_{[0,p)}$ is a Pisot substitution, its incidence matrix is primitive by \cite[Proposition~1.3]{CS:01} and, hence, $\tau_{[0,kp)}$ has positive incidence matrix for some $k \in \NN$.
As $\nu$ is measure-preserving, we have $\nu([\tau_0,\dots,\tau_{kp-1}]) > 0$.
\end{proof}

\begin{theorem}\label{prop:combcond2}
Let $(D,\Sigma,\nu)$, with $D \subset \cS_d^{\ZZ}$, $d \ge 2$, be an ergodic dynamical system that satisfies the Pisot condition and has a periodic Pisot sequence with positive two-sided range and pure discrete spectrum.
Then, for $\nu$-almost every $\bsigma \in D$ the following assertions hold.
\begin{itemize}
\item[(i)]  $\bsigma$ is two-sided primitive.
\item[(ii)]  $\bsigma$ satisfies the tiling condition with compact Rauzy fractals.
\item[(iii)] $(X_{\bsigma},\Sigma)$ has pure discrete spectrum.
\end{itemize} 
\end{theorem}

\begin{proof}
Our main goal is to apply Lemma~\ref{lem:sadic7.9crit}. Thus we have to prove that its conditions are verified.

According to our assumptions, there is a sequence $(\tau_n)_{n\in\ZZ} \in D$ with period~$p$ and two-sided positive range such that $\tau_{[0,p)}$ is a Pisot substitution and the substitutive dynamical system generated by $\tau_{[0,p)}$ has pure discrete spectrum. 
By the two-sided positive range property, there exists $\varepsilon > 0$ such that $\nu(\Sigma^{n}[\tau_0,\dots,\tau_{n-1}] \cap E) > 0$ for all left saturated sets~$E$ with $\nu(E) \ge 1{-}\varepsilon$.
By Lemma~\ref{l:nuBC}, there exists $C \in \NN$ such that $\nu(B_C^{(0)}) \ge 1{-}\varepsilon$.
Since $B_C^{(0)}$ is left saturated, we have $\nu(\Sigma^{n}[\tau_0,\dots,\tau_{n-1}] \cap B_C^{(0)}) > 0$ and thus
\[
\nu([\tau_0,\dots,\tau_{n-1}] \cap B_C^{(n)}) > 0 \qquad (n \in \NN)
\]
by~\eqref{eq:BCkn}. 
Let $m$ be as in Lemma~\ref{lem:sadic7.9crit} for the substitution~$\tau_{[0,p)}$.
By the Poincar\'e Recurrence Theorem and the ergodicity of~$\nu$,  for almost all sequences $\bsigma \in D$, there exists~$n$ such that $\Sigma^n \bsigma \in [\tau_0,\dots,\tau_{mp-1}] \cap B_C^{(mp)}$, which is equivalent to the conditions $\sigma_{[n,n+\ell)} = \tau_{[0,p)}^m$ and $\bsigma \in B_C^{(n+\ell)}$ in the formulation of Lemma~\ref{lem:sadic7.9crit}. 
Thus, since the two-sided property PRICE holds for a.e.\ $\bsigma \in D$ by Proposition~\ref{prop:combcond} (which we may apply because of Lemma~\ref{l:positivecylinder}), Lemma~\ref{lem:sadic7.9crit} yields that the tiling condition is satisfied for a.e.\ $\bsigma \in D$, with compact tiles by Lemma~\ref{lem:RFbalanced}. Moreover, two-sided primitivity directly follows from Proposition~\ref{prop:combcond}.
Finally, pure discrete spectrum follows because, in view of Proposition~\ref{p:tilingw}, we can apply Theorem~\ref{t:tilingpds} for almost all $\bsigma\in D$.
\end{proof}

\subsection{Metric tiling condition and pure discrete spectrum for accelerations}\label{subsec:mettilconpds}

The following theorem shows that we can accelerate each shift in a way that the tiling condition from Definition~\ref{def:tilingcond} as well as pure discrete spectrum hold generically without assuming pure discrete spectrum for a Pisot sequence.
In its statement, for a given $\bM = (M_n)_{n\in\ZZ} \in \cM_d$ and a given $p \in \NN$, we use the notation
\begin{equation} \label{e:Mp}
\bM_p = (M_{[pn,p(n+1))})_{n\in\ZZ} \notx{Mp}{$\bM_p$}{$p$-fold blocking}
\end{equation}
for the $p$-fold blocking of the sequence~$\bM$. 

The pure discrete spectrum result contained in the following theorem is contained in \cite[Theorem~3.5]{BST:23}

\begin{theorem} \label{theo:metrictilingaccel}
Let $(D, \Sigma, \nu)$, with $D \subset \cM_d^{\ZZ}$, $d \ge 2$, be an ergodic dynamical system that satisfies the Pisot condition and contains a periodic Pisot sequence with two-sided positive range.
Then there is a positive integer~$p$ and a substitution assignment\footnote{It is sufficient to define the substitution assignment on matrices occurring in~$D$.}~$\varrho$ on~$\cM_d$ such that, for $\nu$-almost every $\bM \in D$, the sequence of substitutions $\varrho(\bM_p)$ satisfies the tiling condition with compact Rauzy fractals, and $(\varrho(\bM_p),\Sigma)$ has pure discrete spectrum.
\end{theorem}

To prove Theorem~\ref{theo:metrictilingaccel}, we study accelerations $(D,\Sigma^p)$ of $(D,\Sigma)$ and use Proposition~\ref{p:sigmatilde}. In case that $(D,\Sigma^p,\nu)$ is not ergodic (even though $(D,\Sigma,\nu)$ is ergodic), we need the following adaptation of the set~$B_C^{(0)}$. 

For a sequence $\bN = (N_n)_{n\in\ZZ} \in \cM_d^{\ZZ}$ and $p \ge 1$, let 
\[
\begin{aligned}
B_{C,\bN,p} & = \big\{ (M_n)_{n\in\ZZ} \in \cM_d^{\ZZ} \,:\, \bsigma \in B_C^{(0)} \ \mbox{for all $\bsigma = (\sigma_n)_{n\in\ZZ} \in \cS_d^{\ZZ}$ such that} \\
& \hspace{4em} (M_{\sigma_0},M_{\sigma_1},\dots) = (N_{[-i,0)}M_{[0,p-i)}, M_{[p-i,2p-i)}, M_{[2p-i,3p-i)}, \dots) \\ & \hspace{4em} \mbox{for some}\ i \in \{0,\ldots,p{-}1\}\big\},
\end{aligned}
\]
i.e., the set of sequences $(M_n)_{n\in\ZZ} \in \cM_d^{\ZZ}$ such that each substitutive realization of the $p$-fold blocking of a sequence $(N_{-i},\dots,N_{-1},M_0,M_1,\dots)$, $0 \le i < p$, gives a $C$-balanced language. 
Note that $B_{C,\bN,p} \cap D$ is left saturated (and depends only on~$N_i$ with $0 \le i < p$) for all $D \subset \cM_d^{\ZZ}$.

\begin{lemma} \label{l:BCNp}
For $d\ge 2$ let $(D, \Sigma, \nu)$ with $D \subset \cM_d^{\ZZ} $ be an ergodic dynamical system that satisfies the Pisot condition and and has a cylinder of positive measure corresponding to a substitution with positive incidence matrix. Then
\[
\lim_{C\to \infty} \nu(B_{C,\bN,p} \cap D) = 1 \quad \mbox{for all} \quad \bN \in \cM_d^{\ZZ},\, p \ge 1.
\]
\end{lemma}

\begin{proof}
In view of the proof of Lemma~\ref{l:nuBC}, we only have to show that, given $\bM = (M_n)_{n\in\ZZ} \in \cM_d^{\ZZ}$, $\bN = (N_n)_{n\in\ZZ} \in \cM_d^{\ZZ}$, and $p \ge 1$, such that $\bM$ satisfies the Pisot condition and the growth condition $\lim_{n\to\infty} \frac{1}{n} \log \lVert M_n\rVert = 0$, there exists $C \in \NN$ such that $\bsigma \in B_C^{(0)}$ for each substitutive realization~$\bsigma$ of a sequence $\bM' = (M'_n)_{n\in\ZZ} \in \cM_d^{\ZZ}$ with $(M'_0,M'_1,\dots) = (N_{[-i,0)}M_{[0,p-i)}, M_{[p-i,2p-i)}, \dots)$ for some $i \in \{0,\dots,p-1\}$.

We proceed as in the proof of Theorem~\ref{theo:sufcondPisotS}~(\ref{i:scP3}). 
For $\bM, \bM'$ as above, let $\delta_1(n) \ge \cdots \ge \delta_d(n)$ be the singular values of~$M_{[0,n)}$ and $\delta'_1(n) \ge \cdots \ge \delta'_d(n)$ the singular values of~$M'_{[0,n)}$.
Then we have $\beta, C_1 > 0$ such that $\delta_2(n) \le C_1\, e^{-\beta n}$ for all $n \in \NN$, hence there exists $C'_1 > 0$ (that depends only on $\bM$, $\bN$, and~$p$) such that
\[
\delta'_2(n) \le \delta'_1(i)\, \delta_2(pn{-}i) \le C'_1\, e^{-\beta pn} \quad \mbox{for all}\ n \in \NN.
\]
The growth condition implies that there exists $C_2 > 0$ such that $\lVert M_n\rVert_\infty  \le C_2 e^{\beta n/2}$ for all $n \in \NN$.
Since $\lVert M'_n \rVert \le \lVert M_{np-i}\rVert \cdots \lVert M_{np+n-i-1}\rVert$ for all $n \ge 1$, we obtain a constant $C'_2 > 0$ (again depending only on $\bM$, $\bN$, and~$p$) such that 
\[
\lVert M'_n \rVert_\infty  \le C'_2\, e^{\beta pn/2} \quad \mbox{for all}\ n \in \NN.
\]
Then the proof of Theorem~\ref{theo:sufcondPisotS}~(\ref{i:scP3}) shows that there exists $C > 0$ (depending only on $\bM$, $\bN$, and~$p$) such that $\cL_{\bsigma}$ is $C$-balanced for all $\bsigma = (\sigma_n)_{n\in\ZZ} \in \cS_d^{\ZZ}$ such that~$\bM_{\bsigma} = \bM'$.
\end{proof}

We can now conclude the Proof of Theorem~\ref{theo:metrictilingaccel}.

\begin{proof}[Proof of Theorem~\ref{theo:metrictilingaccel}]
Let $\bN = (N_n)_{n\in\ZZ} \in D$ be a periodic Pisot sequence with two-sided positive range, with $p > 0$ such that $\Sigma^p \bN = \bN$ and $N_{[0,p)}$ is a Pisot matrix.
Since $N_{[0,p)}$ and $N_{[-i,p-i)}=N_{[-i,0)} N_{[0,p)} N_{[-i,0)}^{-1}$ are similar matrices, $N_{[-i,p-i)}$ is a Pisot matrix for all $i \in \{0,\dots,p{-}1\}$. 
Then Proposition~\ref{p:sigmatilde} gives substitutions~$\tau_i$ having pure discrete spectrum with incidence matrix $M_{\tau_i} = N_{[-i,kp-i)}$ for some $k \in \NN$.
Since $\{N_{[-i,kp-i)} : 0 \le i < kp\} = \{N_{[-i,kp-i)} : 0 \le i < p\}$, we can replace $p$ by~$kp$ and assume w.l.o.g.\ that $k = 1$. 
We choose $\tau_0,\ldots,\tau_{p-1}$ in a way that $\tau_i = \tau_j$ if $N_{[-i,p-i)}=N_{[-j,p-j)}$, $0\le i,j<p$.

Choose a substitution assignment $\varrho: \cM_d \to \cS_d$ such that $\varrho(N_{[-i,p-i)}) = \tau_i$ for $0 \le i < p$. Then the map
\begin{equation}\label{eq:bpsidef2}
\bpsi:\, D \to \cS_d^{\ZZ}, \quad (M_n)_{n\in\NN} \mapsto \big(\varrho(M_{[np,(n+1)p})\big)_{n\in\ZZ}
\end{equation}
is well defined, and, setting $D' = \bpsi D$ we have the commutative diagram
\[
\begin{tikzcd}
D \arrow[r, "\Sigma^p"]\arrow[d,"\bpsi"] & D \arrow[d, "\bpsi"] \\
D' \arrow[r, "\Sigma"]& D'
\end{tikzcd}
\]
Unfortunately, we cannot apply Theorem~\ref{prop:combcond2} to $(D',\Sigma,\bpsi_*\nu)$ because the acceleration $(D,\Sigma^p,\nu)$ of $(D,\Sigma,\nu)$ may not be ergodic (and it seems difficult to deduce a two-sided positive range property in each of the ergodic components from that in $(D,\Sigma,\nu)$). 
However, each ergodic component of $(D',\Sigma,\bpsi_*\nu)$ has a cylinder of positive measure corresponding to a substitution with positive incidence matrix, hence a.e.\ $\bsigma \in D'$ satisfies the two-sided property PRICE by Proposition~\ref{prop:combcond}; see \cite[Proof of Theorem~3.6]{BST:23} for details on the ergodic components. 

Since $(D,\Sigma^p,\nu)$ need not be ergodic, we consider $(D,\Sigma,\nu)$.
By the two-sided positive range property of~$\bN$, there is $\varepsilon > 0$ such that $\nu(\Sigma^{n}[N_0,\dots,N_{n-1}] \cap F) > 0$ for each left saturated set $F \subset D$ with $\nu(F) \ge 1{-}\varepsilon$ and each $n \in \NN$.
By Lemma~\ref{l:BCNp}, there is $C \in \NN$ such that $\nu(B_{C,\bN,p} \cap D) \ge 1{-}\varepsilon$.
Since $B_{C,\bN,p}$ is left saturated, we have
\[
\nu(\Sigma^n[N_0,\dots,N_{n-1}] \cap B_{C,\bN,p}) > 0 \quad \mbox{for all}\ n \in \NN.
\]
By Lemma~\ref{lem:sadic7.9crit}, there exists $m > 0$ such that $\bpsi(\bM)$ satisfies the tiling condition for each $\bM \in \cM_d^{\ZZ}$ that satisfies the following three conditions for some $n \in \NN$, $i \in \{0,\dots,p{-}1\}$: $\bpsi(\bM)$ satisfies the two-sided property PRICE, $\Sigma^{np} \bM \in [N_{-i},\dots,N_{mp-i-1}]$, and $\bpsi(\bM) \in B_C^{(n+m)}$.
By the previous paragraph, $\bpsi(\bM)$ satisfies the two-sided proprerty PRICE for almost all $\bM \in D$.
For a.e.\ $\bM \in D$, by the ergodicity of $(D,\Sigma,\nu)$ and the Poincar\'e Recurrence Theorem, there also exists $r \ge mp{+}p$ such that $\Sigma^r \bM \in \Sigma^{mp+p}[N_0,\dots,N_{mp+p-1}] \cap B_{C,\bN,p}$; setting $n = \lfloor r/p \rfloor {-} m$, $i = r {-} \lfloor r/p\rfloor p$, we have then $i \in \{0,\dots,p{-}1\}$,
\[
\begin{aligned}
\Sigma^{np} \bM & = \Sigma^{r-mp-i} \bM \in [N_{p-i},\dots,N_{mp+p-1}] \subset [N_{-i},\dots,N_{mp-i-1}], \\
\Sigma^{np+mp} \bM & = \Sigma^{r-i} \bM \in [N_{mp+p-i},\dots,N_{mp+p-1}] = [N_{-i},\dots,N_{-1}],
\end{aligned}
\]
thus $\Sigma^r \bM \in B_{C,\bN,p}$ yields that $\bpsi(\Sigma^{np+mp} \bM) \in B_C^{(0)}$, i.e., $\bpsi(\bM) \in B_C^{(n+m)}$.
Therefore, for a.e.\ $\bM \in D$, $\bpsi(\bM)$ satisfies the three conditions of Lemma~\ref{lem:sadic7.9crit} and therefore satisfies the tiling condition, with compact tiles by Lemma~\ref{lem:RFbalanced}. Pure discrete spectrum of $(\bpsi(\bM), \Sigma)$ follows immediately from Theorem~\ref{t:tilingpds} (recall the equivalences in Prioposition~\ref{p:tilingw}).
\end{proof}

\subsection{Tiling condition and pure discrete spectrum for multidimensional continued fraction algorithms} 
\label{subsec:pdsmcf}
We conclude this section by formulating Theorem~\ref{theo:metrictilingaccel} in the context of multidimensional continued fraction algorithms (compare \cite[Theorem~3.1]{BST:23}).

\begin{theorem} \label{theo:pureMCF}
Let $(X,F,A,\nu)$ be a positive multidimensional continued fraction algorithm satisfying the Pisot condition and $\nu \circ F \ll \nu$. Let $\bphi$ be a faithful substitutive realization of a natural extension $(\hX,\hF,\hA,\hnu)$ of $(X,F,A,\nu)$. Assume that there is a periodic Pisot point $(\bx_0,\by_0) \in \hX$ such that $\bx_0$ has positive range in $(X,F,A,\nu)$ and $\bphi(\bx_0,\by_0)$ has pure discrete spectrum.
Then, for $\hnu$-almost all $(\bx,\by) \in \hX$, $\bphi(\bx_0,\by_0)$ satisfies the tiling condition and $(\bphi(\bx_0,\by_0),\Sigma)$ has pure discrete spectrum.
\end{theorem}

\begin{proof}
By Remark~\ref{rem:1vs2}, the Pisot condition for $(X,F,A,\nu)$ implies that the natural extension $(\hX,\hF,\hA,\hnu)$ satisfies the Pisot condition as well. By \eqref{eq:phi} we have $(\hX,\hF,\hA,\hnu) \overset{\bphi}{\cong} (\bphi(\hX),\Sigma,\bphi_*\hnu)$, thus $(\bphi(\hX),\Sigma,\bphi_*\hnu)$ satisfies the Pisot condition.  By Lemma~\ref{l:positiverange}, for each periodic Pisot point $(\bx_0,\by_0) \in \hX$ such that $\bx_0$ has positive range in $(X,F,A,\nu)$, the sequence $\bphi(\bx_0,\by_0)$ is a periodic Pisot sequence with two-sided positive range in $(\bphi(\hX),\Sigma,\bphi_*\hnu)$. Thus the dynamical system $(\bphi(\hX),\Sigma,\bphi_*\hnu)$ satisfies all the conditions of Theorem~\ref{prop:combcond2}, which proves the theorem.
\end{proof}

Obviously, the result on accelerations contained in Theorem~\ref{theo:metrictilingaccel} can also be formulated in terms of multidimensional continued fraction algorithms. We omit the details.

\chapter{Induced rotations and nonstationary Markov partitions}
\label{chapter:markov}

This chapter is the core part of the present work. In Section~\ref{sec:induced}, we state and prove our results on induced rotations and their relations to multidimensional continued fraction algorithms.
Indeed, we show that these algorithms act as renormalization processes for cascades of toral rotations. In Section~\ref{sec:markov}, we use the tiling property to construct nonstationary Markov partitions for $\cS$-adic mapping families and associate nonstationary edge shifts with them. In Section~\ref{sec:metricMP}, we establish metric results on nonstationary Markov partitions for shifts of $\cS$-adic mapping families. The crucial condition needed here is the Pisot condition. This section also contains interpretations of these results for multidimensional continued fraction algorithms. Again, the guiding example used to illustrate our theory is furnished by the Brun continued fraction algorithm in Section~\ref{subsec:BrunS}.

\section{Induced rotations}\label{sec:induced}
We have announced in Section~\ref{sec:introcf} that, for a multidimensional continued fraction algorithm $F:\, X\to X$, we can define rotations~$\fr_{\bx}$, $\bx \in X$, on~$\TT^{d-1}$ such that the induced map of $\fr_{\bx}$ on a certain subset of~$\TT^{d-1}$ yields the rotation~$\fr_{F(\bx)}$; see Definition~\ref{def:induction}. We provide the details of this construction in the present section. Before we do that, we formulate  according results for $\cS$-adic systems both  for a single sequence of substitutions (cf.\ Section~\ref{subsec:induced}) and in the metric setting (cf.\ Section~\ref{sec:metricrot}), and then consider induced rotations for continued fraction algorithms in Section~\ref{subsec:inductioncf}.

\subsection{Induced rotations for a single sequence}\label{subsec:induced}
Consider a sequence of unimodular substitutions $\bsigma \in \cS_d^{\ZZ}$ that satisfies the tiling condition from Definition~\ref{def:tilingcond} with compact Rauzy fractals~$\cR_n$ as defined in Definition~\ref{def:rauzy}. 
Recall that the $\cS$-adic shift $(X_{\bsigma}^{(n)},\Sigma)$ (defined in Definition~\ref{def:sadicshift}) is conjugate to the rotation by $\pi_{\bu_n,\bone} \be_d$ on $\bone^\perp  / (\ZZ^d \cap \bone^\perp)$ by Theorem~\ref{t:tilingpds} and its proof.
To see the induction process, it is more convenient to view these rotations on the hyperplanes $\bv_n^\perp$.
Using the projections~$\pi_n$ defined in~\eqref{eq:abbrproj}, we obtain the commutative diagram
\begin{equation} \label{e:pinrotation}
\begin{tikzcd}[column sep=4em]
\bone^\perp  / (\ZZ^d \cap \bone^\perp) \arrow[r, "+\pi_{\bu_n,\bone}\be_d"] \arrow[d,"\pi_n"] & \bone^\perp  / (\ZZ^d \cap \bone^\perp) \arrow[d, "\varphi"] \\
\bv_n^\perp  /\pi_n(\ZZ^d \cap \bone^\perp) \arrow[r, "+\pi_n\be_d"]& \bv_n^\perp  /\pi_n(\ZZ^d \cap \bone^\perp)
\end{tikzcd}
\end{equation}
(Here, the vector~$\be_d$ can be replaced by any~$\be_i$, $i\in\cA$, because the difference $\be_d {-} \be_i$ is in~$\bone^\perp$.)
Note that
\begin{equation} \label{e:Tsigman}
\TT_{\bsigma}^{(n)} = \bv_n^\perp  /\pi_n(\ZZ^d \cap \bone^\perp) \notx{Tdsigma}{$\TT_{\bsigma}^{(n)}$}{torus, domain of $\tilde{\fr}_{\bsigma}^{(n)}$}
\end{equation} 
is a $(d{-}1)$-dimensional torus with fundamental domain~$\cR_n$. 
Therefore, we write the rotation ``$+\pi_n\be_d$" on $\TT_{\bsigma}^{(n)}$ in \eqref{e:pinrotation} as
\[
\tilde{\fr}_{\bsigma}^{(n)}:\, \cR_n \to \cR_n,\quad \bz \mapsto \bz + \pi_n \be_d \mod \pi_n(\ZZ^d \cap \bone^\perp) \notx{rota}{$\tilde{\fr}_{\bsigma}^{(n)}$}{cascade of rotations}
\]
and obtain that $(X_{\bsigma}^{(n)},\Sigma)$ is measurably conjugate to $(\cR_n, \tilde{\fr}_{\bsigma}^{(n)})$.
All dynamical systems whose transformations are rotations are equipped with the Lebesgue measure. 

Our aim is to show that the $\cS$-adic structure of the shift $(X_{\bsigma}^{(0)},\Sigma)$ can be geometrically interpreted by the fact that the induced map of~$\tilde{\fr}_{\bsigma}^{(0)}$ on the subset $M_{\sigma_{[0,n)}} \cR_n$ of~$\cR_0$
is conjugate to the rotation~$\tilde{\fr}_{\bsigma}^{(n)}$. In particular, we will show that the following diagram is commutative for all $n \in \NN$:
\[
\begin{tikzcd}[row sep=7ex,column sep=10em]
\cR_n  \arrow[r,"\displaystyle \tilde{\fr}_{\bsigma}^{(n)}"] \arrow[d,"\displaystyle M_{\sigma_{[0,n)}}"] & \cR_n \arrow[d,"\displaystyle M_{\sigma_{[0,n)}}"] \\
M_{\sigma_{[0,n)}}\cR_n \arrow[r,"\displaystyle\text{induced map of}\ \tilde{\fr}_{\bsigma}^{(0)}"] & M_{\sigma_{[0,n)}}\cR_n
\end{tikzcd}
\]
We state this result in the following theorem, which is inspired by the concept of \emph{$\sigma$-structure} introduced in \cite{Arnoux-Ito:01}. It is the exact version of Theorem~\nameref{t:C} formulated in Section~\ref{sec:introresults}. 

\begin{theorem} \label{t:Frot}
Let $\bsigma = (\sigma_n)_{n\in\ZZ} \in \cS_d^{\ZZ}$, with $d \ge 2$, be a sequence of unimodular substitutions that satisfies the tiling condition.
For $m,n\in\ZZ$, $m<n$, the induced map of the rotation~$\tilde{\fr}_{\bsigma}^{(m)}$ on the subset $M_{\sigma_{[m,n)}}\cR_n$ of~$\cR_m$  equals (after renormalization) the rotation~$\tilde{\fr}_{\bsigma}^{(n)}$.
\end{theorem}

\begin{proof}
By the tiling condition and Proposition~\ref{p:ntilingbox}, iterating the subdivision \eqref{e:setequationkl} and taking the union over $a \in \cA$ gives the (measurable) partition
\begin{equation} \label{e:partitionRm}
\cR_m = \bigcup_{\substack{p\in\cA^*,\,b\in\cA:\\p\prec\sigma_{[m,n)}(b)}} \big(\pi_m \bl(p) +  M_{\sigma_{[m,n)}} \cR_n(b)\big). 
\end{equation}
We claim that the induced map of $\tilde{\fr}_{\bsigma}^{(m)}$ on $M_{\sigma_{[m,n)}} \cR_n$ maps points $M_{\sigma_{[m,n)}} \bz$ with $\bz \in \cR_n(b)$, $b \in \cA$, to $M_{\sigma_{[m,n)}} (\bz {+} \pi_n \be_b)$.
Indeed, for $1 \le k < |\sigma_{[m,n)}(b)|$, we have $(\tilde{\fr}_{\bsigma}^{(m)})^k(M_{\sigma_{[m,n)}} \bz) = \pi_m \bl(p) + M_{\sigma_{[m,n)}} \bz$ with the prefix $p$ of $\sigma_{[m,n)}(b)$ of length~$k$.
Because the union in \eqref{e:partitionRm} is disjoint, this implies that $(\tilde{\fr}_{\bsigma}^{(m)})^k(M_{\sigma_{[m,n)}} \bz) \notin M_{\sigma_{[m,n)}} \cR_n$.
For $k = |\sigma_{[m,n)}(b)|$, we have
\[
(\tilde{\fr}_{\bsigma}^{(m)})^k(M_{\sigma_{[m,n)}} \bz) = \pi_m\bl(\sigma_{[m,n)}(b)) + M_{\sigma_{[m,n)}} \bz = M_{\sigma_{[m,n)}} (\bz {+} \pi_n \be_b) \in M_{\sigma_{[m,n)}} \cR_n
\]
by the definition of~$\cR_n(b)$ and~$\cR_n$.
This proves the claim, and renormalization by $M_{\sigma_{[m,n)}}^{-1}$ gives the rotation~$\tilde{\fr}_{\bsigma}^{(n)}$. 
\end{proof}

\subsection{Metric theory for induced rotations}\label{sec:metricrot}
Let $(D,\Sigma,\nu)$ be a dynamical system with $D \subset \cS_d^{\ZZ}$ and a shift invariant Borel measure~$\nu$. We want to prove that, under the Pisot condition, $\nu$-a.e.\ element of $D$ satisfies the conditions of Theorem~\ref{t:Frot}. This leads to the following metric version of this theorem.

\begin{theorem} \label{theo:metricRot}
Let $(D, \Sigma, \nu)$, with $D \subset \cS_d^{\ZZ}$, $d \ge 2$, be an ergodic dynamical system that satisfies the Pisot condition and contains a periodic Pisot sequence having two-sided positive range and pure discrete spectrum.   
Then, for $\nu$-almost every $\bsigma = (\sigma_n)_{n\in\ZZ} \in D$, for all $m,n\in\ZZ$, $m<n$, the induced map of the rotation~$\tilde{\fr}_{\bsigma}^{(m)}$ on the subset $M_{\sigma_{[m,n)}}\cR_n$ of~$\cR_m$  equals (after renormalization) the rotation~$\tilde{\fr}_{\bsigma}^{(n)}$.
\end{theorem}

\begin{proof}
By Theorem~\ref{prop:combcond2}, the tiling condition is satisfied for a.e.\ $\bsigma \in D$, hence we can apply Theorem~\ref{t:Frot}.
\end{proof}

While in Theorem~\ref{theo:metricRot} we started with a shift of sequences of substitutions, in the following theorem, we start with sequences of $d{\times}d$ matrices~$\bM$ and use a substitution assignment $\varrho: \cM_d \to \cS_d$ (see Definition~\ref{d:realization}) to associate a sequence of substitutions~$\varrho(\bM)$ and, by means of these substitutions, a sequence of induced rotations.

\begin{theorem}\label{theo:metricRot2}
Let $(D, \Sigma, \nu)$, with $D \subset \cM_d^{\ZZ} $, $d \ge 2$, be an ergodic dynamical system that satisfies the Pisot condition. Assume that there is a substitution assignment~$\varrho$ on~$\cM_d$ for which there is a periodic Pisot sequence in $\varrho(D)$ having two-sided positive range in $(\varrho(D),\Sigma,\varrho_*\nu)$ and pure discrete spectrum. 
Then, for $\nu$-almost every $\bM = (M_n)_{n\in\ZZ} \in D$, setting $\bsigma = \varrho(\bM)$, for all $m,n\in\ZZ$, $m<n$, the induced map of the rotation~$\tilde{\fr}_{\bsigma}^{(m)}$ on the subset $M_{[m,n)}\cR_n$ of~$\cR_m$  equals (after renormalization) the rotation~$\tilde{\fr}_{\bsigma}^{(n)}$.
\end{theorem}

\begin{proof}
Theorem~\ref{theo:metricRot2} follows immediately from Theorem~\ref{theo:metricRot} because the system $(\varrho(D),\Sigma,\varrho_*\nu)$ satisfies all conditions of Theorem~\ref{theo:metricRot}. 
\end{proof}

Using Theorem~\ref{theo:metrictilingaccel}, we do not need have to assume pure discrete spectrum for a Pisot sequence when we consider appropriate accelerations of the shift map.
Recall that $\bM_p$ denotes the $p$-fold blocking of the sequence~$\bM$, as in \eqref{e:Mp}.

\begin{theorem}\label{theo:metricRotAcc}
Let $(D, \Sigma, \nu)$, with $D \subset \cM_d^{\ZZ}$, $d \ge 2$, be an ergodic dynamical system that satisfies the Pisot condition and has a periodic Pisot sequence with two-sided positive range.
Then there exists a positive integer~$p$ and a substitution assignment~$\varrho$ on~$\cM_d$ with the following property. 
For $\nu$-almost every $\bM = (M_n)_{n\in\ZZ} \in D$, setting $\bsigma = \varrho(\bM_p)$, for all $m,n\in\ZZ$, $m<n$, the induced map of the rotation~$\tilde{\fr}_{\bsigma}^{(m)}$ on the subset $M_{[mp,np)}\cR_n$ of~$\cR_m$  equals (after renormalization) the rotation~$\tilde{\fr}_{\bsigma}^{(n)}$.
\end{theorem}

\begin{proof}
By Theorem~\ref{theo:metrictilingaccel}, all conditions of Theorem~\ref{t:Frot} are satisfied.
\end{proof}

\subsection{Induced rotations for continued fraction algorithms} \label{subsec:inductioncf}
Recall again some facts about multidimensional continued fraction algorithms stated earlier.
Let $(X,F,A,\nu)$ be such an algorithm with natural extension  $(\hX,\hF,\hA,\hnu)$. In Definition~\ref{d:realization}, we defined the faithful substitutive realization $\bphi$ of $(\hX,\hF,\hA,\hnu)$.
If $(X, F, A,\nu)$ converges weakly in the future and in the past at $(\bx, \by)$ for $\hnu$-almost all $(\bx, \by)\in \hX$, then we know from \eqref{eq:phi} that the $\cS$-adic shift $(\bphi(\hX),\Sigma,\bphi_*\nu)$ is measurably conjugate to $(\hX, \hF,\hA,\hnu)$.
The results of Sections~\ref{subsec:induced} and~\ref{sec:metricrot} can be applied to $\bphi(\bx,\by)$.
However, in accordance with the classical case discussed in Section~\ref{sec:introsturm}, we prefer to view multidimensional continued fraction algorithms as induction processes on the set of rotations on the fixed torus $\bone^\perp / (\ZZ^d \cap \bone^\perp)$.

Recall the map $\chi: \PP^{d-1} \to \Delta^{d-1}$ from \eqref{eq:chidef}, and set
\begin{equation} \label{e:rtilde}
\fr_{\bx}:\ \bone^\perp / (\ZZ^d \cap \bone^\perp) \to \bone^\perp / (\ZZ^d \cap \bone^\perp),\ \bz \mapsto \bz + \be_d - \chi(\bx) \bmod{\ZZ^d \cap \bone^\perp} \notx{rot}{$\fr_{\balpha}, \fr_{\bx}$}{toral rotation}
\end{equation}
for $\bx \in \PP^{d-1}$.
Direct calculation yields that $\fr_{\bx}$ is a version of $\tilde{\fr}_{\bphi(\bx,\by)}$ acting on $\bone^\perp / (\ZZ^d \cap \bone^\perp)$ for all $\by$ satisfying $(\bx,\by) \in \hX$.
More precisely, we have
\[
\fr_{\bx} = \pi_{\chi(\bx),\bone} \circ \tilde{\fr}_{\bphi(\bx,\by)} \circ \pi_{\chi(\bx),\chi(\by)}
\]
and, for all $n \in \ZZ$, the following diagram commutes (with $\TT_{\bsigma}^{(n)}$ defined in \eqref{e:Tsigman}):
\begin{equation} \label{e:tilder}
\begin{tikzcd}[column sep=4em]
\TT_{\bphi(\bx,\by)}^{(n)} \arrow[r, "\tilde{\fr}_{\bphi(\bx,\by)}^{(n)}"] \arrow[d,"\pi_{\chi(F^n(\bx)),\bone}"] & \TT_{\bphi(\bx,\by)}^{(n)} \arrow[d, "\pi_{\chi(F^n(\bx)),\bone}"] \\
\bone^\perp / (\ZZ^d \cap \bone^\perp) \arrow[r, "\fr_{F^n(\bx)}"]& \bone^\perp / (\ZZ^d \cap \bone^\perp)
\end{tikzcd}
\end{equation}
(Note that the projections are bijective in this diagram.)
For convenience of notation, the following theorems are formulated in terms of natural extensions, even though the occurring rotations depend only on~$\bx$. 

\begin{theorem} \label{t:FrotMCF}
Let $(X,F,A,\mu)$ be a positive multidimensional continued fraction algorithm with natural extension $(\hX,\hF,A,\hmu)$ and substitutive realization~$\bphi$. 
Let $\bsigma = \bphi(\bx,\by) \in \cS_d^{\ZZ}$ with $(\bx,\by)\in \hX$.
If $\bsigma$ satisfies the tiling condition with compact Rauzy fractals, then, for all $n\in\NN$, the induced map of the rotation $\fr_{\bx}$ on the subset $\pi_{\chi(\bx),\bone} \tr{\!A}^{(n)}(\bx) \cR_n$ of~$\cR_0^{\bone}$ equals (after renormalization) the rotation~$\fr_{F^n(\bx)}$.
\end{theorem}

\begin{proof}
This follows from Theorem~\ref{t:Frot} and \eqref{e:tilder}.
\end{proof}

According to Theorem~\ref{t:FrotMCF}, the multidimensional continued fraction algorithm~$F$ can be viewed as a mapping acting on the parameter space of rotations on the torus $\bone^\perp / (\ZZ^d \cap \bone^\perp) \simeq \TT^{d-1}$.
As in the classical case (discussed in Section~\ref{sec:introsturm}), the recurrence properties of the rotation~$\fr_{\bx}$ are related to the continued fraction expansion of an element~$\bx$.
Indeed, good approximation properties of the convergents correspond to good contraction properties of $\tr{\!A}^{(n)}(\bx)$, which entail that $\pi_{\chi(\bx),\bone} \tr{\!A}^{(n)}(\bx) \cR_n$ is small and that the induced map $\fr_{\bx}$ is close to the identity.
Recall that exponential contraction of $\tr{\!A}^{(n)}(\bx)$ is ensured by the Pisot condition, and the speed of convergence is governed by the first and second Lyapunov exponents of the algorithm.

For the following statement, recall the Pisot condition from Definition~\ref{def:MCF_Pisot} and the notation $\ll$ from \eqref{e:abscontinuity}.

\begin{theorem} \label{theo:metricRotMCF}
Let $(X,F,A,\nu)$ be a positive multidimensional continued fraction algorithm satisfying the Pisot condition and $\nu \circ F \ll \nu$. Let $\bphi$ be a faithful substitutive realization of a natural extension $(\hX,\hF,\hA,\hnu)$ of $(X,F,A,\nu)$. Assume that there is a periodic Pisot point $(\bx_0,\by_0) \in \hX$ such that $\bx_0$ has positive range in $(X,F,A,\nu)$ and $\bphi(\bx_0,\by_0)$ has pure discrete spectrum.
Then, for $\hnu$-almost all $(\bx,\by) \in \hX$, setting $\bsigma = \bphi(\bx,\by)$, for all $n\in\NN$, the induced map of the rotation~$\fr_{\bx}$ on the subset $\pi_{\chi(\bx),\bone} \tr{\!A}^{(n)}(\bx) \cR_n$ of~$\cR_0^\bone$ equals (after renormalization) the rotation~$\fr_{F^n(\bx)}$.
\end{theorem}

\begin{proof}
This follows from Theorem~\ref{theo:metricmarkov} in the same way as Theorem~\ref{theo:pureMCF} follows from Theorem~\ref{prop:combcond2}. By Remark~\ref{rem:1vs2}, the Pisot condition for $(X,F,A,\nu)$ implies that the natural extension $(\hX,\hF,\hA,\hnu)$ satisfies the Pisot condition as well. By \eqref{eq:phi} we have $(\hX,\hF,\hA,\hnu) \overset{\bphi}{\cong} (\bphi(\hX),\Sigma,\bphi_*\hnu)$, thus $(\bphi(\hX),\Sigma,\bphi_*\hnu)$ satisfies the Pisot condition.  By Lemma~\ref{l:positiverange}, for each periodic Pisot point $(\bx_0,\by_0) \in \hX$ such that $\bx_0$ has positive range in $(X,F,A,\nu)$, the sequence $\bphi(\bx_0,\by_0)$ is a periodic Pisot sequence with two-sided positive range in $(\bphi(\hX),\Sigma,\bphi_*\hnu)$. Thus the dynamical system $(\bphi(\hX),\Sigma,\bphi_*\hnu)$ satisfies all the conditions of Theorem~\ref{theo:metricRot}, which proves the theorem in view of \eqref{e:tilder}.
\end{proof}

We saw in Theorem~\ref{theo:metricRotAcc} that shifts on sequences of matrices can always be accelerated in a way that the pure discrete spectrum condition is not required. The same is true in the continued fraction setting. This is the content of the following theorem. 
 
\begin{theorem}  \label{theo:RotAcc}
Let $(X,F,A,\nu)$ be a positive multidimensional continued fraction algorithm satisfying the Pisot condition, $\nu \circ F \ll \nu$. Let $(\hX,\hF,\hA,\hnu)$ be a natural extension of $(X,F,A,\nu)$. Assume that there is a periodic Pisot point $\bx_0 \in X$ with $\bx_0$ having positive range in $(X,F,A,\nu)$. Then there exist a positive integer~$p$ and a (faithful) substitutive realization~$\bphi$ of $(X,F^p,A,\nu)$ with the following property.
For $\hnu$-almost all $(\bx,\by) \in \hX$, setting $\bsigma = \bphi(\bx,\by)$, for all $n\in\NN$, the induced map of the rotation~$\fr_{\bx}$ on the subset $\pi_{\chi(\bx),\bone} \tr{\!A}^{(np)}(\bx) \cR_n$ of~$\cR_0^\bone$, regarded as fundamental domain of $\bone^\perp / (\ZZ^d \cap \bone^\perp)$, equals (after renormalization) the rotation~$\fr_{F^{np}(\bx)}$.
\end{theorem}

\begin{proof}
This follows from Theorem~\ref{theo:metricRot2} in the same way as Theorem~\ref{theo:metricRotMCF} follows from Theorem~\ref{theo:metricRot}.
\end{proof}

\section{Nonstationary Markov partitions for mapping families}\label{sec:markov} 
Given a bi-infinite sequence~$\bM$ of nonnegative unimodular $d{\times}d$ integer matrices, we can interpret this sequence as a mapping family $(\TT,f_{\bM})$, i.e., as a sequence of toral automorphisms. The main results of this section state that (under certain natural conditions) we can define a nonstationary generating Markov partition for $(\TT,f_{\bM})$ by superimposing a combinatorial structure upon~$\bM$, i.e., by considering a sequence of substitutions $\bsigma$ such that $\bM$ is the sequence of incidence matrices~$\bM_{\bsigma}$. This allows to conjugate $(\TT,f_{\bM})$ to a nonstationary edge shift. We make this more precise in the following paragraph.

Let $ \bsigma \in \cS_d^{\ZZ}$, with $d \ge 2$, be a bi-infinite sequence of unimodular substitutions. We interpret the abelianization of~$\bsigma$, i.e., the sequence~$\bM_{\bsigma}$ of its incidence matrices, as an $\cS$-adic mapping family $(\TT,f_{\bsigma}) = (\TT,f_{\bM_{\bsigma}})$, whose mappings are the inverses of the elements of~$\bM_{\bsigma}$; see Definition~\ref{def:SadicmappingS}. In Section~\ref{subsec:MarkovRB}, we discuss the restacking process  used in the proof of Proposition~\ref{p:ntilingbox} and show that the Rauzy boxes of $\bsigma$ can be used to define atoms of a nonstationary Markov partition for $(\TT,f_{\bsigma})$. This means that we use the zero entropy dynamical system $(X_{\bsigma},\Sigma)$ to define the atoms of a Markov partition for the system $(\TT,f_{\bsigma})$ of positive entropy. In order to make this work, we need to assume the tiling condition from Definition~\ref{def:tilingcond} (with compact Rauzy fractals).
If $\bsigma$ is two-sided primitive, then the above mentioned Markov partition of $(\TT,f_{\bsigma})$ can be refined in a way that it becomes generating. This refined partition will be treated in Section~\ref{subsec:refined}.  
In Section \ref{subsec:nsftSub}, we use these Markov partitions to code the mapping family $(\TT,f_{\bsigma})$ by a nonstationary edge shift. 

\subsection{Markov partition defined in terms of Rauzy boxes} \label{subsec:MarkovRB}
Let $\bsigma \in \cS_d^{\ZZ}$, with $d \ge 2$, be a sequence of unimodular substitutions. We want to use Rauzy boxes in order to define a nonstationary Markov partition for the $\cS$-adic mapping family $(\TT,f_{\bsigma})$. In particular, let~$\hR_n$ be the $\cS$-adic Rauzy boxes associated to~$\bsigma$. 
We assume that $\hR_n$ is compact, so that it is also closed when we consider it is a subset of~$\TT^d$. 
We will show, under the tiling condition, that the interiors of the subtiles~$\hR_n(a)$ taken modulo the lattice~$\ZZ^d$ have the nonstationary Property~$M$ (see Definition~\ref{def:M}) and, hence, in view of Proposition~\ref{prop:MMarkov}, form a nonstationary Markov partition for the mapping family $(\TT,f_{\bsigma})$. 

If we assume the tiling condition and that the Rauzy fractals are compact, then the boundary of $\hR_n(a)$ has measure zero for all $a \in \cA$, and the collection 
\begin{equation} \label{e:cPn}
\cP_n = \big\{ \rint(\hR_n(a)) \bmod{\ZZ^d} \,:\, a \in \cA \big\} \notx{Pa}{$\cP_n$}{nonstationary Markov partition}
\end{equation}
forms a topological partition of the $d$-dimensional torus $\TT_n$ for each $n \in \ZZ$.
We recall that criteria for this tiling condition are provided in Section~\ref{sec:rauzy}.

In order to prove that $(\cP_n)_{n\in\ZZ}$ satisfies the nonstationary Property~M (see Definition~\ref{def:M}), we need to come up with suitable horizontal and vertical partitions $H_{n,a}$ and~$V_{n,a}$ for each element of~$\cP_n$, $n \in \ZZ$. 
To this end, we set, for each $n \in \ZZ$, $a \in \cA$, $\bx \in \rint(\hR_n(a))$,
\begin{equation}\label{eq:hvsmall}
\begin{aligned}
h_n(\bx \bmod \ZZ^d) & =  \tilde{\pi}_n \bx - \rint(\cR_n(a)) \bmod \ZZ^d, \notx{hnatom}{$h_n(\cdot), h_n^\gen(\cdot)$}{\hspace{.75em}atom of a horizontal partition} \\
v_n(\bx \bmod \ZZ^d) & = \pi_n \bx + \rint(\tilde{\pi}_n \llbracket\be_a\rrbracket) \bmod \ZZ^d, \notx{vnatom}{$v_n(\cdot), v_n^\gen(\cdot)$}{\hspace{.75em}atom of a vertical partition}
\end{aligned}
\end{equation}
where the projections $\pi_n, \tilde{\pi}_n$ are defined in \eqref{eq:abbrproj}.
\begin{figure}[ht]
\includegraphics{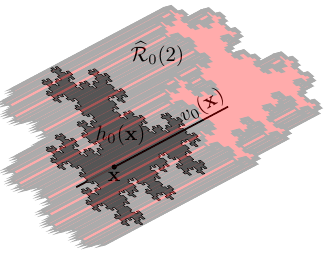}
\caption{Atoms of a transverse partition of a Rauzy box}
\label{fig:bruntransverse}
\end{figure}
See Figure~\ref{fig:bruntransverse} for an example.
This immediately implies that 
\begin{equation}\label{eq:HV}
\begin{aligned}
H_{n,a} & = \big\{h_n(\bx \bmod \ZZ^d) \,:\, \bx\in \rint(\hR_n(a))\big\}, \notx{Hn}{$H_{n,\cdot}, H_{n,\cdot}^\gen$}{horizontal partition} \\
V_{n,a} & = \big\{v_n(\bx \bmod \ZZ^d) \,:\, \bx\in \rint(\hR_n(a))\big\} \notx{Vn}{$V_{n,\cdot}, V_{n,\cdot}^\gen$}{vertical partition}
\end{aligned}
\end{equation}
are transverse partitions of $\rint(\hR_n(a)) \bmod \ZZ^d$ for all $n \in \ZZ$, $a \in \cA$.

\begin{theorem}\label{thm:MarkovCoarse}
Let $\bsigma \in \cS_d^{\ZZ}$, with $d \ge 2$, be a sequence of unimodular substitutions that satisfies the tiling condition with compact Rauzy fractals. 
Then the sequence $(\cP_n)_{n\in\ZZ}$ defined in \eqref{e:cPn} forms a nonstationary Markov partition for the mapping family $(\TT,f_{\bsigma})$ associated to~$\bsigma$.
\end{theorem}

\begin{proof}
By Proposition~\ref{prop:MMarkov}, we have to show that the mapping family $(\TT,f_{\bsigma})$ associated to~$\bsigma$ has Property~M. Looking at Definition~\ref{def:M}, we thus have to prove that, for each $n \in \ZZ$, $a \in \cA$, the atoms of the partitions $H_{n,a}$ and~$V_{n,a}$ defined in \eqref{eq:HV} satisfy
\begin{equation} \label{e:propertyMsigma}
\begin{gathered}
v_n(M_{\sigma_n}\bx \bmod \ZZ^d) \subset M_{\sigma_n} v_{n+1}(\bx \bmod \ZZ^d), \\ 
M_{\sigma_n} h_{n+1}(\bx \bmod \ZZ^d) \subset h_n(M_{\sigma_n}\bx \bmod \ZZ^d),
\end{gathered}
\end{equation}
for all $\bx \in (M_{\sigma_n}^{-1} \rint(\hR_n(a)) {+} \ZZ^d) \cap \rint(\hR_{n+1}(b))$, $b \in \cA$.

Since $M_{\sigma_n} \bx \in (\rint(\hR_n(a)) {+} \ZZ^d) \cap M_{\sigma_n} \rint(\hR_{n+1}(b))$, we obtain by \eqref{e:MnRn} and \eqref{e:Rapb} that $M_{\sigma_n} \bx \in \bl(p) {+} \rint(\hR_n(a,p,b))$ for some $p \in \cA^*$ such that $(a,p,b) \in E_n$. 
Therefore, using \eqref{e:Mneb}, 
\begin{gather*}
\begin{aligned}
v_n(M_{\sigma_n}\bx) & = \pi_n \big(M_{\sigma_n} \bx {-} \bl(p)) + \rint(\tilde{\pi}_n \llbracket\be_a\rrbracket\big) \equiv M_{\sigma_n} \pi_{n+1} \bx + \tilde{\pi}_n \bl(p) {+} \rint(\tilde{\pi}_n \llbracket\be_a\rrbracket) \\
& \subset M_{\sigma_n} \pi_{n+1} \bx + M_{\sigma_n} \rint(\tilde{\pi}_{n+1} \llbracket\be_b \rrbracket) = M_{\sigma_n} v_{n+1}(\bx), 
\end{aligned} \\[.5ex]
\begin{aligned}
M_{\sigma_n} h_{n+1}(\bx) & = M_{\sigma_n} \big(\tilde{\pi}_{n+1} \bx {-} \rint(\cR_{n+1}(b))\big) = \tilde{\pi}_n M_{\sigma_n} \bx - M_{\sigma_n} \rint(\cR_{n+1}(b)) \\
& = \tilde{\pi}_n M_{\sigma_n} \bx + \pi_n \bl(p) - \rint(\cR_n(a,p,b)) \\
& \subset \tilde{\pi}_n (M_{\sigma_n} \bx {-} \bl(p)) -  \rint(\cR_n(a)) \equiv h_n(M_{\sigma_n}\bx),
\end{aligned}
\end{gather*}
where we have omitted $\bmod\ \ZZ^d$ for readability.
Hence, \eqref{e:propertyMsigma} holds, which proves the theorem.
\end{proof}

We now provide a geometric interpretation of the Markov property in terms of the \emph{restacking}\indx{restacking} process on the Rauzy boxes used in the proof of Proposition~\ref{p:ntilingbox}. This reflects the way classical Markov partitions of toral automorphisms are viewed: one first distorts their atoms linearly and then moves them back to the fundamental domain. In our context, this ``moving back'' can be visualized by restacking subtiles of Rauzy boxes modulo the lattice~$\ZZ^d$.  In particular, let $\hR_n$ be the $\cS$-adic Rauzy boxes associated to~$\bsigma$. Looking at the horizontal slices $h_n(\bx) = \tilde{\pi}_n \bx {-} \rint(\cR_n(a))$ and the vertical lines $v_n(\bx) = \pi_n \bx {+} \rint(\llbracket \tilde{\pi}_n\be_a\rrbracket)$ of the subboxes~$\hR_n(a)$ during the restacking process performed in the proof of Proposition~\ref{p:ntilingbox}, we can see Property~M (see Definition \ref{def:M}) in action. Indeed, this restacking process is the nonstationary version of the restacking done in the context of stationary Markov partitions of toral automorphisms; see for instance~\cite[Figures~9 and~10]{A98}. For the case $d = 2$, the nonstationary restacking process is explained in \cite{AF:05}.
Indeed, what we did shows the following.
The ``cylinders'' $\hR_n(a)$ of~$\hR_n$ are ``horizontal'' unions of slices of cylinders of~$M_{\sigma_n} \hR_{n+1}$, as can be seen from the stacking process in Figure~\ref{fig:restack1}. 
Viewing this process in the backwards direction, we see that the cylinders $\hR_{n+1}(b)$ are built by ``vertical'' piles of subtiles of $M_{\sigma_n}^{-1} \hR_n$.

\subsection{Generating Markov partition}\label{subsec:refined}
In the previous section, we saw that the sequence $(\cP_n)_{n\in\ZZ}$ defined in \eqref{e:cPn} forms a Markov partition for the mapping family $(\TT,f_{\bsigma})$. However, in general this Markov partition is too coarse to be \emph{generating}.
To get a generating Markov partition for a large class of sequences of substitutions, we need the finer partitions corresponding to $\hC_n^\gen$. 
On top of this, we need to assume strong convergence in the future and in the past.

Assume again that $\bsigma \in \cS_d^{\ZZ}$, $d \ge 2$, satisfies the tiling condition and that the Rauzy fractals~$\cR_n$ are compact.
Then, by Proposition~\ref{p:ntilingbox}, the collection~$\hC_n^\gen$ forms a tiling of~$\RR^d$ for all $n \in \ZZ$, and 
\begin{equation}\label{eq:markovQ}
\cP_n^\gen = \big\{ \rint(\hR_n(a,p,b)) \bmod \ZZ^d \,:\, (a,p,b) \in E_n 
\big\} \notx{Pagen}{$\cP_n^\gen$}{generating nonstationary Markov partition}
\end{equation}
forms a topological partition of the torus~$\TT_n$, which is a refinement of the topological partition~$\mathcal{P}_n$ defined in~\eqref{e:cPn}. 
Similarly to \eqref{eq:hvsmall}, we set
\[
\begin{aligned}
h_n^\gen(\bx \bmod \ZZ^d) & = \tilde{\pi}_n \bx - \rint(\cR_n(a,p,b)) \bmod \ZZ^d, \notx{hnatom}{$h_n(\cdot), h_n^\gen(\cdot)$}{\hspace{.75em}atom of a horizontal partition} \\
v_n^\gen(\bx \bmod \ZZ^d) & = \pi_n \bx + \rint(\tilde{\pi}_n \llbracket\be_a\rrbracket) \bmod \ZZ^d, \notx{vnatom}{$v_n(\cdot), v_n^\gen(\cdot)$}{\hspace{.75em}atom of a vertical partition} \end{aligned}
\]
for each $n \in \ZZ$, $(a,p,b) \in E_n$, $\bx \in \rint(\hR_n(a,p,b))$.
Again, this implies, for each $n \in \ZZ$ and $(a,p,b) \in E_n$, that 
\[
\begin{aligned}
H_{n,(a,p,b)}^\gen & = \{h_n^\gen(\bx \bmod \ZZ^d) \,:\, \bx\in \rint(\hR_n(a,p,b))\}, \notx{Hn}{$H_{n,\cdot}, H_{n,\cdot}^\gen$}{horizontal partition} \\ 
V_{n,(a,p,b)}^\gen & = \{v_n^\gen(\bx \bmod \ZZ^d) \,:\, \bx\in \rint(\hR_n(a,p,b))\} \notx{Vn}{$V_{n,\cdot}, V_{n,\cdot}^\gen$}{vertical partition}
\end{aligned}
\]
are transverse partitions of $\rint(\hR_n(a,p,b)) \bmod{\ZZ^d}$.
Note that $v_n^\gen(\bx) = v_n(\bx)$ for all $\bx \in \bigcup_{(a,p,b)\in E_n}\rint(\hR_n(a,p,b))$ (but usually $h_n^\gen(\bx) \neq h_n(\bx)$). 
Using these partitions, we shall prove that the mapping family $(\TT,f_{\bsigma})$ associated to~$\bsigma$ has Property~M and, hence, is a Markov partition. We then also show that this Markov partition is generating. Before we state the according result, we need some preparations. Indeed, the crucial part is to prove that the sets considered in Lemma~\ref{lem:inductiveIntersections} have small diameter; we will use strong convergence to prove this. 

\begin{lemma} \label{lem:refinedIntersections} 
Assume that $\bsigma = (\sigma_n)_{n\in\ZZ} \in \cS_d^{\ZZ}$ satisfies the tiling condition with compact Rauzy fractals.
Then, for each $n \in \ZZ$, $(a,p,b) \in E_n$, $(a',p',b') \in E_{n+1}$, we have
\[
\begin{aligned}
& \big(M_{\sigma_n}^{-1} \rint(\hR_n(a,p,b)) {+} \ZZ^d\big) \cap \rint(\hR_{n+1}(a',p',b')) \\
& \qquad = \begin{cases}\tilde{\pi}_{n+1} M_{\sigma_n}^{-1} \big(\bl(p) {+} \rint(\llbracket \be_a \rrbracket)\big) {-} \pi_{n+1} \bl(p') {-} M_{\sigma_{n+1}} \rint(\cR_{n+2}(b')) & \mbox{if}\ a' = b, \\ \emptyset & \mbox{otherwise}. \end{cases}
\end{aligned}
\]
\end{lemma}

\begin{proof}
First note that 
\begin{equation} \label{e:intRapb}
\begin{aligned}
& \big(M_{\sigma_n}^{-1} \rint(\hR_n(a,p,b)) {+} \ZZ^d\big) \cap \rint(\hR_{n+1}(a')) \\
& \qquad = \begin{cases}\tilde{\pi}_{n+1} M_{\sigma_n}^{-1} \big(\bl(p) {+} \rint(\llbracket \be_a \rrbracket)\big) {-} \rint(\cR_{n+1}(b)) & \mbox{if}\ a' = b, \\ \emptyset & \mbox{otherwise}. \end{cases}
\end{aligned}
\end{equation}
Indeed, we have 
\[
\begin{aligned}
 M_{\sigma_n}^{-1} \rint(\hR_n(a,p,b)) + \ZZ^d & = M_{\sigma_n}^{-1} \tilde{\pi}_n \rint(\llbracket \be_a \rrbracket) - M_{\sigma_n}^{-1} \pi_n \bl(p) - \rint(\cR_{n+1}(b)) + \ZZ^d \\
& = M_{\sigma_n}^{-1} \tilde{\pi}_n \big(\bl(p) + \rint(\llbracket \be_a \rrbracket) \big) - \rint(\cR_{n+1}(b)) + \ZZ^d 
\end{aligned}
\]
and $M_{\sigma_n}^{-1} \tilde{\pi}_n \big(\bl(p) {+} \rint(\llbracket \be_a \rrbracket) \big) {-} \rint(\cR_{n+1}(b)) \subset \rint(\hR_{n+1}(b))$, thus the tiling property of~$\hC_{n+1}$ implies \eqref{e:intRapb}.
If we intersecting with $\rint(\hR_{n+1}(a',p',b'))$ instead of $\rint(\hR_{n+1}(a'))$, we obtain the lemma.
\end{proof}

For $\ell \le n \le m$, $(a_k,p_k,b_k) \in E_k$ for $k \in \{\ell,\dots,m\}$, define
\[
\begin{aligned}
& \hR_n((a_k,p_k,b_k)_{\ell\le k\le m}) = \bigcap_{k=\ell}^{n-1} \big(M_{\sigma_{[k,n)}}^{-1} \rint(\hR_k(a_k,p_k,b_k)) {+} \ZZ^d\big) \\
& \hspace{5em} \cap 
\rint(\hR_n(a_n,p_n,b_n)) \cap \bigcap_{k=n+1}^m \big(M_{\sigma_{[n,k)}} \rint(\hR_k(a_k,p_k,b_k)) {+} \ZZ^d\big).
\end{aligned}
\]

\begin{lemma}\label{lem:inductiveIntersections}
Assume that $\bsigma = (\sigma_n)_{n\in\ZZ} \in \cS_d^{\ZZ}$ satisfies the tiling condition with compact Rauzy fractals.
Let $\ell \le n \le m$, $(a_k,p_k,b_k) \in E_k$ for $k \in \{\ell,\dots,m\}$. Then 
\[
\begin{aligned}
& \hR_n((a_k,p_k,b_k)_{\ell\le k\le m}) \\
& = \sum_{k=\ell}^{n-1} \tilde{\pi}_n M_{\sigma_{[k,n)}}^{-1} \bl(p_k) + \tilde{\pi}_n M_{\sigma_{[\ell,n)}}^{-1} \llbracket \be_{a_\ell} \rrbracket - \sum_{k=n}^m \pi_n \bl(\sigma_{[n,k)}(p_k)) - M_{\sigma_{[n,m)}} \cR_{m+1}(b_m)
\end{aligned} 
\]
if $b_k = a_{k+1}$ for all $k \in \{\ell,\dots,m{-}1\}$, and $\hR_n((a_k,p_k,b_k)_{\ell\le k\le m}) = \emptyset$ otherwise.
\end{lemma}

\begin{proof}
This follows by induction from Lemma~\ref{lem:refinedIntersections}. 
\end{proof}

The next theorem is the exact statement  of Theorem~\nameref{t:D}. The improvement with respect to Theorem~\ref{thm:MarkovCoarse} is that  the Markov partition  is now generating.

\begin{theorem}\label{thm:MarkovFine}
Let $\bsigma \in \cS_d^{\ZZ}$, with $d \ge 2$, be a sequence of unimodular substitutions that satisfies the tiling condition with compact Rauzy fractals.
Then the sequence $(\cP_n^\gen)_{n\in\ZZ}$ defined in \eqref{eq:markovQ} forms a nonstationary Markov partition for the mapping family $(\TT,f_{\bsigma})$.  
If moreover $\bsigma$ is two-sided primitive, then this nonstationary Markov partition is generating.
\end{theorem}

\begin{proof}
We first prove Property~M following the lines of the proof of Theorem~\ref{thm:MarkovCoarse}.
For each $\bx \in (M_{\sigma_n}^{-1} \rint(\hR_n(a,p,b)) {+} \ZZ^d) \cap \rint(\hR_{n+1}(a',p',b'))$, we have  $a' = b$, thus 
\[
v_n^\gen(M_{\sigma_n}\bx) = v_n(M_{\sigma_n}\bx) \subset M_{\sigma_n} v_{n+1}(\bx) = M_{\sigma_n} v_{n+1}^\gen(\bx)
\]
by the proof of Theorem~\ref{thm:MarkovCoarse} and, by Lemma~\ref{lem:refinedIntersections}, 
\[
\begin{aligned}
& M_{\sigma_n} h_{n+1}^\gen(\bx) = M_{\sigma_n} \big(\tilde{\pi}_{n+1} \bx {-} \rint(\cR_{n+1}(a',p',b'))\big) \\
& \hspace{3em}\subset \tilde{\pi}_n M_{\sigma_n} \bx - M_{\sigma_n} \rint(\cR_{n+1}(b)) = \tilde{\pi}_n M_{\sigma_n} \bx + \pi_n \bl(p) - \rint(\cR_n(a,p,b)) \\
& \hspace{3em} \equiv \tilde{\pi}_n (M_{\sigma_n} \bx {-} \bl(p)) -  \rint(\cR_n(a,p,b)) \equiv h_n^\gen(M_{\sigma_n} \bx),
\end{aligned}
\]
where we have omitted $\bmod\ \ZZ^d$ for readability.
Therefore, Property~M holds. 

Assume now that $\bsigma$ is two-sided primitive.
To prove that the Markov partition $(\cP_n^\gen)_{n\in\ZZ}$ is generating, we show that $(\hR_n((a_k,p_k,a_{k+1})_{-m\le k\le m}))_{m\in\NN}$, which is a nested sequence of sets, converges to a single point for all $(a_k,p_k,a_{k+1})_{k\in\ZZ} \in \prod_{k\in\ZZ} E_k$.
By Lemma~\ref{lem:inductiveIntersections}, we have
\begin{equation} \label{e:Rpoint}
\hspace{-.5em} \bigcap_{m\to\infty} \hspace{-.5em} \hR_n((a_k,p_k,a_{k+1})_{-m\le k\le m}) = \bigg\{ \sum_{k=-\infty}^{n-1} \hspace{-.5em} \tilde{\pi}_n M_{\sigma_{[k,n)}}^{-1} \bl(p_k) - \hspace{-.25em} \sum_{k=n}^\infty \pi_n \bl(\sigma_{[n,k)}(p_k)) \bigg\} \hspace{-.5em}
\end{equation}
when 
\begin{equation} \label{eq:to0} 
\hspace{-.5em} \lim_{m\to\infty} \hspace{-.25em} \mathrm{diam}(\tilde{\pi}_n M_{\sigma_{[-m,n)}}^{-1} \llbracket \be_{a_{-m}} \rrbracket) = 0 \ \mbox{and} \ \lim_{m\to\infty} \hspace{-.25em} \mathrm{diam}(M_{\sigma_{[n,m+1)}} \cR_{m+1}(b_m)) = 0. \hspace{-.5em}
\end{equation}
The convergence in \eqref{eq:to0} guarantees convergence of the sums in \eqref{e:Rpoint}. 
To prove the first assertion of \eqref{eq:to0}, note that, for each $m \in \NN$ and $a \in \cA$, we have $\tilde{\pi}_n M_{\sigma_{[-m,n)}}^{-1} \be_a \in \RR_{\ge0} \bu_n$ and $M_{\sigma_{[-m,n)}} \tilde{\pi}_n M_{\sigma_{[-m,n)}}^{-1} \be_a = \tilde{\pi}_{-m} \be_a \in \RR_{\ge0}^d \setminus [1,\infty)^d$.
By primitivity in the past, all coordinates of $M_{\sigma_{[-m,n)}}$ tend to~$\infty$ as $m \to \infty$, which implies that $\lim_{m\to\infty} \tilde{\pi}_n M_{\sigma_{[-m,n)}}^{-1} \be_a = \mathbf{0}$. 
To prove the second assertion, we want to apply \cite[Lemma~5.11]{BST:23}. 
Since we have assumed that $\cR_n$ is compact, the language~$\cL_{\bsigma}^{(n)}$ is balanced for each $n \in \ZZ$ by Lemma~\ref{lem:RFbalanced}. 
As we already saw in the proof of Theorem~\ref{t:tilingpds}, the tiling condition implies that $\bu_n$ has rationally independent coordinates. Thus the conditions of \cite[Lemma~5.11]{BST:23} are satisfied, i.e., we have $\lim_{m\to\infty} \sup\big\{ \lVert\pi_{\bu_n,\bone} M_{\sigma_{[n,m+1)}}\, \bl(w)\rVert \,:\,  w \in \cL_{\bsigma}^{(m+1)} \big\} = 0$, which implies the second assertion of \eqref{eq:to0} by the definition of the Rauzy fractals. Therefore, the Markov partition $(\cP_n^\gen)_{n\in\ZZ}$ is generating.
\end{proof}

As in Section~\ref{subsec:induced}, we obtain the following version of Theorem~\ref{thm:MarkovFine} for multidimensional continued fraction algorithms. 

\begin{theorem} \label{t:MpartitionMCF}
Let $(X,F,A,\mu)$ be a positive multidimensional continued fraction algorithm with natural extension $(\hX,\hF,A,\hmu)$. Let $\bphi$ be a substitutive realization of this algorithm. Suppose that, for $(\bx,\by)\in \hX$, the sequence $\bsigma = \bphi(\bx,\by) \in \cS_d^{\ZZ}$ satisfies the tiling condition with compact Rauzy fractals.
Then the sequence $(\cP_n^\gen)_{n\in\ZZ}$ defined in \eqref{eq:markovQ} forms a nonstationary Markov partition for the mapping family $(\TT,f_{\bsigma})$.  
If moreover $\bsigma$ is two-sided primitive, then this nonstationary Markov partition is generating.
\end{theorem}

\subsection{Nonstationary edge shifts} \label{subsec:nsftSub}
Let $\bsigma \in \cS_d^{\ZZ}$ be as in Theorem~\ref{thm:MarkovFine}.
In the present section, we show that the refined Markov partitions~$\cP_n^\gen$ from \eqref{eq:markovQ} can be used to construct a measurable conjugacy between a certain nonstationary edge shift and the mapping family associated to~$\bsigma$. This measurable conjugacy is even one-to-one outside a ``small'' set. This result shows as well that the matrices of the mapping family $(\TT,f_{\bsigma})$ are the transposed inverses of  the transition matrices of this nonstationary edge shift, which is a version of a theorem of Adler~\cite[Theorem~8.4]{A98} and Manning~\cite[Theorem~2.1]{Manning}. 

As in Section~\ref{subsec:nsft}, we will define the required nonstationary edge shift by using 
a Bratteli diagram. We therefore recall how to associate a Bratteli diagram (or Markov compactum) to a sequence of substitutions $\bsigma \in \cS_d^{\ZZ}$ in a natural way. 

\begin{definition}[Brattteli diagram associated to a sequence of substitutions; {see e.g.~\cite{BSTY}}]
\indx{Bratteli diagram}
Let $\bsigma=(\sigma_n)_{n\in\ZZ} \in \cS_d^{\ZZ}$. For $n\in\ZZ$, let $V_n = \cA$ and 
\[
E_n = \big\{ (a,p,b) \in \cA \times \cA^* \times \cA \,:\, pa \preceq \sigma_n(b) \big\}. \notx{En}{$E_n$}{set of edges of a Bratteli diagram or nonstationary edge shift, prefixes of a substitution}
\]
Put $V = \coprod_{n\in\ZZ} V_n$ and $E = \coprod_{n\in\ZZ} E_n$, and define the source and range maps by $s(a,p,b) = a$ and $r(a,p,b) = b$, respectively. Then $\mathscr{B}=(V,E)$ is the \emph{Bratteli diagram of the sequence~$\bsigma$}. 
\end{definition}

It is immediate from the definition that the transition matrices~$A_n$ of the Bratteli diagram of $(\sigma_n)_{n\in\ZZ}$ are equal to $\tr{\!M}_{\sigma_n}$. Indeed, there are $|\sigma_n(b)|_a$ edges with source vertex $a \in \cA$ and range vertex $b \in \cA$ in~$\mathscr{B}$. 

Let $\mathscr{B}$ be the Bratteli diagram of a sequence $\bsigma=(\sigma_n)_{n\in\ZZ}$. We want to give an appropriate symbolic representation for the mapping family $(\TT,f_{\bsigma})$ in terms of the nonstationary edge shift $(X_{\mathscr{B}},\Sigma)$; see Definition~\ref{def:nfst}. 
In particular, we establish the following theorem, generalizing \cite[Proposition~3.6]{AF:05}.
Recall that the Rauzy boxes $\hR_n(a,p,b)$ are defined in~\eqref{e:Re}.

\begin{theorem}\label{th:symbMod}
Let $\bsigma \in \cS_d^{\ZZ}$, with $d \ge 2$, be a sequence of unimodular substitutions that is two-sided primitive and satisfies the tiling condition with compact Rauzy fractals. Let $\mathscr{B} = (V,E)$ be the Bratteli diagram associated to $\bsigma$. Then the maps
\[ 
\psi_n:\, X_{\mathscr{B}}^{(n)} \to \TT_n, \; (e_k)_{k\in\ZZ} \mapsto \bx \quad \mbox{such that}\ 
\bigcap_{\ell\in\NN}\overline{\bigcap_{|k|\le\ell} f_{\bsigma}^{n-k}(\rint(\hR_k(e_k)))} = \{\bx\},
\]
with $\rint(\hR_k(e_k))$ viewed as a subset of~$\TT_k$, are continuous and one-to-one except on the set of the pullbacks of the boundary of the Rauzy boxes~$\hR_n(e)$, $e \in E_n$. Moreover, the diagram 
\[  
\xymatrix{  
\cdots \ar[r]^{\Sigma}& X_{\mathscr{B}}^{(-1)} \ar[d]_{\psi_{-1}}\ar[r]^{\Sigma} &  X_{\mathscr{B}}^{(0)} \ar[d]_{\psi_0}\ar[r]^{\Sigma} & X_{\mathscr{B}}^{(1)} \ar[d]_{\psi_1}\ar[r]^{\Sigma} &  X_{\mathscr{B}}^{(2)}\ar[d]_{\psi_2}  \ar[r]^{\Sigma}&\cdots  \\
\cdots \ar[r]^{f_{-2}}& \TT_{-1} \ar[r]^{f_{-1}} &  \TT_0 \ar[r]^{f_0} & \TT_1 \ar[r]^{f_1} &  \TT_2 \ar[r]^{f_{2}}   &\cdots }  
\]
is commutative. Therefore, $\varphi: X_{\mathscr{B}} \to \TT$, defined by~$\psi_n$ on the component~$X_{\mathscr{B}}^{(n)}$, is a measurable conjugacy between  $(X_{\mathscr{B}},\Sigma)$ and the mapping family $(\TT,f_{\bsigma})$, where $\varphi$ fails to be one-to-one only on the set of the pullbacks of the boundary of the Rauzy boxes~$\hR_n(e)$, $e \in E_n$, $n \in \ZZ$.
\end{theorem}

\begin{proof}
Recall that, because $\rint(\hR_k(e_k) \in \TT_k$, we have   
\begin{equation}\label{eq:f_edgeshiftproof}
f_{\bsigma}^{n-k}(\rint(\hR_k(e_k))) = \begin{cases}
M_{\sigma_{[n,k)}} \hR_k(e_k) & \text{if}\ k \ge n, \\[.5ex]
M_{\sigma_{[k,n)}}^{-1} \hR_k(e_k) & \text{if}\ k < n. \end{cases}
\end{equation}
Then $\bigcap_{\ell\in\NN} \overline{\bigcap_{|k|\le\ell} f_{\bsigma}^{n-k}(\rint(\hR_k(e_k)))} \neq \emptyset$ by Lemma~\ref{lem:refinedIntersections}, and, since $(\cP_n^\gen)_{n\in\ZZ}$ with $\cP_n^\gen$, $n\in\ZZ$, as in \eqref{eq:markovQ} is a generating Markov partition for $(\TT,f_{\bsigma})$ by Theorem~\ref{thm:MarkovFine}, this intersection contains exactly one element of~$\TT_n$ (which is given by \eqref{e:Rpoint}). Thus $\pi_n: X_{\mathscr{B}}^{(n)} \to \TT_n$ is a well-defined mapping for each $n \in \ZZ$. It can be proved in exactly the same way as in the proof of \cite[Theorem~6.5]{A98} that $\pi_n$ is one-to-one, except for the case when $(e_k)_{k\in\ZZ}$ belongs to the preimage of the boundary of some $\hR_k(e)$ for $e \in E_k$ and $k \in \ZZ$. In this case, $\psi_n$ is instead finite-to-one.

Therefore, the disjoint union~$X_{\mathscr{B}}$ projects essentially one-to-one to~$\TT$ via~$\varphi$. It is easy to check that the shift $\Sigma: X_{\mathscr{B}}^{(n)} \to X_{\mathscr{B}}^{(n+1)}$ projects to the map $f_n: \TT_n \to \TT_{n+1}$, and the diagram commutes.
\end{proof}

\begin{definition}\label{symbolic model} \indx{mapping family!symbolic model}
Under the conditions of Theorem~\ref{th:symbMod}, we say that $(X_{\mathscr{B}},\Sigma)$ is a \emph{symbolic model} of  the mapping family $(\TT,f_{\bsigma})$. If $\bM$ is the sequence of incidence matrices of~$\bsigma$, then we also say that $(X_{\mathscr{B}},\Sigma)$ is a \emph{symbolic model}  of the mapping family $(\TT,f_{\bM})$.
\end{definition}

\section[Metric results on Markov partitions]{Metric results on nonstationary Markov partitions} \label{sec:metricMP}
In this section, we present  metric results  on  Markov partitions. The theory for mapping families is provided in Section~\ref{subsec:metricMP}, multidimensional continued fraction algorithms are treated in Section~\ref{sec:MPBrun}.
We will apply the results to the Brun continued fraction algorithm in Section~\ref{subsec:BrunS}.

\subsection{Metric results on Markov partitions for mapping families}\label{subsec:metricMP}
Let $(D,\Sigma,\nu)$ be a dynamical system with $D \subset \cS_d^{\ZZ}$ [$D \subset \cM_d^{\ZZ}$] and a shift invariant Borel measure~$\nu$.
We want to prove that $\nu$-a.e.\ element of $D$ admits a generating Markov partition whose atoms are given by Rauzy boxes. 
Our goal is to apply Theorem~\ref{thm:MarkovFine} to almost all elements of the shift. The crucial condition in this theorem is the tiling condition from Definition~\ref{def:tilingcond}, sufficient conditons of which are provided in Section~\ref{sec:suff-cond-tilings}.

The first result of this section is a version of Theorem~\nameref{t:E} (from the introduction) for sequences of substitutions. It allows to associate a nonstationary Markov partition to almost every element of a shift of substitutive sequences. Recall the definition of Pisot condition from Definition~\ref{def:gengenPisot2} and the definition of pure discrete spectrum in Section~\ref{sec:spectr-prop-sub}.

\begin{theorem} \label{theo:metricmarkov}
Let $(D, \Sigma, \nu)$, with $D \subset \cS_d^{\ZZ}$, $d \ge 2$, be an ergodic dynamical system that satisfies the Pisot condition.  Assume that there is a periodic Pisot sequence having two-sided positive range and pure discrete spectrum.   

Then, for $\nu$-almost every $\bsigma \in D$, the linear eventually Anosov mapping family $(\TT,f_{\bsigma})$ associated to $\bsigma$  admits a generating nonstationary Markov partition, whose atoms are explicitly given by Rauzy boxes. This Markov partition provides a symbolic model for $(\TT,f_{\bsigma})$ as a nonstationary edge shift. 
\end{theorem}

\begin{proof}[Proof of Theorem~\ref{theo:metricmarkov}]
By Theorem~\ref{prop:combcond2}, a.e.\ $\bsigma \in D$ satisfies the tiling condition with compact Rauzy fractals and is two-sided primitive. Therefore, all conditions of Theorem~\ref{thm:MarkovFine} are satisfied, and the result follows.
\end{proof}

We mention that the Pisot condition is the only crucial condition in this theorem. All the other conditions are easily checked.

While in Theorem~\ref{theo:metricmarkov} we started with a shift of sequences of substitutions, in the following theorem, we start with sequences of $d{\times}d$ matrices~$\bM$ and use a substitution assignment $\varrho: \cM_d \to \cS_d$ (see Definition~\ref{d:realization}) to associate a sequence of substitutions~$\varrho(\bM)$ and, by means of these substitutions, a nonstationary Markov partition, with it. This theorem corresponds to Theorem~\nameref{t:E}  from the introduction.

\begin{theorem}\label{theo:metricmarkovM}
Let $(D, \Sigma, \nu)$, with $D \subset \cM_d^{\ZZ} $, $d \ge 2$, be an ergodic dynamical system that satisfies the Pisot condition. Assume that there is a substitution assignment\footnote{It is sufficient to define the substitution assignment on matrices occurring in~$D$.} $\varrho$ on $\cM_d$ for which there is a periodic Pisot sequence in $\varrho(D)$ having two-sided positive range in $(\varrho(D),\Sigma,\varrho_*\nu)$ and pure discrete spectrum. 

Then, for $\nu$-almost every $\bM\in D$, the linear eventually Anosov mapping family $(\TT,f_{\bM})$ associated to~$\bM$  admits a generating nonstationary Markov partition, whose atoms are explicitly given by Rauzy boxes of $\varrho(\bM)$. This Markov partition also provides a symbolic model for $(\TT,f_{\bM})$ as a nonstationary edge shift. 
\end{theorem} 

\begin{proof}
Theorem~\ref{theo:metricmarkovM} follows immediately from Theorem~\ref{theo:metricmarkov} because the system $(\varrho(D),\Sigma,\varrho_*\nu)$ satisfies all conditions of Theorem~\ref{theo:metricmarkov}. 
\end{proof}

The following theorem shows that we can accelerate each shift in a way that the pure discrete spectrum assumption is no longer needed for the acceleration. In its statement, for a given $\bM = (M_n)_{n\in\ZZ} \in \cM_d$ and a given $p \in \NN$, we use the notation $\bM_p= (M_{[pn,p(n+1))})_{n\in\ZZ}$ for the $p$-fold blocking of the sequence~$\bM$. 

\begin{theorem}\label{theo:metricmarkovAcc}
Let $(D, \Sigma, \nu)$, with $D \subset \cM_d^{\ZZ}$, $d \ge 2$, be an ergodic dynamical system that satisfies the Pisot condition and has a periodic Pisot sequence with two-sided positive range.
Then there exists a positive integer~$p$ and a substitution assignment~$\varrho$ on~$\cM_d$ such that, for $\nu$-almost every $\bM \in D$, the linear eventually Anosov mapping family $(\TT,f_{\bM_p})$ associated to~$\bM_p$  admits a generating nonstationary Markov partition, whose atoms are explicitly given by Rauzy boxes of $\varrho(\bM_p)$. This Markov partition also provides a symbolic model for  $(\TT,f_{\bM_p})$ as a nonstationary edge shift. 
\end{theorem}

\begin{proof}
By Theorem~\ref{theo:metrictilingaccel}, all conditions of Theorem~\ref{thm:MarkovFine} are satisfied.
\end{proof}

\subsection[Markov partitions for continued fraction algorithms]{Nonstationary Markov partitions for continued fraction algorithms}  \label{sec:MPBrun}
In this section, we use the results from Section~\ref{subsec:metricMP} to establish results on nonstationary Markov partitions for multidimensional continued fraction algorithms.  
Recall again some facts about multidimensional continued fraction algorithms stated earlier.
Let $(X,F,A,\nu)$ be such an algorithm with natural extension  $(\hX,\hF,\hA,\hnu)$. In Definition~\ref{d:realization}, we defined faithful substitutive realizations~$\bphi$ of $(\hX,\hF,\hA,\hnu)$.
If the algorithm $(\hX, \hF, \hA,\hnu)$ converges a.e.\ weakly in the future and in the past, then we know from \eqref{eq:phi} that the $\cS$-adic shift $(\bphi( \hX),\Sigma,\bphi_*\nu)$ is measurably conjugate to $(\hX, \hF,\hA,\hnu)$.
By Lemma~\ref{lem:MCF_eigenS}, at each point $(\bx,\by)$ of convergence, the generalized right and left eigenvectors $\bu$ and $\bv$ of $\bphi(\bx,\by)$ are representatives of the projective vectors $\bx$ and $\by$ in $\RR^d$, respectively. According to Section~\ref{subsec:realization}, via $\bphi$ we can associate $\cS$-adic mapping families $(\TT,f_{\bphi(\bx,\by)})$ with $(\hX,\hF,\hA,\hnu)$. Our aim is to prove that for almost all $(\bx,\by)\in \hX$ the mapping family $(\TT,f_{\bphi(\bx,\by)})$ admits a generating Markov partition. According to Theorem~\ref{th:symbMod}, this Markov partition provides a symbolic model for  $(\TT,f_{\bphi(\bx,\by)})$. We now state this precisely in the following statement, whose accelerated version, namely  Theorem~\ref{theo:FCAcc},  corresponds to the    detailed formulation for Theorem~\ref{t:F}.  

Note that, when the conditions of Remark~\ref{rem:nu>0} hold, then the set of parameters $(\bx,\by)$ in~$\hX$ has positive Lebesgue measure.
For the following statement, recall the Pisot condition from Definition~\ref{def:MCF_Pisot} and the notation $\ll$ from \eqref{e:abscontinuity}.

\begin{theorem} \label{theo:FC}
Let $(X,F,A,\nu)$ be a positive multidimensional continued fraction algorithm satisfying the Pisot condition and $\nu \circ F \ll \nu$. Let $\bphi$ be a faithful substitutive realization of a natural extension $(\hX,\hF,\hA,\hnu)$ of $(X,F,A,\nu)$. Assume that there is a periodic Pisot point $(\bx_0,\by_0) \in \hX$ such that $\bx_0$ has positive range in $(X,F,A,\nu)$ and $\bphi(\bx_0,\by_0)$ has pure discrete spectrum.
Then, for $\hnu$-almost all $(\bx,\by) \in \hX$,  the mapping family $(\TT,f_{\bsigma})$ associated to $\bsigma=\bphi(\bx,\by)$ is eventually Anosov and admits a generating nonstationary Markov partition, whose atoms are explicitly given by Rauzy boxes. This Markov partition provides a symbolic model for $(\TT,f_{\bsigma})$ as a nonstationary edge shift. 
\end{theorem}

\begin{proof}
Again this follows from Theorem~\ref{theo:metricmarkov} in the same way as Theorem~\ref{theo:pureMCF} follows from Theorem~\ref{prop:combcond2}.
Indeed, as we saw in that proof, the Pisot condition for $(X,F,A,\nu)$ implies that $(\bphi(\hX),\Sigma,\bphi_*\hnu)$ satisfies the Pisot condition, and positive range of the periodic Pisot point $(\bx_0,\by_0) \in \hX$ in $(X,F,A,\nu)$ implies that $\bphi(\bx_0,\by_0)$ is a periodic Pisot sequence with two-sided positive range in $(\bphi(\hX),\Sigma,\bphi_*\hnu)$.
Thus the dynamical system $(\bphi(\hX),\Sigma,\bphi_*\hnu)$ satisfies all the conditions of Theorem~\ref{theo:metricmarkov}, which proves the theorem.
\end{proof}
 
Again, we can omit the pure discrete spectrum condition for a Pisot point when we accelerate the algorithm appropriately. 
 
\begin{theorem} \label{theo:FCAcc}
Let $(X,F,A,\nu)$ be a positive multidimensional continued fraction algorithm satisfying the Pisot condition, $\nu \circ F \ll \nu$. Let $(\hX,\hF,\hA,\hnu)$ be a natural extension of $(X,F,A,\nu)$. Assume that there is a periodic Pisot point $\bx_0 \in X$ with $\bx_0$ having positive range in $(X,F,A,\nu)$. Then there exist a positive integer~$p$ and a (faithful) substitutive realization~$\bphi$ of $(X,F^p,A,\nu)$,  such that the cocycle of the acceleration $(\hX,\hF^p,\hnu)$ generates sequences~$\bM$ of matrices whose eventually Anosov mapping family $(\TT,f_{\bM})$ $\hnu$-almost always admits a generating nonstationary Markov partition, with atoms that are explicitly given by  Rauzy boxes. This Markov partition also provides a symbolic model for $(\TT,f_{\bM})$ as a nonstationary edge shift. 
\end{theorem}

\begin{proof}
This follows from Theorem~\ref{theo:metricmarkovM} in the same way as Theorem~\ref{theo:FC} follows from Theorem~\ref{theo:metricmarkov}.
\end{proof}

\section{Nonstationary Markov partitions and induced rotations for the Brun algorithm} \label{subsec:BrunS}
In this section, we want to apply Theorems~\ref{theo:pureMCF},~\ref{theo:metricRotMCF}, and~\ref{theo:FC}  to the 2- and 3-dimensional cases of the versions of the Brun continued fraction algorithm studied in Section~\ref{sec:Brun}. We start with the unordered case. The substitutive realization~$\bphi_{\rU}$ for this algorithm was set up in Example~\ref{ex:UnBrunSubsRe}.
Recall that, for $\bx \in \PP^{d-1}$, the rotation~$\fr_{\bx}$ is defined on $\bone^\bot / (\ZZ^d \cap \bone^\bot)$ by $\fr_{\bx}(\bz) = \bz + \be_d - \chi(\bx)$; see \eqref{e:rtilde}.
Moreover, under the tiling condition, the Rauzy fractal~$\cR_0^{\bone}$ is a fundamental domain of $\bone^\bot / (\ZZ^d \cap \bone^\bot)$.
 
\begin{corollary} \label{cor:FCB1}
For $d\in\{3,4\}$, let $(X_{\rU},F_{\rU},A_{\rU},\nu_{\rU})$ be the unordered Brun algorithm. Let $\bphi_{\rU}$ be the faithful substitutive realization of this algorithm that is defined by the unordered Brun substitutions
\[
\sigma_{\rU,ij} : 
\begin{cases} j \mapsto ij, \\ k \mapsto k & (k\in \cA\setminus\{j\}) \end{cases} \qquad (i,j \in \cA,\, i \neq j).
\]
Then, for $\hnu_{\rU}$-almost all $(\bx,\by) \in \hX_{\rU}$, the following assertions hold.
\begin{itemize}
\item[(i)] The $\cS$-adic dynamical system $(X_{\bphi_{\rU}(\bx,\by)},\Sigma)$ has pure discrete spectrum.
\item[(ii)]  The mapping family $(\TT,f_{\bsigma})$ associated to $\bsigma_{\rU}=(\sigma_n)_{n\in\ZZ}=\bphi_{\rU}(\bx,\by)$ is eventually Anosov and admits a generating nonstationary Markov partition, whose atoms are explicitly given by Rauzy boxes. This Markov partition provides a symbolic model for $(\TT,f_{\bsigma})$ as a nonstationary edge shift. 
\item[(iii)] The induced map of the rotation $\fr_{\bx}$ on the subset $\pi_{\chi(\bx),\bone} \tr{\!A}_{\rU}^{(n)}(\bx) \cR_n$ of~$\cR_0^{\bone}$ equals (after renormalization) the rotation~$\fr_{F_{\rU}^n(\bx)}$ for each $n \in \NN$.
\end{itemize}
\end{corollary}

\begin{proof}
We need to make sure that the conditions of Theorems~\ref{theo:pureMCF}, \ref{theo:metricRotMCF}, and~\ref{theo:FC} are satisfied for the algorithm $(X_{\rU},F_{\rU},A_{\rU},\nu_{\rU})$ in the 2- and 3-dimensional case. The generic Pisot condition holds by Proposition~\ref{prop:Brun23Pisot}. The fact that $\nu_{\rU} \circ F_{\rU} \ll \nu_{\rU}$ follows from the fact that, according to Lemma~\ref{lem:unorderedinvariant}, $\nu_{\rU}$~is defined in terms of a measure with full support that is absolutely continuous w.r.t.\ the Lebesgue measure.  It remains to find a periodic Pisot point $(\bx_0,\by_0) \in \hX_{\rU}$ such that $\bphi_{\rU}(\bx_0,\by_0)$ has pure discrete spectrum. 

For $d = 3$, we consider
\begin{equation} \label{e:tau3}
\tau = \sigma_{\rU,12}\circ\sigma_{\rU,23}\circ\sigma_{\rU,31}:\, \begin{cases}1 \mapsto 1321, \\ 2 \mapsto 21, \\ 3 \mapsto 321.\end{cases}
\end{equation}
One easily checks that $\tau$ is a Pisot substitution and, by standard algorithms like for instance the balanced pair algorithm from~\cite{Livshits:87,Livshits:92}, one verifies that the substitutive dynamical system generated by~$\tau$ has pure discrete spectrum.
Let $\btau \in \cS_3^{\ZZ}$ be the periodic sequence with cycle $(\sigma_{\rU,12}, \sigma_{\rU,23}, \sigma_{\rU,31})$.
Then $\btau$ satisfies the admissibility property~\eqref{eq:unBrunMarkovSub}.
Let $(\bx_0,\by_0) \in \hX_{\rU}$ such hat $\chi(\bx_0)$ and $\chi(\by_0)$ are the generalized right and left eigenvectors of $\btau$.
Then $\bphi_{\rU}(\bx_0,\by_0) = \btau$, hence $(\bx_0,\by_0)$ is a periodic Pisot point, which has positive range because of the Markov property~\eqref{eq:unBrunMarkov}, and $\bphi_{\rU}(\bx_0,\by_0)$ has pure discrete spectrum.

For $d = 4$, we know from \cite[Section~6.5]{BST:23}\footnote{In comparison with \cite{BST:23}, we revert the words in the images of the substitutions. This does not change the arguments.} that for 
\begin{equation} \label{e:tau4}
\tau = \sigma_{\rU,12} \circ \sigma_{\rU,23} \circ \sigma_{\rU,34} \circ \sigma_{\rU,41}:\ \begin{cases}1 \mapsto 14321, \\ 2 \mapsto 21, \\ 3 \mapsto 321, \\ 4 \mapsto 4321,\end{cases}
\end{equation}
we can create a Pisot point with positive range and pure discrete spectrum in the same way. Thus all the conditions of Theorems~\ref{theo:pureMCF},~\ref{theo:metricRotMCF}, and~\ref{theo:FC} are satisfied for $d \in \{3,4\}$ and the corollary is established.
\end{proof}

An analogous result is obtained for the ordered version of Brun's algorithm. Namely, we get the following corollary. The substitutive realization~$\bphi_{\rB}$ for this algorithm was set up in Example~\ref{ex:OrdBrunSubsRe}. Moreover, the conditions of Remark~\ref{rem:nu>0} hold, i.e., the set~$\hX_{\rB}$ has positive Lebesgue measure.

\begin{corollary} \label{cor:FCB2}
For $d \in \{3,4\}$, let $(X_{\rB},F_{\rB},A_{\rB},\nu_{\rB})$ be the ordered Brun continued fraction algorithm. Let $\bphi_{\rB}$ be the faithful substitutive realization of this algorithm defined by the ordered Brun substitutions~$\sigma_{\rB,k}$, $k \in \cA$, as in~\eqref{eq:brunsubsallg}. 

Then, for $\hnu_{\rB}$-almost all $(\bx,\by) \in \hX_{\rB}$, the following assertions hold.
\begin{itemize}
\item[(i)] The $\cS$-adic dynamical system $(X_{\bphi_{\rB}(\bx,\by)},\Sigma)$ has pure discrete spectrum.
\item[(ii)]  The mapping family $(\TT,f_{\bsigma})$ associated to $\bsigma_{\rB}=(\sigma_n)_{n\in\ZZ}=\bphi_{\rB}(\bx,\by)$ is eventually Anosov and admits a generating nonstationary Markov partition, whose atoms are explicitly given by Rauzy boxes. This Markov partition provides a symbolic model for $(\TT,f_{\bsigma})$ as a nonstationary edge shift. 
\item[(iii)] The induced map of the rotation $\fr_{\bx}$ on the subset $\pi_{\chi(\bx),\bone} \tr{\!A}_{\rB}^{(n)}(\bx) \cR_n$ of~$\cR_0^{\bone}$ equals (after renormalization) the rotation~$\fr_{F_{\rB}^n(\bx)}$ for each $n \in \NN$.
\end{itemize}
\end{corollary}

\begin{proof}
Again, we verify the conditions of Theorems~\ref{theo:pureMCF}, \ref{theo:metricRotMCF}, and~\ref{theo:FC}. As in Corollary~\ref{cor:FCB1} we see from Proposition~\ref{prop:Brun23Pisot} and Lemma~\ref{lem:orderedinvariant} that the generic Pisot condition holds and that $\nu_{\rB}\circ F_{\rB} \ll \nu_{\rB}$.
Here, the constant sequence $(\sigma_{\rB,1})$  gives a periodic Pisot point for $d \in \{3,4\}$; note that $\sigma_{\rB,1}^d$ is the substitution~$\tau$ from \eqref{e:tau3} and \eqref{e:tau4} for $d=3$ and $d=4$ respectively.
\end{proof}

From Corollary~\ref{cor:FCB2}, we derive the following exact formulation of Corollary~\nameref{c:G}.

\begin{corollary}\label{cor:foliMarkov}
Let $d \in \{3,4\}$. For Lebesgue a.e.\ pair of vectors $(\bu,\bv) \in \RR_{\ge0}^d \times \RR_{\ge0}^d$, we can provide an eventually Anosov mapping family $$(\TT,f_{\bM}) = ((\TT_n)_{n\in\ZZ}, (M_n)_{n\in\ZZ})$$ such that for each $n \in \ZZ$, in each $d$-dimensional torus~$\TT_n$ the mappings $M_m^{-1} \cdots M_n^{-1}$ ($m<n$) \emph{expand}~$\bu_n$ and \emph{contract}~$\bv_n^\perp$, where
\[
\bu_n = \begin{cases}M_{[0,n)}^{-1}\, \bu & \text{for } m \ge 0, \\ M_{[n,0)}\, \bu & \text{for } m < 0,\end{cases} \qquad \bv_n = \begin{cases}\tr{\!M}_{[0,n)}\, \bv & \text{for } m \ge 0, \\ \tr{\!M}_{[n,0)}^{-1}\, \bv & \text{for } m < 0.
\end{cases}
\]
The mapping family can be constructed in a way that it admits a generating Markov partition and, hence, a symbolic model as a nonstationary edge shift.
\end{corollary}

\begin{proof}
We saw in Section~\ref{subsec:BrunOrd} that the domain~$\hX_{\rB}$ of the natural extension of the ordered Brun algorithm equals $\hX_{\rB} = X_{\rB} \times X_{\rB}^*$ with $X_{\rB}$ and~$X_{\rB}^*$ given in \eqref{eq:BOX} and~\eqref{eq:BOXstar}.
Clearly, up to a set of zero measure (here we take the Lebesgue measure on~$\Delta^{d-1}$ transferred to~$\PP_{\ge0}^{d-1}$ via the chart~$\chi$),
\[
\bigcup_{\zeta,\eta\in\mathfrak{S}_d} (\zeta X_{\rB} \times \eta X_{\rB}^*) = \PP_{\ge0}^{d-1} \times  \PP_{\ge0}^{d-1}.
\]
For $\zeta \in \mathfrak{S}_d$, we consider the algorithm $(\zeta X_{\rB}, \zeta F_{\rB} \zeta^{-1}, \zeta A_{\rB} \zeta^{-1}, \zeta_*\nu_{\rB})$. 
For each $\eta \in \mathfrak{S}_d$, the map
\begin{equation*}
\hF_{\zeta,\eta}:\, \zeta X_{\rB} \times \eta X_{\rB}^* \to  \zeta X_{\rB} \times \eta X_{\rB}^*, \quad (\bx, \by) \mapsto (\zeta \tr{\!A}(\bx)^{-1} \zeta^{-1} \bx, \eta A(\bx)\eta^{-1}\by),
\end{equation*}
defines a natural extension $(\zeta X_{\rB} \times \eta X_{\rB}^*,\hF_{\zeta,\eta}, \hA_{\zeta,\eta},\hnu_{\zeta,\eta})$, and Corollary~\ref{cor:FCB2} remains true for the algorithms $(\zeta X_{\rB}, \zeta F_{\rB} \zeta^{-1}, \zeta A_{\rB} \zeta^{-1}, \zeta_*\nu_{\rB})$, $\zeta \in \mathfrak{S}_d$, with each of these natural extensions. Thus we get eventually Anosov mapping families for a.e.\ $(\bx,\by) \in \PP_{\ge0}^{d-1} \times \PP_{\ge0}^{d-1}$ (w.r.t.\ the appropriate measures~$\hnu_{\zeta,\eta}$ taken in each of the sets $\zeta X_{\rB} \times \eta X_{\rB}^*$). Because the generalized right and left eigenvectors of $\bphi(\bx,\by)$ are affine representatives of $\bx$ and~$\by$ respectively, the corollary is true for $\hmu_{\rB}$-a.e.\ $(\bu,\bv) \in \RR_{\ge0}^d  \times  \RR_{\ge0}^d$. The result follows because $\hmu_{\rB}$ is equivalent to the Lebesgue measure on $\RR_{\ge0}^d \times \RR_{\ge0}^d$.
\end{proof}

It remains to treat the modified Jacobi--Perron algorithm; for a substitutive realization~$\bphi_{\rM}$ of this algorithm we refer to Example~\ref{ex:MultBrunSubsRe}. The according result reads as follows.

\begin{corollary} \label{cor:FCB3}
For $d \in \{3,4\}$, let $(X_{\rM},F_{\rM},A_{\rM},\nu_{\rM})$ be the modified Jacobi--Perron continued fraction algorithm. Let $\bphi_{\rM}$ be the faithful substitutive realization of this algorithm that is defined by the multiplicative Brun substitutions~$\sigma_{\rM,k,m}$, $1 \le k < d$, $m \ge 1$, defined in \eqref{eq:brunsubsallgMult}. 

Then, for $\hnu_{\rM}$-almost all $(\bx,\by) \in \hX_{\rM}$, the following assertions hold.
\begin{itemize}
\item[(i)] The $\cS$-adic dynamical system $(X_{\bphi_{\rM}(\bx,\by)},\Sigma)$ has pure discrete spectrum.
\item[(ii)]  The mapping family $(\TT,f_{\bsigma})$ associated to $\bsigma_{\rM}=(\sigma_n)_{n\in\ZZ}=\bphi_{\rM}(\bx,\by)$ is eventually Anosov and admits a generating nonstationary Markov partition, whose atoms are explicitly given by Rauzy boxes. This Markov partition provides a symbolic model for $(\TT,f_{\bsigma})$ as a nonstationary edge shift. 
\item[(iii)] The induced map of the rotation $\fr_{\bx}$ on the subset $\pi_{\chi(\bx),\bone} \tr{\!A}_{\rM}^{(n)}(\bx) \cR_n$ of~$\cR_0^{\bone}$ equals (after renormalization) the rotation~$\fr_{F_{\rM}^n(\bx)}$ for each $n \in \NN$.
\end{itemize}
\end{corollary}

\begin{proof}
Again, we have to assure the conditions of Theorems~\ref{theo:pureMCF}, \ref{theo:metricRotMCF}, and~\ref{theo:FC}. As in Corollary~\ref{cor:FCB1}, we see from Proposition~\ref{prop:Brun23Pisot} and Lemma~\ref{lem:multinvariant} that the generic Pisot condition holds and that $\nu_{\rM}\circ F_{\rM} \ll \nu_{\rM}$. 
As in Corollary~\ref{cor:FCB2}, the constant sequence $(\sigma_{\mathrm{M,1,1}})$ gives a periodic Pisot point for $d \in \{3,4\}$; note that $\sigma_{\mathrm{M,1,1}} = \sigma_{\mathrm{B,1}}$.
\end{proof}

\chapter{Concluding remarks}\label{sec:conclude}

We end this work by  providing open problems and perspectives for further research. Here, problems around the Pisot condition (see Definition~\ref{def:gengenPisot} for its generic form) and the tiling condition (see Definition~\ref{def:tilingcond}) play an important role. We also propose to view multidimensional continued fraction algorithms as cross sections of the Weyl chamber flow (or a variant of it). This is inspired by the interpretation of the classical continued fraction algorithm as  a cross-section of the geodesic flow on the unit tangent bundle of the modular surface. Finally, we dwell on different kinds of dynamics related to our theory and their interplay.

\section{Open problems}\label{sec:open}

In all our results on nonstationary Markov partitions in Section~\ref{sec:markov}, we assume the tiling condition from Definition~\ref{def:tilingcond}, which is closely related to the notion of pure discrete spectrum, as recalled in Section~\ref{subsec:dspectrum}.
In \cite[Conjecture~3.5]{BST:19}, we have formulated an $\cS$-adic Pisot conjecture stating that this tiling condition holds under weak assumptions. 
We slightly modify this conjecture using the local Pisot condition (see Definition~\ref{def:genPisot}). 

\begin{conjecture} \label{c:SadicPisotconj}
Let $\bsigma \in \cS^{\ZZ}$ be a primitive sequence of unimodular substitutions satisfying the local Pisot condition and $\lim_{n\to\infty} \frac{1}{n} \log \lVert M_n\rVert = 0$.
Then the tiling condition holds, and $(X_{\bsigma}, \Sigma, \mu)$ has pure discrete spectrum. 
\end{conjecture}

This conjecture generalizes the classical Pisot conjecture for unimodular irreducible Pisot substitutions (see \cite{AkiBBLS})  because a stationary directive sequence formed by a unimodular irrreducible Pisot substitution satisfies the conditions of this conjecture. 
Since all infinite factor balanced shifts that we are aware of are defined by sequences of substitutions satisfying the conditions of Conjecture~\ref{c:SadicPisotconj}, we even ask the following question.

\begin{openproblem}
Let $(X,\Sigma,\mu)$ be a factor balanced subshift.
Under which  conditions does  $(X,\Sigma,\mu)$ have pure discrete spectrum? 
\end{openproblem}
 
Like in the substitutive setting, also in the primitive  $\cS$-adic setting, the tiling condition implies pure discrete measure-theoretic spectrum by Theorem~\ref{t:tilingpds}.
However, contrary to the substitutive case (see \cite{Barge-Kwapisz:06} and  \cite[Lemma~4.9]{BST:23}), it is not clear if the converse is true. 

\begin{conjecture}
Let $\bsigma \in \cS_d^{\ZZ}$, with $d \ge 2$, be a primitive sequence of unimodular substitutions that admits generalized right and left eigenvectors. 
If $(X_{\bsigma}^{(n)},\Sigma,\mu)$ has pure discrete spectrum  for each $n \in \ZZ$, then the tiling condition holds.
\end{conjecture}

In Chapter~\ref{sec:cf}, we saw that multidimensional continued fraction algorithms give rise to important families of $\cS$-adic shifts. For our theory  to be applicable to these shifts,  it is important that they satisfy the Pisot condition, according to 
 Section~\ref{sec:MPBrun}. In \cite{BST21}, we considered this question which is intimately related to strong convergence on these algorithms. In particular, we proved that the Selmer continued fraction algorithm (see \cite{Selmer:61})\indx{Selmer algorithm}\indx{continued fraction algorithm!Selmer} satisfies the Pisot condition for dimension $d\in\{3,4\}$. Besides that, in \cite{BST21}  we gave some numerical results that indicate that the second Lyapunov exponent of many well-known multidimensional  continued fraction algorithms become positive from a certain dimension onwards. 
More precisely these heuristic results point toward the following conjecture.

\begin{conjecture}\label{conj:convfc}
We conjecture the following convergence behavior of some classical multidimensional continued fraction algorithms. 
\begin{itemize}
\item The Brun continued fraction algorithm satisfies the Pisot condition if and only if $d \le 10$.
\item The Jacobi--Perron continued fraction algorithm (see \cite{Perron:07}) satisfies the Pisot condition if and only if $d \le 10$. \indx{Jacobi--Perron!algorithm}\indx{continued fraction algorithm!Jacobi--Perron}
\item The Selmer continued fraction algorithm satisfies the Pisot condition if and only if $d\le 4$.
\end{itemize}
\end{conjecture}
 
While the Pisot property of the Brun and the Selmer continued fraction algorithms would suffice to apply our theory to these algorithms in higher dimensions, there is an additional problem with the Jacobi--Perron algorithm. Indeed, it might make no sense to try to establish nonstationary Markov partitions for Lebesgue almost every orbit of the Jacobi--Perron algorithm because the domain $\hDelta$ defined in Section~\ref{sec:natex}, that would be a candidate for a natural extension, seems to have ``fractal structure'' and may well have Lebesgue measure 0; see \cite[Section~8]{AL18}. In particular, the following is a long standing open question.

\begin{openproblem}
Define the candidate  $\hDelta$ for a natural extension of the Jacobi--Perron algorithm as described in Section~\ref{sec:natex}.
Does $\hDelta$ have positive Lebesgue measure?
\end{openproblem}

This is also of interest because, according to Section~\ref{sec:natex}, the natural extension can be used to calculate the invariant measure of a continued fraction algorithm. According to Lemmas~\ref{lem:unorderedinvariant} and \ref{lem:orderedinvariant}, this invariant measure is known for the Brun algorithm in any dimension. For the Selmer algorithm, the invariant measure is provided in \cite{BFK:15,BFK:19} for any dimension. However, for the Jacobi--Perron algorithm, even for $d=3$, we do not know an explicit formula for the invariant measure. The only thing we know is that its density is piecewisely analytic; see~\cite[Theorem~1]{Schweiger:90} and \cite[Th\'eor\`eme~2.24]{Broise:96}.

\begin{openproblem}
Is there an explicit formula for the density of the invariant measure of the Jacobi--Perron algorithm?
\end{openproblem}

Let us focus  now on  the topological properties of the  boundaries of  the Markov partitions  we get. The subtiles of our $\cS$-adic Rauzy fractals seem to have fractal boundaries in all our examples. When looking at the stationary case, according to \cite{Bow78}, each atom of a Markov partition has nonsmooth boundary for $d\ge 3$. We therefore formulate the following conjecture for the nonstationary case.
 
\begin{conjecture}[{cf.~\cite[Section~10]{AF:01}}]
For $d\ge 3$, let $(X,f)$ be an $\cS$-adic mapping family which satisfies the Pisot condition from Definition~\ref{def:gengenPisot2subs}. Assume that $(X,f)$ admits a nonstationary Markov partition $(\cP_n)_{n\in\ZZ}$.
Then none of the atoms of~$\cP_n$, $n \in \ZZ$, has a smooth boundary. 
\end{conjecture}

Another direction of further research concerns the topology of the atoms of a nonstationary Markov partition or, equivalently, of the underlying Rauzy fractals. In the stationary case, such properties are studied for instance in \cite{ST:09}. We formulate the following problem.

\begin{openproblem}[{cf.~\cite[Section~10]{AF:01}}]
Let $\bsigma\in \cS^\ZZ$ be a sequence of substitution than satisfies the tiling condition. Study topological properties of $\cS$-adic Rauzy fractals and their subtiles. For instance, can we give criteria for connectedness of $\cS$-adic Rauzy fractals? 
\end{openproblem}

Finally, we ask whether we can go beyond the present unimodular Pisot case. Indeed, in the present paper we only dealt with linear eventually Anosov mapping families satisfying the Pisot condition. It is natural to ask the following question.

\begin{openproblem}
Is it possible to construct nonstationary Markov partitions for linear eventually Anosov mapping families whose hyperbolic splitting is \emph{not} Pisot, i.e., if each part of the hyperbolic splitting has dimension at least~$2$?
\end{openproblem}

For examples in the stationary case, see e.g.~\cite{KV:98,AFHI:11}.
  
\section{Arithmetic and geometric coding of the Weyl chamber flow}\label{sec:flows}
\indx{Weyl chamber flow}
We aim at linking a continued fraction map  (in the sense of Definition~\ref{def:cf})  to a flow on a homogeneous space generalizing the modular surface.  
According to \cite[Section~6]{Gorodnik:07}, Ledrappier and Mozes proposed to use fractal tilings in order to get an arithmetic coding of the Weyl chamber flow, a higher-dimensional analog of the Teichm\"uller flow. 
One motivation for the present work is to start with the realization of this program by using the $\ZZ^d$-tilings induced by the Rauzy boxes~$\hR_n$ and their Markov properties.

We recall that the Weyl chamber flow is an action of the diagonal subgroup of $\SL(d,\RR)$ on the space of lattices. More precisely, the \emph{Weyl chamber flow}\indx{Weyl chamber flow} is the $\RR^{d-1}$-action on the right on the space of $d$-dimensional lattices $\SL(d,\ZZ)\backslash \SL(d,\RR)$ of the diagonal subgroup of $\SL(d,\RR)$ of matrices 
\begin{equation} 
W(\bt) = W(t_1,t_2,\dots,t_d) = \mathrm{diag}\big(e^{t_1}, e^{t_2},\dots, e^{t_d}\big)
\end{equation}
such that $t_1 {+} t_2 {+} \cdots {+} t_d = 0$. For $d=2$, the flow $W(\bt)$ is just the geodesic flow on the unit tangent bundle of the modular surface $\SL(2,\ZZ) \backslash \HH^2 \oplus \mathrm{SO}(2,\RR)\cong \SL(2,\ZZ)\backslash \SL(2,\RR)$. Here, $\HH^2$~denotes the upper half plane.

The Weyl chamber flow acts on the space $\SL(d,\ZZ)\backslash \SL(d,\RR)$ of lattices while the natural extension of a multidimensional continued fraction algorithm acts on a subset of $\PP^{d-1} {\times} \PP^{d-1}$, as detailed in Section ~\ref{sec:natex}.  
If we want to code the Weyl chamber flow, then we need to relate these objects. The idea is to realize the natural extension of a continued fraction algorithm  as a Poincar\'e section of this flow  in terms of a geometric coding with Rauzy fractals. To make the Rauzy fractals fit together with this flow, a base change has to be performed by using the vectors of the Oseledets splitting defined in terms of the cocycle of the two-sided $\cS$-adic system  associated to the continued fraction. 
The point is, for almost every lattice, to define a fundamental domain given by a Rauzy box. The flow contracts the vertical direction of the box and increases the measure of its horizontal slices; every time this measure becomes larger than~1, we can renormalize by a restacking construction, as illustrated in Figure~\ref{fig:restack1}, corresponding to a change of basis in the lattice.

\section{An interplay between three dynamical systems}\label{subsec:dynamics}
Let us come back to the dynamical setting described in the introduction.  We have reached the goal described in Chapter~\ref{sec:intro}: We found a parametrized family of rotations which is stable by induction, and for which the renormalization is realized by the matrices $(M_n)_{n \in {\mathbb Z}}$ provided by the orbit of the natural extension of  the underlying  multidimensional  continued fraction algorithm; see Theorems~\ref{t:FrotMCF} and~\ref{theo:metricRotMCF}.  Moreover, by using substitutive realizations and $\cS$-adic shifts, we exhibited families of nonstationary Markov partitions to obtain measurable conjugacy to a symbolic  model for the multidimensional  continued fraction algorithm; see Sections~\ref{sec:markov}  and~\ref{sec:metricMP}. 
This is a  rich dynamical situation, in the sense that these results  involve  the interplay of different kinds of dynamics, each of which has a symbolic, a geometric, and an arithmetic interpretation. 

The first kind of dynamics is a family of $\cS$-adic shifts. Each of these $\cS$-adic shifts $(X_{\bsigma},\Sigma,\mu)$ is defined on the orbit closure of an $\cS$-adic  sequence, which is defined in terms of a sequence of substitutions~$\bsigma$; see Definition~\ref{def:sadicshift}. We assume that the tiling condition from Definition~\ref{def:tilingcond} holds. 
This entails that the shift map~$\Sigma$ acting on the $\cS$-adic shift $(X_{\bsigma},\Sigma,\mu)$ is one-to-one and onto, of entropy~0, and measurably conjugate to a rotation on the torus~$\TT^{d-1}$, by Theorem~\ref{t:tilingpds}. By suspension, this corresponds to a ``vertical''  linear flow on a torus of dimension~$d$. This first dynamics can thus be understood geometrically as follows in  dimension $d=2$,  in the basic case of Sturmian sequences (see Example~\ref{ex:sturm}): The Poincar\'e section of the vertical flow inside a single $L$-shaped region gives the rotation $\fr_\alpha$ on the unit circle~$\TT^1$;
one has in fact a collection of dynamical systems, parametrized by the rotation vector~$\alpha$.

Secondly, we have continued fraction algorithms. These are chaotic dynamical systems which are far from being one-to-one and have positive entropy. To make a continued fraction algorithm $(X,F,A,\nu)$ invertible, we build a geometric natural extension $(\hX,\hF,\hA,\hnu)$ of $(X,F,A,\nu)$; see Section~\ref{sec:natex}.  Provided  that the tiling condition holds, using a substitutive realization~$\bphi$, such as  defined in Section~\ref{subsec:realization},  we can interpret the natural extension $(\hX,\hF,\hA,\hnu)$ as a bi-infinite shift $(\bphi(\hX),\Sigma,\bphi_*\nu)$ that acts on a collection of $\cS$-adic shifts and, hence, on a collection of dynamical systems parametrized by a rotation vector~$\balpha$, by Theorem~\ref{t:tilingpds}. The continued fraction algorithm governs a normalization dynamics that allows one to go from $\cS$-adic shifts to $\cS$-adic shifts. By taking a suspension of the natural extension, we can build a flow of positive entropy, and this is an expanding dynamical system. Moreover this flow has a nice interpretation in the case of  regular continued fractions; see \cite{AF:01}.  Indeed, this second dynamics can be expressed as a Poincar\'e section of the geodesic flow on the tangent bundle of the modular surface, i.e.,  on  the  space of  $L$-shaped regions, with this  Poincar\'e section being geometrically described in terms of a nonstationary Markov partition.
In our general Pisot setting, we want to suspend the nonstationary Markov partitions we obtained; in other words, we want to use them for defining a Poincar\'e section of a natural flow defined on some space of unimodular lattices in~$\RR^d$. Here, Rauzy boxes play the role of $L$-shaped regions. Some ideas for this general case  are discussed just above in Section~\ref{sec:flows}. 

There is a third kind of dynamics obtained by putting the first two kinds of dynamics into one large system. In fact, this third dynamics is an enrichment of the second one, in the sense  that one follows \emph{pointed} $L$-shapes (and pointed Rauzy boxes in higher dimension) through the second dynamics. In symbolical terms, we follow given sequences in a shift through the  normalization process given by the continued fraction algorithm (the second dynamics). Consider an orbit of a continued fraction expansion, or equivalently a rotation vector~$\balpha$. In the shift of all $\cS$-adic sequences corresponding to this rotation~$\fr_{\balpha}$, consider an $\cS$-adic sequence, to which we associate an infinite sequence of renormalized  sequences obtained by a suitable recoding process, also based on substitutions. This recoding constitutes the shift on the Ostrowski expansion in the Sturmian case (see \cite[Section~6]{AF:01}) and the shift on the Dumont--Thomas expansion (cf.\ \cite{Dumont-Thomas,CS01a}) in the periodic case (i.e., in the substitutive case). In other words, we enrich the Gauss map  and turn it into the Ostrowski map, which is a skew product of the Gauss map; see \cite{BerLee:23} for more on the Ostrowski skew product. This third dynamics is different from the second one, since we now  consider a fixed orbit of the  given continued fraction algorithm, but it is also  expanding and of positive entropy. It is given by a nonstationary edge shift corresponding to the digit strings of the Ostrowski expansion, and in the periodic case, we recover a usual edge shift in the sense of \cite[Definitions~2.2.5]{Lind-Marcus}. 
For the case of Sturmian sequences, the whole picture is presented in \cite{AF:01}.

\bibliographystyle{amsalpha}
\bibliography{sadic3}

\printindex[index]
\makeatletter
\renewcommand{\@idxitem}{\par\hangindent 5em}
\makeatother
\printindex[notation]
\end{document}